\documentclass[reqno,11pt]{amsart}

\usepackage{xcolor}
\usepackage{graphicx}
\usepackage[hidelinks]{hyperref}
\usepackage{amscd}

\usepackage{amsmath}
\usepackage{amsthm}
\usepackage{amssymb}
\usepackage{latexsym,array}
\usepackage{amsfonts}
\usepackage{shadow}
\usepackage{amsbsy}
\usepackage{amssymb}
\usepackage{dsfont}
\usepackage{mathtools}
\usepackage{enumitem}
\usepackage{doi}
\usepackage{color}
\usepackage{graphicx}
\usepackage[hidelinks]{hyperref}
\usepackage{cleveref}
\usepackage{fullpage}
\usepackage{comment}

\usepackage{dsfont}

\usepackage{cleveref}

\numberwithin{equation}{section}

\newcommand{\id}{\mathcal{I}}

\newcommand{\qcs}{\mathcal Q_{c,s}(x)}

\newcommand{\dom}{\mathcal{D}}

\newtheorem{theorem}{Theorem}[section]
\newtheorem{lemma}[theorem]{Lemma}
\newtheorem{corollary}[theorem]{Corollary}
\newtheorem{proposition}[theorem]{Proposition}

\theoremstyle{definition}
\newtheorem{remark}[theorem]{{\bf Remark}}
\newtheorem{definition}[theorem]{Definition}

\newcommand{\cc}{\mathbb{C}}

\newcommand{\nn}{\mathbb{N}}

\newcommand{\pp}{\partial}
\newcommand{\rr}{\mathbb{R}}

\newcommand{\unx}{\underline{x}}

\renewcommand{\Re}{\mathrm{Re}}

\newcommand{\qcsa}[1]{\mathcal Q^{#1}_{c,s}(x)}
\usepackage{amscd}
\usepackage{tikz-cd}
\usepackage{tikz}

\newtheorem{example}[theorem]{Example}

\crefname{enumi}{}{}
\crefname{enumii}{}{}

\title[]{Functions and operators of the polyharmonic and polyanalytic Clifford fine structures on the $S$-spectrum}
\author[F. Colombo]{Fabrizio Colombo}
\address{(FC)
	Politecnico di Milano\\Dipartimento di Matematica\\Via E. Bonardi, 9\\20133
	Milano, Italy}
\email{fabrizio.colombo@polimi.it}

\author[A. De Martino]{Antonino De Martino}
\address{(ADM)
	Politecnico di Milano\\Dipartimento di Matematica\\Via E. Bonardi, 9\\20133
	Milano, Italy
} \email{antonino.demartino@polimi.it}

\author[S. Pinton]{Stefano Pinton}
\address{(SP)
	Politecnico di Milano\\Dipartimento di Matematica\\Via E. Bonardi, 9\\20133
	Milano, Italy
} \email{stefano.pinton@polimi.it}

\date{}

\begin{document}
	
	\maketitle

	\begin{abstract}
		The spectral theory on the $S$-spectrum originated to give quaternionic quantum mechanics a precise mathematical foundation
and as a spectral theory for linear operators in vector analysis.
		This theory has proven to be significantly more general than initially anticipated, naturally extending to fully Clifford operators and revealing unexpected connections with the spectral theory based on the monogenic spectrum, developed over forty years ago by A. McIntosh and collaborators.
		In recent years, we have combined slice hyperholomorphic functions
with the Fueter-Sce mapping theorem, also called Fueter-Sce extension theorem, to broaden the class of functions and operators to which the theory can be applied. This generalization has led to the definition of what we call the {\em fine structures on the $S$-spectrum}, consisting of classes of functions that admit an integral representation and their associated functional calculi.
		In this paper, we focus on the fine structures within the Clifford algebra setting, particularly addressing polyharmonic functions, polyanalytic functions,
holomorphic Cliffordian functions and their associated functional calculi defined via integral representation formulas.
		Moreover, we demonstrate that the monogenic functional calculus, defined via the monogenic Cauchy formula, and the $F$-functional calculus of the fine structures, defined via the Fueter-Sce mapping theorem in integral form, yield the same operator.
		
	\end{abstract}
	
	

	\tableofcontents
	
	\section{Introduction}\label{Intro}

	A new branch of spectral theory, called {\em fine structures on the $S$-spectrum}, has recently emerged.
	This development involves function spaces linked to the Fueter-Sce mapping theorem, also called Fueter-Sce extension theorem, see \cite{Fueter, TaoQian1, Sce} and also \cite{DDG, DDG1},
	which in the Clifford algebra setting $\mathbb{R}_n$ with $n$ imaginary units $e_\mu$ for $\mu=1,\ldots,n$,
	bridges slice hyperholomorphic functions, denoted by $ \mathcal{SH}(U)$, and axially monogenic functions, denoted by $\mathcal{AM}(U)$, using
	the Laplace operator's powers. The Fueter-Sce extension theorem can be visualized by the following diagram
$$
	 	\begin{CD}
	 		&& \textcolor{black}{\mathcal{SH}(U)}  @>\ \     \Delta^{\frac{n-1}{2}}_{n+1}>>\textcolor{black}{\mathcal{AM}(U)},
	 	\end{CD}
$$
	For odd $n$, the operator $\Delta^{\frac{n-1}{2}}_{n+1}$ acts as a pointwise differential operator,
	while for even $n$, fractional powers of the Laplace operator are used.
	The spectral theory based on the $S$-spectrum for quaternionic operators,
	initially driven by Birkhoff and von Neumann’s work on quaternionic quantum mechanics, see \cite{BF},
	saw a major development with the introduction of the concepts of the $S$-spectrum and of the $S$-resolvent operators in 2006,
	see the introduction of the book \cite{CGK} for more details.
	This concepts has advanced the field significantly, with applications now extending into
	areas such as fractional diffusion problems, where fractional powers of vector operators are crucial.
	For comprehensive discussions, the reader is referred to the books \cite{CGK, FJBOOK,ColomboSabadiniStruppa2011}.

	Quaternionic fine structures on the $S$-spectrum are now well studied and understood. To illustrate the essence of these theories, we have to consider
	to the Fueter extension theorem in the quaternionic setting or the Fueter-Sce extension theorem in the Clifford Algebra setting.
	In hypercomplex analysis, two different theories of hyperholomorphic functions have emerged as crucial in several fields of science: the class of slice hyperholomorphic functions and the class of Fueter regular functions in the quaternionic setting, as well as the monogenic functions for Clifford algebra setting.
	Roughly speaking, the Fueter or Fueter-Sce mapping theorem provides an explicit link between these two classes. This theorem offers a constructive method to extend holomorphic functions of one complex variable to slice hyperholomorphic functions and subsequently to axially Fueter regular functions that has been discussed in details in Remark \ref{QUTFINSRUC} for the convenience of the readers.
	
\medskip
	We will call {\em fine structures of the spectral theory on the $S$-spectrum}:
\begin{itemize}
\item
the set of function spaces, and
\item
 the associated functional calculi
\end{itemize}
induced by all possible factorizations of the operator $\Delta^{\frac{n-1}{2}}_{n+1}$ in the Fueter-Sce extension theorem.

	\medskip
In this paper, we focus our investigation on the Clifford algebra framework.
This exploration reveals several challenges, with respect to the quaternionic setting, arising from the numerous fine structures involved.

	\medskip
	Specifically, for odd $n$ in the Clifford algebra $\mathbb{R}_n$, one of the most significant factorizations of the second map in the Fueter-Sce mapping theorem (see Theorem \ref{FS1}), that is $T_{FS2}:=\Delta_{n+1}^{h_n}$, where  $h_n:=(n-1)/2$  is the Sce exponent,
leads to the so called Dirac fine structures.
	This structure involves an alternating sequence of  Dirac operator $D$ and its conjugate $\overline{D}$, repeated $(n-1)/2$ times.
	The operators $D$ and $\overline{D}$ are defined as
	\begin{equation}\label{DIRACeBARDNN}
		D= \frac{\partial}{\partial x_0}+ \sum_{i=1}^{n} e_i \frac{\partial}{\partial x_i}, \ \ {\rm and} \ \
		\overline{D}= \frac{\partial}{\partial x_0}- \sum_{i=1}^{n} e_i \frac{\partial}{\partial x_i}.
	\end{equation}
	As an example of a factorization of the operator $T_{FS2}$ we have:
	$
	T_{FS2}=\Delta_{n+1}^{h_n}= {D} {\overline{D}}\cdots {D}{\overline{D}}
	$
	where $D$ and $\overline{D}$ appear $(n-1)/2$  times.
	Since $D$ and $\overline{D}$ commute with each other,
	various configurations of their order can be considered, as long as their product results in $\Delta_{n+1}^{h_n}$.
	As a further example also the configuration
	$
	\Delta_{n+1}^{h_n}= {D} \cdots {D} {\overline{D}}  \cdots{\overline{D}}
	$
	is possible
	and if we allow in certain positions $D\overline{D}$ or $ \overline{D}D$
	to be the Laplacian $\Delta_{n+1}$, since
	$
	\Delta_{n+1}=D\overline{D}=\overline{D}D
	$
	several different fine structures emerge from the operators
	$D$, $\overline{D}$, $\Delta_{n+1}$ and their powers, depending on the dimension of the Clifford algebra.

\medskip
Thanks to the above considerations, the fine structures of spectral theory related to the \( S \)-spectrum generate several classes of functions, some of which have already been explored in the literature.
 For $k$ and $\ell$ belong to $\mathbb{N}_0=\mathbb{N}\cup\{0\}$,
we can summarize these classes of functions by defining the polyanalytic holomorphic Cliffordian functions $\mathcal{PHC}_{k, \ell}$ of order \((k, \ell)\), as follows.
Let $U$ be an open set in $\mathbb{R}^{n+1}$.
A function $f: U \subset \mathbb{R}^{n+1} \to \mathbb{R}_n $ that is of class \(\mathcal{C}^{2k + \ell}(U)\) is called (left) polyanalytic holomorphic Cliffordian of order \((k, \ell)\) if
\[
\Delta_{n+1}^k D^{\ell} f(x) = 0 \quad \forall x \in U.
\]
Observe that, thanks to the factorization
$\Delta_{n+1}=D\overline{D}=\overline{D}D$
 the operator $\Delta_{n+1}^k D^{\ell}$ can also be written as
 $$
\Delta_{n+1}^k D^{\ell} = \overline{D}^{k}D^{\ell+k} = D^{\ell+k}\overline{D}^{k}.
$$
It is important to note that this definition requires only the regularity $\mathcal{C}^{2k + \ell}(U)$; however, in all our considerations, the operators $\Delta_{n+1}^k D^{\ell}$ are applied to slice hyperholomorphic functions and the indexes $k$ and $\ell$  have suitable limitations due to the Fueter-Sce mapping theorem.

\medskip
	Clearly,  polyharmonic functions of degree $k$ are a particular case of polyanalytic holomorphic Cliffordian functions, they are those for which $\ell=0$. Similarly polyanalytic functions of order $\ell$ are those for which $k=0$.
In dimension five we have studied Dirac fine structures, and their variations, with great details in \cite{Fivedim},
	for dimensions greater than five, an additional function space not mentioned in \cite{Fivedim} comes into play.

\medskip
In this paper the class of operators, for which the functional calculi of the fine structures are defined,
consist of bounded linear operators in paravector form with commuting components.
Precisely, let \( V \) be a real Banach space and let
$\mathcal{B}(V)$ is the Banach space of all bounded $\mathbb{R}$-linear operators on $V$.
 Denote by \( e_\mu \) the imaginary units of the Clifford algebra $\mathbb{R}_n$, then
paravector operators have the form
$
T = \sum_{\mu=0}^n e_\mu T_\mu,
$
where the components \( T_\mu \in \mathcal{B}(V) \) for \( \mu = 0, 1, \dots, n \),  are assumed to commute with each other.
The operator $T$ acts on Banach module $V\otimes \mathbb{R}_n$,  the conjugate of such an operator is defined as
$\overline{T} = T_0 - \sum_{\mu=1}^n e_\mu T_\mu,$
where \( T_\mu \in \mathcal{B}(V) \) for \( \mu = 0, 1, \dots, n \).
The notion of spectrum for paravector operators with commuting components,
for  \( \|T\| < |s| \) with \( s \in \mathbb{R}^{n+1} \),
is suggested by the sum of the \( S \)-resolvent series:
\[\sum_{m=0}^\infty T^m s^{-1-m} = (s \mathcal{I} - \overline{T}) \mathcal{Q}_{c,s}(T)^{-1},\]
where
$
\mathcal{Q}_{c,s}(T) := s^2 \mathcal{I} - s(T + \overline{T}) + T \overline{T}.
$
For such operators, we define the \( F \)-resolvent set of \( T \) as:
\[\rho_{F}(T) = \{ s \in \mathbb{R}^{n+1} : \mathcal{Q}_{c,s}(T)^{-1} \in \mathcal{B}(V_n) \},\]
where $V_n=V \otimes \mathbb{R}_n$ is a two sided Banach module over $\mathbb{R}_n$. The \( F \)-spectrum of \( T \) is defined as
$\sigma_{F}(T) = \mathbb{R}^{n+1} \setminus \rho_{F}(T)$. This notion of spectrum is associated with
the left $S$-resolvent operator used for the $S$-functional calculus and is given by
$$
S_L^{-1}(s,T) = (s\id- \overline{T})(s^2 \mathcal{I} - s(T + \overline{T}) + T \overline{T})^{-1}
\ \ \text{for}\  \  s\in \rho_{F}(T).
$$
Similar definition holds for the right $S$-resolvent operator.
We point out that $\sigma_F(T)$ is the commutative version of the \( S \)-spectrum. When the components of \( T \) are not required to commute, for  \( \|T\| < |s| \) with \( s \in \mathbb{R}^{n+1} \), the sum of the series becomes:
\[\sum_{m=0}^\infty T^m s^{-1-m} = -\mathcal{Q}_s(T)^{-1}(T - \overline{s} \mathcal{I}),\]
where
$
\mathcal{Q}_s(T) := T^2 - 2s_0 T + |s|^2 \mathcal{I}.
$
Remarkably, the \( S \)-resolvent series has a closed form even when the operators \( T_\mu \in \mathcal{B}(V) \) for \( \mu = 1, \dots, n \), defining \( T \), do not commute.
For such operators  we define the
\( S \)-resolvent set of \( T \) as
\[
\rho_S(T) = \{ s \in \mathbb{R}^{n+1} : \mathcal{Q}_s(T)^{-1} \in \mathcal{B}(V_n) \},
\]
and the \( S \)-spectrum of \( T \) as
$\sigma_S(T) = \mathbb{R}^{n+1} \setminus \rho_S(T).$

The key result is that for bounded operators with commuting components, we have:
\[\sigma_S(T) = \sigma_F(T),\]
showing that \( \sigma_F(T) \) is the commutative version of the \( S \)-spectrum.
We recall that the $F$-spectrum is not well defined when the paravector operators have
noncommuting components.

\medskip	
{\em The content of the paper.}
Section \ref{Intro} provides the introduction and
in order to illustrate our results here we will consider left slice hyperholomorphic functions $\mathcal{SH}_L(U)$ and the related function spaces, since the right slice hyperholomorphic functions generate analogous spaces.

\medskip
In Section \ref{Prel}, we present the preliminary results on Clifford algebra, function theory and operator theory. Specifically, in the two subsections we discuss slice monogenic functions and the associated $S$-functional calculus for $n$-tuples of operators, respectively.

\medskip
In Section \ref{MONandFUN}, we show that, despite being based on different notions of the spectrum for paravector operators and employing distinct integral representations, the monogenic functional calculus and the $F$-functional calculus for axially monogenic functions produce the same operator.

\medskip
{\em Axially polyharmonic functions.}
In Section \ref{SEC4} we demonstrate an integral representation formula
 for {axially polyharmonic functions} of degree $h_n-\ell+1$ defined by
$$
\mathcal{APH}^L_{h_n-\ell+1}(U)=\{D\Delta^{\ell-1}_{n+1}f \ :\ f\in \mathcal{SH}_L(U)\},
$$
where, $n$ is a odd number, $h_n:=(n-1)/2$ is the Sce exponent  and $1 \leq \ell \leq h_n$.
Precisely these functions are obtained by applying the operator
$T_{FS}^{(I)}:= \Delta_{n+1}^{\ell-1} D=D\Delta_{n+1}^{\ell-1} $ to slice hyperholomorphic functions
and the fine structure is represented by the diagram:
\begin{equation}
		\begin{CD}
			&& \textcolor{black}{\mathcal{SH}_L(U)}  @>\ \    D \Delta^{\ell-1}_{n+1}>>\textcolor{black}{\mathcal{APH}^L_{h_n-\ell+1}(U)}@>\ \   \overline{D} \Delta_{n+1}^{h_n-\ell}>>\textcolor{black}{\mathcal{AM}_L(U)}, \qquad 1 \leq \ell \leq h_n,
		\end{CD}
	\end{equation}
where the set of axially monogenic function is defined by
$$
\mathcal{AM}_L(U)=\{\Delta^{h_n}_{n+1}f \ :\ f\in \mathcal{SH}_L(U)\}.
$$
This fine structure follows from the Fueter-Sce mapping theorem since, for $f\in \mathcal{SH}_L(U)$,
we have that $D\Delta^{h_n}f(x)=0$ and we obtain the factorization
$$
D( \overline{D} \Delta_{n+1}^{h_n-\ell})(D \Delta^{\ell-1}_{n+1})f=
 \Delta_{n+1}^{h_n-\ell+1}( D\Delta^{\ell-1}_{n+1})f=0
 $$
which implies $D \Delta^{\ell-1}_{n+1}f\in \mathcal{APH}^L_{h_n-\ell+1}(U)$.

\medskip
 In Section \ref{Polyharmonic fun}, based on the integral representation of the class of functions studied
 in Section \ref{SEC4}, we introduce a polyharmonic functional calculus based on the S-spectrum.
The polyharmonic resolvent operator for  $1 \leq \ell \leq h_n$ where $h_n:=(n-1)/2$, is defined as
$$
		\mathbf{H}_{\ell}(s,T):=2^{2\ell-1} (\ell-1)! (-h_n)_{\ell-1} (s^2\mathcal{I}-s(T+\overline{T})+T\overline{T})^{-\ell}, \ \ \
\text{for} \ \ s\in\rho_S(T),
$$
where $(-h_n)_{\ell-1} $ is the Pochhammer symbol, and since $-h_n$ is a negative integer,  is defined as
$$
(-h_n)_{\ell-1}=(-1)^{\ell-1}  \frac{\Gamma(h_n+1)}{\Gamma(h_n-\ell+2)}.
$$
We point out that the resolvent operator $\mathbf{H}_{\ell}(s,T)$ is the same for the left and also for the right case.

\medskip
{\em Axially holomorphic Cliffordian functions.}
 In Section \ref{HOCLIFIN_INT_FORM} we introduce
axially holomorphic Cliffordian  functions of order $(\ell+1, h_n-s-\ell)$ in integral form,
for $ 0 \leq \ell \leq h_n-2$, $0 \leq s \leq h_n-1 $.
These spaces, for left slice hyperholomorphic functions are defined as
$$
\mathcal{AHC}_{\ell+1,h_n-s-\ell}^L(U)=\{D^s \overline{D}^{h_n-\ell-1}f \ :\ f\in \mathcal{SH}_L(U)\}
$$
and are associated with the operator given by
$
T_{FS}^{(II)}= D^s \overline{D}^{h_n-\ell-1}.
$
The following diagram illustrates this fine structure
\begin{equation}
\begin{CD}
			&& \textcolor{black}{\mathcal{SH}_L(U)}  @>\ \    D^s \overline{D}^{h_n-\ell-1} >>\textcolor{black}{\mathcal{AHC}_{\ell+1,h_n-s-\ell}^L(U)}@>\ \
 \Delta ^{\ell+1}D^{h_n-s-\ell-1}>>\textcolor{black}{\mathcal{AM}_L(U)}.
		\end{CD}
	\end{equation}
This is motivated by the chain of equalities:
$$
D (\Delta ^{\ell+1}D^{h_n-s-\ell-1}) (D^s \overline{D}^{h_n-\ell-1}f) =
(\Delta ^{\ell+1}_{n+1}D^{h_n-s-\ell}) (D^s \overline{D}^{h_n-\ell-1}f)
$$
$$
 =
(\Delta ^{\ell+1}_{n+1}D^{h_n-\ell}) ( \overline{D}^{h_n-\ell-1}f)
 =
(D^{\ell+1} \overline{D}^{\ell+1}D^{h_n-\ell}) ( \overline{D}^{h_n-\ell-1}f) =
D\Delta^{h_n}_{n+1}f=0.
$$
The investigation of all the fine structures depending on the parameters appearing in the operator
$T_{FS}^{(II)} := D^s \overline{D}^{h_n - \ell - 1}$ is quite a challenging problem that is still under investigation. However, in order to obtain the product rule for the $F$-functional calculus, we have investigated the case where the indices $\ell$ and $s$ in the operator
$T_{FS}^{(II)} := D^s \overline{D}^{h_n - \ell - 1}$
are $s = \alpha$ and $\ell = h_n - 1 - \alpha$.
With this choice we have $1 \leq \alpha \leq h_n-1$
 and the fine structure becomes
 \begin{equation}
\begin{CD}
			&& \textcolor{black}{\mathcal{SH}_L(U)}  @>\ \    D^\alpha \overline{D}^{\alpha}=\Delta_{n+1}^\alpha >>\textcolor{black}{\mathcal{AHC}_{h_n-\alpha,1}^L(U)}@>\ \
  D^{h_n-\alpha} \overline{D}^{h_n-\alpha}=\Delta_{n+1}^{h_n -\alpha} >>\textcolor{black}{\mathcal{AM}_L(U)},
		\end{CD}
	\end{equation}
with
$$
\mathcal{AHC}^L_{h_n-\alpha,1}(U)=\{ \Delta_{n+1}^{\alpha} f(x) \ :\ f\in \mathcal{SH}_L(U)\}.
$$

\medskip
In Section \ref{HOLCLIFFUNCAL} we study the associated  holomorphic Cliffordian functional calculus
and in a subsection we obtain the product rule for the $F$-functional calculus.
This is important in various aspects especially when we will consider the extension of the calculus to sectorial operators.

The left resolvent operator for holomorphic Cliffordian functional calculus $\mathcal{K}^L_{\alpha}(s,T)$,
for $\alpha$ integer and such that $1 \leq \alpha \leq h_n-1$, is defied as
$$
\mathcal{K}^L_{\alpha}(s,T):= 4^{\alpha} \alpha! (-h_n)_{\alpha}(s-\overline{T}) (s^2\mathcal{I}-s(T+\overline{T})+T\overline{T})^{-\alpha-1}, \ \ \
\text{for} \ \ s\in\rho_S(T).
$$

\medskip
{\em Axially  polyanalytic functions.}
In Section \ref{polyanalytic functions} we investigate
 the integral representation of axially polyanalytic functions
of order $h_n-\ell+1$ defined by
$$
\mathcal{AP}_{h_n-\ell+1}^L(U)=\{\Delta_{n+1}^{\ell} \overline{D}^{h_n- \ell}f\ :\ f\in \mathcal{SH}_L(U)\}
$$
where $0 \leq \ell \leq h_n-1$. These functions are associated with the operator
$
T^{(III)}_{FS}=\Delta_{n+1}^{\ell} \overline{D}^{h_n- \ell},
$
and this fine structure is illustrated by the diagram
\begin{equation}
		\begin{CD}
			&& \textcolor{black}{\mathcal{SH}_L(U)}  @>\ \    \overline{D}^{h_n-\ell} \Delta^{\ell}_{n+1}>>\textcolor{black}{\mathcal{AP}_{h_n-\ell+1}^L(U)}@>\ \   D^{h_n-\ell}>>\textcolor{black}{\mathcal{AM}_L(U)}
		\end{CD}
	\end{equation}
that is based on the equalities:
$$
 D(D^{h_n-\ell}) (\overline{D}^{h_n-\ell} \Delta^{\ell}_{n+1}f) =
(D^{h_n-\ell+1}) (\overline{D}^{h_n-\ell} \Delta^{\ell}_{n+1}f)=D\Delta^{h_n}_{n+1}f=0.
$$

\medskip
In Section \ref{Polyanalytic functional} we introduce the polyanalytic functional calculus based on the $S$-spectrum,
assuming that the integer $\ell$ is such that $0 \leq \ell \leq h_n-1$.
 We define the left $\mathcal{P}$-resolvent operator as
$$
\mathcal{P}^L_\ell(s,T)=\frac{(-1)^{h_n-\ell}}{(h_n-\ell)!}F^L_n(s,T)(s \mathcal{I}-T_0)^{h_n-\ell}, \ \ \
\text{for} \ \ s\in\rho_S(T),
$$
where $F_n^L(s,T)$ is the left resolvent operator of the $F$-functional calculus defined by
$$
			F_n^L(s,T):=\left[\Gamma\left(\frac{n+1}{2}\right)\right]^2 2^{n-1}(-1)^{\frac{n-1}{2}}(s\mathcal{I}-\overline{ T})(s^2\mathcal{I}-s(T+\overline{T})+T\overline{T})^{-\frac{n+1}{2}}, \ \ \
\text{for} \ \ s\in\rho_S(T).
$$
\newline
\newline
It is a remarkable fact that all the functional calculi associated with the function spaces
of the fine structures discussed in this paper  are all based on the commutative version of the $S$-spectrum, i.e., all these resolvent operators
are explicit functions of the map
$$
s\mapsto (s^2\mathcal{I}-s(T+\overline{T})+T\overline{T})^{-1}, \ \ \ \text{for}\ \ \  s\in\rho_S(T).
$$
\newline
\newline
The definitions of the function spaces and the identification of an integral transform for the class of functions, based on the factorization of the second map in the Fueter-Sce mapping theorem, are crucial steps in defining the related functional calculi. Once the integral representation of this class of functions has been established and suitable series expansions for the kernel in the integral representation have been found, we can define the series expansion of the resolvent operators associated with the various functional calculi.
The well-posedness of the functional calculi is guaranteed by showing that they are independent of both the complex plane of integration and the open set containing the $S$-spectrum.
Furthermore, it is essential to demonstrate the independence of the operator's kernel that acts on slice hyperholomorphic functions and produce the desired function spaces.
\newline
\newline
The spaces of fine structures defined via the Fueter-Sce-Qian extension theorem for even $n$ can also be constructed. However, these spaces involve fractional powers of the Laplace operator, as $T_{FS2} = \Delta_{n+1}^{h_n}$. This requires different techniques, which will be addressed in future work.
\newline
\newline
Finally, we provide some concluding remarks in Section \ref{Concluding remarks}, regarding also the quaternionic fine structures. In Section \ref{Appendix}, we include an appendix where we present a proof of a crucial technical result that is too elaborated to be included in the main text where it is applied.

\section{Preliminary results}\label{Prel}
		We will work within the framework of the real Clifford algebra $\rr_n$, constructed over $n$ imaginary units denoted as $e_1,\dots,e_n$,
	satisfying the relations $e_ie_j+e_je_i=-2\delta_{ij}$. An element in the Clifford algebra $\mathbb{R}_n$ is represented as $\sum_A e_Ax_A$, where $A=\{ i_1\ldots i_r\}$, $i_\ell\in \{1,2,\ldots, n\}$, $i_1<\ldots <i_r$ forms a multi-index, $e_A=e_{i_1} e_{i_2}\ldots e_{i_r}$, and $e_\emptyset =1$. Notably, when $n=1$, $\rr_1$ corresponds to the commutative algebra of complex numbers $\mathbb{C}$, and for $n=2$, it yields the division algebra of real quaternions $\mathbb{H}$. However, for $n>2$, the Clifford algebras $\rr_n$ contain zero divisors.
	
	Within the Clifford algebra $\rr_n$, certain elements can be identified with vectors in Euclidean space $\rr^n$: a vector $(x_1,x_2,\ldots,x_n)\in\rr^n$ maps to a "1-vector" in the Clifford algebra via $\unx=x_1e_1+\ldots+x_ne_n$. Additionally, an element $(x_0,x_1,\ldots,x_n)\in \rr^{n+1}$ is identified as a paravector, represented as $x=x_0+\unx=x_0+ \sum_{j=1}^nx_je_j$, the conjugate of $x$ is define as $\overline{x}:=x_0-\unx$.
The norm of $x\in\rr^{n+1}$ is defined as $|x|^2=x_0^2+x_1^2+\ldots +x_n^2$, and the real part $x_0$ of $x$ is denoted as ${\rm Re} (x)$.
	We introduce the notation $\mathbb{S}$ for the sphere of unit 1-vectors in $\mathbb{R}^n$, defined as
	$$
	\mathbb{S}=\{ \unx=e_1x_1+\ldots +e_nx_n\ :\ x_1^2+\ldots +x_n^2=1\}.
	$$
	Note that $\mathbb{S}$ is an $(n-1)$-dimensional sphere in $\rr^{n+1}$.
	
	The vector space $\mathbb{R}+I\mathbb{R}$, passing through $1$ and $I\in \mathbb{S}$, is denoted by $\mathbb C_I$, with elements represented as $u+Iv$, where $u$, $v\in \mathbb{R}$. Notably, $\mathbb C_I$ is a 2-dimensional real subspace of $\rr^{n+1}$ isomorphic to the complex plane, with the isomorphism being an algebraic isomorphism.

	\noindent Given a paravector $x=x_0+\unx\in\rr^{n+1}$ let us set
	$$
	I_x=\left\{\begin{array}{l}
		\displaystyle\frac{\unx}{|\unx|}\quad{\rm if}\ \unx\not=0,\\
		{\rm any\ element\ of\ } \mathbb{S}{\rm\ otherwise,}\\
	\end{array}
	\right.
	$$
	so, by definition we have $x\in \mathbb C_{I_x}$.
	\par\noindent
	\begin{definition}\label{sphere}
		Given an element $x\in\rr^{n+1}$, we define
		$$
		[x]=\{y\in\rr^{n+1}\ :\ y={\rm Re}(x)+I |\underline{x}|,\, I\in \mathbb{S}\}.
		$$
	\end{definition}
	
	\begin{remark}{\rm The set $[x]$ is a $(n-1)$-dimensional sphere in $\rr^{n+1}$. When $x\in\rr$, then $[x]$ contains $x$ only. In this case, the
			$(n-1)$-dimensional sphere has radius equal to zero. }\end{remark}
	\begin{definition}
		Let  $U \subseteq \mathbb{R}^{n+1}$.
		\begin{itemize}
			\item We say that $U$ is {\em axially symmetric} if $[x]\subset U$  for any $x \in U$.
			\item We say that $U$ is a {\em slice domain} if $U\cap\rr\neq\emptyset$ and if $U\cap\cc_{I}$ is a domain in $\cc_{I}$ for any $I\in\mathbb{S}$.
		\end{itemize}
	\end{definition}
	
	\subsection{Slice monogenic functions}\label{SMONG}
	
	\begin{definition}[Slice Cauchy domain]
		An axially symmetric open set $U\subset \mathbb{R}^{n+1}$ is called a slice Cauchy domain,
		if $U\cap\cc_I$ is a Cauchy domain in $\cc_I$ for any $I\in\mathbb{S}$.
		More precisely, $U$ is a slice Cauchy domain if for any $I\in\mathbb{S}$
		the boundary ${\partial( U\cap\cc_I)}$ of $U\cap\cc_I$ is the union a finite number of non-intersecting piecewise continuously differentiable Jordan curves in $\cc_{I}$.
	\end{definition}
	
	\begin{definition}[Slice hyperholomorphic functions (or slice monogenic functions)]\label{sh}
		Let $U\subseteq \mathbb{R}^{n+1}$ be an axially symmetric open set and let
		\begin{equation}\label{SETcalU}
		\mathcal{U} = \{ (u,v)\in\rr^2: u+ \mathbb{S} v\subset U\}.
		\end{equation}
		A function $f:U\to \mathbb{R}^{n+1}$ is called a left
		slice function, if it is of the form
		\begin{equation}
			\label{form}
			f(x) = \alpha(u,v) + I\beta(u,v)\qquad \text{for } x = u + I v\in U,
		\end{equation}
		where the functions $\alpha, \beta: \mathcal{U}\to \mathbb{R}_{n}$ that satisfy the compatibility condition
		\begin{equation}
			\label{EO}
			\alpha(u,v)=\alpha(u,-v), \qquad \beta(u,v)=- \beta(u,-v), \qquad \forall (u,v) \in \mathcal{U}.
		\end{equation}
		If in addition $\alpha$ and $\beta$ satisfy the Cauchy-Riemann-equations
		\begin{equation}
			\label{CR}
			\partial_u \alpha(u,v)- \partial_v \beta(u,v)=0, \quad \partial_v \alpha(u,v)+ \partial_u \beta(u,v)=0.
		\end{equation}
		then $f$ is called left slice hyperholomorphic or slice monogenic.
		A function $f:U\to \mathbb{R}_n$ is called a right slice function if it is of the form
		\[
		f(x) = \alpha(u,v) + \beta(u,v)I\qquad \text{for } x = u + I v\in U
		\]
		where $ \alpha$ and $ \beta: \mathcal{U}\to \mathbb{R}_{n}$ that satisfy \eqref{EO}.
		If in addition $\alpha$ and $\beta$ satisfy the Cauchy-Riemann-equations (\ref{CR}), then $f$ is called right slice hyperholomorphic or slice monogenic.
	\end{definition}
	\begin{definition}
		Let $U$ be an axially symmetric open set $\mathbb{R}^{n+1}$.
		\begin{itemize}
			\item
			We denote the sets of left
			slice hyperholomorphic functions (or left slice monogenic functions) on $U$ by $\mathcal{SH}_L(U)$
			while  right slice hyperholomorphic functions (or right slice monogenic functions) on $U$ will be denoted by $\mathcal{SH}_R(U)$.
			\item
			A slice hyperholomorphic function \eqref{form} such that $\alpha$ and $\beta$ are real-valued functions is called intrinsic slice hyperholomorphic or intrinsic slice monogenic functions and will be denoted by $\mathcal{N}(U)$.
		\end{itemize}
		\begin{remark}
			The set of intrinsic left slice monogenic functions coincides with the set of right slice monogenic functions, so when we mention the set $\mathcal{N}(U)$ we do not distinguish the left case from the right one.
		\end{remark}
	\end{definition}
	In the theory of slice monogenic functions, a crucial aspect is the representation formula,
	which shows that any slice monogenic function, as in Definition \ref{sh},
	on an axially symmetric set can be entirely characterized by its values on two complex planes $\mathbb{C}_I$, for $I\in \mathbb{S}$, because of the  book structure of $\mathbb{R}^{n+1}$, i.e.,
	$\mathbb{R}^{n+1}=\bigcup_{I\in \mathbb{S}} \mathbb{C}_I$.

	Now, we recall that also for slice monogenic functions holds a counterpart of the Cauchy integral theorem, see \cite{CGK, ColomboSabadiniStruppa2011}.
	\begin{theorem}
		\label{Cif}
		Let $U \subset \mathbb{R}^{n+1}$ be an open set and $I \in \mathbb{S}$. We assume that $f$ is a left slice monogenic function and $g$ is a right slice monogenic function in $U$. Furthermore, let $D_I \subset U \cap \mathbb{C}_I$
		be an open and bounded subset of $ \mathbb{C}_I$ with $ \overline{D}_I \subset U \cap \mathbb{C}_I$
		such that $\partial D_I$ is a finite union of piecewise continuously differentiable Jordan curves. Then, for $ds_I=ds(-I) $, we have
		$$ \int_{\partial D_I}g(s)ds_I f(s)=0.$$
	\end{theorem}
	
	\medskip
	
	We now revisit the slice monogenic Cauchy formulas, which serve as the foundation for developing hyperholomorphic spectral theories on the $S$-spectrum. The following result will be required, see \cite{CGK, ColomboSabadiniStruppa2011}.
	\begin{proposition}[Cauchy kernel series]
		\label{cauchyseries}
		Let $s$, $x \in \mathbb{R}^{n+1}$ such that $|x|<|s|$ then
		$$ \sum_{n=0}^{\infty}x^n s^{-1-n}=-(x^2 -2x {\rm Re} (s)+|s|^2)^{-1}(x-\overline s)$$
		and
		$$ \sum_{n=0}^{\infty}s^{-1-n}x^n=-(x-\overline s)(x^2 -2x {\rm Re} (s)+|s|^2)^{-1}.$$
	\end{proposition}
	An important fact for our considerations is that the sum of the Cauchy kernel series can be written in two equivalent ways.
	\begin{proposition}\label{secondAA}
		Suppose that $x$ and $s\in\rr^{n+1}$ are such that $x\not\in [s]$. We have that
		\begin{equation}\label{second}
			-(x^2 -2x {\rm Re} (s)+|s|^2)^{-1}(x-\overline s)=(s-\bar x)(s^2-2{\rm
				Re}(x)s+|x|^2)^{-1},
		\end{equation}
		and
		\begin{equation}\label{third}
			(s^2-2{\rm Re}(x)s+|x|^2)^{-1}(s-\bar x)=-(x-\bar s)(x^2-2{\rm Re}(s)x+|s|^2)^{-1} .
		\end{equation}
	\end{proposition}
	These facts justify the following definitions.
	\begin{definition}
		\label{Ckernel}
		Let $x$, $s\in \rr^{n+1}$
		such that $x\not\in [s]$.
		\begin{itemize}
			\item
			We say that  $S_L^{-1}(s,x)$ is written in the form I if
			$$
			S_L^{-1}(s,x):=-(x^2 -2x {\rm Re} (s)+|s|^2)^{-1}(x-\overline s).
			$$
			\item
			We say that $S_L^{-1}(s,x)$ is written in the form II if
			$$
			S_L^{-1}(s,x):=(s-\bar x)(s^2-2{\rm Re}(x) s+|x|^2)^{-1}.
			$$
			\item
			We say that  $S_R^{-1}(s,x)$ is written in the form I if
			$$
			S_R^{-1}(s,x):=-(x-\bar s)(x^2-2{\rm Re}(s)x+|s|^2)^{-1} .
			$$
			\item
			We say that $S_R^{-1}(s,x)$ is written in the form II if
			$$
			S_R^{-1}(s,x):=(s^2-2{\rm Re}(x)s+|x|^2)^{-1}(s-\bar x).
			$$
		\end{itemize}
	\end{definition}
	\begin{remark}
		In the following the polynomial
		\begin{equation}\label{polynomQCSX}
			\mathcal{Q}_{c,s}(x):=s^2-2{\rm Re}(x) s+|x|^2, \ \ \ s,x\in \mathbb{R}^{n+1}
		\end{equation}
		and its inverse, defined for $s,x\in \mathbb{R}^{n+1}$ with $x\not\in [s]$,
		\begin{equation}\label{QCSX}
			\mathcal{Q}_{c,s}(x)^{-1}:=(s^2-2{\rm Re}(x) s+|x|^2)^{-1}
		\end{equation}
		play a crucial role in several results. The term $\mathcal{Q}_{c,s}(x)^{-1}$, that appears in the Cauchy kernel in form II  of slice monogenic functions, is often called commutative pseudo-Cauchy kernel to distinguish it from the pseudo-Cauchy kernel that is given by
		$\mathcal{Q}_{s}(x)^{-1}:=(x^2-2{\rm Re}(s)x+|s|^2)^{-1}$, defined for $s,x\in \mathbb{R}^{n+1}$ with
		$x\not\in [s]$.
	\end{remark}
	\begin{remark}
		Thanks to Proposition \ref{secondAA}, the Cauchy kernels can be expressed in two equivalent ways from the perspective of the Clifford Algebra.
		However, this equivalence does not hold when the paravector $x$
		is replaced by a paravector operator.
		Specifically, the Cauchy kernels in form I are well-suited for paravector
		operators $T=\sum_{\ell=0}^n T_\ell e_\ell$, where the operators $T_\ell$, for $\ell=0,\cdots,n$,
		do not commute among themselves.
		In contrast, the Cauchy kernels in  form II requires
		the operators $T_\ell$, for $\ell=0,\cdots,n$, to commute among themselves.
		This commuting restriction on the operators has the advantage of allowing the resolvent operators of the fine structure to be computed explicitly in closed form.
	\end{remark}
	
	The Cauchy formula for slice monogenic functions is a crucial tool
	for defining the integral representation of functions within the fine structure and their associated functional calculi.
	\begin{theorem}[The Cauchy formulas for slice monogenic functions, see \cite{CGK, ColomboSabadiniStruppa2011}]
		\label{Cauchy}
		Let $U\subset\mathbb{R}^{n+1}$ be a bounded slice Cauchy domain, let $I\in\mathbb{S}$ and set  $ds_I=ds (-I)$.
		If $f$ is a (left) slice monogenic function on a set that contains $\overline{U}$ then
		\begin{equation}\label{cauchynuovo}
			f(x)=\frac{1}{2 \pi}\int_{\partial (U\cap \mathbb{C}_I)} S_L^{-1}(s,x)\, ds_I\,  f(s),\qquad\text{for any }\ \  x\in U.
		\end{equation}
		If $f$ is a right slice hyperholomorphic function on a set that contains $\overline{U}$,
		then
		\begin{equation}\label{Cauchyright}
			f(x)=\frac{1}{2 \pi}\int_{\partial (U\cap \mathbb{C}_I)}  f(s)\, ds_I\, S_R^{-1}(s,x),\qquad\text{for any }\ \  x\in U.
		\end{equation}
		These integrals  depend neither on $U$ nor on the imaginary unit $I\in\mathbb{S}$.
	\end{theorem}

	\subsection{The $S$-functional calculus for $n$-tuples of  operators}\label{SFUNC}
	By $V$ we denote a Banach space over $\mathbb{R}$ with norm $\|\cdot \|$.
	The tensor product  $V_n:=V\otimes \mathbb{R}_n$ is a
	two-sided Banach module  over $\mathbb{R}_n$.
	An element in $V_n$ is of the type $\sum_A v_A\otimes e_A$ (where
	$A=\{ i_1\ldots i_r$\}, $i_\ell\in \{1,2,\ldots, n\}$, $i_1<\ldots <i_r$ is a multi-index).
	The multiplications (right and left) of an element $v\in V_n$ by a scalar
	$a\in \mathbb{R}_n$ are defined as
	$$
	va=\sum_A v_A \otimes (e_A a),\ \  {\rm and}\ \  av=\sum_A v_A \otimes (ae_A ).
	$$
	For short, we will write
	$\sum_A v_A e_A$ instead of $\sum_A v_A \otimes e_A$. Moreover, we define
	$\| v\|^2_{V_n}=
	\sum_A\| v_A\|^2_V.$
	
	Let $\mathcal{B}(V)$ be the space
	of bounded $\mathbb{R}$-homomorphisms of the Banach space $V$ into itself
	endowed with the natural norm denoted by $\|\cdot\|_{\mathcal{B}(V)}$.
	\noindent
	If $T_A\in \mathcal{B}(V)$, we can define the operator
	$$
	T=\sum_A T_Ae_A
	$$ and
	its action on
	$
	v=\sum_B v_Be_B
	$
	as
	$$
	T(v)=\sum_{A,B}T_A(v_B)e_Ae_B.
	$$
	The set of all such right bounded operators is denoted by $\mathcal{B}(V_n)$ and a
	norm is defined by
	$$
	\|T\|_{\mathcal{B}(V_n)}=\sum_A \|T_A\|_{\mathcal{B}(V)}.
	$$

	The set of bounded operators, in paravector form,
	$T=\sum_{\mu=0}^ne_\mu T_\mu$, where $T_\mu\in\mathcal{B}(V)$ for $\mu=0,1,...,n$,
	will be denoted by $\mathcal{B}^{\small 0,1}(V_n)$. We also recall that the conjugate is defined by
	$ \overline{T}= T_0- \sum_{\mu=0}^n e_{\mu}T_{\mu},$ where $T_\mu\in\mathcal{B}(V)$ for $\mu=0,1,...,n$.
	The set of bounded operators, in vector form,
	$T=\sum_{j=1}^ne_jT_j$, where $T_\mu\in\mathcal{B}(V)$ for $\mu=1,...,n$,
	will be denoted by $\mathcal{B}^{1}(V_n)$.
	We will embed an $(n+1)$-tuple $(T_0,T_1,\ldots, T_n)$ of bounded $\mathbb{R}$-linear operators into
	$\mathcal{B}(V_n)$ as
	$$
	(T_0,T_1,\ldots, T_n)\mapsto T=\sum_{\mu=0}^ne_\mu T_\mu.
	$$
	Throughout the rest of this section, and unless otherwise specified,
	we will consider operators of the form
	$T=\sum_{\mu=0}^ne_\mu T_\mu$ where $T_\mu\in\mathcal{B}(V)$ for $\mu=0,1,...,n$.
	The subset of those operators in  ${\mathcal{B}(V_n)}$ whose components commute among themselves will be denoted by
	${\mathcal{BC}(V_n)}$. In the same spirit we denote by $\mathcal{BC}^{\small 0,1}(V_n)$ the
	set of paravector operators with commuting components and by $\mathcal{BC}^{\small 1}(V_n)$ vector operators with commuting components. When $ T \in \mathcal{BC}^{\small 0,1}(V_n)$ the operator $T \overline{T}$ is well defined and $T \overline{T}=\overline{T}T= \sum_{\mu=0}^{n}T_{\mu}^2$ and $T+\overline{T}=2 T_0$.
	
	\medskip
	We now arrive at the definition of the spectrum and resolvent operators
	within the Clifford framework. For bounded commuting operators, there are two possible definitions of the $S$-spectrum
	and the $S$-resolvent operators. We will explore the key concepts essential for defining the functional calculi of the Clifford fine structures.
	We have proved, see \cite{ColomboSabadiniStruppa2011}, that for
	$T\in\mathcal{B}^{\small 0,1}(V_n)$,  $s\in \mathbb{R}^{n+1}$, such that $\|T\|< |s|$,  then
	\begin{equation}\label{ciao}
		\begin{split}
			\sum_{m= 0}^\infty T^ms^{-1-m}=-\mathcal{Q}_s(T)^{-1}(T-\overline{s}\mathcal{I})
			\ \ \ {\rm and} \ \ \
			\sum_{m= 0}^\infty T^ms^{-1-m}=-(T-\overline{s}\mathcal{I})\mathcal{Q}_s(T)^{-1},
		\end{split}
	\end{equation}
	where
	\begin{equation}\label{QST}
		\mathcal{Q}_s(T):=T^2-2s_0T +|s|^2\mathcal{I}.
	\end{equation}
	Observe that
	the second hand sides of (\ref{ciao}) there are the Cauchy kernels $S^{-1}_L(s,x)$ and $S^{-1}_R(s,x)$ written in form I,
	where we have formally replaced the paravector $x\in \mathbb{R}^{n+1}$ by the paravector operator $T$ belonging to
	$\mathcal{B}^{\small 0,1}(V_n)$.
	\begin{remark}
		It is a remarkable fact that the $S$-resolvent operators series
		\begin{equation}\label{Sresolvdsf}
			\sum_{m= 0}^\infty T^m s^{-1-m}\ \ \ {\rm and} \ \ \ \sum_{m= 0}^\infty s^{-1-m} T^m,
		\end{equation}
		for $\|T\|< |s|$,
		have a closed form, also when the operators
		 $T_\mu\in\mathcal{B}(V)$ for $\mu=1,...,n$, who define $T$, do not commute among themselves.
	\end{remark}

\begin{remark}
The definition of slice hyperholomorphicity operator valued, see Definition 2.19, is as follows.
Let $U\subseteq\mathbb{R}^{n+1}$ be an axially symmetric open set, denote by $\mathcal{B}(V_n)$
the set of all bounded right linear Clifford operators. An operator valued function $K:U\rightarrow\mathcal{B}(V_n)$ is called \textit{left} (resp. \textit{right}) \textit{slice hyperholomorphic}, if there exists operator valued functions $A,B:\mathcal{U}\rightarrow\mathcal{B}(V_n)$,
with $\mathcal{U}$ in \eqref{SETcalU}, such that for every $(x,y)\in\mathcal{U}$:

\begin{enumerate}
\item[i)] The operators $K$ admit for every $J\in\mathbb{S}$ the representation
\begin{equation}\label{Eq_Holomorphic_decomposition_operators}
K(u+Jv)=A(u,v)+JB(u,v),\quad\Big(\text{resp.}\;K(u+Jv)=A(u,v)+B(u,v)J\Big).
\end{equation}

\item[ii)] The operators $A,B$ satisfy the even-odd conditions
\begin{equation}\label{Eq_Symmetry_condition_operators}
A(u,-v)=A(u,v)\quad\text{and}\quad B(u,-v)=-B(u,v).
\end{equation}

\item[iii)] The operators $A,B$ satisfy the Cauchy-Riemann equations
\begin{equation}\label{Eq_Cauchy_Riemann_equations_operators}
\frac{\partial}{\partial u}A(u,v)=\frac{\partial}{\partial v}B(u,v)\quad\text{and}\quad\frac{\partial}{\partial v}A(u,v)=-\frac{\partial}{\partial u}B(u,v),
\end{equation}
where the derivatives are understood in the norm convergence sense.
\end{enumerate}

We moreover, call $s\mapsto K(s)$ \textit{intrinsic}, if the operators $A,B$ are two-sided linear.
\end{remark}
	So the following definitions naturally arise from (\ref{ciao}), see \cite{ColomboSabadiniStruppa2011}.
	\begin{definition}[The $S$-resolvent set, the $S$-spectrum and the $S$-resolvent operators]
		Let $T\in\mathcal{B}^{\small 0,1}(V_n)$. We define the $S$-resolvent set of $T$ as
		$$
		\rho_S(T)=\{s\in \mathbb{R}^{n+1}\ :\ \mathcal{Q}_s(T)^{-1}\in \mathcal{B}(V_n)\}
		$$
		and the $S$-spectrum  of $T$ as
		$$
		\sigma_S(T):= \mathbb{R}^{n+1}\setminus \rho_S(T).
		$$
		Moreover, for $s\in \rho_S(T)$, we define the left and the right $S$-resolvent operators as
		\begin{equation}\label{ciaoTT}
			\begin{split}
				S^{-1}_L(s,T)=-\mathcal{Q}_s(T)^{-1}(T-\overline{s}\mathcal{I})
				\ \ \ {\rm and} \ \ \
				S^{-1}_R(s,T)=-(T-\overline{s}\mathcal{I})\mathcal{Q}_s(T)^{-1},
			\end{split}
		\end{equation}
		respectively.
	\end{definition}
	\begin{remark}
		As mentioned in the introduction, to obtain the functions of the fine structures, we need to apply the operators $D$, $\overline{D}$, $\Delta_{n+1}$ and their powers to slice monogenic functions. A crucial aspect in deriving an integral representation of such functions is to understand the action of $D$, $\overline{D}$, $\Delta_{n+1}$, and their powers on the Cauchy kernel of slice monogenic functions.
		
		The application of these operators $D$, $\overline{D}$, $\Delta_{n+1}$ to the Cauchy kernels of slice monogenic functions written in form II produces the desired kernels in closed form. These new kernels, combined with the Cauchy formula, provide  integral representation formulas for the functions of the fine structure.
	\end{remark}	
	The Cauchy kernel series written in the form II will be of crucial importance to define functional calculi,
but we must require that the components of the paravector operator $T$ commute among themselves.
	
	\begin{theorem}\label{sumserSCa}
		Let $T\in\mathcal{BC}^{\small 0,1}(V_n)$ and  $s\in \mathbb{R}^{n+1}$ be such that $\|T\|< |s|$. Then, we have
		\begin{equation}\label{ciaooo}
			\begin{split}
				\sum_{m=0}^\infty T^m s^{-1-m}= (s\mathcal{I}- \overline{T})  \mathcal{Q}_{c,s}(T)^{-1},
				\ \ \ {\rm and} \ \ \
				\sum_{m=0}^\infty T^m s^{-1-m}= (s\mathcal{I}- \overline{T})  \mathcal{Q}_{c,s}(T)^{-1},
			\end{split}
		\end{equation}
		where
		\begin{equation}\label{QCST}
			\mathcal{Q}_{c,s}(T):=s^2\mathcal{I}-s(T+\overline{T})+T\overline{T}.
		\end{equation}
	\end{theorem}
	Theorem \ref{sumserSCa} suggests the notions of spectrum and of resolvent set of $T$ for paravector operators with commuting components.
	
	\begin{definition}[The ${F}$-resolvent set and the $F$-spectrum]
		Let $T\in\mathcal{BC}^{\small 0,1}(V_n)$.
		We define the ${F}$-resolvent set of $T$  as:
		$$
		\rho_{F}(T)=\{ s\in \mathbb{R}^{n+1}\ \ :\ \ \mathcal{Q}_{c,s}(T)^{-1}\in \mathcal{B}(V_n)\}.
		$$
		and the $F$- spectrum of $T$ as
		$$
		\sigma_{F}(T)=\rr^{n+1}\setminus\rho_{F}(T).
		$$
	\end{definition}
	
	\begin{theorem}\label{ResCom}
		Let $T \in\mathcal{BC}^{\small 0,1}(V_n)$. Then $\mathcal{Q}_{c,s}(T)$, defined in (\ref{QCST}), in invertible if and only if $\mathcal{Q}_{s}(T)$ is invertible and so
		\begin{equation}\label{SSpecCommut}
			\sigma_S(T) = \sigma_F(T).
		\end{equation}
		Moreover, for $s\in\rho_{S}(T)$, we have
		\begin{align}\label{ZwoaDreiVia}
			S_L^{-1}(s,T) = &(s\id- \overline{T})\mathcal{Q}_{c,s}(T)^{-1}\\
			\intertext{and}
			\label{ZwoaDreiVia1}
			S_R^{-1}(s,T) = & \mathcal{Q}_{c,s}(T)^{-1}(s\id -\overline{T}).
		\end{align}
	\end{theorem}
	\begin{proof}
		The proof of this results is a immediate adaptation of the analogue one for quaternionic operators that can be found in \cite{CGK} Theorem 4.5.6.
	\end{proof}
	
	\begin{definition}\label{SHONTHEFS}
		Let $T\in \mathcal{BC}^{0,1}(V_n)$ and let $U\subset\mathbb{R}^{n+1}$ be a bounded slice Cauchy domain.
		We denote by $\mathcal{SH}^L_{\sigma_S(T)}(U)$, $\mathcal{SH}^R_{\sigma_S(T)}(U)$
		and $\mathcal{N}_{\sigma_S(T)}(U)$ the set of all  left, right and intrinsic slice hyperholomorphic functions
		$f$ with $\sigma_S(T)\subset U\subset \dom(f)$, where $\dom(f)$ is the domain of the function $f$.
	\end{definition}
	
	Keeping in mind Definition \ref{SHONTHEFS} we define the $S$-functional calculus.
	\begin{definition}[$S$-functional calculus]
		\label{Sfun}
		Let $T\in \mathcal{BC}^{0,1}(V_n)$, $U\subset\mathbb{R}^{n+1}$ be a bounded slice Cauchy domain as in
		Definition \ref{SHONTHEFS},
		let $I\in\mathbb{S}$ and set  $ds_I=ds (-I)$.
		For any function $f\in \mathcal{SH}^L_{\sigma_S(T)}(U)$, we define
		\begin{equation}\label{SCalcL}
			f(T) := \frac{1}{2\pi}\int_{\partial(U\cap \mathbb{C}_I)}S_L^{-1}(s,T)\,ds_I\,f(s).
		\end{equation}
		For any $f\in \mathcal{SH}^R_{\sigma_S(T)}(U)$, we define
		\begin{equation}\label{SCalcR}
			f(T) := \frac{1}{2\pi}\int_{\partial(U\cap\mathbb{C}_I)}f(s)\,ds_I\,S_R^{-1}(s,T).
		\end{equation}
	\end{definition}
	\begin{remark}
		The $S$-functional calculus is well defined since the  integrals in (\ref{SCalcL}) and (\ref{SCalcR}) depend neither on $U$ nor on the imaginary unit $I\in\mathbb{S}$.
	\end{remark}

	Because of non-commutativity, the right and left versions of the $S$-functional calculus typically produce different operators. The only exception occurs when the $S$-functional calculus is applied to intrinsic slice hyperholomorphic functions, see \cite[Thm. 3.2.11]{CGK}.
	
	\begin{theorem}
	\label{SLR}
Let $T\in \mathcal{BC}^{0,1}(V_n)$, $U\subset\mathbb{R}^{n+1}$ be a bounded slice Cauchy domain as in
Definition \ref{SHONTHEFS}. Then for $f \in \mathcal{N}_{\sigma_S(T)}(U)$  we have
$$ \frac{1}{2 \pi} \int_{\partial (U\cap \mathbb C_I)} S_L^{-1}(s,T) ds_I f(s)=\frac{1}{2 \pi} \int_{\partial (U\cap \mathbb C_I)}f(s)ds_I  S_R^{-1}(s,T).$$
	\end{theorem}

	The left and the right $S$-resolvent operators satisfy the equations:
	\begin{theorem}
		Let $T\in\mathcal{BC}^{0,1}(V_n)$ and let $s\in\rho_S(T)$.
		The left $S$-resolvent operator satisfies the {\em left $S$-resolvent equation}
		\begin{equation}\label{LeftSREQ}
			S_L^{-1}(s,T)s - TS_L^{-1}(s,T) = \id
		\end{equation}
		and the right $S$-resolvent operator satisfies the {\em right $S$-resolvent equation}
		\begin{equation}\label{RightSREQ}
			sS_R^{-1}(s,T) - S_R^{-1}(s,T)T = \id.
		\end{equation}
	\end{theorem}

	The left and right $S$-resolvent equations cannot be considered generalizations of the classical resolvent equation:
	\begin{equation}
		\label{ClassicalREQ}
		R_{\lambda}(A) - R_{\mu}(A) = (\mu - \lambda)R_{\lambda}(A)R_{\mu}(A), \ \ \ \text{for} \ \ \ \ \lambda,\mu\in\rho(A),
	\end{equation}
	where $R_{\lambda}(A) = (\lambda\id-A)^{-1}$ is the resolvent operator of the complex operator $A:V\to V$ evaluated at $\lambda\in\rho(A)$.
	This equation allows to write of the product of resolvent operator of $A$ in $\lambda,\mu\in\rho(A)$, i.e. $R_{\lambda}(A)R_{\mu}(A)$
in terms of the sum of these operators. However, this property does not hold for the left or right $S$-resolvent equations.
	The proper generalization to the $S$-functional calculus of \eqref{ClassicalREQ}, which retains this characteristic, is the $S$-resolvent equation
	presented in the following theorem. It is notable that this equation involves both the left and right $S$-resolvent operators.
	To date, no generalization of \eqref{ClassicalREQ} that includes only one of these $S$-resolvent operators
has been discovered with the properties that
	(\ref{ClassicalREQ}) has.

	\begin{theorem}[The $S$-resolvent equation]\label{SREQ}
		Let $T\in\mathcal{BC}^{0,1}(V_n)$
		and let $s,x\in \rho_S(T)$ with $x\notin[s]$. Then the equation
		\begin{multline}\label{SREQ1}
			S_R^{-1}(s,T)S_L^{-1}(x,T)=\left[\left(S_R^{-1}(s,T)-S_L^{-1}(x,T)\right)x\right.\\
			\left.-\overline{s}\left(S_R^{-1}(s,T)-S_L^{-1}(x,T)\right)\right](x^2-2\Re(s)x+|s|^2)^{-1}
		\end{multline}
		holds true. Equivalently, it can also be written as
		\begin{multline}\label{SREQ2}
			S_R^{-1}(s,T)S_L^{-1}(x,T)=(s^2-2\Re(x)s+|x|^2)^{-1}
			\\
			\cdot\left[\left(S_L^{-1}(x,T) - S_R^{-1}(s,T)\right)\overline{x}-s\left(S_L^{-1}(x,T) - S_R^{-1}(s,T)\right)
			\right].
		\end{multline}
	\end{theorem}
	\begin{remark} We point out that the left $S$-resolvent equation (\ref{LeftSREQ}), the
		right $S$-resolvent equation  (\ref{RightSREQ}) and the $S$-resolvent equation (\ref{SREQ}) hold true without the limitation that
		$T =\sum_{\ell =0}^n T_{\ell} e_{\ell}\in\mathcal{BC}(V_n)$. This condition is required for Theorem \ref{ResCom}. In fact,
		the relations  (\ref{LeftSREQ}), (\ref{RightSREQ}), (\ref{SREQ1}) and (\ref{SREQ2}) hold for general fully Clifford operators $T=\sum_A T_Ae_A$ with noncommuting components as well as the $S$-functional calculus, as it was proved in \cite{ADVCGKS}.
	\end{remark}

	\section{The monogenic functional calculi and their connections}\label{MONandFUN}

	The monogenic functional calculus, based on the Cauchy formula for monogenic functions,
	was introduced in \cite{JM}, but for a systematic treatment we refer the reader to the books \cite{JBOOK} and \cite{TAOBOOK} where one can also find the connection of the monogenic functional calculus with the Weyl functional calculus that has applications in quantum mechanic.
	
	\medskip
	Following a different approach based on spectral theory on the $S$-spectrum and the Fueter-Sce mapping theorem in integral form,
	in \cite{CSS}, was introduced the so-called $F$-functional calculus.
	This is also a functional calculus of monogenic type.
	
	\medskip
	In this section, we show that these two different approaches produce the same functional calculus for axially monogenic functions,
	despite being based on different integral representations of axially monogenic functions whose definition is as follows:
	
	\begin{definition}
	Let $U \subseteq \mathbb{R}^{n+1}$ be an axially symmetric domain.
		A function $f:U \to \mathbb{R}^{n+1}$ is of left (resp. right) axial form  if it is of the following type
		\begin{equation}
			\label{axx}
			f(x)=A(x_0,|\underline{x}|)+\underline{\omega}B(x_0,|\underline{x}|)\quad \left( f(x)=A(x_0,|\underline{x}|)+B(x_0,|\underline{x}|)\underline{\omega} \right),
			\quad \underline{\omega}= \frac{\underline{x}}{|\underline{x}|},
		\end{equation}
		where the functions $A$ and $B$ satisfy the even-odd conditions, see \eqref{EO}.
	\end{definition}
	\begin{definition}\label{AXIALM}
		Let $U \subseteq \mathbb{R}^{n+1}$ be an axially symmetric domain.
		A function $f:U \to \mathbb{R}^{n+1}$ of class $ \mathcal{C}^1$ is said to be left  axially monogenic if it is monogenic:
		\begin{equation}\label{LEFTMONFUNC}
		Df(x)= \left(\frac{\partial}{\partial x_0}+ \sum_{i=1}^{n} e_i \frac{\partial}{\partial x_i} \right)f(x)=0,
\end{equation}
i.e. it is in the kernel of the Dirac operator $D$ and it is of left axial form. Similarly, we define  a right axially monogenic function if
$$f(x)D=\frac{\partial}{\partial x_0}f(x)+ \sum_{i=1}^{n} \frac{\partial}{\partial x_i}f(x) e_i  =0
		$$
and
it is of right axial form. The class of left (resp. right) axially monogenic is denoted by $ \mathcal{AM}_L(U)$ (resp. $\mathcal{AM}_R(U)$).
	\end{definition}
See \cite{red, green} for more details on axially monogenic functions.
	\begin{remark}
		In the literature related to axially monogenic functions for the elements of the sphere $\mathbb{S}$, the symbol $\underline{\omega}$ is used instead of $I$, which is used in the definition of slice monogenic functions. We maintain this dual notation throughout the paper to clearly distinguish between slice monogenic functions and axially monogenic functions.
	\end{remark}
	From Definition \ref{AXIALM} it is clear that the class of axially monogenic functions is a subclass of monogenic function.
	Thus one can find monogenic functions that are not of axial type as for example the so-called Fueter polynomials.
	We denote by $z_j(x)=x_j-x_0e_j$, with $x \in \mathbb{R}^{n+1}$, the Fueter variables, we define the Fueter polynomials as
	\begin{equation}
		\label{fueterpoly}
		\mathcal{P}_{\underline{k}}(x)= \frac{1}{k!} \sum_{\sigma \in \hbox{perm}(k)}z_{j_{\sigma(1)}}(x)\cdot \cdot \cdot z_{j_{\sigma(k)}}(x) ,
	\end{equation}
	where $k=|\underline k|$ (see pag. 111 in \cite{GHS}).
	
	An example of axially monogenic functions is given by the so-called Clifford-Appell
	polynomials. These are homogeneous polynomials of degree $k$, that were introduced in
\cite{CFM, CMF} and can be defined as follows:
	\begin{equation}
		\label{cliffApp}
		P_k^n(x)= \sum_{\ell=0}^{k} \mathcal{C}_\ell^k(n) x^{k-\ell} \overline{x}^\ell, \qquad \mathcal{C}_\ell^k(n):= \binom{k}{\ell} \frac{\left(\frac{n+1}{2} \right)_{k-\ell} \left( \frac{n-1}{2} \right)_\ell}{(n)_k}.
	\end{equation}
where  $(n)_k=\Gamma(n+k)/\Gamma(n)$ is the  Pochhammer symbol.
Since $ \sum_{\ell=0}^{k} \mathcal{C}_{\ell}^k(n)=1$ we have
\begin{equation}
\label{Real}
P_k^n(x_0)=x_0^k.
\end{equation}

	We point out that in \cite{AKS3}, the authors demonstrated that Clifford-Appell polynomials
	are a basis for the space of axially monogenic functions, while
	in \cite{ACDDS}, the authors utilized Clifford-Appell polynomials
	to extend Schur analysis techniques to the realm of axially monogenic functions.
	
	We observe that	it is possible to write the Clifford-Appell polynomials in terms of Fueter polynomials.
	This can be achieved by using \cite[Prop. 9.21]{GHS}, that states that any homogeneous monogenic polynomial of degree $k$ can be written as a
	$ \mathbb{R}_n$-linear combination of Fueter polynomials.
	\begin{proposition}
		\label{rel}
		Let $n$, $k \in \mathbb{N}$ and $x \in \mathbb{R}^{n+1}$. Then for $ \underline{k}=(0,k_1,..., k_n)$ we can write the Clifford-Appell polynomials in terms of Fueter polynomials, i.e.
		\begin{equation}
			\label{rel1}
			P^n_k(x)= \sum_{| \underline{k}|=k} \frac{1}{\underline{k}!} \mathcal{P}_{\underline{k}}(x)\left( \nabla^{\underline{k}} P^n_k \right)(0),
		\end{equation}
		where $ \nabla^{\underline{k}}:= \partial_{x_1}^{k_1}\ldots \partial_{x_n}^{k_n}$ is the so-called Nabla operator.
	\end{proposition}
	
	\begin{remark}
		Observe that even though the Fueter polynomials $\mathcal{P}_{\underline{k}}(x)$ defined in (\ref{fueterpoly}) are not axially monogenic, the linear combination of such polynomials as in
		(\ref{rel1}) give rise to Clifford-Appell polynomials that are axially monogenic functions, see the generic result \cite{green}.
	\end{remark}
	
	\begin{definition}[The monogenic Cauchy kernel]
		{\rm
		  The monogenic Cauchy kernel on $\mathbb{R}^{n+1}$ is defined as
			\begin{equation}\label{Caukermon}
				{G}(x)=\frac{1}{\Sigma_{n+1}}\,\, \frac{\overline{x}}{|x|^{n+1}},\ \ \ \ \ \ \ \ x\in \mathbb{R}^{n+1}\setminus \{0\},
			\end{equation}
			where
			\begin{equation}\label{COSTSIGMA}
				\Sigma_{n+1}=\frac{2 \pi^{(n+1)/2}}{\Gamma(\frac{n+1}{2})}
			\end{equation}
			is the area of the unit sphere in $\mathbb{R}^{n+1}$.
		}
	\end{definition}

	\begin{theorem}[Monogenic Cauchy formula]
		\label{CF}
		Let $f$ be a left monogenic function in an open set that contains $\overline{W}$. We suppose that $W$ is a four-dimensional compact, oriented manifold with smooth boundary $\partial W$, then, for $x \in W$, we have
		\begin{equation}
			\label{cr}
			f(x)= \frac{1}{2 \pi} \int_{\partial W} G(y-x)Dyf(y).
		\end{equation}
		For a right monogenic function we have
\begin{equation}
	\label{nH}
		f(x)= \frac{1}{2 \pi} \int_{\partial W}f(y) DyG(y-x),
\end{equation}
		where the differential form is given by
$Dy= \sum_{j=0}^n (-1)^j e_j d \widehat{y}_j$ where
$$d \widehat{y}_j=dy_0\wedge \dots \wedge dy_{j-1}\wedge dy_{j+1}\wedge \dots \wedge dy_n.
$$
	\end{theorem}
	\begin{remark}
Formulas \eqref{cr} and \eqref{nH} also hold for left and right axially monogenic functions, as the sets of left and right axially monogenic functions are subsets of the set of monogenic functions.
	\end{remark}
	
	A connection between slice monogenic functions and axially monogenic functions is given by the Fueter-Sce theorem, see \cite{Sce}.
	
	\begin{theorem}[Fueter-Sce mapping theorem (also called Fueter-Sce extension theorem)]
		\label{FS1}
		Let $n$ be an odd number and set $h_n:=(n-1)/2$. We assume that $f_0(z)=\alpha(u,v)+i \beta(u,v)$ is a holomorphic function of one variable defined in a domain $\Pi$
		in the upper-half complex plane. Lets us consider
		$$ U_{\Pi}:= \{x=x_0+ \underline{x} \ \  : \ \ (x_0, | \underline{x}|) \in \Pi\},$$
		 the open set induced by $\Pi$ in $ \mathbb{R}^{n+1}$. The operator $T_{FS1}$, called slice operator, and defined by
		$$
		f(x)=T_{FS1}(f_0):= \alpha(x_0, | \underline{x}|)+ \frac{\underline{x}}{| \underline{x}|} \beta(x_0, | \underline{x}|), \ \ \ x\in  U_{\Pi}
		$$
		maps the set of holomorphic functions in the set of slice hyperholomorphic functions. Furthermore,
the operator
$$
T_{FS2}:=\Delta^{h_n}_{n+1}
$$
maps the slice hyperholomorphic functions $f(x)=T_{FS1}(f_0)$ into the set of axially monogenic function, i.e.,
 the function
		$$
		\breve{f}(x):=\Delta^{h_n}_{n+1} \left(\alpha(x_0, | \underline{x}|)
+ \frac{\underline{x}}{| \underline{x}|} \beta(x_0, | \underline{x}|) \right), \ \ \ x\in  U_{\Pi}
		$$
		is in the kernel of the Dirac operator, i.e., $D\breve{f}(x)=0$ for $x\in  U_{\Pi}$.
	\end{theorem}

	\begin{remark}
		Theorem \ref{FS1} was first proved for quaternions in 1934, see \cite{Fueter}. In the late 1950s, the theorem was extended to Clifford algebras in the case of odd dimensions, see \cite{ColSabStrupSce, Sce}. Theorem \ref{FS1} was generalized to all dimensions in 1997. In this generalization, the operator $\Delta^{h_n}_{n+1}$ is a fractional operator, defined through the Fourier multipliers, see \cite{TaoQian1, TAOBOOK}.
	\end{remark}
	\begin{remark}
		The requirement, in the Fueter or Fueter-Sce mapping theorem,
		which typically considers holomorphic functions defined only on the upper-half complex plane,
		can be relaxed. Moreover, our focus is on the second map, i.e. $\Delta^{(n-1)/2}_{n+1} $, of the  Fueter-Sce mapping theorem,
		where we consider slice hyperholomorphic functions defined on axially symmetric open sets that can,
		in principle, intersect the real line. This issue is addressed by considering functions of the form
		$$
		f(x) = \alpha(u,v) + I\beta(u,v)\qquad \text{for } x = u + I v\in U
		$$
		where the functions $\alpha, \beta: \mathcal{U}\to \mathbb{R}_n$ satisfy the compatibility condition (\ref{EO}).
	\end{remark}
	
	\begin{remark}
		In \cite{DDG1} it was proved that the kernel of the Fueter-Sce map is given by the vector space of
polynomials with real coefficients in $x$ with degree at most $n-2$.
	\end{remark}
	A direct application of the second map $T_{FS2}:=\Delta^{h_n}_{n+1}$ of the Fueter-Sce mapping theorem
to the Cauchy kernels of slice hyperholomorphic functions  leads to the following result, see \cite{CSS},
that leads to an integral representation form of the Fueter-Sce mapping theorem.
	
	\begin{theorem}\label{Laplacian_comp}
		Let $n$ be an odd number set $h_n:=(n-1)/2$ and consider $x$, $s\in \rr^{n+1}$ be such that $x\not\in [s]$ and  recall that $\mathcal{Q}_{c,s}(x)^{-1}$ is defined in (\ref{QCSX}).
		Then we have:
		\begin{itemize}
			\item[(a)]
			Applying the second map of Fueter-Sce mapping theorem, i.e. $\Delta_{n+1}^{h_n}$, where
			$$
			\Delta_{n+1}=\sum_{\ell=0}^n \frac{\partial^2}{\partial x_\ell^2},
			$$
			to the left slice hyperholomorphic Cauchy kernel in  form II we get
			\begin{equation}\label{hLaplacian}
				\Delta^{h_n}_{n+1} S_L^{-1}(s,x)=\gamma_n(s-\bar x) \mathcal{Q}_{c,s}(x)^{-\frac{n+1}{2}},
			\end{equation}
			where the constants $\gamma_n$ are defined as
			\begin{equation}\label{gamman}
				\gamma_n:=\left[\Gamma\left(\frac{n+1}{2}\right)\right]^2 2^{n-1}(-1)^{\frac{n-1}{2}}.
			\end{equation}
			\item[(b)]
			Applying the operator $\Delta_{n+1}^{h_n}$ to the right slice hyperholomorphic Cauchy kernel in  form II
			we obtain
			\begin{equation}\label{hLaplacianR}
				\Delta^{h_n}_{n+1}S_R^{-1}(s,x)=\gamma_n  \mathcal{Q}_{c,s}(x)^{-\frac{n+1}{2}}(s-\bar x).
			\end{equation}
		\end{itemize}
		
	\end{theorem}
	
	\begin{remark}
		The previous result was generalized to all dimensions in \cite{CMQS}, by means of the Fourier multipliers.
	\end{remark}

	\begin{definition}[The $F_n$-kernel]
		\label{Fkernel}
		Let $n$ be an odd number  and assume that $x$, $s\in \rr^{n+1}$ be such that $s\not\in[x]$.
		Then we define the $F_n^L$-kernel as
		\begin{equation}
			\label{FFL}
			F_n^L(s,x):=\gamma_n(s-\bar x)(s^2-2{\rm Re}(x)s +|x|^2)^{-\frac{n+1}{2}},
		\end{equation}
		and the $F_n^R$-kernel as
		\begin{equation}
			\label{FFR}
			F_n^R(s,x):=\gamma_n(s^2-2{\rm Re}(x)s +|x|^2)^{-\frac{n+1}{2}}(s-\bar x),
		\end{equation}
		where the constant $\gamma_n$ are given in \eqref{gamman}.
	\end{definition}
	
	\begin{proposition}	\label{reg2}
		Let $x$, $s\in \rr^{n+1}$ be such that $x\not\in [s]$.
		Then the left $F_n$-kernel $F_n^L(s,x)$
		is a right slice hyperholomorphic function in the variable $s$ and left axially monogenic in the variable $x$.
		The right $F_n$-kernel $F_n^R(s,x)$
		is a left slice hyperholomorphic function in the variable $s$ and right axially monogenic in the variable $x$.
	\end{proposition}

	A series expansion of the $F_n$-kernel in terms of the Clifford-Appell polynomials was proved in \cite{CDS1}.
	\begin{proposition}
		\label{exseries}
		Let $n$ be an odd number and $h_n:=(n-1)/2$. Then for $x$, $s \in \mathbb{R}^{n+1}$ such that $|x|<|s|$ we can write the expansion in series of $F_n^L(s,x)$ as
		$$ F_n^{L}(s,x)=  \gamma_n \sum_{m=2h_n}^{\infty}  \binom{m}{m-2h_n} P^n_{m-2h_n}(x)s^{-1-m},$$
		and  the expansion in series of $F_n^R(s,x)$ as
		$$ F_n^{R}(s,x)=  \gamma_n \sum_{m=2h_n}^{\infty} \binom{m}{m-2h_n} s^{-1-m} P^n_{m-2h_n}(x).$$
		where $P^n_{m-2h_n}(x)$ are defined in \eqref{cliffApp}.
	\end{proposition}
This is based on the application of the Fueter-Sce-map to the monomial $x^m$, with $m \geq 0$, i.e.
		\begin{equation}
			\label{app11}
			\Delta_{n+1}^{h_n} x^m= \gamma_n \binom{m}{m-2h_n}P^n_{m-2h_n}(x), \ \ \ \ m \geq 2h_n,
		\end{equation}
		see \cite{DDG, DDG1}.
In the sequel it will play a crucial role the kernel of the Fueter-Sce map, see \cite{DDG1}.

\begin{lemma}
\label{kk}
Let $n$ be an odd number and set $h_n= \frac{n-1}{2}$. Let $U$ be a connected slice Caucy domain and $f \in \mathcal{SH}_L(U)$ (resp.$f \in \mathcal{SH}_R(U)$). Then $f$ belongs to $\operatorname{ker} (\Delta_{n+1}^{h_n})$ if and only if $ f(x)= \sum_{i=0}^{n-2} x^i \alpha_i$ (resp. $ f(x)= \sum_{i=0}^{n-2}\alpha_i x^i $) where $ \{\alpha_i\}_{0 \leq i \leq n-2} \subseteq \mathbb{R}_n$.
\end{lemma}
	
	By combining Theorem \ref{Laplacian_comp} and the Cauchy Theorem for slice hyperholomoprhic functions we get the following fundamental result.
	\begin{theorem}[The Fueter-Sce mapping theorem in integral form]
		\label{intform}
		Let $n$ be an odd number and $h_n:=(n-1)/2$. Let $U\subset\mathbb{R}^{n+1}$ be a bounded slice Cauchy domain, let $I\in\mathbb{S}$ and set  $ds_I=ds (-I)$.
		\begin{itemize}
			\item
			If $f$ is a left  slice hyperholomorphic function on a set that contains $\overline{U}$, then
			$\breve{f}(x)=\Delta^{h_n}_{n+1}f(x)$
			is left monogenic and it admits the integral representation
			\begin{equation}
				\label{ff}
				\breve{f}(x)=\frac{1}{2 \pi}\int_{\pp (U\cap \mathbb{C}_I)} F_n^L(s,x)ds_I f(s) \qquad\text{for any }\ \  x\in U.
			\end{equation}
			\item
			If $f$ is a right  slice hyperholomorphic function on a set that contains $\overline{U}$, then
			$\breve{f}(x)=\Delta^{h_n}_{n+1}f(x)$
			is right monogenic and it admits the integral representation
			\begin{equation}
				\breve{f}(x)=\frac{1}{2 \pi}\int_{\pp (U\cap \mathbb{C}_I)} f(s)ds_I F_n^R(s,x)\qquad\text{for any }\ \  x\in U.
			\end{equation}
		\end{itemize}
		The integrals are independent of $U$, the imaginary unit $I\in \mathbb{S}$, and the kernel of the operator $\Delta^{h_n}_{n+1}$.
	\end{theorem}
\begin{proof}
We focus exclusively on the case of the left slice hyperholomorphic functions. Formula \eqref{ff} follows from \eqref{hLaplacian} and the Cauchy Theorem (see Theorem \ref{Cauchy}). Independence from the set $U$ and from the imaginary unit $I \in \mathbb{S}$ also follows from the Cauchy Theorem. We now demonstrate the independence from the kernel of the operator $\Delta^{h_n}_{n+1}$. Let $f_1$ be a another slice hyperholomorphic function that contains $\bar{U}$ and such that $\breve{f}(x)= \Delta_{n+1}^{h_n} f_1(x)$. Thus, by Lemma \ref{kk} we have that
$$ f_1(x)=f(x)+ \sum_{i=0}^{n-2} x^i a_i, \qquad \{a_i\}_{0 \leq i \leq n-2} \subseteq \mathbb{R}_n.$$
Hence by \eqref{ff} we have
\begin{eqnarray*}
\breve{f}_1(x)&=& \int_{\partial (U \cap \mathbb{C}_I)}F_n^L(s,x) ds_I f_1(s)\\
&=&  \int_{\partial (U \cap \mathbb{C}_I)}F_n^L(s,x) ds_I f(s)+ \sum_{i=0}^{n-2}\int_{\partial (U \cap \mathbb{C}_I)}F_n^L(s,x) ds_I s^i a_i\\
&=& \breve{f}(x)+  \sum_{i=0}^{n-2} \Delta_{n+1}^{h_n}s^i a_i\\
&=& \breve{f}(x)
\end{eqnarray*}
This proves the independence from the kernel of the operator $\Delta^{h_n}_{n+1}$.
\end{proof}

	\subsection{The F-functional calculus}

	The definition of $F$-resolvent operators are suggested by the Fueter-Sce mapping theorem in integral form and it was introduced in \cite{CSS}.
	From the definition of the Clifford-Appell polynomials defined in (\ref{cliffApp}) we give the following notion of Clifford-Appell operators.
	\begin{definition}[The Clifford-Appell operators]
		Let $n$, $k\in \mathbb{N}$ and
		$T\in\mathcal{BC}^{\small 0,1}(V_n)$. We define the Clifford-Appell operators $P_k^n(T)$ as
		\begin{equation}
			\label{Appellope}
			P_k^n(T)=\sum_{\ell=0}^{k} \mathcal{C}_{\ell}^k(n) T^{k-\ell} \overline{T}^{\ell}, \qquad \mathcal{C}_{\ell}^k(n)= \binom{k}{\ell} \frac{\left(\frac{n+1}{2}\right)_{k-\ell} \left(\frac{n-1}{2} \right)_{\ell}}{(n)_{k}}.
		\end{equation}
	\end{definition}
	\begin{remark}
	\label{conv}
	We assume by convention that $P^n_k(T)=0$ if $k<0$.
\end{remark}

	\begin{definition}[$F$-resolvent operators]
		
		Let $n$ be an odd number and
		$T\in\mathcal{BC}^{\small 0,1}(V_n)$.
		For $s\in \rho_S(T)$
		we define the left $F$-resolvent operator by
		\begin{equation}\label{FresBOUNDL}
			F_n^L(s,T):=\gamma_n(s\mathcal{I}-\overline{ T})(s^2\mathcal{I}-s(T+\overline{T})+T\overline{T})^{-\frac{n+1}{2}},
		\end{equation}
		and the right $F$-resolvent operator by
		\begin{equation}\label{FresBOUNDR}
			F_n^R(s,T):=\gamma_n(s^2\mathcal{I}-s(T+\overline{T})+T\overline{T})^{-\frac{n+1}{2}}(s\mathcal{I}-\overline{ T}),
		\end{equation}
		where the constants $\gamma_n$ are given by (\ref{gamman}).
	\end{definition}
	By Proposition \ref{reg2} it follows that
	\begin{lemma}
		\label{reg21}
		Let $n$ be an odd number and $T\in\mathcal{BC}^{\small 0,1}(V_n)$. Then, for $s\in \rho_S(T)$,
		the left (resp right) $F$-resolvent operator is a $\mathcal{B}(V_n)$-valued left (resp. right) slice hyperholomorphic function in the variable $s$.
	\end{lemma}
	
	The left (resp. right)  $F$-resolvent operator admits an expansion in series in terms of Clifford-Appell operators, see \cite{CDS1}.
	\begin{proposition}
		\label{exseriesOPR}
		Let $n$ be an odd number and $h_n:=(n-1)/2$.
		Let
		$T\in\mathcal{BC}^{\small 0,1}(V_n)$.
		Then, for $s \in \mathbb{R}^{n+1}$ such that $\|T\|<|s|$, we can write the following expansion in series of the left $F$-resolvent operator
		$$
		F_n^L(s,T)= \gamma_n\sum_{m=2h_n}^{\infty} \binom{m}{m-2h_n} P^n_{m-2h_n}(T)s^{-1-m},
		$$
		and of the right $F$-resolvent operator	
		$$F_n^R(s,T)= \gamma_n \sum_{m=2h_n}^{\infty} \binom{m}{m-2h_n} s^{-1-m} P^n_{m-2h_n}(T).$$
	\end{proposition}

	Now, we provide the definition of the $F$-functional calculus for bounded operators with commuting components.

	\begin{definition}
		\label{Ffun}	
		Let $n$ be an odd number and set $h_n:=(n-1)/2$. Let $T\in \mathcal{BC}^{1}(V_n)$ be such that its components $T_{i}$, for $ i=1,...,n$, have real spectra. Let $U$ be a bounded slice Cauchy domain as in
		Definition \ref{SHONTHEFS} and set $ds_I=ds/I$.
		\begin{itemize}
			\item
			For any  $f\in \mathcal{SH}^L_{\sigma_S(T)}(U)$ we set
			$\breve{f}(x)=\Delta^{h_n}_{n+1}f(x)$.
			We define the $F$-functional calculus for the operator $T$ as
			\begin{equation}\label{integ311_mon}
				\breve{f}_L(T)=\frac{1}{2\pi}\int_{\pp(U\cap \mathbb{C}_I)} F_n^L(s,T) \, ds_I\, f(s).
			\end{equation}
			
			\item
			For any  $f\in \mathcal{SH}^R_{\sigma_S(T)}(U)$ we set
			$\breve{f}(x)=\Delta^{h_n}_{n+1}f(x)$. We define the $F$-functional calculus for the operator $T$ as
			\begin{equation}\label{integ311_monRIGHT}
				\breve{f}_R(T)=\frac{1}{2\pi}\int_{\pp(U\cap \mathbb{C}_I)} f(s) \, ds_I\, F_n^R(s,T).
			\end{equation}
		\end{itemize}
	\end{definition}
	\begin{remark} The $F$-functional calculus is well defined because
		the integrals are independent of $U$, the imaginary unit $I\in \mathbb{S}$, and the kernel of the operator $\Delta^{h_n}_{n+1}$.
	\end{remark}
Due to the lack of commutativity, the right and the left versions of the $F$-functional calculus do not yield the same operator. The only exception occurs when considering the $F$-functional calculus for intrinsic slice hyperholomorphic functions.

 \begin{proposition}
\label{FLR}
Let $n$ be an odd number and set $h_n:=(n-1)/2$. We assume that  $T \in \mathcal{BC}^{0,1}(V_n)$ and $f \in \mathcal{N}_{\sigma_S(T)}(U)$, where $U$ is a bounded slice Cauchy domain as in
Definition \ref{SHONTHEFS}. Then we have
$$\frac{1}{2\pi}\int_{\pp(U\cap \mathbb{C}_I)} F_n^L(s,T) \, ds_I\, f(s)=\frac{1}{2\pi}\int_{\pp(U\cap \mathbb{C}_I)} f(s) ds_I F_n^R(s,T).$$
 \end{proposition}
\begin{proof}
From the definition of the $F$-resolvent operators we have that $F_n^L(s,T)=S^{-1}_L(s,T) \mathcal{Q}_{c,s}^{- \frac{n-1}{2}}(T)$ and $F_n^R(s,T)= \mathcal{Q}_{c,s}^{- \frac{n-1}{2}}(T)S^{-1}_R(s,T)$. We observe that $\mathcal{Q}_{c,s}^{- \frac{n-1}{2}}(T)$ is a $ \mathcal{B}(V_n)$-valued intrinsic slice hyperholomorphic function. We set $g(s,T):=\mathcal{Q}_{c,s}^{- \frac{n-1}{2}}(T) f(s)$, that is also a $\mathcal{B}(V_n)$-valued intrinsic slice hyperholomorphic function. Thus by Theorem \ref{SLR} we have
 \begin{eqnarray*}
\frac{1}{2\pi}\int_{\pp(U\cap \mathbb{C}_I)} F_n^L(s,T) \, ds_I\, f(s)&=& \frac{1}{2\pi}\int_{\pp(U\cap \mathbb{C}_I)} S^{-1}_L(s,T) \, ds_I\, g(s,T)\\
&=&\frac{1}{2\pi}\int_{\pp(U\cap \mathbb{C}_I)} g(s,T) ds_I S^{-1}_R(s,T) \, \\
&=& \frac{1}{2\pi}\int_{\pp(U\cap \mathbb{C}_I)}  f(s) ds_I \mathcal{Q}_{c,s}^{- \frac{n-1}{2}}(T)S^{-1}_R(s,T) \\
&=& \frac{1}{2\pi}\int_{\pp(U\cap \mathbb{C}_I)}  f(s) ds_I F_n^R(s,T).
 \end{eqnarray*}
This proves the result.
\end{proof}

	The resolvent equation of the quaternionic version of the $F$-functional calculus can be found in \cite{CG}.  Its extension to $F$-functional calculus based on Clifford algebras $\mathbb{R}_n$, for $n$ odd,  is highly non-trivial and was proved in \cite{CDS, CDS1}.
	This resolvent equation is important in proving the existence of Riesz projectors for this functional calculus. We now recall the $F$-resolvent equation, as presented in
	\cite[Lemma 3.1]{CDS}, which will be used subsequently.
	
	\begin{theorem}
		\label{resolvent}
		Let $n$ be an odd number and set $h_n:=(n-1)/2$. We assume that $T \in \mathcal{BC}^{0,1}(V_n)$. Then for $p$, $s \in \rho_S(T)$ the following equation holds
		\begin{eqnarray}
			&&F_n^R(s,T)S^{-1}_L(p,T)+S^{-1}_R(s,T)F_n^L(p,T)\\
			\nonumber
			&&+\gamma_n \left[ \sum_{i=0}^{h_n-2}{Q}_{c,s}^{-h_n+i+1}(T)S^{-1}_R(s,T)S^{-1}_L(p,T) {Q}_{c,p}^{-i-1}(T)+\sum_{i=0}^{h_n-1} {Q}_{c,s}^{-h_n+i}(T) {Q}_{c,p}^{-i-1}(T)\right]\\
			\nonumber
			&=&\left\{ \left[F_n^R(s,T)-F_n^L(p,T)\right]p- \bar{s}\left[F_n^R(s,T)-F_n^L(p,T)\right]\right\}(p^2-2s_0p+|s|^2)^{-1},
		\end{eqnarray}
		where $\gamma_n$ are given by (\ref{gamman}).
The left and the right $S$-resolvent operators are defined in \eqref{ZwoaDreiVia} and \eqref{ZwoaDreiVia1},
respectively, while the left and the right $F$-resolvent operators are given in \eqref{FFL} and \eqref{FFR}, respectively.
Finally, the operator $ \mathcal{Q}_{c,s}(T) $ is defined in \eqref{QCST}.
	\end{theorem}
	
	\begin{remark}
		The $F$-functional calculus has been also studied for unbounded operators, see \cite{CSF}. In \cite{CPS1} the $H^{\infty}$-functional calculus
		within the framework of the $F$-functional calculus in the quaternionic setting has been established.
	\end{remark}

\subsection{The monogenic functional calculus}
	
	The monogenic functional calculus based on the monogenic Cauchy formula has been investigated by A. McIntosh and his collaborators (see \cite{JBOOK, JM}). To recall the definition of this functional calculus, we first need to define the monogenic spectrum and the resolvent operator in this setting that is related with the monogenic Cauchy kernel, see \eqref{Caukermon}.

	\begin{definition}[Monogenic spectrum and resolvent operator]
		Let $n$ be an odd number and $T \in \mathcal{BC}^{1}(V_n)$ be such that its components  $T_i$, $i=1,...,n$ have real spectra. We define the monogenic spectrum as
		\begin{equation}\label{monspect}
			\gamma(T):= \left\{ (0, y_1,..., y_n) \in \mathbb{R}^{n+1} \, : \, 0 \in \sigma \left(\sum_{i=1}^{n} (y_i \mathcal{I}-T_i)^2 \right) \right\}.
		\end{equation}
		For $y \notin \gamma(T)$ we define the monogenic resolvent operator as
		\begin{equation}\label{MONGRESOLV}
			G(y,T)=\frac{1}{\Sigma_n} \left( \bar{y}\mathcal{I}+T\right) \left(y_0^2 \mathcal{I}+ \sum_{i=1}^{n} (y_i \mathcal{I}-T_i)^2 \right)^{-\frac{n+1}{2}},
		\end{equation}
		where the constants $\Sigma_n$ are given by (\ref{COSTSIGMA}).
	\end{definition}
\begin{remark}\label{RESOLVprime}
 We point out that, for bounded commuting operators, we can also assume that $T_0\not=0$ and that one of the $T_i$ for $i=1,\cdots,n$ is the zero operator. For example if we assume $T_n=0$ we would get
\begin{equation}\label{MONGRESOLVprime}
			G(y,T)=\frac{1}{\Sigma_n} \left( \bar{y}\mathcal{I}+T\right) \left(  \sum_{i=0}^{n-1} (y_i \mathcal{I}-T_i)^2+y_n^2\mathcal{I} \right)^{-\frac{n+1}{2}}
		\end{equation}
 with obvious adaptation of the definition of $\gamma(T)$.
\end{remark}
	\begin{remark} We recall that $y \mapsto G(y,T)$, for $y\not\in \gamma(T)$,
		is a monogenic function $\mathcal{B}(V_n)$-operator valued.
	\end{remark}
	\begin{definition}\label{locmongspetrum}
		We say that a function $f$ is locally left (resp. right) monogenic on $\gamma(T)$ with $T \in \mathcal{BC}^{1}(V_n)$ if there exists an open set $W \subset \mathbb{R}^{n+1}$ containing $\gamma(T)$ whose boundary is a rectifiable $(n-1)$-cell and such that $f$ is left (resp. right) monogenic in every connected component of $W$. We denote the set of locally left (resp. right) monogenic functions by $ \mathcal{M}_L(\gamma(T))$ (resp. $ \mathcal{M}_R(\gamma(T))$).
	\end{definition}
	\begin{remark}
		In the sequel we will consider only axially monogenic functions so we will use the symbols
		$ \mathcal{AM}_L(\gamma(T))$ (resp. $ \mathcal{AM}_R(\gamma(T))$).
	\end{remark}
	\begin{definition}[Monogenic functional calculus]
		\label{McI}
		Let $n$ be an odd number and let $W \subset \mathbb{R}^{n+1}$ be as in Definition \ref{locmongspetrum}.
		Assume that $f \in \mathcal{AM}_L(\gamma(T))$ and let $T \in \mathcal{BC}^{1}(V_n)$ be such that its components  $T_i$, $i=1,...,n$ have real spectra.
		Then we define the monogenic functional calculus as
		\begin{equation}\label{defmonogfc}
			f^\sharp_L(T):=\frac{1}{2 \pi} \int_{\partial W} G(x,T) Dxf(x),
		\end{equation}
		For $f \in \mathcal{AM}_R(\gamma(T))$ we have
		$$f^\sharp_R(T):=\frac{1}{2 \pi} \int_{\partial W}f(x)DxG(x,T),$$
		where the differential form $Dx$ is as in the monogenic Cauchy formula \eqref{cr}.
		
	\end{definition}
	\begin{remark}
		The above definition is well posed, since the integrals that define the monogenic functional calculus do not depend on the open set $W$. This follows from the Cauchy formula for
		monogenic functions, see Theorem \ref{CF}.
	\end{remark}
	\begin{remark}
		We recall that the monogenic functional calculus can be extended to $n$-tuples of noncommuting operators, specifically for operators $T\in \mathcal{B}^{1}(V_n)$.
		However, this extension requires the modification of the definition of the monogenic resolvent operators  $G(y,T)$ given in (\ref{MONGRESOLV}). By employing the plane wave decomposition of the monogenic Cauchy kernel $G$ (see \cite[Proposition 3.4]{JBOOK}), the monogenic resolvent operator $G(y,T)$ is defined according to \cite[Lemma 4.14]{JBOOK}.
		
		Since the $F$-functional calculus is defined for $n$-tuples of commuting operators, i.e.,
		for $T\in \mathcal{BC}^{1}(V_n)$, and in order to compare this calculus with the monogenic functional calculus, we provide the definitions and properties of the monogenic functional calculus specifically for $T\in \mathcal{BC}^{1}(V_n)$. In this context, both the monogenic resolvent (\ref{MONGRESOLV}) and the monogenic spectrum (\ref{monspect}) assume a particularly simple form.
	\end{remark}

	In the sequel we will need the following property of the monogenic functional calculus, see \cite[Prop.4.21]{JBOOK}.
	\begin{proposition}
		\label{McIapp}
		Let $n$ be an odd number and let $W \subset \mathbb{R}^{n+1}$ be as in Definition \ref{locmongspetrum}.
		Assume the operator $T \in \mathcal{BC}^{1}(V_n)$ be such that its components  $T_i$, $i=1,...,n$ have real spectra. Then, we have
		$$\frac{1}{2 \pi}\int_{\partial W} G(s,T) Ds \mathcal{P}_{\underline{k}}(s)=\mathcal{P}_{\underline{k}}(T),$$
		where $\mathcal{P}_{\underline{k}}(T)$ is the operator obtained by replacing  the operator $T$ in the definition of the Fueter polynomials, see \eqref{fueterpoly}.
	\end{proposition}

	\subsection{Equivalence between monogenic and the F-functional calculi}
	
	We show that the monogenic functional calculus and the $F$-functional calculus lead to the same operator.
	First we show the equivalence of the two functional calculi
	on simple functions that is of its own interest.
	Let us see the connection of the $F$-functional calculus with the Clifford-Appell operators $P_m^n(T)$,
	introduced in \eqref{Appellope}, for $m \geq 0$. To this aim it is important the integral representation involving the left and right $F$-resolvent operators.

	\begin{theorem}
		\label{intrepp}
		Let $n$ be an odd number and set $h_n:=(n-1)/2$.
		Let $T \in \mathcal{BC}^{1}(V_n)$ be such that its components  $T_i$, $i=1,...,n$ have real spectra. Let $U$ be a bounded slice Cauchy domain as in Definition \ref{SHONTHEFS} and set $ds_I=(-I)ds$.
		Then, for $m \in \mathbb{N}_0$, we have
		\begin{equation}
			\label{inteAppe}
			P_m^n(T)= c_{m,n} \int_{\partial(U \cap \mathbb{C}_I)} F_n^L(s,T) ds_I s^{m+2h_n},
		\end{equation}
		and
		\begin{equation}
			\label{inteAppe1}
			P_m^n(T)=c_{m,n} \int_{\partial(U \cap \mathbb{C}_I)} s^{m+2h_n}ds_I F_n^R(s,T),
		\end{equation}
		where
		\begin{equation}\label{costcmn}
			c_{m,n}:= \frac{ m! (2h_n)!}{ 2 \pi \gamma_n (m+2h_n)!}
		\end{equation}
		and $\gamma_n$ are given by (\ref{gamman}).
	\end{theorem}
	\begin{proof}
		We prove only \eqref{inteAppe}, since \eqref{inteAppe1} can be proved by using similar arguments.
  We first consider
  $U= B_{r}(0)$ with $\|T \| <r$ and moreover we take advantage  of the expansion in series of the $F$-resolvent operator in terms of Clifford-Appell operators,
		see Proposition \ref{exseriesOPR}.
		For   $n$ be an odd number, $T\in\mathcal{BC}^{\small 0,1}(V_n)$ and for $s \in \mathbb{R}^{n+1}$ such that $\|T\|<|s|$,
we have the following series expansion of the left $F$-resolvent operator
		$$
		F_n^L(s,T)= \gamma_n \sum_{m=2h_n}^{\infty} \binom{m}{m-2h_n} P^n_{m-2h_n}(T)s^{-1-m}.
		$$
		From the fact that this series is uniform convergent for $s \in \partial B_r(0)$ we obtain
		\begin{eqnarray*}
			c_{m,n} \int_{\partial(B_r(0) \cap \mathbb{C}_I)} F_n^L(s,T) ds_Is^{m+2h_n} &=& c_{m,n} \gamma_n \sum_{\ell=2h_n}^\infty \binom{\ell}{\ell-2h_n} P^n_{\ell-2h_n}(T) \int_{\partial(B_r(0) \cap \mathbb{C}_I)} s^{m+2h_n-1-\ell} ds_I\\
			&=&  c_{m,n} \gamma_n \sum_{\ell=0}^\infty \binom{\ell+2h_n}{\ell} P^n_{\ell}(T) \int_{\partial(B_r(0) \cap \mathbb{C}_I)} s^{-1-\ell+m} ds_I
		\end{eqnarray*}
Since
		$$ \int_{\partial(B_r(0)\cap \mathbb{C}_I)} s^{-1-\ell+m}ds_I= \begin{cases}
			0, \quad \hbox{if} \quad \ell \neq m\\
			2 \pi \quad \hbox{if} \quad \ell=m
		\end{cases}
		$$
		we get
		\begin{equation}
			\label{aux2}
			c_{m,n} \int_{\partial(B_r(0) \cap \mathbb{C}_I)} F_n^L(s,T) ds_Is^{m+2h_n}=2 \pi c_{m,n} \gamma_n \binom{m+2h_n}{m} P^n_m(T)\\
			= P^n_m(T).
		\end{equation}
		Now, we consider a set $U$ that is an arbitrary bounded slice Cauchy domain that contains $\sigma_S(T)$. We also assume that there exists a radius $r$ such that $\bar{U} \subset B_r(0)$. By Lemma \ref{reg21} we know that the left $F$-resolvent operator is right slice hyperholomorphic in $s$ and the monomial $s^{m+2h_n}$ is left slice hyperholomorphic on $ B_r(0) \setminus U$. By  the Cauchy's integral theorem, see Theorem \ref{Cif}, we get
		\begin{eqnarray*}
			&&-c_{m,n} \int_{\partial(B_r(0) \cap \mathbb{C}_I)} F_n^L(s,T) ds_I s^{m+2h_n}
			+c_{m,n} \int_{\partial(U \cap \mathbb{C}_I)} F_n^L(s,T) ds_I s^{m+2h_n}\\
			&=& c_{m,n} \int_{\partial( (B_r(0) \setminus U) \cap \mathbb{C}_I)} F_n^L(s,T) ds_I s^{m+2h_n}\\
			&=&0.
		\end{eqnarray*}
		Finally, by \eqref{aux2} we get
		$$
		c_{m,n} \int_{\partial(U \cap \mathbb{C}_I)} F_n^L(s,T) ds_I s^{m+2h_n}
		= c_{m,n} \int_{\partial(B_r(0) \cap \mathbb{C}_I)} F_n^L(s,T) ds_I s^{m+2h_n}=P^n_m(T).
		$$
	\end{proof}

	It is also possible to write the Clifford-Appell operators by means of an integral representation involving the monogenic resolvent operator.
	
	\begin{proposition}
		\label{newinto}
		Let $m \in \mathbb{N}_0$, $n$ be an odd number and let $W \subset \mathbb{R}^{n+1}$ be as in Definition \ref{locmongspetrum}.
		Assume that $T \in \mathcal{BC}^{1}(V_n)$ be such that its components  $T_i$, $i=1,...,n$ have real spectra. Then we have
		\begin{equation}
			\frac{1}{2  \pi}\int_{\partial W}G(s,T)Ds P_m^n(s)= P_m^n(T).
		\end{equation}
	\end{proposition}
	\begin{proof}
		By Proposition \ref{rel} and Proposition \ref{McIapp} we have
		\begin{eqnarray*}
			\frac{1}{2  \pi}	\int_{\partial W}G(s,T)Ds P_m^n(s)&=& \sum_{| \underline{k}|=m} \frac{1}{\underline{k}!} \left(\frac{1}{2 \pi} \int_{\partial W} G(s,T) Ds \mathcal{P}_{\underline{k}}(s) \right) \left(\nabla^{\underline{k}} P_m^n(s) \right)(0)\\
			&=& \sum_{| \underline{k}|=m} \frac{1}{\underline{k}!} \mathcal{P}_{\underline{k}}(T) \left(\nabla^{\underline{k}} P_m^n(s) \right)(0)\\
			&=&  P_m^n(T).
		\end{eqnarray*}
	\end{proof}
	
	\begin{theorem}
		Let $n$ be an odd number, set $h_n:=(n-1)/2$, and $T \in \mathcal{BC}^{1}(V_n)$ be such that its components  $T_i$, $i=1,...,n$ have real spectra. We consider the left slice hyperholomorphic polynomial  $P_L(s)=\sum_{\ell=0}^{N} s^{\ell}a_{\ell}$ (resp. $P_R(s)=\sum_{\ell=0}^{N} a_{\ell}s^{\ell})$ where $ a_{\ell} \in  \mathbb{R}_n$, for $\ell=0,...,N$.
		\begin{itemize}
			\item We assume that $U$ is a bounded slice Cauchy domain as in Definition \ref{SHONTHEFS}, we set $ds_I=(-I)ds$, and let  $\breve{P}_L(x)=\Delta^{h_n}_{n+1}P_L(x)$ (resp. $\breve{P}_R(x)=\Delta^{h_n}_{n+1}P_R(x)$).
			We consider the $F$-functional calculus defined in (\ref{integ311_mon}):
			$$ \breve{P}_L(T)=\int_{\partial(U \cap \mathbb{C}_I)}F_n^L(s,T) ds_I P_L(s), \quad \left(\hbox{resp.} \quad \breve{P}_R(T)=\int_{\partial(U \cap \mathbb{C}_I)}P_R(s) ds_IF_n^R(s,T) \right).$$
			\item We suppose that $W \subset \mathbb{R}^{n+1}$ is as in Definition \ref{locmongspetrum}, and consider the monogenic functional calculus as defined in (\ref{defmonogfc}):
		\end{itemize}
		$$
		P^{\sharp}_L(T)=\int_{\partial W} G(s,T) Ds \Delta_{n+1}^{h_n} P_L(s),
 \qquad \left(\hbox{resp.} \quad P^{\sharp}_R(T)=\int_{\partial W}\Delta_{n+1}^{h_n}P_R(s)Ds G(s,T)  \right).
		$$
		Then, we have that
		\begin{equation}
			\label{Rpoly}
			\breve{P}_L(T)=P^{\sharp}_L(T), \qquad \left(\hbox{resp.} \quad \breve{P}_R(T)=P^{\sharp}_R(T)\right).
		\end{equation}
	\end{theorem}
	\begin{proof}
		We prove the result for left slice hyperholomorphic polynomials. The proof of the right case can be proved using similar arguments. By \eqref{app11} and Proposition \ref{newinto} we obtain
		\begin{eqnarray}
			\nonumber
			P_L^{\sharp}(T)&=&\frac{1}{2 \pi}\int_{\partial W} G(s,T) Ds \Delta_{n+1}^{h_n} f(s)\\
			\nonumber
			&=& \gamma_n \sum_{m=2h_n}^{\infty}   \binom{m}{m-2h_n} \left(\frac{1}{2 \pi}\int_{\partial W} G(s,T) Ds P^n_{m-2h_n}(s) \right)a_m\\
			\label{aux22}
			&=&\gamma_n \sum_{m=2h_n}^{\infty} \binom{m}{m-2h_n} P^n_{m-2h_n}(T)a_m.
		\end{eqnarray}
		On the other side by Theorem \ref{intrepp} and the independence of $F$-functional calculus from the kernel of $\Delta_{n+1}^{h_n}$ we have
		\begin{eqnarray}
			\nonumber
			\breve{F}_L(T)&=&\int_{\partial(U\cap \mathbb{C}_I)}F_n^L(s,T) ds_I P_L(s)\\
			\nonumber
			&=& \sum_{m=2h_n}^{\infty} \left( \frac{1}{2 \pi} \int_{\partial(U \cap \mathbb{C}_I)}F_n^L(s,T)s^m \right)a_m\\
			\label{aux11}
			&=& \gamma_n \sum_{m=2h_n}^{\infty} \binom{m}{m-2h_n}P^n_{m-2h_n}(T)a_m.
		\end{eqnarray}
		Finally the equality \eqref{Rpoly} follows from the fact that \eqref{aux22} and \eqref{aux11} coincide.
	\end{proof}
	
	The previous case shows that at least on specific polynomial the two functional calculi define the
	same operator. Those results can obviously be extended when we consider series expansions
	centred at the origin  instead of polynomials.	
	\\Now, we show a result that puts in correlation the $F$-resolvent operator and the monogenic resolvent operator and it will be of crucial importance in the rest of the paper.
	\begin{proposition}
	\label{NewR}
	Let $n$ be an odd number and let $W \subset \mathbb{R}^{n+1}$ as in Definition \ref{locmongspetrum}. Assume that $T \in \mathcal{BC}^1(V_n)$ be such that its components $T_i$, $i=1,...,n$ have real spectra. Then for $s \in \mathbb{R}^{n+1}$ we have
	$$ F_n^L(s,T)=\frac{1}{2 \pi}\int_{\partial W} G(\omega, T) D \omega F_n^L(s, \omega).$$
	\end{proposition}
	\begin{proof}
	By Proposition \ref{exseries} and the fact that the series converges uniformly we have
	$$\frac{1}{2 \pi}\int_{\partial W} G(\omega, T) D \omega F_n^L(s, \omega)= \sum_{m=2h_n}^{\infty} \binom{m}{m-2h_n} \left(\frac{1}{2 \pi} \int_{\partial W} G(\omega,T) D \omega P^n_{m-2h_n}(\omega)\right)s^{1-m}.$$
	Finally, by Proposition \ref{newinto} and Proposition \ref{exseriesOPR}, we have
$$ \frac{1}{2 \pi}\int_{\partial W} G(\omega, T) D \omega F_n^L(s, \omega)=\sum_{m=2h_n}^{\infty} \binom{m}{m-2h_n}P^n_{m-2h_n}(T)s^{1-m}=F_n^L(s,T).$$
This proves the result.
	\end{proof}
	
	Now, we prove that the $F$-functional calculus defined in Definition \ref{Ffun} and the monogenic
	functional calculus for bounded operators introduced in Definition \ref{McI} are equivalent.

	\begin{theorem}
		Let $n$ be an odd number and set $h_n:=(n-1)/2$.
		Let $T \in \mathcal{BC}^{1}(V_n)$ be such that its components  $T_i$, $i=1,...,n$ have real spectra.
		
		\begin{itemize}
			\item[(I)] Let $U$ be a bounded slice Cauchy domain as in Definition \ref{SHONTHEFS} and set $ds_I=ds/I$.
			For any  $f\in \mathcal{SH}^L_{\sigma_S(T)}(U)$ we set $\breve{f}(x)=\Delta^{h_n}_{n+1}f(x)$.
			We consider the $F$-functional calculus defined in (\ref{integ311_mon}), i.e.
			\begin{equation}
				\label{FF}
				\breve{f}_L(T)=\frac{1}{2\pi}\int_{\pp(U\cap \mathbb{C}_I)} F_n^L(s,T) \, ds_I\, f(s), \quad \left(\hbox{resp.} \, \,\breve{f}_R(T):=\frac{1}{2 \pi} \int_{\partial U}f(s)ds_I F_n^R(s,T) \right)
			\end{equation}
			
			\item[(II)]
			Let $W \subset \mathbb{R}^{n+1}$ be as in Definition \ref{locmongspetrum}.
			Assume that $\breve{f} \in \mathcal{AM}_L(\gamma(T))$ (resp. $\breve{f} \in \mathcal{AM}_R(\gamma(T))$).
			Consider the monogenic functional calculus as defined in (\ref{defmonogfc}), i.e.
			\begin{equation}
				\label{monoMc}
				f^\sharp_L(T):=\frac{1}{2 \pi} \int_{\partial W} G(x,T) Dx\check{f}(x) \quad \left(\hbox{resp.} \, \, f^\sharp_R(T):=\frac{1}{2 \pi} \int_{\partial W}\check{f}(x)DxG(x,T) \right),
			\end{equation}
		\end{itemize}
		Then, under the above assumptions, the $F$-functional calculus and the monogenic functional calculi for bounded operators are equivalent, i.e.,
		$$f^\sharp_L(T)=\breve{f}_L(T), \qquad \left(\hbox{resp.} \quad f^\sharp_R(T)=\breve{f}_R(T) \right)$$
	\end{theorem}
	\begin{proof}
We prove only formula \eqref{FF}, since formula \eqref{monoMc} can be obtained by using similar arguments. By formula \eqref{monoMc}, Theorem \ref{intform} and the Fubini theorem we have
		\begin{eqnarray}
			\nonumber
			f^\sharp_L(T)&=& \frac{1}{2 \pi}\int_{\partial W} G(\omega,T) d D \omega \left( \frac{1}{2 \pi}\int_{\partial(U \cap \mathbb{C}_I)} F_n^L(s, \omega) d \omega_I f(\omega) \right)\\
			\label{aux33}
			&=& \frac{1}{2 \pi}\int_{\partial(U \cap \mathbb{C}_I)} \left( \frac{1}{2 \pi}\int_{\partial W} G(\omega,T)  D \omega F_n^L(s, \omega)\right) d \omega_I f(\omega),
		\end{eqnarray}
		Finally Proposition \ref{NewR} we have
		$$  f^\sharp_L(T)=\frac{1}{2 \pi}\int_{\partial(U\cap \mathbb{C}_I)} F_n^L(\omega,T)d \omega_I f(\omega)=\breve{f}_L(T).$$
	\end{proof}

	It is therefore possible to define monogenic functions of $n$-tuples of commuting operators through two functional calculi,
each based on distinct notions of the spectrum for quaternionic operators or for paravector operators
in the Clifford  algebra setting.

	\section{Integral representation of axially polyharmonic functions}\label{SEC4}

Polyharmonic functions has been intensively studied in \cite{Aro}, where the authors developed a comprehensive theory of polyharmonic functions, particularly focusing on their analytical properties.
	They investigated the conditions under which polyharmonic functions are unique and the types of boundary conditions that can be imposed to ensure well-posed problems.
In this section our goal is to provide an integral representation for axially polyharmonic functions in the Clifford algebra setting.
Precisely, this class of functions is defined as follows.
	\begin{definition}[Axially polyharmonic function]
		Let $k \geq2$ be an integer. Let $U$ be an open set in $\mathbb{R}^{n+1}$. We say that a function $f:U \subset \mathbb{R}^{n+1} \to \mathbb{R}_n$, of class $\mathcal{C}^{2k}$ is left (resp. right) axially polyharmonic function (of degree $k$) if it is of axial type, see \eqref{axx}, and if
		$$ \Delta^k_{n+1} f(x)=0, \qquad \forall x \in U.$$	
	\end{definition}
The definition above will be used in consideration of the Fueter-Sce mapping theorem. Now, we give some particular example of polyharmonic functions.

	\begin{example}[Axially harmonic polynomials]
		An example of axially harmonic functions, i.e. polyharmonic functions of degree $1$,
		is given by the following polynomials
		\begin{equation}
			\label{Harmo}
			\mathcal H_{k}^n(x):=\sum_{\ell=0}^k  \mathcal{T}_{\ell}^k(n) x^{k-\ell}\bar x^\ell, \qquad \mathcal{T}_{\ell}^k(n):=\binom{k}{\ell}\frac{\left(\frac{n-1}{2}\right)_{k-\ell} \left(\frac{n-1}{2}\right)_\ell}{(n-1)_k}, \ \ \ x\in  \mathbb{R}^{n+1}.
		\end{equation}
An interesting property of these polynomials is that
\begin{equation}
\label{RE}
 		\mathcal{H}_k^n(x_0)=x_0^k
\end{equation}
since $ \sum_{\ell=0}^k  \mathcal{T}_{\ell}^k(n)=1$, see \cite{DA}.
	\end{example}
	\begin{example}[Axially polyharmonic polynomials]
		A second example of axially polyharmonic polynomials of degree $k$, first introduced in \cite{B}, is given by
		\begin{equation}
			\label{polyharm}
			H_k(x):= \sum_{\ell=0}^{k}x^{k- \ell} \overline{x}^{\ell}, \ \ \ x\in  \mathbb{R}^{n+1}.
		\end{equation}
	\end{example}
	We show that through an appropriate factorization of the second map in the Fueter-Sce extension theorem, the class of left (or right) axially polyharmonic functions emerges naturally within the Fueter-Sce framework.
\begin{remark}
	Throughout this section, we will prove each result only for left slice hyperholomorphic functions, as the corresponding result for right slice hyperholomorphic functions can be derived using similar arguments.
\end{remark}
Our goal now is to determine a suitable factorization of the Fueter-Sce map to obtain axially polyharmonic functions. First, however, we need some preliminary results on the spherical decomposition of the Laplace operator in $n+1$ variables raised to an appropriate power and on the Dirac operator $D$. These results will be crucial for the remainder of the section.

 {\color{black}

	\begin{lemma}
		\label{compo}
		Let $U \subseteq \mathbb{R}^{n+1}$ be an axially symmetric open set that intersects the real line and
		let $f\in \mathcal{SH}_L(U)$. Then,
		for any $(u,0)\in U\cap\mathbb R$ there exists a ball $B(u,r)\subseteq U$, centred in $u$ and radius $r>0$,
		where for any $x=u+Iv\in B(u,r)$
		the functions $\alpha(u,v)$ and $\beta(u,v)$ can be written as
		$$ \alpha(u,v)= \sum_{j=0}^{\infty} \frac{(-1)^j v^{2j}}{(2j)!} \partial_{u}^{2j}[f(u)],\qquad \beta(u,v)= \sum_{j=0}^{\infty} \frac{(-1)^j v^{2j+1}}{(2j+1)!} \partial_{u}^{2j+1}[f(u)],$$
		where $f(u)=\alpha(u,0)$.
	\end{lemma}
	\begin{proof}
		Since, by assumption, $f$ is a slice hyperholomorphic function on an axially symmetric slice domain $U$, for any $a\in U\cap\mathbb R$ there exists a ball $B(a,r)\subseteq U$ where for any $x\in B(a,r)$ we have
		$$
		f(x)=\sum_{k=0}^\infty \frac{(x-a)^{k}}{k!} \partial_S^k f(a)
		$$
		where $\partial_S$ is the left slice derivative of $f$ (see \cite{ColomboSabadiniStruppa2011} or  \cite[Def. 2.1.10 , Eq. (2.15)]{FJBOOK}) and it is given by $\partial_S^k f(a)=\partial_u^k f(a)$.
		Now we choose $a=u$ so in the ball $B(u,r)\subseteq U$ for any $x=u+Iv$ such that
		$(u,v)\in\rr^2$ with  $u+ \mathbb{S} v\subset U$,
		and for $v$ small enough we have
		$$
		f(x)=\sum_{k=0}^\infty \frac{(x-u)^{k}}{k!} \partial_u^k f(u)=\sum_{k=0}^\infty \frac{(Iv)^{k}}{k!} \partial_u^k f(u)=\sum_{j=0}^\infty \frac{(-1)^j v^{2j}}{(2j)!}\partial_u^{2j} f(u) + I\sum_{j=0}^{\infty} \frac{(-1)^j v^{2j+1}}{(2j+1)!}\partial_u^{2j+1} f(u),
		$$
		so, recalling that $f(x)=\alpha(u,v)+I\beta(u,v)$, from the last equality, we deduce the series expansions of $\alpha(u,v)$ and $\beta(u,v)$.
	\end{proof}
	
	\begin{remark}
		Let $f$ be a slice hyperholomorphic function of the form $f(x)=\alpha(u,v)+I\beta(u,v)$. Then $f(u)=\alpha(u,0)+I\beta(u,0)$. Due to the even-odd conditions, see \eqref{EO},we have $f(u)=\alpha(u,0)$.
	\end{remark}
	
The following result was originally proved in \cite{Sce} in a more general setting, and later addressed in the context of Clifford algebras in \cite[Thm. 11.33]{GHS}. We include the proof here for completeness
	
	\begin{lemma}\label{laplacian_sf}
		Let $n$ be an odd number and set $h_n:=(n-1)/2$.
		Let $U \subseteq \mathbb{R}^{n+1}$ be an axially symmetric open set that intersects the real line and
		let $f\in \mathcal{SH}_L(U)$.
		Then, for $m\in \mathbb{N}$, we have
		\begin{equation}
			\label{app111}
			\Delta^m_{n+1}f(x)=2^m  (h_n-m+1)_m \left[\left(\frac{1}{v}\partial_v\right)^m\alpha(u,v)
			+I\left(\partial_v\frac{1}{v}\right)^m\beta(u,v)\right],
		\end{equation}
		where $(h_n-m+1)_m$ is the Pochhammer symbol and recall that $f(x)=\alpha(u,v)+I\beta(u,v)$ for
		$x=u+Iv$.
	\end{lemma}
	\begin{proof}
		We prove the result by induction on $m$. We start proving the case $m=1$. By \cite{D, green} we know that if $f(x)=\alpha(u,v)+I \beta(u,v)$ then we can write
		\begin{equation}
			\label{d1}
		Df(x)= \left(\partial_{u}\alpha(u,v)- \partial_v \beta(u,v)- \frac{2h_n}{v} \beta(u,v)\right)+I \left(\partial_{u} \beta(u,v)+\partial_{v} \alpha(u,v)\right),
		\end{equation}
		and
			\begin{equation}
				\label{d2}
			\overline{D}f(x)= \left(\partial_{u}\alpha(u,v)+\partial_v \beta(u,v)+ \frac{2h_n}{v} \beta(u,v)\right)+I \left(\partial_{u} \beta(u,v)-\partial_{v} \alpha(u,v)\right).
		\end{equation}
By \eqref{d1} and \eqref{d2} we have
\begin{eqnarray}
	\label{d3}
&&\Delta_{n+1} f(x)=D\overline{D}f(x)\\
\nonumber
&&= \left( \partial_{u}^2 \alpha(u,v)+ \partial_v^2 \alpha(u,v)+ \frac{2 h_n}{v}\partial_v \alpha(u,v)\right)+I \left(\partial_{u}^2 \beta(u,v)+ \partial_v^2 \beta(u,v)+ \partial_{v} \left( \frac{2h_n}{v} \beta(u,v)\right)\right).
\end{eqnarray}
Now since the function $f$ is left slice hyperholomorphic, by \eqref{CR} we have
\begin{equation}
\label{eo1}
\partial_{u}^2 \alpha(u,v)=- \partial_v^2 \alpha(u,v), \qquad \partial_u^2 \beta(u,v)=-\partial_v^2 \beta(u,v).
\end{equation}
Thus by plugging \eqref{eo1} into \eqref{d3} we have
$$
\Delta_{n+1} f(x)= 2 h_n \left( \frac{1}{v}\partial_v \alpha(u,v)+I \partial_v \left(\frac{1}{v}\beta(u,v)\right)\right).
$$
This proves the result for $m=1$. In order to show the result for all $m$ we define $ \alpha_0(u,v)=\alpha(u,v)$ and $ \beta_0(u,v)=\beta(u,v)$ and for $m=1,2,...$
\begin{equation}
	\label{not}
\alpha_m(u,v):= 2^m m!  \left( \frac{1}{v} \partial_v\right)^m \alpha(u,v), \quad \beta_m(u,v):= 2^m m!  \left(  \partial_v\frac{1}{v}\right)^m \beta(u,v).
\end{equation}
By \cite{ColSabStrupSce, Sce} we know that
\begin{equation}
\label{GCR}
\partial_u \alpha_m(u,v)= \partial_v \beta_m(u,v)+2m \frac{\beta_m(u,v)}{v}, \qquad \partial_v \alpha_m(u,v)=-\partial_{u} \beta_m(u,v).
\end{equation}
By using the notaions in \eqref{not} we have to prove that
\begin{equation}
\label{nn1}
\Delta_{n+1}^m f(x)= \frac{(h_n-m+1)_m}{m!} (\alpha_m(u,v)+I\beta_m(u,v)).
\end{equation}
We suppose the statement true for $m-1$ and we prove for $m$. By the inductive hypothesis and \eqref{d3} we have
\begin{eqnarray}
\nonumber
\Delta_{n+1}^mf(x)&=& \frac{(h_n-m+2)_{m-1}}{(m-1)!} \Delta_{m+1} \left(\alpha_{m-1}(u,v)+I\beta_{m-1}(u,v)\right)\\
\nonumber
&=&\frac{(h_n-m+2)_{m-1}}{(m-1)!}\left[\left( \partial_u^2 \alpha_{m-1}(u,v)+\partial_v^2 \alpha_{m-1}(u,v)+ \frac{2 h_n}{v} \partial_v \alpha_{m-1}(u,v)\right)\right.\\
\label{nnn}
&& \left.+I\left( \partial_u^2 \beta_{m-1}(u,v)+\partial_v^2 \beta_{m-1}(u,v)+ 2 h_n \partial_v\left( \frac{\alpha_{m-1}(u,v)}{v}\right) \right) \right].
\end{eqnarray}
By \eqref{GCR} we deduce that
\begin{equation}
\label{GCR1}
\partial_u^{2} \alpha_m(u,v)=-\partial_v^2 \alpha_{m-1}(u,v)- \frac{2(m-1)}{v} \partial_v \alpha_{m-1}(u,v)
\end{equation}
and
\begin{equation}
	\label{GCR2}
\partial_u^2 \beta_{m-1}(u,v)=- \partial_{v}^2 \beta_{m-1}(u,v)-2 (m-1) \partial_v \left( \frac{\beta_{m-1}(u,v)}{v}\right).
\end{equation}
Now, we plug \eqref{GCR1} and \eqref{GCR2} into \eqref{nnn} and we get
\begin{equation}
\label{minus}
\Delta_{n+1}^mf(x)= \frac{2(h_n-m+2)_{m-1} (h_n-m+1)}{(m-1)!} \left[ \left(\frac{1}{v} \partial_v \alpha_{m-1}(u,v)\right)+I\partial_v \left( \frac{\beta_{m-1}(u,v)}{v}\right)\right].
\end{equation}
Now, we observe that
\begin{equation}
\label{rel11}
 (h_n-m+2)_{m-1} (h_n-m+1)=(h_n-m+1)_{m},
\end{equation}
and
\begin{equation}
 \label{add1}
 \frac{2}{(m-1)!}\left(\frac{1}{v} \partial_v \alpha_{m-1}(u,v)\right)=2^m \left(\frac{1}{v} \partial_v \right)^{m} \alpha(u,v)=\frac{\alpha_m(u,v)}{m!},
\end{equation}
\begin{equation}
 \label{add2}
 \frac{2}{(m-1)!}\partial_v \left( \frac{\beta_{m-1}(u,v)}{v}\right)=2^m  \left( \partial_v \frac{1}{v}\right)^m \beta(u,v)=\frac{\beta_m(u,v)}{m!}.
\end{equation}
Thus by plugging \eqref{rel11}, \eqref{add1} and \eqref{add2} into \eqref{minus} we get \eqref{nn1}.

	\end{proof}
	\begin{lemma}\label{norm_derivative}
		Let $m$, $j \in \mathbb{N}$, such that $j \geq m$. Then, we have
		\begin{equation}\label{norm_derivative1}
			\left(\frac{1}{v}\partial_v\right)^m v^{2j}= \frac{2^m j!}{(j-m)!} v^{2j-2m},
		\end{equation}
		and
		\begin{equation}\label{norm_derivative2}
			\left(\partial_v\frac{1}{v}\right)^m v^{2j+1}= \frac{2^m j!}{(j-m)!} v^{2j-2m+1}.
		\end{equation}
	\end{lemma}
	\begin{proof}
		We proceed by induction to prove equation \eqref{norm_derivative1}. For $m=1$ we have
		$$
		\left( \frac 1 v \partial_v\right) v^{2j}= 2j \frac 1v v^{2j-1}=2 \frac {j!}{(j-1)!} v^{2j-2}.
		$$
		We suppose the equation \eqref{norm_derivative1} holds for $m$ and we want to prove it for $m+1$. Thus we have
		\[
		\begin{split}
			\left( \frac 1 v \partial_v\right)^{m+1} v^{2j}& = \left( \frac 1 v \partial_v\right) \left( \frac 1 v \partial_v\right)^{m} v^{2j}= \left( \frac 1 v \partial_v\right) \left( 2^m \frac {j!}{(j-m)!} v^{2j-2m}\right)\\
			&= 2^m \frac {j!}{(j-m)!} 2(j-m) v^{2j-2m-2}=2^{m+1} \frac{j!}{(j-(m+1))!}  v^{2j-2(m+1)}.
		\end{split}
		\]
		For the proof of equation \eqref{norm_derivative2} we proceed by induction in a similar way. For $m=1$ we have
		$$
		\left(  \partial_v \frac 1 v\right) v^{2j+1}= \partial_v  v^{2j}=2 j v^{2j-1}=2\frac {j!}{(j-1)!} v^{2j-1}.
		$$
		We suppose the equation \eqref{norm_derivative2} holds for $m$ and we want to prove it for $m+1$. Thus we have
		\[
		\begin{split}
			\left( \partial_v \frac 1 v \right)^{m+1} v^{2j+1}& = \left(  \partial_v \frac 1 v \right) \left( \partial_v \frac 1 v \right)^{m} v^{2j+1}= \left( \partial_v  \frac 1 v \right) \left( 2^m \frac {j!}{(j-m)!} v^{2j-2m+1}\right)\\
			&= 2^m \frac {j!}{(j-m)!} 2(j-m) v^{2j-2m-1}=2^{m+1} \frac{j!}{(j-(m+1))!}  v^{2j-2(m+1)+1}.
		\end{split}
		\]
	\end{proof}
	
	\begin{proposition}
		\label{aaf}
		Let $n$ be an odd number and set $h_n:=(n-1)/2$.
		Let $U \subseteq \mathbb{R}^{n+1}$ be an axially symmetric open set that intersects the real line and
		let $f\in \mathcal{SH}_L(U)$.
		Then, for $m \in \mathbb{N}$,  we have
		\begin{eqnarray}
			\label{Dapp}
			D \Delta_{n+1}^{m}f(x)&=& 2^{m} (h_n-m+1)_{m} \left[ \left(\frac{1}{v}\partial_v\right)^{m} \partial_v \beta(u,v)-\partial_v\left(\partial_v \frac{1}{v}\right)^{m}  \beta(u,v)+ \right.\\
			\nonumber
			&&\left.-\frac{2h_n}{v}\left(\partial_v \frac{1}{v}\right)^{m} \beta(u,v)\right],
		\end{eqnarray}
		where $f(x)=\alpha(u,v)+I\beta(u,v)$ for any $x=u+Iv$ such that
		$(u,v)\in\rr^2$ with  $u+ \mathbb{S} v\subset U$.
	\end{proposition}
	\begin{proof}
		By combining  \eqref{app111}  and \eqref{d1} we get
		\begin{eqnarray*}
			D\Delta_{n+1}^{m}f(x)&=&2^{m}  (h_n-m+1)_{m} \left\{ \left[\partial_u \left(\frac{1}{v}\partial_v\right)^{m} \alpha(u,v)- \partial_v \left(\partial_v\frac{1}{v}\right)^{m}\beta(u,v) \right. \right.\\
			&&\left. \left.- \frac{2h_n}{v}\left(\partial_v\frac{1}{v}\right)^{m} \beta(u,v)\right] \right. \left.+I \left[\partial_{u}\left(\partial_v\frac{1}{v}\right)^{m}\beta(u,v)+ \partial_v \left(\frac{1}{v}\partial_v\right)^{m}\alpha(u,v)\right]\right\}.
		\end{eqnarray*}
		 Now, since the function $f$ is slice hyperholomorphic, by the Cauchy-Riemann equation \eqref{CR}, we have
		\begin{eqnarray}
			\label{starr}
			D\Delta_{n+1}^{m}f(x)&=&2^{m}  (h_n-m+1)_{m} \left\{ \left[ \left(\frac{1}{v}\partial_v\right)^{m} \partial_v\beta(u,v)- \partial_v \left(\partial_v\frac{1}{v}\right)^{m}\beta(u,v) \right. \right.\\
			\nonumber
			&&\left. \left.- \frac{2h_n}{v}\left(\partial_v\frac{1}{v}\right)^{m} \beta(u,v)\right] \right. \left.+I \left[-\left(\partial_v\frac{1}{v}\right)^{m}\partial_v \alpha(u,v)+ \partial_v \left(\frac{1}{v}\partial_v\right)^{m}\alpha(u,v)\right]\right\}.
		\end{eqnarray}
		Since
		$$
		\left(\partial_v\frac{1}{v}\right)^{m}\partial_v \alpha(u,v)=\partial_v \left(\frac{1}{v}\partial_v\right)^{m}\alpha(u,v),
		$$
		from \eqref{starr} we get \eqref{Dapp}.
	\end{proof}
Now, we have all the tools to show that axially polyharmonic function arise naturally in the Fueter-Sce construction.
	\begin{theorem}

		\label{FSharm}
		Let $n$ be an odd number, set $h_n:=(n-1)/2$ and assume that $\ell$ is an integer such that $1 \leq \ell \leq h_n$.
		Then, the application of the following operator
				\begin{equation}
	\label{OPERFSI}
	{T}_{FS}^{(I)}= D \Delta^{\ell-1}_{n+1}.
\end{equation}
		on a left (resp. right) slice hyperholomorphic function defined on an axially symmetric open set $U \subseteq \mathbb{R}^{n+1}$, we obtain a left (resp. right) axially polyharmonic of degree $h_n-\ell+1$.
	\end{theorem}
	\begin{proof}

		 Let $f$ be a left slice hyperholomorphic function. The fact that $T_{FS}^{(I)}f(x)$ is of axial form follows by Proposition \ref{aaf}. Now, we need to show that
		$$ \Delta_{n+1}^{h_n-\ell+1} T_{FS}^{(I)}f(x)=0, \qquad \forall x \in U.$$
		By the Fueter-Sce theorem, see Theorem \ref{FS1}, and the fact that $\Delta_{n+1}=D \overline{D}$ we obtain
		\begin{eqnarray*}
			\Delta_{n+1}^{h_n-\ell+1} T_{FS}^{(I)}f(x)&=& \Delta_{n+1}^{h_n-\ell+1} D \Delta_{n+1}^{\ell-1}f(x)\\
			&=& \Delta_{n+1}^{h_n} D f(x)
            \\
			&=& D\Delta_{n+1}^{h_n}  f(x)
            \\
			&=& 0.
		\end{eqnarray*}
		
	\end{proof}
The result above leads to the following notion.
	\begin{definition}[Axially polyharmonic functions associated with $\mathcal{SH}_L(U)$ and $\mathcal{SH}_R(U)$]
		Let $U$ be an open set. Let $n$ be an odd number, set $h_n:=(n-1)/2$ and consider any integer $\ell$ such that $1 \leq \ell \leq h_n$.
We denote the set of  left axially polyharmonic functions of degree $h_n-\ell+1$ as
$$
\mathcal{APH}^L_{h_n-\ell+1}(U)=\{D\Delta^{\ell-1}_{n+1}f \ :\ f\in \mathcal{SH}_L(U)\},
$$
and the class of right axially polyharmonic functions of degree $h_n-\ell+1$ as
$$
\mathcal{APH}^R_{h_n-\ell+1}(U)=\{f\Delta^{\ell-1}_{n+1}D \ :\ f\in \mathcal{SH}_R(U)\}.
$$
\end{definition}
As a consequence of  Theorem \ref{FSharm} we have the function spaces of the axially polyharmonic functions fine structure.
	\begin{corollary}
	Let $U$ be an open set.	Let $n$ be an odd number, set $h_n:=(n-1)/2$ and assume that $\ell$ is an integer such that $1 \leq \ell \leq h_n$.
		Then, we have the following factorization of the Fueter-Sce construction:
		\begin{equation}
			\begin{CD}\label{FIRSTFACTORIZ}
				&& \textcolor{black}{\mathcal{SH}_L(U)}  @>\ \    D \Delta^{\ell-1}_{n+1}>>\textcolor{black}{\mathcal{APH}^L_{h_n-\ell+1}(U)}@>\ \   \overline{D} \Delta_{n+1}^{h_n-\ell}>>\textcolor{black}{\mathcal{AM}_L(U)}.
			\end{CD}
		\end{equation}
		Moreover, the factorization (\ref{FIRSTFACTORIZ}) can be further expanded as follows:
		\begin{equation}
			\begin{CD}
				&& \textcolor{black}{\mathcal{SH}_L(U)}  @>\ \    D \Delta^{\ell-1}_{n+1}>>\textcolor{black}{\mathcal{APH}^L_{h_n-\ell+1}(U)}@>\ \   \Delta_{n+1}>>\textcolor{black}{\mathcal{APH}^L_{h_n-\ell}(U)}@>\ \   \Delta_{n+1} >>\textcolor{black}{\mathcal{APH}^L_{h_n-\ell-1}(U)}\\
				&&@>\ \  \Delta_{n+1}>>\textcolor{black}{...}@>\ \  \Delta_{n+1}>>\textcolor{black}{\mathcal{APH}^L_1(U)}@>\ \  \overline{D}>>\textcolor{black}{\mathcal{AM}_L(U)}.
			\end{CD}
		\end{equation}
	\end{corollary}
	\begin{proof}
		It is a direct consequence of Theorem \ref{FSharm}.
	\end{proof}
	
	In order to define functions of operators within the context of axially polyharmonic functions, it is crucial to obtain an integral representation of these functions.
	The key step is to apply the operator $T_{FS}^{(I)}$ to the left (or right) slice hyperholomorphic Cauchy kernel in its second form. To achieve this, we first need to establish some preliminary results.
	
	\begin{proposition}
		\label{Dirac}
		Let $n$ be an odd number and set $h_n:=(n-1)/2$. For $s$, $x\in\rr^{n+1}$, such that $s\notin [x]$, and $m \in \mathbb{N}_0$ we have
		\begin{equation}\label{dirac1}
			D((s-\overline x)\qcs^{-m})=2(m-h_n-1)\qcs^{-m}
		\end{equation}
		and
		\begin{equation}\label{dirac2}
			\overline {D}(\qcs^{-m})=2m(s-\overline x)\qcs^{-m-1},
		\end{equation}
		where
		Dirac operator $D$ and its conjugate $\overline{D}$ are defined in (\ref{DIRACeBARDNN}) and
		$\mathcal{Q}_{c,s}(x)^{-1}$ is defined in (\ref{QCSX}).
	\end{proposition}
	\begin{proof}
		In order to  prove \eqref{dirac1} we have to observe that
		$$
		\partial_{x_0}((s-\overline x)\qcs^{-m})=-\qcs^{-m}-m(s-\overline x)\qcs^{-m-1}(-2s+2x_0)
		$$
		and
		$$
		\partial_{x_i}((s-\overline x)\qcs^{-m})=e_i\qcs^{-m}-m(s-\overline x)\qcs^{-m-1}(x)(2x_i).
		$$
		Thus we obtain
		\[
		\begin{split}
			D((s-\overline x)\qcs^{-m})&=\partial_{x_0}((s-\overline x)\qcs^{-m})+\sum_{i=1}^n e_i\partial_{x_i}((s-\overline x)\qcs^{-m})\\
			&= (-1-n)\qcs^{-m}+2m((s-\overline x)s-x(s-\overline x))\qcs^{-m-1}\\
			&=(-1-n)\qcs^{-m}+2m(\qcs)\qcs^{-m-1}\\
			&=(-1-n+2m)\qcs^{-m}\\
			&=2(m-h_n-1)\qcs^{-m}
		\end{split}
		\]
		and this gives formula \eqref{dirac1}. Now we prove formula \eqref{dirac2} observing that
		$$
		\partial_{x_0}(\qcs^{-m})=2m(s-x_0)\qcs^{-m-1}
		$$
		and
		$$
		\partial_{x_i}(\qcs^{-m})=-2mx_i\qcs^{-m-1}.
		$$
		Thus we have
		$$ \overline {D}(\qcs^{-m})=\partial_{x_0}(\qcs^{-m})-\sum_{i=1}^n e_i\partial_{x_i}(\qcs^{-m})=2m(s-\overline x)\qcs^{-m-1}$$
		and this concludes the proof.
	\end{proof}
	
	\begin{remark}
		In the following we will use the Pochhammer symbol defined as
		$$
		(x)_\ell=\frac{\Gamma(x+\ell)}{\Gamma(x)}.
		$$
		When we choose $x=-h_n$, where $h_n$ is the Sce exponent and
		where $\ell$ is an integer such that $0 \leq \ell \leq h_n$.
		In order to avoid the poles of the Gamma function in the points $0, -1,-2,\ldots$
		the symbol $(-h_n)_\ell$ has to be interpreted as
		$$
		(-h_n)_{\ell}=(-1)^\ell(h_n-\ell+1)_{\ell}=(-1)^\ell  \frac{\Gamma(h_n+1)}{\Gamma(h_n-\ell+1)}
		$$
		which excludes the poles of the function $\Gamma$.
	\end{remark}
	For the sake of completeness, we give a short proof of Theorem 3.3 in \cite{CSS} that we will use in the sequel.
	\begin{proposition}
		\label{Laplace}
		Let $n$ be an odd number and set $h_n:=(n-1)/2$. For $s$, $x\in\rr^{n+1}$ such that $s\notin [x]$,
and for any integer $\ell$ such that $0 \leq \ell \leq h_n$ we have
		\begin{equation}\label{dirac_1bis}
			\Delta^\ell_{n+1}(S^{-1}_L(s,x))=4^\ell  \ell! (-h_n)_{\ell}(s-\overline x)\qcs^{-\ell-1},
		\end{equation}
		and
		$$ 		\Delta^\ell_{n+1} (S^{-1}_R(s,x))=4^\ell  \ell! (-h_n)_{\ell}(s-\overline x)\qcs^{-\ell-1},$$
		where $\mathcal{Q}_{c,s}(x)^{-1}$ is defined in (\ref{QCSX}).
	\end{proposition}
	\begin{proof}
		First we prove \eqref{dirac_1bis} by induction with respect to the index $\ell$. When $\ell=1$, applying \eqref{dirac1} and \eqref{dirac2} we have
		\begin{eqnarray*}
			\Delta_{n+1}(S^{-1}_L(s,x))&=&\overline{D}(D(S^{-1}_L(s,x)))\\
			&=&-2h_n\overline{D} (\qcs^{-1})\\
			&=&-4h_n(s-\overline x)\qcs^{-2}.
		\end{eqnarray*}
		Assuming formula \eqref{dirac_1bis} holds to be true for $\ell$, we will prove it for $\ell+1$. By the induction principle, Proposition \ref{Dirac} and the definition of the slice monogenic Cauchy kernel  we have
		\[
		\begin{split}
			\Delta^{\ell+1}_{n+1}(S^{-1}_L(s,x))&=\overline {D}D(\Delta^\ell_{n+1} (S^{-1}_L(s,x)))\\
			&=4^\ell  \ell! (-h_n)_{\ell}\overline{D} D((s-\overline x)\qcs^{-\ell-1})\\
			&=2 \cdot 4^\ell  \ell! (-h_n)_{\ell}\overline{D} \left((\ell-h_n)\qcs^{-\ell-1}\right)\\
			&= 4^{\ell+1} \ell! (\ell+1) (-h_n)_{\ell}(\ell-h_n)(s-\overline x)\qcs^{-\ell-2}\\
			&=	4^{\ell+1}  (\ell+1)! (-h_n)_{\ell+1}(s-\overline x)\qcs^{-\ell-2},
		\end{split}
		\]
		that concludes the proof.
	\end{proof}
	
	Now, we have all the tools to apply the operator ${T}_{FS}^{(I)}$,
	i.e $D\Delta^{\ell-1}_{n+1}$, to the second form of the slice hyperholomorphic Cauchy kernel $S^{-1}_L(s,x)$.	
	
	\begin{theorem}
		\label{harmkernel}
		Let $n$ be an odd number and set $h_n:=(n-1)/2$. The, for any $s$, $x\in\rr^{n+1}$ with $s\notin [x]$, and for any integer $\ell$ such that $1 \leq \ell \leq h_n$, we have
		\begin{equation}\label{dirac_2bis}
			D\Delta^{\ell-1}_{n+1}(S^{-1}_L(s,x))=2^{2 \ell-1} (\ell-1)! (-h_n)_{\ell}\qcs^{-\ell},
		\end{equation}
		and
		$$	(S^{-1}_R(s,x))D\Delta^{\ell-1}_{n+1}=2^{2 \ell-1} (\ell-1)! (-h_n)_{\ell}\qcs^{-\ell} $$
	\end{theorem}	
	\begin{proof}
		By formula \eqref{dirac1}, Proposition \ref{Laplace}, and the properties of the Pochhammer symbol $ (-h_n)_{\ell-1}$ we have
		\[
		\begin{split}
			D(\Delta^{\ell-1}_{n+1}(S^{-1}_L(s,x)))&= 4^{\ell-1} (\ell-1)! (-h_n)_{\ell-1}D((s-\overline x)\qcs^{-\ell})\\
			&=2^{2\ell-1} (\ell-1)! (-h_n)_{\ell-1} (\ell-h_n-1)\qcs^{-\ell}\\
			&=2^{2\ell-1} (\ell-1)! (-h_n)_{\ell}\qcs^{-\ell}.
		\end{split}
		\]
	\end{proof}
	
	\begin{remark}
		If we consider $\ell=1$ and $n=3$ in Theorem \ref{harmkernel} we re-obtain \cite[Theorem 4.1]{CDPS}.
	\end{remark}
	
	\begin{definition}[The kernel $\mathbf{H}_{\ell}(s,x)$]\label{FUNHBL}
		Let $n$ be an odd number and set $h_n:=(n-1)/2$. Then, for any $s$, $x\in\rr^{n+1}$ with $s\notin [x]$,
 and for any integer $\ell$ such that $1 \leq \ell \leq h_n$, we define function
		\begin{equation}\label{FUNHBLLLL}
			\mathbf{H}_{\ell}(s,x):=\sigma_{n,\ell} \qcs^{-\ell},
		\end{equation}
and we set
		\begin{equation}
			\label{gammaL}
			\sigma_{n,\ell}:=2^{2\ell-1} (\ell-1)! (-h_n)_{\ell}.
		\end{equation}
	\end{definition}
	\begin{proposition}
		\label{regH}
		Let $n$ be a fixed odd number and $h_n:=(n-1)/2$. For $s$, $x \in \mathbb{R}^{n+1}$ such that $s \notin [x]$ and $ 1 \leq \ell \leq h_n$ we have that the function $\mathbf{H}_{\ell}(s,x)$, defined in (\ref{FUNHBLLLL}), is intrinsic slice hyperholomorphic in $s$ and axially left and right polyharmonic in $x$ of degree $h_n-\ell+1$.
	\end{proposition}
	\begin{proof}
Since the function $\qcs^{-1}$ is intrinsic slice hyperholomorphic we get that $\qcs^{-\ell}$ is intrinsic slice hyperholomorphic as well.
Thus the function  $\mathbf{H}_{\ell}(s,x)$ is intrinsic slice  hyperholomorphic in the variable $s$ for $s\notin [x]$.
Now, we show that the function $\mathbf{H}_{\ell}(s,x)$ is of axial type. By \cite[Thm. 7.3.1]{CGK} we deduce that
		$$
		\mathbf{H}_{\ell}(s,x)=\sigma_{n,\ell} \qcs^{-\ell}=\sigma_{n,\ell} \mathcal Q^{-h_n}_{c,s}(x) \mathcal Q^{h_n-\ell}_{c,s}(x)=\frac{\sigma_{n,\ell}}{\gamma_n}[F_n^L(s,x)s-xF_n^L(s,x)]\mathcal Q^{h_n-\ell}_{c,s}(x),
		$$
		where $\gamma_n$ is defined in \eqref{gamman}.
		By Proposition \ref{reg2} we know that the left $F$-kernel is a function of axial type, so we can write
		$$ \mathbf{H}_{\ell}(s,x)= A_1(x_0,| \underline{x}|)+ \underline{\omega}B_1(x_0,| \underline{x}|),$$
		where $A_1(x_0,| \underline{x}|)$ and $B_1(x_0,| \underline{x}|)$ are suitable functions.
		Finally the fact that the function $\mathbf{H}_{\ell}(s,x)$ is left axially polyharmonic in the variable $x$ follows by Theorem \ref{FSharm} and by Theorem \ref{harmkernel}.
	\end{proof}
	
	Now, our goal is to give a series expansion of the axially polyharmonic function $ \mathbf{H}_{\ell} (s,x)$. To do this we need some preliminary results.
	\\In \cite{DDG} the authors have already studied the application of the operator $D \Delta^{h_n-1}_{n+1}$ to the monomial $x^k$.
	\begin{theorem}
		\label{newapp}
		Let $n$ be a fixed odd number and $h_n:=(n-1)/2$. For $x \in \mathbb{R}^{n+1}$ and $k \geq 2h_n-1$ we have
		$$ D \Delta_{n+1}^{h_n-1} x^k= \frac{\gamma_n}{(2h_n)!} \frac{k!}{(k-2h_n+1)!} \mathcal{H}_{k-2h_n+1}^n(x)$$
	where $\mathcal{H}_{k-2h_n+1}^n(x)$ are defined in \eqref{Harmo}.
 	\end{theorem}
	  Now, we prove a relation between the coefficients of the Clifford-Appell polynomials, see \eqref{cliffApp}, and the coefficients of the axially harmonic polynomials in \eqref{Harmo}.
	\begin{lemma}\label{lem1}
		Let $n$ be an odd number and set $h_n:=(n-1)/2$. For $\ell=0,1,\ldots, k$
		let
		$\mathcal{C}_\ell^k(n)$	be the coefficient of the homogeneous polynomials of degree $k$, given by (\ref{cliffApp})
		and let $\mathcal{T}_{\ell}^k(n)$ be
		the coefficients  axially harmonic polynomials,
		given by (\ref{Harmo}). Then we have the following relations.
		\begin{enumerate}
			\item For $m\in\nn$,  with $m-2h_n+1\geq 0$, we have
			$$\mathcal C^{m-2h_n+1}_{m-2h_n+1} (2h_n+1)\binom{m+1}{m+1-2h_n}=\binom{m}{2h_n-1} \mathcal T^{m-2h_n+1}_{m-2h_n+1}(2h_n+1).$$
			\item
			For $m\in\nn$ and for $\ell\in \mathbb{N}_0$ such that $0\leq \ell\leq m-2h_n$, with
			$ m-2h_n\geq 0$, we have
			$$
			\mathcal C^{m-2h_n+1}_\ell(2h_n+1)\binom{m+1}{m+1-2h_n} - \mathcal C^{m-2h}_\ell(2h_n+1)\binom{m}{m-2h_n}=\binom{m}{2h_n-1}\mathcal T^{m-2h_n+1}_\ell(2h_n+1).
			$$
		\end{enumerate}
	\end{lemma}
	\begin{proof}
		We prove point (1) recalling that by definition of the coefficients $\mathcal{C}_\ell^k(n)$ in
		\eqref{cliffApp}
		$$
		\mathcal{C}_\ell^k(n):= \binom{k}{\ell} \frac{\left(\frac{n+1}{2} \right)_{k-\ell} \left( \frac{n-1}{2} \right)_\ell}{(n)_k} \ \ \ {\rm for}\ \ \ \ell=0,1,\ldots, k, \ \ \ {\rm and} \ \ \  n\ {\rm odd\ number},
		$$
		setting $\ell=k=m-2h_n+1$, we have
		\begingroup\allowdisplaybreaks
		\begin{equation}\label{ec}
			\begin{split}
				\mathcal C^{m-2h_n+1}_{m-2h_n+1}(2h_n+1)\binom{m+1}{m+1-2h_n} & =\frac{(h_n)_{m-2h_n+1}}{(2h_n+1)_{m-2h_n+1}}\binom{m+1}{m+1-2h_n}\\
				& =\frac{(h_n)_{m-2h_n+1}}{(m+1)!} (2h_n)!\frac{(m+1)!}{(2h_n)!(m+1-2h_n)!}\\
				& =\frac{(h_n)_{m-2h_n+1}}{(m+1-2h_n)!}.
			\end{split}
		\end{equation}
		\endgroup
		On the other side form considering the coefficients
		$$
		\mathcal{T}_{\ell}^k(n):=\binom{k}{\ell}\frac{\left(\frac{n-1}{2}\right)_{k-\ell} \left(\frac{n-1}{2}\right)_\ell}{(n-1)_k} \ \ \ {\rm for}\ \ \ \ell=0,1,\ldots k, \ \ \ {\rm and} \ \ \  n\ {\rm odd\ number}
		$$
		given by  \eqref{Harmo} and setting $\ell=k=m-2h_n+1$
		we have
		\begingroup\allowdisplaybreaks
		\begin{equation}\label{et}
			\begin{split}
				\binom{m}{2h_n-1} \mathcal T^{m-2h_n+1}_{m-2h_n+1}(2h_n+1)&=\frac{m!(h_n)_{m-2h_n+1}}{(2h_n-1)! (m-2h_n+1)!(2h_n)_{m-2h_n+1}}\\
				&=\frac{m!(2h_n-1)! (h_n)_{m-2h_n+1}}{(2h_n-1)! m! (m-2h_n+1)!}\\
				&=\frac{(h_n)_{m-2h_n+1}}{(m-2h_n+1)!}.
			\end{split}
		\end{equation}
		\endgroup
		The first relation is proved since \eqref{ec} and \eqref{et} coincide.
		
		\medskip
		We prove point (2). First we consider the first hand side in the relation in point (2) that contains the coefficients $\mathcal{C}_\ell^k(n)$ so
		by \eqref{cliffApp} with the restrictions on
		$m\in\nn$ and for $\ell\in \mathbb{N}_0$ such that $0\leq \ell\leq m-2h_n$, with
		$ m-2h_n\geq 0$, we have
		\begingroup\allowdisplaybreaks
		\begin{eqnarray}
			\nonumber
			 &&\mathcal C^{m-2h_n+1}_\ell (2h_n+1)\binom{m+1}{m+1-2h_n}-\mathcal C^{m-2h_n}_\ell (2h_n+1)\binom{m}{m-2h_n}  \\
			\nonumber
			&=& \binom{m-2h_n+1}{\ell} \frac{(h_n+1)_{m-2h_n+1-\ell}(h_n)_\ell}{(2h_n+1)_{m-2h_n+1}} \binom{m+1}{m+1-2h_n}
			\\
			\nonumber
			\ \ \ \ &&-\binom{m-2h_n}{\ell} \frac{(h_n+1)_{m-2h_n-\ell}(h_n)_\ell}{(2h_n+1)_{m-2h_n}} \binom{m}{m-2h_n}\\
			\nonumber
			&=&\frac{(m-2h_n+1)!(m-h_n+1-\ell)! (h_n+\ell-1)! (2h_n)! (m+1)!}{\ell! (m-2h_n+1-\ell)! h_n! (h_n-1)! (m+1)! (2h_n)! (m+1-2h_n)!}\\
			\nonumber
			&&-\frac{(m-2h_n)!(m-h_n-\ell)! (h_n+\ell-1)! (2h_n)! m!}{\ell! (m-2h_n-\ell)! h_n! (h_n-1)! m! (2h_n)! (m-2h_n)!}\\
			\nonumber
			&=&\frac{(h_n+\ell-1)! (m-h_n+1-\ell)!}{\ell! h_n! (h_n-1)! (m-2h_n+1-\ell)!}- \frac{(h_n+\ell-1)! (m-h_n-\ell)!}{\ell! h_n! (h_n-1)! (m-2h_n-\ell)!}\\
			\nonumber
			&=& \frac{(m-h_n-\ell)!(h_n+\ell-1)!}{(m-2h_n-\ell)! \ell! h_n! (h_n-1)!} \left[ \frac{m-h_n+1-\ell}{m-2h_n+1-\ell}-1 \right]\\
			\label{ec2}
			&=& \frac{(m-h_n-\ell)!(h_n+\ell-1)!}{(m-2h_n+1-\ell)! \ell!  [(h_n-1)!]^2}.
		\end{eqnarray}
		\endgroup
		On the other side for the
		the coefficients $\mathcal{T}_{\ell}^k(n)$ of axially harmonic polynomials,
		given by (\ref{Harmo}), we have
		\begingroup\allowdisplaybreaks
		\begin{eqnarray}
			\nonumber
			\binom{m}{2h_n-1}\mathcal T^{m-2h_n+1}_\ell(2h_n+1) &=&\frac{m! (h_n)_{m-2h_n+1-\ell } (h_n)_\ell (m-2h_n+1)!}{(2h_n-1)! (m-2h_n+1)! (m-2h_n+1-\ell)! \ell! (2h_n)_{m-2h_n+1}}\\
			\nonumber
			&=& \frac{m! (m-h_n-\ell)! (h_n+\ell-1)! (2h_n-1)!}{\ell! (m-2h_n+1-\ell)! (2h_n-1)! \left[(h_n-1)!\right]^2 m!}\\
			\label{et2}
			&=&\frac{(m-h_n-\ell)!(h_n+\ell-1)!}{\ell! (m-2h_n+1-\ell)! \left[(h_n-1)!\right]^2}.
		\end{eqnarray}
		\endgroup
		Since the expressions \eqref{ec2} and \eqref{et2} coincide, we get the result.
		
	\end{proof}

	\begin{proposition}\label{prop1}
		Let $n$ be an odd number and set $h_n:=(n-1)/2$. For any $m \in \mathbb{N}_0$ and $x \in \mathbb{R}^{n+1}$ we have
		\begingroup\allowdisplaybreaks
	\begin{equation}
		\label{pf1}
		\Delta^{h_n}_{n+1}(x^{1+m})-x\Delta^{h_n}_{n+1}(x^m)=2h_nD\Delta^{h_n-1}_{n+1} x^m.
	\end{equation}
		\endgroup
	\end{proposition}
	\begin{proof}
		The proof is based on the relation between
		of the coefficient of the homogeneous polynomials $P_k^n(x)$ given by (\ref{cliffApp})
		 with
		the coefficients  axially harmonic polynomials $\mathcal{H}_k^n(x)$
		given by (\ref{Harmo}).
		Applying the second map of the Fueter-Sce mapping theorem, i.e. $T_{FS2}:=\Delta^{h_n}_{n+1}$,
 to the monomial  $x^m$ (see formula \eqref{app11}) and using the first and second relations from Lemma \ref{lem1} we get
		\begingroup\allowdisplaybreaks
		\begin{eqnarray*}
			&&\Delta^{h_n}_{n+1}(x^{1+m})-x\Delta^{h_n}_{n+1}(x^m)=\gamma_n \binom{m+1}{m+1-2h_n} P^{2h_n+1}_{m+1-2h_n}(x)-\gamma_n x P^{2h_n+1}_{m-2h_n}(x)\\
			&&=\gamma_n \binom{1+m}{1+m-2h_n}\sum_{\ell=0}^{m-2h_n+1} \mathcal{C}^{m-2h_n+1}_\ell(2h_n+1) x^{m+1-2h_n-\ell} \bar x^\ell\\
			&& - \gamma_n\binom{m}{m-2h_n}\sum_{\ell=0}^{m-2h_n} \mathcal{C}^{m-2h_n}_\ell(2h_n+1) x^{m+1-2h_n-\ell} \bar x^\ell \\
			&& =\gamma_n \sum_{\ell=0}^{m-2h_n} \left[\mathcal C^{m-2h_n+1}_\ell(2h_n+1) \binom{m+1}{m+1-2h_n}- \mathcal C^{m-2h_n}_\ell(2h_n+1) \binom{m}{m-2h_n}  \right] x^{m-2h_n-\ell+1} \bar x^\ell \\
			&&\quad\quad\quad\quad\quad\quad\quad\quad\quad\quad\quad\quad +\gamma_n \binom{m+1}{m+1-2h_n} \mathcal C^{m-2h_n+1}_{m-2h_n+1}(2h_n+1) \bar x^{m-2h_n+1}\\
			&&=\gamma_n\frac{m!}{(2h_n-1)!(m-2h_n+1)!} \Big[ \sum_{\ell=0}^{m-2h_n} \mathcal T_\ell ^{m-2h_n+1}(2h+1) x^{m-2h_n-\ell+1} \bar x^\ell
			\\
			&&\quad\quad\quad\quad\quad\quad\quad\quad\quad\quad\quad\quad + \mathcal T^{m-2h+1}_{m-2h_n+1} (2h_n+1)\bar x^{m-2h_n+1} \Big]\\
			&&=(2h_n)\gamma_n \frac{m!}{(2h_n)!(m-2h_n+1)!} \sum_{\ell=0}^{m-2h_n+1} \mathcal T^{m-2h_n+1}_\ell(2h_n+1) x^{m-2h_n-\ell+1} \bar x^{\ell}\\
			&&= (2h_n)D\Delta^{h_n-1}_{n+1} x^m,
		\end{eqnarray*}
		where the constant $\gamma_n$ are given in \eqref{gamman}.
		\endgroup
		In the last equality we used the application of $D\Delta^{h_n-1}_{n+1}$ to the monomial $x^m$, see Theorem \ref{newapp}.
	\end{proof}
	
	Now, we have all the tools to give a series expansion of the function $\mathbf{H}_{\ell}(s,x)$.
	\begin{proposition}
		\label{harmapp}
		Let $n$ be an odd number, set $h_n:=(n-1)/2$
		and assume that  $\ell$ is an integer such that $1 \leq \ell \leq h_n$.
		Then, for any $x$, $s \in \mathbb{R}^{n+1}$ with $|x|<|s|$,
		the function $\mathbf{H}_{\ell}(s,x)$ defined in (\ref{FUNHBLLLL})
		has the series expansion
		\begin{equation}
			\label{newharmoseries}
			\mathbf{H}_{\ell}(s,x)=\sigma_{n,\ell} \sum_{k=2 \ell-1}^{\infty}\sum_{\alpha=0}^{h_n-\ell} \sum_{\beta=0}^\alpha K_{k,\alpha, \beta, \ell, h_n} \mathcal H^n_{k-2\alpha+\beta-2\ell+1}(x) (-2x_0)^{\beta} |x|^{2(\alpha-\beta)}s^{-1-k},
		\end{equation}	
		where the axially harmonic polynomials $\mathcal H_{k-2\alpha+\beta-2\ell+1}^n(x)$ are given by (\ref{Harmo}),
		the constants $K_{k,\alpha, \beta, \ell, h_n}$ are defined as
		\begin{equation}
			\label{kon}
			K_{k,\alpha, \beta, \ell, h_n}:=\binom{2h_{n}-2\alpha+\beta-2\ell+k}{2h_n-1} \binom{h_n-\ell}{\alpha}\binom{\alpha}{\beta},
		\end{equation}
		and  $\sigma_{n,\ell}$ are given in \eqref{gammaL}.
	\end{proposition}
	\begin{proof} By Proposition \ref{regH} we deduce that
		\begin{equation}
			\label{ax1}
			\mathbf{H}_{\ell}(s,x)=\frac{\sigma_{n,\ell}}{\gamma_n}[F_n^L(s,x)s-xF_n^L(s,x)]\mathcal Q^{h_n-\ell}_{c,s}(x).
		\end{equation}
		By Proposition \ref{exseries} and \eqref{app11} we have
		$$
		F_n^L(s,x)=\sum_{k=2h_n}^\infty \Delta^{h_n}_{n+1}(x^k) s^{-1-k},
		$$
		thus by Proposition \ref{prop1}, we obtain
		\begin{eqnarray}
			\nonumber
F_n^L(s,x)s-xF_n^L(s,x)&=& \sum_{k=2h_n}^{\infty} \Delta_{n+1}^{h_n}(x^k) s^{-k}-x\sum_{k=2h_n}^{\infty} \Delta_{n+1}^{h_n}(x^k) s^{-1-k}\\
\nonumber
&=&\sum_{k=2h_n-1}^{\infty} \Delta_{n+1}^{h_n}(x^{k+1}) s^{-1-k}-x\sum_{k=2h_n-1}^{\infty} \Delta_{n+1}^{h_n}(x^k) s^{-1-k}\\
\label{abb}
&=& 2h_n\sum_{k=2h_n-1}^{\infty} (D\Delta^{h_n-1}_{n+1} x^k) s^{-1-k}.
		\end{eqnarray}

		By \eqref{ax1}, \eqref{abb} and the binomial formula, we have
		\begingroup\allowdisplaybreaks
		\begin{eqnarray}
			\nonumber
			\mathbf{H}_{\ell}(s,x)&=&\frac{2h_n \sigma_{n,\ell}}{\gamma_n}\left( \sum_{k=2h_n-1}^\infty (D\Delta^{h_n-1}_{n+1} x^k) s^{-1-k} \right) (s^2-2x_0s+|x|^2)^{h_n-\ell}\\
			\nonumber
			&=& \frac{2h_n \sigma_{n,\ell}}{\gamma_n}\left( \sum_{k=2h_n-1}^\infty (D\Delta^{h_n-1}_{n+1} x^k) s^{-1-k} \right) \\
			\nonumber
			&&\times \left[ \sum_{\alpha=0}^{h_n-\ell}\binom{h_n-\ell}{\alpha} s^{2h_n-2\ell-2\alpha} (|x|^2-2x_0s)^\alpha \right] \\
			\nonumber
			&=& \frac{2h_n \sigma_{n,\ell}}{\gamma_n}\left( \sum_{k=2h_n-1}^\infty (D\Delta^{h_n-1}_{n+1} x^k) s^{-1-k} \right) \\
			\nonumber
			&&\times \left[ \sum_{\alpha=0}^{h_n-\ell} \sum_{\beta=0}^{\alpha} \binom{h_n-\ell}{\alpha}\binom{\alpha}{\beta} s^{2h_n-2\ell-2\alpha} (-2x_0s)^\beta |x|^{2\alpha-2\beta} \right]\\
			\nonumber
			&=& \frac{2h_n \sigma_{n,\ell}}{\gamma_n} \sum_{\alpha=0}^{h_n-\ell} \sum_{\beta=0}^{\alpha} \sum_{k=2h_n-1}^\infty \binom{h_n-\ell}{\alpha}\binom{\alpha}{\beta} (D\Delta^{h_n-1}_{n+1} x^k) s^{-1-k+2h_n -2\ell -2\alpha+\beta}\\
			\label{last_sum}
			&&\times (-2x_0)^{\beta} |x|^{2\alpha-2\beta}.
		\end{eqnarray}
		\endgroup
		We change the index $k$ in the index $u$ setting: $k=u+2h_n -2\ell -2\alpha+\beta$. Thus, the previous summation becomes
		\begin{equation}\label{first_sum}
		\begin{split}
			\mathbf{H}_{\ell}(s,x)&=\frac{2h_n \sigma_{n,\ell}}{\gamma_n} \sum_{\alpha=0}^{h_n-\ell} \sum_{\beta=0}^{\alpha} \sum_{u=2\alpha-\beta +2\ell-1}^\infty \binom{h_n-\ell}{\alpha}\binom{\alpha}{\beta} \left (D\Delta^{h_n-1}_{n+1} x^{u+2h_n -2\ell -2\alpha+\beta}\right)\times\\
&\times s^{-1-u}(-2x_0)^{\beta} |x|^{2\alpha-2\beta}.
		\end{split}
		\end{equation}
		If $2\ell-1\leq u<2\alpha-\beta +2\ell-1$, we have that
		$$
		D\Delta^{h_n-1}_{n+1} x^{u+2h_n -2\ell -2\alpha+\beta}=0,
		$$
		due to the fact that
		$$
		u+2h_n -2\ell -2\alpha+\beta\leq 2h_n-2
		$$
		and, thus, in the previous summation the index $u$ can start from $2\ell-1$. So by \eqref{first_sum} and  Theorem \ref{newapp} we conclude that
		\begin{eqnarray*}
\mathbf{H}_{\ell}(s,x)&=& 2h_n \frac{\sigma_{n,\ell}}{\gamma_n} \sum_{u=2\ell-1}^\infty \sum_{\alpha=0}^{h_n-\ell} \sum_{\beta=0}^{\alpha}  \binom{h_n-\ell}{\alpha}\binom{\alpha}{\beta} \left (D\Delta^{h_n-1}_{n+1} x^{u+2h_n -2\ell -2\alpha+\beta}\right) (-2x_0)^{\beta} |x|^{2\alpha-2\beta} s^{-1-u}\\
&=&\sigma_{n,\ell} \sum_{u=2\ell-1}^\infty \sum_{\alpha=0}^{h_n-\ell} \sum_{\beta=0}^{\alpha} \binom{2h_n-2\alpha+\beta-2\ell+u}{2h_n-1} \binom{h_n-\ell}{\alpha}\binom{\alpha}{\beta}\\
&&\times \mathcal{H}^n_{u-2\alpha+\beta-2\ell+1}(x) (-2x_0)^{\beta} |x|^{2\alpha-2\beta} s^{-1-u}
		\end{eqnarray*}
This proves the result.
		
	\end{proof}
From the previous result, we can determine how the operator $D\Delta^{\ell-1}_{n+1}$ acts on the monomial $x^k$.
	
	\begin{corollary}
		Let $n$ be an odd number set $h_n:=(n-1)/2$ and let
		$\ell$ be an integer such that $1 \leq \ell \leq h_n$. Then, we have the following facts.
For every $x \in \mathbb{R}^{n+1}$ and $ k\geq 2(\ell -1) $, it follows that
\begin{equation}
	\label{appH}
	D\Delta^{\ell-1}_{n+1}x^k=\sigma_{n,\ell} \sum_{\alpha=0}^{h_n-\ell} \sum_{\beta=0}^\alpha K_{k,\alpha, \beta, \ell, h_n} \mathcal H^n_{k-2\alpha+\beta-2\ell+1}(x) (-2x_0)^{\beta} |x|^{2(\alpha-\beta)},
\end{equation}
where $K_{k,\alpha, \beta, \ell, h_n}$ is defined in \eqref{kon}, the
axially harmonic polynomials $\mathcal H_{k-2\alpha+\beta-2\ell+1}^n(x)$ are given by (\ref{Harmo}) and
$\sigma_{n,\ell}$ are defined in \eqref{gammaL}.
	\end{corollary}
	\begin{proof}
By Theorem \ref{harmkernel} and Proposition \ref{cauchyseries} we have
$$
			\mathbf{H}_{\ell}(s,x)=\Delta_{n+1}^{\ell-1}D x^k=\sum_{k=2\ell-1}^\infty \left( D\Delta^{\ell-1}_{n+1} x^k \right) s^{-1-k}, \qquad |x|<|s|.
$$
Thus, by using the series expansion of the kernel $\mathbf{H}_{\ell}(s,x)$, see \eqref{newharmoseries}, in the expression above, we obtain \eqref{appH} by equating the coefficients of the power series.

	\end{proof}
	\begin{remark}
		The expression \eqref{appH} can be considered the Almansi decomposition of the polyharmonic polynomials $D \Delta_{n+1}^{\ell-1}x^k$, see \cite[Prop. 1.3]{Aro}.
	\end{remark}
	Now, we want to show that the sum of the coefficient of the series \eqref{appH} are different from zero.
\begin{proposition}

Let $n$ be an odd number and $h_n:=(n-1)/2$.
Let $U \subseteq \mathbb{R}^{n+1}$ be an axially symmetric open set that intersects the real line and
let $f\in \mathcal{SH}_L(U)$.
Then, for $1 \leq \ell \leq h_n$,  we have
\begin{equation}
\label{limit}
\lim_{v \to 0} D\Delta^{\ell-1}_{n+1}f(x)= \frac{4^\ell h!(-1)^\ell (\ell-1)!}{2(2\ell-1)!(h_n-\ell)!} \partial_{u}^{2\ell-1} f(u),
\end{equation}
where $f(x)=\alpha(u,v)+I\beta(u,v)$, for any $x=u+Ix$ such that $(u,v) \in \mathbb{R}^2$ and $f(u)=\alpha(u,0)$.
\end{proposition}
\begin{proof}
We plug the expansion of $\beta(u,v)$, see Lemma \ref{compo}, into \eqref{Dapp} (with $m=\ell-1$) and by using the relations in Lemma \ref{norm_derivative} we get
\begin{eqnarray}
\nonumber
\lim_{v \to 0} \left(\frac{1}{v} \partial_v\right)^{\ell-1} \partial_v \beta(u,v)&=&\lim_{v \to 0} \left(\frac{1}{v} \partial_v\right)^{\ell-1} \sum_{j=0}^{\infty} \frac{(-1)^j}{(2j)!} v^{2j} \partial_{u}^{2j+1}[f(u)]\\
\nonumber
&=& 2^{\ell-1}\lim_{v \to 0}  \sum_{j=\ell-1}^{\infty} \frac{(-1)^j j!}{(2j)!(j-\ell+1)!} v^{2j-2(\ell-1)} \partial_{u}^{2j+1}[f(u)]\\
\label{news}
&=& 2^{\ell-1} \frac{(-1)^{\ell-1} (\ell-1)!}{(2\ell-2)!} \partial_{u}^{2 \ell-1}[f(u)],
\end{eqnarray}
and
\begin{eqnarray}
\nonumber
\lim_{v \to 0}  \partial_v\left( \partial_v\frac{1}{v}\right)^{\ell-1} \beta(u,v)
\nonumber
&=& 2^{\ell-1}\lim_{v \to 0}  \sum_{j=\ell-1}^{\infty} \frac{(-1)^j j!(2j-2(\ell-1)+1)}{(2j+1)!(j-\ell+1)!} v^{2j-2(\ell-1)} \partial_{u}^{2j+1}[f(u)]\\
\label{news1}
&=& 2^{\ell-1} \frac{(-1)^{\ell-1} (\ell-1)!}{(2\ell-1)!} \partial_{u}^{2 \ell-1}[f(u)]
\end{eqnarray}
and
\begin{eqnarray}
	\nonumber
	2h_n  \lim_{v \to 0}  \frac{1}{v}
	\left( \partial_v\frac{1}{v}\right)^{\ell-1} \beta(u,v)&=&2^{\ell}h_n  \lim_{v \to 0}  \sum_{j=\ell-1}^{\infty} \frac{(-1)^j j!}{(2j+1)!(j-\ell+1)!} v^{2j-2(\ell-1)} \partial_{u}^{2j+1}[f(u)]\\
	\label{news2}
	&=&  2^\ell h_n \frac{(-1)^{\ell-1} (\ell-1)!}{(2\ell-1)!} \partial_{u}^{2 \ell-1}[f(u)].
\end{eqnarray}
Finally, by putting together \eqref{news}, \eqref{news1}, \eqref{news2} and using Proposition \ref{aaf} we obtain
\begin{eqnarray*}
\lim_{v \to 0} D \Delta^{\ell-1}_{n+1} f(x)&=& 2^{\ell-1}(h_n-\ell+2)_{\ell-1}\left[\frac{2^{\ell-1}(-1)^{\ell-1}(\ell-1)!}{(2\ell-2)!}-\frac{2^{\ell-1}(-1)^{\ell-1}(\ell-1)!}{(2\ell-1)!} \right.\\
&& \left. - \frac{2^\ell h_n (-1)^{\ell-1} (\ell-1)!}{(2 \ell-1)!}\right]\partial_{u}^{2\ell-1} f(u)\\
&=& \frac{4^{\ell-1} h_n! (\ell-1)! (-1)^{\ell-1}}{(h_n-\ell+1)!(2 \ell-1)!} \left[1-\frac{1}{2\ell-1}-\frac{2h_n}{2 \ell-1}\right]\partial_{u}^{2\ell-1} f(u)\\
&=&\frac{4^{\ell-1} h_n! (\ell-1)! (-1)^{\ell-1}}{(h_n-\ell+1)!(2 \ell-1)!} \frac{2(\ell-1-h_n)}{(2 \ell-1)}\partial_{u}^{2\ell-1} f(u)\\
&=& \frac{4^\ell h_n!(-1)^\ell (\ell-1)!}{2(2\ell-1)!(h_n-\ell)!} \partial_{u}^{2\ell-1} f(u).
\end{eqnarray*}
\end{proof}
\begin{remark}
The above results based on Lemma \ref{compo}, where we assume that
$U \subseteq \mathbb{R}^{n+1} $ is an axially symmetric open set that intersects the real line,
are of independent interest but are used in this paper only for the following Proposition \ref{seriesw}. In the rest of the paper, we relax the assumption that the domain has to intersect the real line.
\end{remark}
We now have all the tools to show that the sum of the coefficients in  \eqref{appH} are different from zero, and actually the sum of the coefficients have closed compact expression.
\begin{proposition}
\label{seriesw}
Let $n$ be an odd number and let $h_n:=(n-1)/2$. Then, for $1 \leq \ell \leq h_n$, we have
$$\sum_{\alpha=0}^{h_n-\ell} \sum_{\beta=0}^{\alpha} K_{k,\alpha, \beta, \ell, h_n} (-2)^\beta= \frac{1}{ (2\ell-1)!} \frac{k!}{(k-2\ell+1)!}, \qquad k  \geq 2\ell-1,$$
where $K_{k,\alpha, \beta, \ell, h_n}$ is defined in \eqref{kon}.
\end{proposition}
\begin{proof}
We consider in \eqref{limit} the function $f(x)=x^k$, with $x=u+Iv$ and $k \geq 2 \ell-1$. Thus we have
\begin{equation}
\label{ant1}
D\Delta^{\ell-1}_{n+1}x^k |_{\underline{x}=0}=\lim_{v \to 0} D\Delta^{\ell-1}_{n+1}x^k= \frac{4^\ell h_n! (-1)^\ell (\ell-1)!}{2(2\ell-1)!(h_n-\ell)!} \frac{k!}{(k-2\ell+1)!} x_{0}^{k-2\ell+1}.
\end{equation}
If we restrict the expression \eqref{appH} to the real, by \eqref{RE} we obtain
\begin{equation}
\label{ant2}
D\Delta^{\ell-1}_{n+1}x^k |_{\underline{x}=0}=\sigma_{n,\ell}\sum_{\alpha=0}^{h_n-\ell} \sum_{\beta=0}^{\alpha} K_{k,\alpha, \beta, \ell, h_n} (-2)^\beta x_{0}^{k-2\ell+1}.
\end{equation}
By equating the right-hand sides of the relations \eqref{ant1} and \eqref{ant2} we get
$$ \sigma_{n,\ell}\sum_{\alpha=0}^{h_n-\ell} \sum_{\beta=0}^{\alpha} K_{k,\alpha, \beta, \ell, h_n} (-2)^\beta x_{0}^{k-2\ell+1}=\frac{4^\ell h_n!(-1)^\ell (\ell-1)!}{2(2\ell-1)!(h_n-\ell)!} \frac{k!}{(k-2\ell+1)!} x_{0}^{k-2\ell+1}.$$
Hence, by using the definition of the constant $\sigma_{n,\ell}$, see \eqref{gammaL}, we get the result.
\end{proof}
}

Now, we  characterize of the kernel of the operator $D \Delta_{n+1}^{\ell-1}$.
\begin{lemma}
	\label{kernel0}
	Let $n$ be an odd number and set $h_n:=(n-1)/2$, with $1 \leq \ell \leq h_n$. We assume that $U$ is a connected slice Cauchy domain and $f \in \mathcal{SH}_L(U)$ (resp. $f \in \mathcal{SH}_R(U)$). Then $f \in \hbox{ker} (D \Delta^{\ell-1}_{n+1})$ if and only if $f(x)= \sum_{\nu=0}^{2 \ell-2} x^{\nu} \alpha_{\nu}$ (resp. $f(x)= \sum_{\nu=0}^{2 \ell-2} \alpha_{\nu}x^{\nu} $), with $\{\alpha_{\nu}\}_{1 \leq \nu \leq 2 \ell-2} \subseteq \mathbb{R}_n $.
\end{lemma}
\begin{proof}
	We focus on the case $f \in \mathcal{SH}_L(U)$, since for $f \in \mathcal{SH}_R(U)$ it follows by similar arguments. We proceed by double inclusion.
	If we take $f(x)=\sum_{\nu=0}^{2 \ell-2} \alpha_{\nu}x^{\nu} $, then it is clear that $f \in \hbox{ker} (D \Delta^{\ell-1}_{n+1})$. Now, we assume that $f \in \hbox{ker} (D \Delta^{\ell-1}_{n+1}) $. Since the function $f$ is left slice hyperholomorphic on a connected slice Cauchy domain, after a change of variable if needed, there exists $r>0$ such that the function $f$ can be expanded in a convergent series at the origin as follows
	$$ f(x)= \sum_{k=0}^{\infty} x^k \alpha_k, \quad \hbox{for} \quad \{\alpha_k\}_{k \in \mathbb{N}_0} \subset \mathbb{R}_n,$$
	for any $x \in B_r(0)$. Now, we apply the differential operator $ \Delta^{\ell-1}_{n+1} D$ and we get
	\begin{equation}
		\label{aux}
		0= \Delta^{\ell-1}_{n+1}Df(x)= \sum_{k=2 \ell-1}^{\infty} \Delta^{\ell-1}_{n+1} D x^k \alpha_k, \qquad \forall x \in B_r(0).
	\end{equation}
	{\color{black}
		We restrict the polynomials $\Delta^{\ell-1}_{n+1} D x^k$ to a neighbourhood of the real axis $\Omega$. By \eqref{appH} and Proposition \ref{seriesw} we have
		$$ \Delta^{\ell-1}_{n+1} D x^k|_{\underline{x}=0}= d_k x_0^{k-2\ell+1}, \qquad d_k:=\frac{4^\ell h_n! (-1)^\ell (\ell-1)!}{2(2\ell-1)!(h_n-\ell)!} \frac{k!}{(k-2\ell+1)!}.$$
		Thus if we restrict \eqref{aux} to a neighbourhood of the real axis we have
		$$ 0=  \sum_{k=2 \ell-1}^{\infty} d_k x_0^{k-2 \ell+1}\alpha_k, \qquad \forall x \in B_r(0).$$
		Since for $ k \geq 2 \ell-1$ the coefficients $d_k$ are different from zero we get
		$$ \alpha_k=0, \qquad k \geq 2 \ell-1.$$
		This implies that $f(x)= \sum_{\nu=0}^{2 \ell-2} x^{\nu} \alpha_{\nu}$ in $\Omega$, and since $U$ is connected we get that $f(x)= \sum_{\nu=0}^{2 \ell-2} x^{\nu} \alpha_{\nu}$ with $x \in U$.
	}
\end{proof}

	It is also possible to obtain a  series expansion of the function $ \mathbf{H}_{\ell}(s,x)$  in terms of the polyharmonic polynomials of degree $k$ introduced in \eqref{polyharm}.
	\begin{theorem}\label{THERABOVE}
		Let $n$ be an odd number and let $h_n:=(n-1)/2$. Then, for $s$, $x \in \mathbb{R}^{n+1}$,
		such that $|x|<|s|$, the function $\mathcal{Q}_{c,s}(x)^{-1}$ defined in (\ref{QCSX}), has the series expansion
		\begin{equation}
			\label{series1}
			\mathcal{Q}_{c,s}(x)^{-1}= \sum_{k=0}^{\infty} H_{k}(x) s^{-2-k}.
		\end{equation}
		Moreover, for $ 1 \leq \ell \leq h_n$, the function $\mathbf{H}_{\ell}(s,x)$ defined in (\ref{FUNHBLLLL}), can be written as
		\begin{equation}
			\label{series3}
			\mathbf{H}_{\ell}(s,x)= \sigma_{n,\ell}\left( \sum_{k=1}^{\infty} H_{k}(x)s^{-2-k} \right)^{\ell},
		\end{equation}
		where $H_k(x)$ are defined in (\ref{polyharm}) and constant $\sigma_{n,\ell}$ are given in \eqref{gammaL}.		
	\end{theorem}
	\begin{proof}
		We start proving the uniform convergence of the series
		$$ \mathbf{S}(s,x):=\sum_{k=0}^{\infty} H_{k}(x) s^{-2-k}.$$
		By the definition of the polynomials $H_k(x)$ it is clear that $| H_k(x)| \leq (k+1)|x|^k.$ Then we have
		$$ |\mathbf{S}(s,x)| \leq \sum_{k=0}^{\infty}| H_k(x) s^{-2-k}| \leq  \sum_{k=0}^{\infty} (k+1)|x|^k |s|^{-2-k}.$$	
		Since $|x|<|s|$, by using the ratio test we get that the series $\mathbf{S}(s,x)$ is uniformly convergent.
		Now, we prove the equality in \eqref{series1}.  By Theorem \ref{harmkernel}, with $\ell=1$, we have
		$$ D S^{-1}_L(s,x)=2 (-h_n)_{1} \mathcal{Q}_{c,s}(x)^{-1}.$$
		By using the fact that $D(x^k)=-2h_n H_{k-1}(x)$, see \cite[Lemma 1]{B} and Theorem \ref{cauchyseries} we get
		$$ \mathcal{Q}_{c,s}(x)^{-1}=\frac{DS^{-1}_L(s,x)}{2 (-h_n)_1}=- \frac{1}{2h_n} \sum_{k=0}^{\infty} D(x^k )s^{-1-k}=\sum_{k=0}^{\infty} H_{k}(x) s^{-2-k}.$$
		Finally, by Theorem \ref{harmkernel} and the equality \eqref{series1} we get
		$$ \mathbf{H}_{\ell}(s,x)= \sigma_{n,\ell}\left(\mathcal{Q}_{c,s}(x)^{-1} \right)^{\ell}=\sigma_{n,\ell}\left( \sum_{k=0}^{\infty} H_{k}(x)s^{-2-k} \right)^{\ell}.$$
	\end{proof}
	\begin{remark}
		The convergence of the series in \eqref{series3} follows from the fact that is finite product of converging series.
	\end{remark}
	By equating the two series expansions of the kernel $ \mathbf{H}_{\ell}(s,x)$ we get an interesting connection between the axially harmonic polynomials $ \mathcal{H}_k^n(x)$ and the axially polyharmonic polynomials $H_k(x)$.
	\begin{corollary}
\label{harmapp1}
Let $n$ be an odd number, set $h_n:=(n-1)/2$ and assume that $1 \leq \ell \leq h_n$.
 Then, for $s$, $x \in \mathbb{R}^{n+1}$ such that $|x|<|s|$ we have
$$
\left( \sum_{k=1}^{\infty} H_{k}(x)s^{-2-k} \right)^{\ell}=\frac{1}{(2h_{n}-1)!} \sum_{k=2 \ell-1}^{\infty}\sum_{\alpha=0}^{h_n-\ell} \sum_{\beta=0}^\alpha K_{k,\alpha, \beta, \ell, h_n} \mathcal H^n_{k-2\alpha+\beta-2\ell+1}(x) (-2x_0)^{\beta} |x|^{2(\alpha-\beta)}s^{-1-k},
$$
where the constants $K_{k,\alpha, \beta, \ell, h_n}$ are defined in (\ref{kon})
	\end{corollary}
	\begin{proof}
		The result follows by setting the second hand side in \eqref{newharmoseries}  equal  to the second hand side in \eqref{series3}.
	\end{proof}
	Finally we have all the tools to provide an integral representation for axially polyharmonic functions.
	
	\begin{theorem}[Integral representation for polyharmonic functions]
		\label{IP}
		Let $n$ be an odd number, set $h_n:=(n-1)/2$ and $s$, $x \in \mathbb{R}^{n+1}$ be such that $s \notin [x]$. We assume that $U$ is a  bounded slice Cauchy domain in $\mathbb{R}^{n+1}$. We set $ds_I=ds(-I)$ with $I \in \mathbb{S}$ and we recall that $\mathbf{H}_{\ell}(s,x)$ is defined in (\ref{FUNHBLLLL}).
		\begin{itemize}
			\item[1)] If $f$ is a left slice hyperholomorphic function on a set that contains $\overline{U}$, then $\tilde{f}(x)=D\Delta^{\ell-1}_{n+1}f(x)$, with $1 \leq \ell \leq h_n$, is axially polyharmonic of degree $h_n-\ell+1$ and it admits the integral representation
			\begin{equation}
				\label{left2}
				\tilde{f}(x)= \frac{1}{2 \pi} \int_{\partial (U \cap \mathbb{C}_I)}  \mathbf{H}_{\ell}(s,x) ds_I f(s), \qquad \hbox{for any} \quad x \in U.
			\end{equation}
			\item[2)] If $f$ is a right slice hyperholomorphic function on a set that contains $\overline{U}$, then $\tilde{f}(x)=f(x)D\Delta^{\ell-1}_{n+1}$, with $1 \leq \ell \leq h_n$, is axially polyharmonic of degree $h_n-\ell+1$ and it admits the integral representation
			\begin{equation}
				\label{right2}
				\tilde{f}(x)= \frac{1}{2\pi} \int_{\partial (U \cap \mathbb{C}_I)} f(s)ds_I \mathbf{H}_{\ell}(s,x), \qquad \hbox{for any} \quad x \in U.
			\end{equation}
		\end{itemize}
		Moreover, the integrals in \eqref{left2} and \eqref{right2} depend neither on $U$ and nor on the imaginary unit $I \in \mathbb{S}$ and are independent on the kernel of $D\Delta^{\ell-1}_{n+1}$.
	\end{theorem}
	\begin{proof}
		We prove only formula \eqref{left2}, since formula \eqref{right2} follows by similar arguments. By Theorem \ref{harmkernel} and Theorem \ref{Cauchy} we have
		$$ \tilde{f}(x)=D\Delta^{\ell-1}_{n+1}f(x)=\frac{1}{2 \pi} \int_{\partial(U \cap \mathbb{C}_I)}D\Delta^{\ell-1}_{n+1} S^{-1}_L(s,x) ds_I f(s)=\frac{1}{2\pi} \int_{\partial (U \cap \mathbb{C}_I)}  \mathbf{H}_{\ell}(s,x)ds_If(s).$$
		The independence of the integrals \eqref{left2} and \eqref{right2} from the set $U$ and the imaginary unit follows from Theorem \ref{Cauchy}. Finally, the regularity follows by Proposition \ref{regH}. Now, we show that formula \eqref{left2} is independent from the kernel of $D\Delta^{\ell-1}_{n+1}$. Let $f_1$ be a function defined as $f$ such that $ \tilde{f}(x)=D\Delta^{\ell-1}_{n+1} f_1(x)$. Thus by Lemma \ref{kernel0} we have
		$$ f_1(x)=f(x)+ \sum_{\nu=0}^{2 \ell-2}x^\nu \alpha_\nu, \qquad \{\alpha_{\nu}\}_{0 \leq \nu \leq 2 \ell-2}.$$
		Thus by \eqref{left2} we have
		\begingroup\allowdisplaybreaks
		\begin{eqnarray*}
			\tilde{f}_1(x)&=& \int_{\partial (U \cap \mathbb{C}_I)}\mathbf{H}_{\ell}(s,x) ds_I f_1(s)\\
			&=&  \int_{\partial (U \cap \mathbb{C}_I)}\mathbf{H}_{\ell}(s,x) ds_I f(s)+ \sum_{i=0}^{2 \ell-2}\int_{\partial (U \cap \mathbb{C}_I)}\mathbf{H}_{\ell}(s,x) ds_I s^i a_i\\
			&=& \tilde{f}(x)+ \sum_{i=0}^{2 \ell-2}D\Delta^{\ell-1}_{n+1}s^i a_i\\
			&=& \tilde{f}(x)
		\end{eqnarray*}
		\endgroup
	\end{proof}

\section{A Polyharmonic functional calculus based on the $S$-spectrum}\label{Polyharmonic fun}

In this section our goal is to introduce a polyharmonic functional calculus by means of the integral representation of polyharmonic function studied in Section \ref{SEC4}.
The following diagram allows to visualize the polyharmonic fine structure, i.e., the function spaces and the related functional calculi:
\begin{equation*}
    \begin{CD}
        \mathcal{SH}_L(U)  @> {D} \Delta^{\ell-1}_{n+1} >> \mathcal{APH}^L_{h_n-\ell+1}(U) @> \overline{D} \Delta_{n+1}^{h_n-\ell}>> \mathcal{AM}_L(U) \\
        @V VV    @V VV  @V VV \\
        S\text{-Functional calculus} @. \text{Polyharmonic Functional calculus } @. F\text{-Functional calculus} \\
    \end{CD}
\end{equation*}
A similar diagram holds for functions in $\mathcal{SH}_R(U)$.
 In this section we develop the functional calculus of the central part of the diagram.
From the definition of the axially harmonic polynomials, given by \eqref{Harmo}, we now provide the notion of Clifford-harmonic operators.
\begin{definition}[The Clifford harmonic operators]
Let $n$, $k \in \mathbb{N}$ and $T \in \mathcal{BC}^{0,1}(V_n)$. We define the Clifford-harmonic operators as
\begin{equation}
\label{Clifford-harmonic operators}
\mathcal{H}_k^n(T)= \sum_{\ell=0}^{k} \mathcal{T}_{\ell}^k(n) T^{k-\ell} \overline{T}^\ell, \qquad \mathcal{T}_{\ell}^k(n)= \binom{k}{\ell}\frac{\left(\frac{n-1}{2}\right)_{k-\ell} \left(\frac{n-1}{2}\right)_\ell}{(n-1)_k}.
\end{equation}
\end{definition}
Now, we show a relation between the Clifford Appell operators (see \eqref{Appellope}) and the Clifford harmonic-operators.
\begin{lemma}
Let $n$ be an odd number and set $h_n:=(n-1)/2$. For $T \in \mathcal{BC}^{0,1}(V_n)$ and $\ell \geq 0$ we have
\begin{equation}
\label{rell}
\binom{\ell+1+2h_n}{\ell+1} P^n_{\ell+1}(T)-\binom{\ell+2h_n}{\ell}TP_{\ell}^n(T)=\binom{\ell+2h_n}{2h_n-1} \mathcal{H}^n_{\ell+1}(T).
\end{equation}
\begin{proof}
By the definition of the Clifford-Appell operators, see \eqref{Appellope}, we have
\begin{eqnarray*}
&&\binom{\ell+1+2h_n}{\ell+1} P^n_{\ell+1}(T)-\binom{\ell+2h_n}{\ell}TP_{\ell}^n(T)\\
&=&\binom{\ell+1+2h_n}{\ell+1} \sum_{\nu=0}^{\ell+1} \mathcal{C}_{\nu}^{\ell+1}(n)T^{\ell+1-\nu} \overline{T}^{\nu}-\binom{\ell+2h_n}{\ell}\sum_{\nu=0}^{\ell} \mathcal{C}_{\nu}^{\ell}(n)T^{\ell+1-\nu}\overline{T}^{\nu}\\
&=&\binom{\ell+1+2h_n}{\ell+1} \mathcal{C}_{\ell+1}^{\ell+1}(n) \overline{T}^{\ell+1}+ \binom{\ell+1+2h_n}{\ell+1} \sum_{\nu=0}^{\ell} \mathcal{C}_{\nu}^{\ell+1}(n)T^{\ell+1-\nu} \overline{T}^{\nu}
\\
&-&\binom{\ell+2h_n}{\ell}\sum_{\nu=0}^{\ell} \mathcal{C}_{\nu}^{\ell}(n)T^{\ell+1-\nu}\overline{T}^{\nu}.
\end{eqnarray*}
By Lemma \ref{lem1}, with $m=\ell+2 h_n$, we get
\begin{eqnarray*}
&&\binom{\ell+1+2h_n}{\ell+1} P^n_{\ell+1}(T)-\binom{\ell+2h_n}{\ell}TP_{\ell}^n(T)\\
&=& \binom{\ell+2h_n}{2h_n-1} \mathcal{T}_{\ell+1}^{\ell+1}(n) \overline{T}^{\ell+1}+
\binom{\ell+2h_n}{2h_n-1}\sum_{\nu=0}^{\ell} \mathcal{T}_{\nu}^{\ell+1}(n)T^{\ell+1-\nu} \overline{T}^{\nu}\\
&=& \binom{\ell+2h_n}{2h_n-1}\mathcal{H}^n_{\ell+1}(T).
\end{eqnarray*}
\end{proof}
\end{lemma}
\begin{definition}[Polyharmonic resolvent operator]
Let $n$ be an odd number and set $h_n:=(n-1)/2$. We assume $T \in \mathcal{BC}^{0,1}(V_n)$ and $1 \leq \ell \leq h_n$. Then, for $s \in \rho_S(T)$, we define the polyharmonic resolvent operator by
$$
 \mathbf{H}_{\ell}(s,T)=\sigma_{n,\ell} (s^2-2s (T+ \overline{T})+T \overline{T})^{-\ell},
 $$
where $\sigma_{n,\ell}$ is defined in \eqref{gammaL}.
\end{definition}

Now we give a suitable  series expansion for the polyharmonic resolvent operator
$\mathbf{H}_{\ell}(s,T)$.

\begin{proposition}
	\label{series4}
	Let $n$ be an odd number and $h_n:=(n-1)/2$, with $1 \leq \ell \leq h_n$.
Assume that $T \in \mathcal{BC}^{0,1}(V_n)$ and $s \in \mathbb{R}^{n+1}$ is such that $\|T \|< |s|$. Then, we have
$$
		\mathbf{H}_\ell(s,T)
		= \sigma_{n,\ell} \sum_{m=2h_n-1}^{\infty} \sum_{k=0}^{h_n-\ell} \sum_{t=0}^{h_n-\ell-k} \mathbf{k}_{m,h_n,t, \ell, k}  |T|^{2t}\mathcal{H}^n_{m-2h_n+1}(T) (-2T_0)^{h_n-\ell-k-t} s^{-1+h_n-\ell-k-t-m},
$$
where
$$
\mathbf{k}_{m,h_n,t, \ell, k}:=\binom{m}{2h_n-1} \binom{h_n-\ell-k}{t}\binom{h_n-\ell}{k},
$$
the constants $\sigma_{n,\ell}$ are given in \eqref{gammaL}.
\end{proposition}

\begin{proof}
From the definition of the polyharmonic resolvent operator we have
\begin{equation}
\label{Hpoly}
\mathbf{H}_\ell(s,T)= \sigma_{n,\ell} \mathcal{Q}_{c,s}(T)^{-\ell}=\sigma_{n,\ell} \mathcal{Q}_{c,s}^{-h_n}(T) \mathcal{Q}_{c,s}^{h_n-\ell}(T).
\end{equation}
Thus to get the series expansion of $\mathbf{H}_\ell(s,T)$ is crucial to obtain an expansion in series of $\mathcal{Q}_{c,s}^{-h_n}(T)$. By \cite[Thm. 5.1]{CG}, Proposition \ref{exseriesOPR} and Remark \ref{conv} we get
\begin{eqnarray*}
\mathcal{Q}_{c,s}^{-h_n}(T)&=& \frac{1}{\gamma_n} \left[F_n^L(s,T)s-TF_n^L(s,T) \right]\\
&=&\left[ \sum_{m=2h_n}^{\infty} \binom{m}{m-2h_n}P^n_{m-2h_n}(T)s^{-m}-\sum_{m=2h_n}^{\infty}T \binom{m}{m-2h_n}P^n_{m-2h_n}(T)s^{-1-m} \right]\\
&=&  \sum_{m=2h_n-1}^{\infty} \left[\binom{m+1}{m+1-2h_n}P^n_{m+1-2h_n}(T)-T \binom{m}{m-2h_n}P^n_{m-2h_n}(T) \right]s^{-1-m}.
\end{eqnarray*}
By using \eqref{rell} with $\ell=m-2h_n$  we obtain
\begin{equation}
\label{Hpoly1}
\mathcal{Q}_{c,s}^{-h_n}(T)=\sum_{m=2h_n-1}^{\infty} \binom{m}{2h_n-1} \mathcal{H}^n_{m-2h_n+1}(T) s^{-1-m}.
\end{equation}
By plugging \eqref{Hpoly1} into \eqref{Hpoly} and by using the binomial theorem we obtain
\begin{eqnarray*}
\mathbf{H}_\ell(s,T)&=& \sigma_{n,\ell} \sum_{m=2h_n-1}^{\infty} \binom{m}{2h_n-1} \mathcal{H}^n_{m-2h_n+1}(T) s^{-1-m} \sum_{k=0}^{h_n-\ell} \binom{h_n-\ell}{k}s^{2k}(|T|^2-2T_0s)^{h_n-\ell-k}\\
&=& \sigma_{n,\ell} \sum_{m=2h_n-1}^{\infty} \sum_{k=0}^{h_n-\ell} \sum_{t=0}^{h_n-\ell-k} \binom{m}{2h_n-1} \binom{h_n-\ell-k}{t}\binom{h_n-\ell}{k} \times\\
&& \times |T|^{2t}\mathcal{H}^n_{m-2h_n+1}(T) (-2T_0)^{h_n-\ell-k-t} s^{-1+h_n-\ell-k-t-m}.
\end{eqnarray*}
\end{proof}

We can also write the polyharmonic resolvent operator in terms of the $F$-resolvent operator.

\begin{proposition}
	\label{kernels}
	Let $n$ be an odd number and set $h_n:=(n-1)/2$. For $s \in \mathbb{R}^{n+1}$, $T \in \mathcal{BC}^{0,1}(V_n)$ with $1 \leq \ell \leq h_n$ we have
	\begin{eqnarray*}
		\mathbf{H}_\ell(s,T)&=&\frac{\sigma_{n,\ell}}{\gamma_n} \sum_{k=0}^{h_n-\ell}\sum_{t=0}^{h_n-\ell-k} \binom{h_n-\ell}{k} \binom{h_n-\ell-k}{t} |T|^{2t} (-2T_0)^{h_n-\ell-k-t} F_n^L(s,T)s^{h_n+k-\ell-t+1}\\
		&&-\frac{\sigma_{n,\ell}}{\gamma_n}\sum_{k=0}^{h_n-\ell}   \sum_{t=0}^{h_n-\ell-k}\binom{h_n-\ell}{k} \binom{h_n-\ell-k}{t} T|T|^{2t} (-2T_0)^{h_n-\ell-k-t} F_n^L(s,T)   s^{h_n+k-\ell-t}.
	\end{eqnarray*}
\end{proposition}
\begin{proof}
	By  \cite[Thm. 5.1]{CG} and the binomial Newton we have
	\begin{eqnarray*}
		&\,&\mathbf{H}_\ell(s,T)= \sigma_{n,\ell} \mathcal{Q}_{c,s}^{-h_n}(T) \mathcal{Q}_{c,s}^{h_n-\ell}(T)\\ &=&\frac{\sigma_{n,\ell}}{\gamma_n}\left(F_n^L(s,T)s-T F_n^L(s,T) \right)  \sum_{k=0}^{h_n-\ell} \binom{h_n-\ell}{k} s^{2k} \sum_{t=0}^{h_n-\ell-k} \binom{h_n-\ell-k}{t} |T|^{2t} (-2T_0s)^{h_n-\ell-k-t}\\
		&=&\frac{\sigma_{n,\ell}}{\gamma_n} \sum_{k=0}^{h_n-\ell}\sum_{t=0}^{h_n-\ell-k} \binom{h_n-\ell}{k} \binom{h_n-\ell-k}{t} |T|^{2t} (-2T_0)^{h_n-\ell-k-t} F_n^L(s,T)s^{h_n+k-\ell-t+1}\\
		&&-\frac{\sigma_{n,\ell}}{\gamma_n}\sum_{k=0}^{h_n-\ell}   \sum_{t=0}^{h_n-\ell-k} \binom{h_n-\ell}{k}\binom{h_n-\ell-k}{t} T|T|^{2t} (-2T_0)^{h_n-\ell-k-t} F_n^L(s,T)   s^{h_n+k-\ell-t}.
	\end{eqnarray*}
\end{proof}

\begin{remark}
\label{right}
Observe that it is possible to show the result in Proposition \ref{kernels} by means of the right $F$-resolvent operator $F_n^R(s,T)$ by using the relation
$$ \gamma_n \mathcal{Q}_{c,s}^{-h_n}(T)=sF_n^R(s,T)-F_n^R(s,T)T, $$
see \cite[Thm. 5.1]{CG}.
\end{remark}

For the definition of the functional calculus we have to
show  the  slice hyperholomorphicity in the variable $s$ of the polyharmonic resolvent operator
$\mathbf{H}_{\ell}(s,T)$, with $1 \leq \ell \leq h_n$.
This fact is stated in the next lemma.

\begin{lemma}
\label{regope}
	Let $n$ be an odd number, set $h_n:=(n-1)/2$ and $T \in \mathcal{BC}^{0,1}(V_n)$. The polyharmonic resolvent operator $\mathbf{H}_{\ell}(s,T)$, with $1 \leq \ell \leq h_n$, is a $ \mathcal{B}(V_n)$-valued right and left slice hyperholomorphic function of the variable $s$ in $ \rho_S(T)$.
\end{lemma}
\begin{proof}
Since the operator $F_n^L(s,T)$ is a $ \mathcal{B}(V_n)$-valued right slice hyperholomorphic in the variable $s$, see \cite{CG}. By Proposition \ref{kernels} we get that the polyharmonic resolvent operator $\mathbf{H}_{\ell}(s,T)$ is $\mathcal{B}(V_n)$-valued right slice hyperholomorphic in the variable $s$. By Remark \ref{right} we know that we can write the polyharmonic resolvent operator $\mathbf{H}_{\ell}(s,T)$ in terms of $F_n^R(s,T)$. Since the right $F$-resolvent operator is $ \mathcal{B}(V_n)$-valued left slice hyperholomorphic in the variable $s$, see \cite{CG}, we get that  $\mathbf{H}_{\ell}(s,T)$ is also $\mathcal{B}(V_n)$-valued left slice hyperholomorphic in the variable $s$.
\end{proof}

\begin{definition}[Polyharmonic functional calculus on the $S$-spectrum]
\label{polyH}
Let $n$ be an odd number, set $h_n:=(n-1)/2$ and assume that $0 \leq \ell \leq h_n$. Let $T \in \mathcal{BC}^1(V_n)$ be such that its components $T_i$, $i=1,...,n$ have real spectra. Let $U$ be a bounded slice Cauchy domain as in Definition \ref{SHONTHEFS} and set $ds_I=ds(-I)$ for $I\in \mathbb{S}$.
\begin{itemize}
\item For any $f \in \mathcal{SH}^L_{\sigma_S(T)}(U)$ we set $\tilde{f}(x)=\overline{D} \Delta_{n+1}^{\ell-1}f(x)$. We define the polyharmonic functional calculus for the operator $T$ as

\begin{equation}
\label{I1}
\tilde{f}(T):=\frac{1}{2\pi} \int_{\partial(U \cap \mathbb{C}_I)}  \mathbf{H}_{\ell}(s,T)ds_I f(s),
\end{equation}
\item For any $f \in \mathcal{SH}^R_{\sigma_S(T)}(U)$ we set $\tilde{f}(x)=\overline{D} \Delta_{n+1}^{\ell-1}f(x)$. We define the polyharmonic functional calculus for the operator $T$ as
\begin{equation}
\label{I2}
	\tilde{f}(T):=\frac{1}{2 \pi} \int_{\partial(U \cap \mathbb{C}_I)}f(s) ds_I\mathbf{H}_{\ell}(s,T).
\end{equation}
\end{itemize}
\end{definition}
\begin{remark}
When we set $n$ equal to $3$ in the definition above, we get the harmonic functional calculus based on the $S$-spectrum, as it was shown in \cite{CDPS}.
\end{remark}

\begin{theorem}
	\label{well1}
Let $n$ be an odd number.
The polyharmonic functional calculus based on the $S$-spectrum is well defined, i.e. the integrals \eqref{I1} and \eqref{I2} depend neither on the imaginary units $J \in \mathbb{S}^{n-1}$ nor on the slice Cauchy domain $U$.
\end{theorem}
\begin{proof}
We show the result only for left slice hyperholomorphic functions, since the case for right slice hyperholomorphic function can be deduced by similar arguments.
We first show that the integral \eqref{I1} does not depend on the slice domain $U$. Let $U'$ be another bounded slice Cauchy domain with $\sigma_S(T)\subset U'$ and $\overline{U'}\subset \operatorname{dom}(f)$, and let us assume for the moment $\overline{U'}\subset U$. Then $O=U\setminus U'$ is again a bounded slice Cauchy domain and we have $\overline O\subset\rho_S(T)$ and $\overline O\subset \operatorname{dom}(f)$. Hence the function $f(s)$ is left slice hyperholomorphic and the kernel $\mathbf{H}_{\ell}(s,T)$ is right slice hyperholomorphic in the variable $s$ on $\overline O$.
The Cauchy integral theorem therefore implies
\begin{eqnarray*}
	0 & =& \frac 1{2\pi} \int_{\partial (O\cap \mathbb C_I)} \mathbf{H}_{\ell}(s,T) \, ds_I\, f(s)\\
	&=& \frac 1{2\pi} \int_{\partial (U\cap \mathbb C_I)} \mathbf{H}_{\ell}(s,T) \, ds_I\, f(s)-\frac 1{2\pi} \int_{\partial (U'\cap \mathbb C_I)} \mathbf{H}_{\ell}(s,T) \, ds_I\, f(s).
\end{eqnarray*}
If $\overline {U'} \not\subset U$, then $O:=U'\cap U$ is an axially symmetric that contains $\sigma_S(T)$. As in Remark 3.2.4 in \cite{CGK}, we can hence find a third slice Cauchy domain $U''$ with $\sigma_S(T)\subset U''$ and $\overline{U''}\subset O=U\cap U'$. The above arguments show that the integrals over the boundaries of all three sets agree.

In order to show the independence of the imaginary units, we choose two units $I, J\in\mathbb S$ and the two slice Cauchy domains $U_x, U_s\subset \operatorname{dom}(f)$ with $\sigma_S(T)\subset U_x$ and $\overline{U_x} \subset U_s$ (the subscripts $s$ and $x$ are chosen in order to indicate the variable of integration in the following computation). The set $U^c_x=\mathbb R^{n+1}\setminus U_x$ is then an unbounded axially symmetric slice Cauchy domain with $\overline{U^c_x}\subset \rho_S(T)$. The left $\mathcal K$-resolvent is right slice hyperholomorphic on $\rho_S(T)$, see Lemma \ref{regope}, and also at infinity because
$$
\lim_{s\to\infty} \mathbf{H}_{\ell}(s,T)=0.
$$
The right slice hyperholomorphic Cauchy formula implies therefore
$$
\mathbf{H}_{\ell}(s,T)=\frac 1{2\pi}\int_{\partial (U^c_x\cap\mathbb C_I)} \mathcal \mathbf{H}_{\ell}(x,T)\, dx_I S^{-1}_R(x,s)
$$
for every $s\in U^c_x$. Since by Corollary 2.1.26 in \cite{CGK} we have $\partial(U^c_x\cap\mathbb C_J)=-\partial(U_x\cap \mathbb C_J)$ and $S^{-1}_R(x,s)=-S^{-1}_L(s,x)$, we therefore obtain
\[
\begin{split}
	\tilde{f}(T)&=\frac 1{2\pi} \int_{\partial(U_s\cap \mathbb C_J)} \mathbf{H}_{\ell}(s,T)\, ds_J\, f(s)\\
	&= \frac 1{(2\pi)^2}\int_{\partial(U_s\cap\mathbb C_J)} \left( \int_{\partial (U^c_x\cap\mathbb C_I)} \mathbf{H}_{\ell}(x,T) dx_I S^{-1}_R(x,s) \right) \,ds_J f(s)\\
	&= \frac 1{(2\pi)^2}\int_{\partial(U_x\cap\mathbb C_I)} \mathbf{H}_{\ell}(x,T) dx_I \left( \int_{\partial (U_s\cap\mathbb C_J)}  S^{-1}_L(s,x)  \,ds_J f(s) \right)\\
	&=\frac 1{2\pi} \int_{\partial( U_x\cap \mathbb C_I)}\mathbf{H}_{\ell}(x,T)\, dx_I\, f(x),
\end{split}
\]
where the last identity follows again from the slice hyperholomorphic Cauchy formula because we chose $\overline{U_x}\subset U_s$.
\end{proof}

\begin{proposition}
\label{QLR}
Let \( n \) be an odd integer, and define $ h_n := (n-1)/2 $ with $ 1 \leq \ell \leq h_n $. Suppose $T \in \mathcal{BC}^{0,1}(V_n) $ and $f \in \mathcal{N}_{\sigma_S(T)}(U) $, where $ U $ is a bounded slice Cauchy domain as specified in Definition \ref{SHONTHEFS}.
 Then we have
$$\frac{1}{2\pi}\int_{\pp(U\cap \mathbb{C}_I)} \mathbf{H}_{\ell}(s,T) \, ds_I\, f(s)=\frac{1}{2\pi}\int_{\pp(U\cap \mathbb{C}_I)} f(s) ds_I \mathbf{H}_{\ell}(s,T).$$
\end{proposition}
\begin{proof}
The result follows easily from the by Lemma \ref{regope} and the fact that we are considering an intrinsic slice hyperholomorphic function.
\end{proof}

The following technical result based on the monogenic functional calculus will be important in sequel.

\begin{theorem}
\label{zerothm}
Let $n$ be an odd number, $h_n:=(n-1)/2$ and $ 1 \leq \ell \leq h_n$. Let $T \in \mathcal{BC}^1(V_n)$  and assume that the operators $T_{i}$, for $i=1,2,...,n$, have real spectrum. Let $G$ be a bounded slice Cauchy domain such that $(\partial G) \cap \sigma_S(T)= \emptyset$. Thus, for every $I \in \mathbb{S}$ we have
$$ \int_{\partial (G \cap \mathbb{C}_I)} \mathbf{H}_{\ell}(s,T) ds_I s^{\alpha}=0, \qquad \hbox{if} \quad 0 \leq \alpha \leq 2(\ell-1).$$
\end{theorem}
\begin{proof}
By the Fueter-Sce theorem, see Theorem \ref{intform},  we have
\begin{equation}
\label{intzero}
 \int_{\partial(G \cap \mathbb{C}_I)} F_n^L(s,x) ds_I s^{m}= \Delta^{h_n}_{n+1}(x^m)=0, \qquad \hbox{if} \quad m \leq 2h_n-1,
\end{equation}
for all $x \notin \partial G$ and $I \in \mathbb{S}$. Now, by using Proposition \ref{kernels}, Fubini theorem and Proposition \ref{NewR} we have
\begingroup\allowdisplaybreaks
\begin{eqnarray}
\nonumber
\int_{\partial(G \cap \mathbb{C}_I)} \mathbf{H}_{\ell}(s,T) ds_I s^{\alpha}
&=&\frac{\sigma_{n,\ell}}{\gamma_n} \sum_{k=0}^{h_n-\ell}\sum_{t=0}^{h_n-\ell-k} \binom{h_n-\ell}{k} \binom{h_n-\ell-k}{t} |T|^{2t} (-2T_0)^{h_n-\ell-k-t}\\
\nonumber
&& \times \int_{\partial(G \cap \mathbb{C}_I)}F_n^L(s,T)ds_I s^{h_n+k-\ell-t+1+\alpha}\\
\nonumber
&&-\frac{\sigma_{n,\ell}}{\gamma_n}\sum_{k=0}^{h_n-\ell}   \sum_{t=0}^{h_n-\ell-k} \binom{h_n-\ell-k}{t}\binom{h_n-\ell}{k} T|T|^{2t} (-2T_0)^{h_n-\ell-k-t}\\
\nonumber
&& \times
 \int_{\partial(G \cap \mathbb{C}_I)} F_n^L(s,T)  ds_I s^{h_n+k-\ell-t+\alpha}\\
 \nonumber
 &=& \frac{\sigma_{n,\ell}}{\gamma_n} \sum_{k=0}^{h_n-\ell}\sum_{t=0}^{h_n-\ell-k} \binom{h_n-\ell}{k} \binom{h_n-\ell-k}{t} |T|^{2t} (-2T_0)^{h_n-\ell-k-t}\\
 \nonumber
 && \times \int_{\partial W} G(\omega, T) D \omega \left( \int_{\partial(G \cap \mathbb{C}_I)}F_n^L(s,\omega)ds_I s^{h_n+k-\ell-t+1+\alpha}\right)\\
 \nonumber
 &&-\frac{\sigma_{n,\ell}}{\gamma_n}\sum_{k=0}^{h_n-\ell}   \sum_{t=0}^{h_n-\ell-k} \binom{h_n-\ell-k}{t}\binom{h_n-\ell}{k} T|T|^{2t} (-2T_0)^{h_n-\ell-k-t}\\
 \label{int0}
 && \times
 \int_{\partial W} G(\omega, T) D \omega \left( \int_{\partial(G \cap \mathbb{C}_I)}F_n^L(s,\omega)ds_I s^{h_n+k-\ell-t+\alpha}\right),
\end{eqnarray}
\endgroup
where $W \subset \mathbb{R}^{n+1}$ as in Definition \ref{locmongspetrum}.
Since $ k \leq h_n- \ell$, $\alpha \leq 2( \ell-1)$ and $t \geq 0$ we have
$$ h_n+k-\ell-t+1+\alpha \leq 2h_n-2\ell-t+1 +2 \ell-2= 2h_n-t-1 \leq 2h_n-1.$$
Thus by \eqref{intzero} we have
\begin{equation}
\label{int01}
\int_{\partial(G \cap \mathbb{C}_I)}F_n^L(s,\omega)ds_I s^{h_n+k-\ell-t+1+\alpha}=0.
\end{equation}
Similarly, we have that $h_n+k-\ell-t+\alpha \leq 2h_n-1$ and so by \eqref{intzero} we have
\begin{equation}
\label{int02}
\int_{\partial(G \cap \mathbb{C}_I)}F_n^L(s,\omega)ds_I s^{h_n+k-\ell-t+\alpha}=0.
\end{equation}
Finally, the result follows by plugging into \eqref{int0} the expressions \eqref{int01} and \eqref{int02}.
\end{proof}

\begin{proposition}
	Let $n$ be an odd number and set $h_n:=(n-1)/2$. We assume $T \in \mathcal{BC}^1(V_n)$ be such $T_{i}$, $i=1,...,n$ have real spectra. Let $U$ be slice Cauchy domain with $\sigma_S(T) \subset U$. Let $f$, $g \in \mathcal{SH}^L_{\sigma_S(T)}(U)$ (resp. $f$, $g \in \mathcal{SH}^R_{\sigma_S(T)}(U)$) and such that
$ D \Delta^{\ell-1}_{n+1} f(x)=D \Delta^{\ell-1}_{n+1} g(x)$ ( resp.  $f(x)D \Delta^{\ell-1}_{n+1}= g(x)D \Delta^{\ell-1}_{n+1} )$ with $1 \leq \ell \leq h_n$, then we have $\tilde{f}(T)=\tilde{g}(T)$.
\end{proposition}
\begin{proof}
	We show the result only for the left slice hyperholomorphic functions, the case for right slice hyperholomorphic functions follows by using similar arguments. We assume that $U$ is connected. By the definition of the polyharmonic functional calculus, see Definition \ref{polyH}, we have
	$$ \tilde{f}(T)-\tilde{g}(T)= \frac{1}{2\pi} \int_{\partial(U \cap \mathbb{C}_I)} \mathbf{H}_{\ell}(s,T) ds_I(f(s)-g(s)).$$
	By Lemma \ref{regope} we know that the  resolvent operator $\mathbf{H}_{\ell}(s,T)$ is slice hyperholomorphic in $s$, thus we can change the domain of integration to $B_r(0) \cap \mathbb{C}_I$ for $r>0$ such that $\| T \| <r$.  We know by hypothesis that $f-g$ belongs to the kernel of the operator $ D\Delta^{\ell-1}_{n+1}$, thus by Lemma \ref{kernel0} we get
	$$\tilde{f}(T)-\tilde{g}(T)=\frac{1}{2\pi}  \sum_{\nu=0}^{2 \ell-2} \int_{\partial(B_r(0) \cap \mathbb{C}_I)} \mathbf{H}_{\ell}(s,T) ds_I s^{\nu} \alpha_{\nu}, $$
	where $\{\alpha_{\nu}\}_{0 \leq \nu \leq 2 \ell-2} \subseteq \mathbb{R}_n$. By Theorem \ref{zerothm} (since $ \nu \leq 2(\ell-1)$) we know that the previous integral is zero.
\\
In  the case the set $U$ is not connected we can write $U=\bigcup_{\ell=1}^n U_{\ell}$ where $U_{\ell}$ are the connected components of the set $U$. So by Lemma \ref{kernel0} we have
$$
f(s)-g(s)=\sum_{\tau=1}^{n} \sum_{\nu=0}^{2 \ell-2} \chi_{U_{\tau}} s^{\nu}\alpha_{\nu,\tau}.
$$
 Hence by the definition of the polyharmonic functional calculus, see Definition \ref{polyH}, we have
$$ \tilde{f}(T)-\tilde{g}(T)= \frac{1}{2\pi} \sum_{\tau=1}^{n} \sum_{\nu=0}^{2\ell-2} \int_{\partial(U_{\tau} \cap \mathbb{C}_I)}\mathbf{H}_{\ell}(s,T)ds_I s^{\nu}\alpha_{\nu,\tau}.$$
Finally, by Theorem \ref{zerothm} we get that $\tilde{f}(T)=\tilde{g}(T)$.
\end{proof}

Now, we show some algebraic properties of the polyharmonic functional calculus.

\begin{lemma}
Let $n$ be an odd number, set $h_n:=(n-1)/2$ and let $1 \leq \ell \leq h_n$. We assume $T \in \mathcal{BC}^1(V_n)$ be such that $T_{i}$, for $i=1,...,n$,
 have real spectra. Then, if $\tilde{f}(x)= D \Delta_{n+1}^{\ell-1}f(x)$ (resp. $\tilde{f}(x)= f(x)D \Delta_{n+1}^{\ell-1}$) and $\tilde{g}(x)= D \Delta_{n+1}^{\ell-1}g(x)$ (resp. $\tilde{g}(x)= g(x)D \Delta_{n+1}^{\ell-1}$) with $f$, $g \in \mathcal{SH}^L_{\sigma_S(T)}(U)$ (resp. $\mathcal{SH}^R_{\sigma_S(T)}(U)$) and $a \in \mathbb{R}_n$,
 then
$$ (\tilde{f}a+\tilde{g})(T)=\tilde{f}(T)a+\tilde{g}(T), \qquad (resp.\ \  (a\tilde{f}+\tilde{g})(T)=a\tilde{f}(T)+\tilde{g}(T)).$$
\end{lemma}
\begin{proof}
The result follows from the linearity of the integral \eqref{I1} (resp. \eqref{I2}).
\end{proof}

\begin{proposition}
Let $n$ be an odd number, set $h_n:=(n-1)/2$ and let $1 \leq \ell \leq h_n$. We assume $T \in \mathcal{BC}^1(V_n)$  such that $T_{i}$, $i=1,...,n$ have real spectra.
Recall that the constants $	\sigma_{n,\ell}$ and $\gamma_n$ are defined in (\ref{gammaL}) and
 (\ref{gamman}), respectively.
\begin{itemize}
\item Let  $\tilde{f}(x)= D \Delta_{n+1}^{\ell-1}f(x)$ where  $f(x)= \sum_{m=0}^{\infty} x^m a_m$ with $a_m \in \mathbb{R}_n$,  converges on a ball $B_r(0)$, for $r>0$, with $\sigma_S(T) \subset B_r(0)$. Then, we have
$$ \tilde{f}(T)=\sigma_{n,\ell} \sum_{m=2h_n-1}^{\infty} \sum_{k=0}^{h_n-\ell} \sum_{t=0}^{h_n-\ell-k} \mathbf{k}_{m,h_n,t, \ell, k}  |T|^{2t}\mathcal{H}^n_{m-2h_n+1}(T) (-2T_0)^{h_n-\ell-k-t} a_{\ell-k+t+m-h_n}.$$
\item Let  $\tilde{f}(x)= f(x)D \Delta_{n+1}^{\ell-1}$ where $f(x)= \sum_{m=0}^{\infty} a_mx^m $ with $a_m \in \mathbb{R}_n$,  converges on a ball $B_r(0)$, for $r>0$, with $\sigma_S(T) \subset B_r(0)$.
    Then, we have
$$ \tilde{f}(T)=\sigma_{n,\ell} \sum_{m=2h_n-1}^{\infty} \sum_{k=0}^{h_n-\ell} \sum_{t=0}^{h_n-\ell-k} a_{\ell-k+t+m-h_n}\mathbf{k}_{m,h_n,t, \ell, k}  |T|^{2t}\mathcal{H}^n_{m-2h_n+1}(T) (-2T_0)^{h_n-\ell-k-t} .$$
\end{itemize}
\end{proposition}
\begin{proof}
We prove only the first statement since the second one follows from similar arguments. We pick an imaginary unit
$I \in \mathbb{S}$ and a radius $R$, with $0<R< r$, such that $\sigma_S(T) \subset B_R(0)$. Thus, by Definition \ref{polyH} and the fact that the series expansion of $f$ converges uniformly on $\partial (B_R(0) \cap \mathbb{C}_I)$, we have
\begin{eqnarray*}
\tilde{f}(T)&=&\frac{1}{2\pi} \int_{\partial(B_R(0) \cap \mathbb{C}_I)}  \mathbf{H}_{\ell}(s,T)ds_I \sum_{\tau=0}^{\infty} s^\tau a_{\tau}\\
&=& \frac{1}{2\pi}\sum_{\tau=0}^{\infty} \int_{\partial(B_R(0) \cap \mathbb{C}_I)}  \mathbf{H}_{\ell}(s,T)ds_I  s^\tau a_{\tau}.
\end{eqnarray*}
By using the  series expansion  of the resolvent operator $\mathbf{H}_{\ell}(s,T)$ obtained in Proposition \ref{series4} we have
\begin{eqnarray*}
\tilde{f}(T)&=&\frac{\sigma_{n,\ell}}{2\pi }\sum_{\tau=0}^{\infty}  \sum_{m=2h_n-1}^{\infty} \sum_{k=0}^{h_n-\ell} \sum_{t=0}^{h_n-\ell-k} \mathbf{k}_{m,h_n,t, \ell, k}  |T|^{2t}\mathcal{H}^n_{m-2h_n+1}(T) (-2T_0)^{h_n-\ell-k-t}\times\\
&&\times \int_{\partial(B_R(0) \cap \mathbb{C}_I)} s^{-1+h_n-\ell+k-t-m+ \tau} a_{\tau} ds_I.
\end{eqnarray*}
Now, by using the fact that
$$ \int_{\partial(B_R(0) \cap \mathbb{C}_I)} s^{-1-\alpha+\tau}ds_I=\begin{cases}
2 \pi, \qquad \hbox{if} \quad \tau=\alpha \\
0, \qquad \hbox{if} \quad \tau\neq \alpha,
\end{cases}$$
and by setting $\alpha:=-h_n+\ell-k+t+m$ we get
$$ \tilde{f}(T)=\sigma_{n,\ell} \sum_{m=2h_n-1}^{\infty} \sum_{k=0}^{h_n-\ell} \sum_{t=0}^{h_n-\ell-k} \mathbf{k}_{m,h_n,t, \ell, k}  |T|^{2t}\mathcal{H}^n_{m-2h_n+1}(T) (-2T_0)^{h_n-\ell-k-t} a_{\ell-k+t+m-h_n}.$$
This proves the result.
\end{proof}

\section{Holomorphic Cliffordian functions in integral form}\label{HOCLIFIN_INT_FORM}

The notion of the holomorphic Cliffordian function was first introduced by G. Laville and I. Ramadanoff in \cite{LR}. In that paper, the authors study integral representations and the counterparts of the Taylor and Laurent expansions for this type of functions. In subsequent papers, see \cite{L, LL,LR1, LR2}, the same authors further develop the concept of the holomorphic Cliffordian functions.
\begin{definition}[Axially holomorphic Cliffordian]
Let $U$ be an open set in $\mathbb{R}^{n+1}$ and $k \geq 0$. A function $f:\mathbb{R}^{n+1} \to \mathbb{R}_n$ of class $\mathcal{C}^{2k+1}(U)$ is said to be left (resp. right) axially holomorphic Cliffordian if it is of axial type, see \eqref{axx}, and if we have
$$ \Delta^k_{n+1} D f(x)=0, \qquad \forall x \in U, \qquad \left(\hbox{reps.} \quad  f(x)\Delta^k_{n+1} D=0, \qquad \forall x \in U \right)$$
\end{definition}

In \cite{Fivedim} the authors pointed out that this class of functions arises naturally in the factorization of the Fueter-Sce map.
Moreover, by making a suitable factorization of $\Delta^{h_n}_{n+1}$ we also have the class of functions given by:
\begin{definition}[Axially polyanalytic holomorphic Cliffordian of order $(k, \ell)$]
\label{polyC}
Let $U$ be an open set in $\mathbb{R}^{n+1}$ and $k$, $\ell \geq 0$. A function $f:\mathbb{R}^{n+1} \to \mathbb{R}_n$ of class $\mathcal{C}^{2k+\ell}(U)$ is said to be left (resp. right)
axially polyanalytic holomorphic Cliffordian of order $(k, \ell)$ if  it is of axial type, see \eqref{axx},  and if we have
$$ \Delta^k_{n+1} D^{\ell} f(x)=0, \qquad \forall x \in U, \qquad \left(f(x)\Delta^k_{n+1} D^{\ell} =0, \qquad \forall x \in U \right).$$
\end{definition}

\begin{remark}
Our goal is to introduce a holomorphic Cliffordian  functional calculus based on the factorization of the Fueter-Sce map. This will be of crucial importance in Subsection \ref{PROD} also for the product rule for the $F$-functional calculus.
\end{remark}

However, before we proceed, we need to determine the appropriate factorization
of the Fueter-Sce map in order to obtain functions as in Definition \ref{polyC}.
\newline
\newline
In this section, we will demonstrate each result solely for left slice hyperholomorphic functions, as analogous results for right slice hyperholomorphic functions follow from similar reasoning.
\newline
\newline
Before to figure out the suitbale factorization of the Fueter-Sce map that leads to the axially polyanalytic holomorphic Cliffordian functions, we need the following result.

\begin{proposition}
\label{axxx}
Let  $k \in \mathbb{N}$. If we assume that $f$ is a function of axial type then we have that $\overline{D}^kf$ is also of axial type. Let $\alpha\in \mathbb{N}$, then we have that $D^\alpha \overline{D}^kf$ is of axial type as well.
\end{proposition}
\begin{proof}
First we show that $\overline{D}^kf$ is of axial type by induction on $k$. If $k=1$, by \eqref{d2} we get that $\overline{D}^kf$ is of axial type. We suppose that $\overline{D}^kf$ is of axial type, i.e.
$$\overline{D}^kf(x)= A(u,v)+IB(u,v), \qquad x=u+Iv,$$
where $A(u,v)$ and $B(u,v)$ are Clifford-valued functions. We have to show that $\overline{D}^{k+1}f$ is also of axial type. By \eqref{d2} and the inductive hypothesis we have
\begin{eqnarray*}
\overline{D}^{k+1}f(x)&=& \overline{D}\left(A(u,v)+IB(u,v) \right)\\
&=&  \left(\partial_u A(u,v)+ \partial_v B(u,v)+ \frac{2 h_n}{v} B(u,v))\right)+I \left( \partial_u B(u,v)-\partial_v A(u,v)\right).
\end{eqnarray*}
This shows that $\overline{D}^{k}f(x)$ is of axial type. Now we prove that $D^\alpha \overline{D}^kf$ is of axial type by induction on $\alpha$. If $\alpha=1$, by \eqref{d1} we have that
$$ D \overline{D}^kf(x)= \left(\partial_u A(u,v)- \partial_v B(u,v)- \frac{2 h_n}{v} B(u,v))\right)+I \left( \partial_u B(u,v)+\partial_v A(u,v)\right).$$
Now, we suppose that $D^\alpha \overline{D}^kf$ is of axial type:
$$D^\alpha \overline{D}^kf=C(u,v)+ID(u,v),$$
where $C(u,v)$ and $D(u,v)$ are Clifford-valued functions. Thus by \eqref{d1} and the inductive hypothesis we get
\begin{eqnarray*}
D^{\alpha+1} \overline{D}^kf(x)&=&D \left(C(u,v)+ID(u,v)\right)\\
&=& \left(\partial_u C(u,v)- \partial_v D(u,v)- \frac{2 h_n}{v} D(u,v))\right)+I \left( \partial_u D(u,v)+\partial_v C(u,v)\right).
\end{eqnarray*}
This shows that $D^{\alpha+1} \overline{D}^kf(x)$ is of axial type, which completes the proof of the result.
\end{proof}

\begin{theorem}
\label{splitC}
Let $n$ be an odd number and set $h_n:=(n-1)/2$. We consider the parameters $\ell$ and $s$ in the set of natural numbers such that $0 \leq \ell \leq h_n-2$ and $0 \leq s \leq h_n-1$ with $s+ \ell <h_n$. Then, the action of the operator
\begin{equation}
\label{ope3}
T_{FS}^{(II)}:=D^s \overline{D}^{h_n-\ell-1}
\end{equation}
on the set of  left (resp. right) slice hyperholomorphic functions ,defined on an axially symmetric open set $U \subseteq \mathbb{R}^{n+1}$,
gives left (resp. right) holomorphic Cliffordian functions of order $(\ell+1, h_n-s-\ell)$ in $U$.
\end{theorem}
\begin{proof}
 Let $f\in \mathcal{SH}_L(U)$. By Lemma \ref{axxx} it is clear that $T_{FS}^{(II)} f(x)$ is of axial type. Now, by Definition \ref{polyC} we have to show that
$$ \Delta^{\ell+1}_{n+1} D^{h_n-s-\ell} T_{FS}^{(II)} f(x)=0, \qquad \forall x \in U.$$
By the Fueter-Sce mapping theorem, see Theorem \ref{Laplacian_comp}, and the fact that $\Delta_{n+1}=D \overline{D}$ we have
\begin{eqnarray*}
\Delta^{\ell+1}_{n+1} D^{h_n-s-\ell} T_{FS}^{(II)} f(x)&=& \Delta^{\ell+1}_{n+1}D^{h_n-\ell}\overline{D}^{h_n-\ell-1}f(x)\\
&=&D^{\ell+1} \overline{D}^{\ell+1} D^{h-\ell}\overline{D}^{h_n-\ell-1}f(x)\\
&=& D D^{h_n} \overline{D}^{h_n} f(x)\\
&=& \Delta_{n+1}^{h_n} Df(x)\\
&=&0.
\end{eqnarray*}
\end{proof}
\begin{definition}[Axially holomorphic Cliffordian of order $(k, \ell)$ associated with  $\mathcal{SH}_L(U)$ and $\mathcal{SH}_R(U)$ ]\label{DEFwithsandell1}
Let $n$ be an odd number and $h_n:=(n-1)/2$. We consider the parameters $\ell$ and $s$ in the set of natural numbers such that $0 \leq \ell \leq h_n-2$ and $0 \leq s \leq h_n-1$ with $s+ \ell <h_n$.
Let $U$ be an axially symmetric open set in $\mathbb{R}^{n+1}$ and let $k$, $\ell \geq 0$.
The set of functions
$$
\mathcal{AHC}^L_{\ell+1, h_n-s-\ell}(U)=\{ D^s \overline{D}^{h_n-\ell-1} f \ :\ f\in \mathcal{SH}(U)\}
$$
is called left axially holomorphic Cliffordian of order $(\ell+1, h_n-s-\ell)$ associated with  $\mathcal{SH}_L(U)$ and the set
$$
\mathcal{AHC}^R_{\ell+1, h_n-s-\ell}(U)=\{ f D^s \overline{D}^{h_n-\ell-1}  \ :\ f\in \mathcal{SH}_R(U)\}
$$
is called right axially holomorphic Cliffordian of order $(\ell+1, h_n-s-\ell)$ associated with  $\mathcal{SH}_R(U)$.
\end{definition}

As consequence of Theorem \ref{splitC} we have the following result.

\begin{corollary}
Let $U$ be an open set and $n$ be an odd number such that $h_n:=(n-1)/2$. Then for $ 0 \leq \ell \leq  h_n-2$ and $0 \leq s \leq h_n-1$ with $s+ \ell <h_n$ we have the following facatorization of the Fueter-Sce construction:
$$
	\begin{CD}
		&& \textcolor{black}{\mathcal{SH}_L(U)}  @>\ \    D^s \overline{D}^{h_n-\ell-1}>>\textcolor{black}{\mathcal{ACH}^L_{\ell+1,h_n-s-\ell}(U)}@>\ \   D^{h_n-s}\overline{D}^{\ell+1}>>\textcolor{black}{\mathcal{AM}_L(U)}.
	\end{CD}
	$$
\end{corollary}

\begin{remark}
For our purposes we define a holomorphic Cliffordian functional calculus for some specific values of the
 indexes $\ell$ and $s$ in the operator $T_{FS}^{(II)}:=D^s \overline{D}^{h_n-\ell-1}$ defined in
 (\ref{ope3}). Let us consider $1 \leq \alpha \leq h_n-1$.
We take $s=\alpha$ and $\ell=h_n-1-\alpha$. These conditions on the indexes $s$ and $\ell$ stated in Theorem \ref{splitC}
are satisfied and it is also clear that $s+\ell <h_n$.
Thus, in this case, we get a holomorphic Cliffordian function of order
$(h_n-\alpha, 1)$ that leads to focus on a particular operator of $T_{FS}^{(II)}$, that is given by
$$
D^\alpha \overline{D}^{\alpha}=\Delta_{n+1}^{\alpha}, \qquad 1 \leq \alpha \leq h_n-1.
$$
The fine structure we study is associated with the function space of axially holomorphic Cliffordian functions of order $(h_n-\alpha,1)$. The reason why we are considering the above operator among those that we can pick in \eqref{ope3} will be clarified in Subsection \ref{PROD} where we show the product rule for the $F$-functional calculus.
\end{remark}
\begin{remark}\label{DEFwithsandell}
Let $n$ be an odd number and $h_n:=(n-1)/2$ and assume the natural number $\alpha$ be such that
$1\leq \alpha<h_n-1$.
Let $U$ be an axially symmetric open set in $\mathbb{R}^{n+1}$. We will study the fine structure associated with
 the set of functions
$$
\mathcal{AHC}^L_{h_n-\alpha,1}(U)=\{ \Delta_{n+1}^{\alpha} f \ :\ f\in \mathcal{SH}_L(U)\}.
$$
that are called left axially holomorphic Cliffordian of order $(h_n-\alpha,1)$ associated with  $\mathcal{SH}_L(U)$.
Similarly
$$
\mathcal{AHC}^R_{h_n-\alpha,1}(U)=\{ \Delta_{n+1}^{\alpha} f \ :\ f\in \mathcal{SH}_R(U)\}.
$$
is called right axially holomorphic Cliffordian of order $(h_n-\alpha,1)$ associated with  $\mathcal{SH}_R(U)$.
\end{remark}

By Proposition \ref{Laplace} we know that the action of the operator $\Delta_{n+1}^{\alpha}$ to the left (resp. right) slice hyperholomorphic kernel Cauchy is given by

\begin{equation}
\label{holCliff}
 \Delta_{n+1}^{\alpha}S^{-1}_L(s,x)= 4^{\alpha} \alpha! (-h_n)_{\alpha} (s-\overline{x}) \mathcal{Q}_{c,s}(x)^{-\alpha-1},
\end{equation}
$$\left(\hbox{resp.} \quad \Delta_{n+1}^{\alpha}S^{-1}_R(s,x)= 4^{\alpha} \alpha! (-h_n)_{\alpha}  \mathcal{Q}_{c,s}(x)^{-\alpha-1}(s-\overline{x}) \right).$$
The above formulas suggest the definition of the following kernels.
\begin{definition}
Let $n$ be an odd number, set $h_n:=(n-1)/2$ and let $\alpha\in \mathbb{N}$ be such that $1 \leq \alpha \leq h_n-1$.
Let $s$, $x \in \mathbb{R}^{n+1}$ with $ s \notin [x]$. We define the left $\mathcal{K}$-kernel as
$$ \mathcal{K}^L_{\alpha}(s,x)= k_{\alpha} (s-\overline{x}) \mathcal{Q}_{c,s}(x)^{-\alpha-1}$$
and the right  $\mathcal{K}$-kernel as
$$ \mathcal{K}^R_{\alpha}(s,x)= k_{\alpha} \mathcal{Q}_{c,s}(x)^{-\alpha-1}(s-\overline{x}),$$
where
\begin{equation}
\label{ka}
 k_{\alpha}= 4^{\alpha} \alpha! (-h_n)_{\alpha}.
\end{equation}
\end{definition}

As for the $F$-kernels and the polyharmonic kernel also the $\mathcal{K}$-kernels have
 two different regularities with respect to the two variables $x$ and $s$.

\begin{proposition}
\label{reg}
Let $n$ be an odd number, set $h_n:=(n-1)/2$ and let $\alpha\in \mathbb{N}$ be such that $1 \leq \alpha \leq h_n-1$.
 Let $s$, $x \in \mathbb{R}^{n+1}$ for $ x \notin [s]$. Then, we have that $ \mathcal{K}_\alpha^L(s,x)$ (resp. $ \mathcal{K}_\alpha^R(s,x)$ ) is right (resp. left) slice hyperholomorphic in the variable $s$ and left (resp. right) holomorphic Cliffordian function of order $(h_n-\alpha, 1)$ in the variable $x$.
\end{proposition}
\begin{proof}
By the definition of the left slice hyperholomorphic Cauchy kernel $S^{-1}_L(s,x)$, see Definition \ref{Ckernel}, we can write the left $\mathcal{K}$-kernel as
$$ \mathcal{K}^L_{\alpha}(s,x)= k_{\alpha}S^{-1}_L(s,x) \mathcal{Q}_{c,s}(x)^{-\alpha}.$$
Hence, by Proposition \ref{regH} and the fact that $S^{-1}_L(s,x)$ is right slice hyperholomorphic in the variable $s$ we get that also the left $\mathcal{K}$-kernel is right slice hyperholomorphic in the variable $s$. Now, we show that the kernel $\mathcal{K}_\alpha^L(s,x)$ if of axial type.
By easy manipulations we get
\begin{eqnarray*}
	\nonumber
	\mathcal{K}^L_{\alpha}(s,x)&=&k_{\alpha} (s- \overline{x}) \mathcal{Q}_{c,s}(x)^{- \alpha-1}\\
	\nonumber
	&=& \frac{k_{\alpha}}{\gamma_n} \gamma_n (s-\overline{x}) \mathcal{Q}_{c,s}^{-h_n-1}(x) \mathcal{Q}_{c,s}^{h_n-\alpha}(x)\\
	&=& \frac{k_{\alpha}}{\gamma_n}F_n^L(s,x) \mathcal{Q}_{c,s}^{h_n-\alpha}(x).
\end{eqnarray*}
By Proposition \ref{reg2} we know that $F_n^L(s,x)$ is of axial type and so by using the Newton binomial we can write
$$
\mathcal{K}^L_{\alpha}(s,x)=A_2(x_0,| \underline{x}|)+\underline{\omega}B_2(x_0,| \underline{x}|),
$$
where $A_2(x_0,| \underline{x}|)$ and $B_2(x_0,| \underline{x}|)$ are suitable functions. The regularity of $\mathcal{K}^L_{\alpha}(s,x)$ on the variable $x$ follows by Theorem \ref{splitC} and formula \eqref{holCliff}.
\end{proof}
 We provide a series expansion of the left (resp. right) $\mathcal{K}$- kernel.
  This can be obtained by using the series expansion  of the $F$-kernel, see Proposition \ref{exseriesOPR}.

\begin{proposition}
	\label{clifff}
	Let $n$ be an odd number, set $h_n:=(n-1)/2$ and let $\alpha\in \mathbb{N}$ be such that $1 \leq \alpha \leq h_n-1$.  Then for $x$, $s \in \mathbb{R}^{n+1}$ such that $|x| <|s|$ we have
	\begin{equation}
		\label{Cliffser}
		\mathcal{K}^{L}_{\alpha}(s,x)= k_{\alpha}\sum_{k=2 \alpha}^\infty  \sum_{\ell=0}^{h_n-\alpha} \sum_{\nu=0}^{h_n-\alpha-\ell}\kappa_{k,h_n, \alpha, \ell, \nu}P^n_{k-2 \alpha-2 \ell+\nu}(x)(-2x_0)^{\nu} |x|^{2(\ell-\nu)}s^{-1-k},
	\end{equation}
	and
	\begin{equation}
		\label{Cliffser1}
		\mathcal{K}^{R}_{\alpha}(s,x)= k_{\alpha}\sum_{k=2 \alpha}^\infty  \sum_{\ell=0}^{h_n-\alpha} \sum_{\nu=0}^{h_n-\alpha-\ell}\kappa_{k,h_n, \alpha, \ell, \nu}s^{-1-k}P^n_{k-2 \alpha-2 \ell+\nu}(x)(-2x_0)^{\nu} |x|^{2(\ell-\nu)},
	\end{equation}
	where
		\begin{equation}
		\label{kappa}
		\kappa_{k,h_n,\alpha, \ell, \nu}:=\binom{h_n-\alpha}{\ell} \binom{h_n -\alpha-\ell}{\nu} \binom{k+2h_n-2\alpha-2 \ell+\nu}{k-2 \alpha-2 \ell+\nu},
	\end{equation}
$k_{\alpha}$ is given in \eqref{ka} and $\gamma_n$ is defined in \eqref{gamman}.
\end{proposition}

\begin{proof}
By Proposition \ref{reg} we deduce that
$$
\mathcal{K}^L_{\alpha}(s,x)=\frac{k_{\alpha}}{\gamma_n}F_n^L(s,x) \mathcal{Q}_{c,s}^{h_n-\alpha}(x).
$$
Now by Proposition \ref{exseries} and the Newton binomial we have
		\begingroup\allowdisplaybreaks
\begin{eqnarray*}
\mathcal{K}^L_{\alpha}(s,x)&=&\frac{k_{\alpha}}{\gamma_n}F_n^L(s,x) \mathcal{Q}_{c,s}^{h_n-\alpha}(x)\\
&=& \frac{k_{\alpha}}{\gamma_n} \sum_{k=2h_n}^\infty \binom{k}{k-2h_n} P^n_{k-2h_n}(x)s^{-1-k} (s^2-2x_0s+|x|^2)^{h_n-\alpha}\\
&=& k_{\alpha} \sum_{k=2h_n}^\infty \binom{k}{k-2h_n} P^n_{k-2h_n}(x)s^{-1-k} \sum_{\ell=0}^{h_n-\alpha} \binom{h-\alpha}{\ell} s^{2 h_n-2 \alpha -2 \ell} (|x|^2-2x_0s)^{\ell}\\
&=&k_{\alpha} \sum_{k=2h_n}^\infty \binom{k}{k-2h_n} \sum_{\ell=0}^{h_n-\alpha}  \binom{h_n-\alpha}{\ell} P^n_{k-2h_n}(x) s^{-1-k+2 h_n-2 \alpha -2 \ell} \times\\
&& \times \sum_{\nu=0}^{h_n-\alpha-\ell}  \binom{h_n -\alpha-\ell}{\nu}  (-2x_0s)^{\nu} |x|^{2\ell-2 \nu}\\
&=&k_{\alpha} \sum_{\ell=0}^{h_n-\alpha} \sum_{\nu=0}^{h_n-\alpha-\ell}\sum_{k=2h_n}^\infty \binom{k}{k-2h_n} \binom{h_n-\alpha}{\ell} \binom{h_n -\alpha-\ell}{\nu}
 \times\\
&& \times P^n_{k-2h_n}(x) s^{-1-k+2 h_n-2 \alpha- 2 \ell + \nu}(-2x_0)^{\nu} |x|^{2(\ell-\nu)}.
\end{eqnarray*}
\endgroup
Now, we change index $k$ in the above sum with $u=k-2h_n+2 \alpha+ 2 \ell- \nu$, and by \eqref{app11} we get
\begin{eqnarray}
	\nonumber
	\mathcal{K}^L_{\alpha}(s,x)&=& k_{\alpha} \sum_{\ell=0}^{h_n-\alpha} \sum_{\nu=0}^{h_n-\alpha-\ell}\sum_{u=2\alpha+2\ell-\nu}^\infty \binom{u+ 2h_n-2 \alpha- 2 \ell + \nu}{u-2\alpha-2 \ell+\nu} \binom{h_n-\alpha}{\ell} \binom{h_n -\alpha-\ell}{\nu}  \times\\
	\nonumber
	&& \times P^n_{u-2\alpha-2 \ell+\nu}(x) s^{-1-u}(-2x_0)^{\nu} |x|^{2(\ell-\nu)}\\
	\nonumber
	&=&\frac{k_{\alpha}}{\gamma_n} \sum_{\ell=0}^{h_n-\alpha} \sum_{\nu=0}^{h_n-\alpha-\ell}\sum_{u=2\alpha+2\ell-\nu}^\infty  \binom{h_n-\alpha}{\ell} \binom{h_n -\alpha-\ell}{\nu} \Delta_{n+1}^{h_n} (x^{u+ 2h_n-2 \alpha- 2 \ell + \nu})  \times\\
	\label{series5}
	&& \times s^{-1-u}(-2x_0)^{\nu} |x|^{2(\ell-\nu)}.
\end{eqnarray}

Now, if $2 \alpha \leq u <  2\alpha+2 \ell-\nu$ then $\Delta_{n+1}^{h_n}(x^{u+ 2h_n-2 \alpha- 2 \ell + \nu})=0$. This follows from the fact that $u+ 2h_n-2 \alpha- 2 \ell + \nu <2h_n$. Thus we can write the series in \eqref{series5} as

\begin{eqnarray}
	\nonumber
	\mathcal{K}^L_{\alpha}(s,x)&=& k_{\alpha} \sum_{\ell=0}^{h_n-\alpha} \sum_{\nu=0}^{h_n-\alpha-\ell}\sum_{u=2\alpha}^\infty  \binom{h_n-\alpha}{\ell} \binom{h_n -\alpha-\ell}{\nu} \Delta_{n+1}^{h_n} (x^{u+ 2h_n-2 \alpha- 2 \ell + \nu})  \times\\
	\label{series6}
	&& \times (-2x_0)^{\nu} |x|^{2(\ell-\nu)}s^{-1-u}.
\end{eqnarray}

We get the final result by applying \eqref{app11}.
\end{proof}

From the previous result we can get how the operator $\Delta^{\alpha}_{n+1}$ acts on the monomial $x^k$.
\begin{corollary}
	\label{appp}
	Let $n$ be an odd number, set $h_n:=(n-1)/2$ and let $\alpha\in \mathbb{N}$ be such that $1 \leq \alpha \leq h_n-1$.  Let $x \in \mathbb{R}^{n+1}$,
 then, for $k \geq 1$, we have
\begin{equation}
\label{capp}
	\Delta^{\alpha}_{n+1}x^k=k_{\alpha}\sum_{\ell=0}^{h_n-\alpha} \sum_{\nu=0}^{h_n-\alpha-\ell}\kappa_{k,h_n, \alpha, \ell, \nu}P^n_{k-2 \alpha-2 \ell+\nu}(x)(-2x_0)^{\nu} |x|^{2(\ell-\nu)},
\end{equation}
	where $\kappa_{k,h_n,\alpha, \ell, \nu}$ is defined in \eqref{kappa},
$k_{\alpha}$ is given in \eqref{ka} and $\gamma_n$ is defined in \eqref{gamman}.
\end{corollary}
\begin{proof}
	By \eqref{holCliff} we get that $\Delta^{\alpha}_{n+1} S^{-1}_L(s,x)=\mathcal{K}_{L}(s,x)$. By Proposition \ref{cauchyseries} we have that
$$
		\mathcal{K}^L_{\alpha}(s,x)= \sum_{k=2 \alpha}^\infty \Delta^{\alpha}_{n+1} x^k s^{-1-k}, \qquad |x|<|s|.
$$
By using the expansion in \eqref{clifff}, and the fact that two power series coincide when they have the same coefficients, we get \eqref{capp}.
\end{proof}
 {\color{black}
Now, we focus on finding the sum of the coefficients in \eqref{capp}. This will be of crucial importance in Section 7. Before we proceed we need the following preliminary result.
\begin{theorem}\label{lim_laplacian}
Let $n$ be an odd number and set $h_n:=(n-1)/2$.
Let $U \subseteq \mathbb{R}^{n+1}$ be an axially symmetric open set that intersects the real line and
let $f\in \mathcal{SH}_L(U)$.
Thus, we have
	\begin{equation}\label{laplacian_slice_f}
		\lim_{v\to 0} \Delta^m_{n+1} f(x)=2^{2m} (-1)^m m! \frac{(h_n-m+1)_m}{(2m)!}\partial_{u}^{2m}(f(u)),
	\end{equation}
where $f(x)=\alpha(u,v)+I\beta(u,v)$, for any $x=u+Ix$ such that $(u,v) \in \mathbb{R}^2$ and $f(u)=\alpha(u,0)$.
\end{theorem}
\begin{proof}
	By Lemma \ref{laplacian_sf} we have
	\begin{equation}\label{laplacian_slice_f2}
		\lim_{v\to 0} \Delta^m_{n+1} f(x) = 2^m  (h_n-m+1)_m \lim_{v\to 0} \left[ \left( \frac 1v \partial_v \right)^m \alpha(u,v) +I \left( \partial_v\frac 1v \right)^m\beta(u,v)\right].
	\end{equation}
	By Lemma \ref{compo} and Lemma \ref{norm_derivative} we have
	\begin{equation}\nonumber
		\lim_{v\to 0} \left(\frac 1v \partial_v\right)^m \alpha(u,v)=\lim_{v\to 0} \sum_{j= m}^\infty\frac{2^m(-1)^j j! v^{2j-2m}}{(2j)! (j-m)!} \partial_u^{2j} (f(u)).
	\end{equation}
	In the previous summation the only term that is not vanishing when $v\to 0$ is the one for $j=m$, thus we obtain
	$$
	\lim_{v\to 0} \left(\frac 1v \partial_v\right)^m \alpha(u,v)=2^m(-1)^m \frac{m!}{(2m)!} \partial_u^{2m} (f(u)).
	$$
	On the other hand we also have
	$$
	\lim_{v\to 0} \left( \partial_v\frac 1v \right)^m\beta (u,v)=\lim_{v\to 0} 2^m\sum_{J= m}^\infty\frac{ (-1)^j j! v^{2j-2m+1}}{(2j+1)! (j-m)!} \partial_{u}^{2j+1} (f(u))=0,
	$$
	since all the terms in the summation are vanishing when $v\to 0$. Plugging these last two equations in \eqref{laplacian_slice_f2} we can conclude that
	$$
	\lim_{v\to 0} \Delta^m_{n+1} f(x) = 2^{2m} (-1)^m m! \frac{(h_n-m+1)_m}{(2m)!}\partial_{u}^{2m}(f(u)),
	$$
where $x=u+Iv$.
\end{proof}
\begin{proposition}
	\label{summ1}
Let $n$ be an odd number and set $h_n:=(n-1)/2$.
 Let $\alpha\in \mathbb{N}$ be such that $1 \leq \alpha \leq h_n-1$. For $ k \geq 2\alpha$ we have
\begin{equation}
\label{summ}
	\sum_{\ell=0}^{h_n-\alpha} \sum_{\nu=0}^{h_n-\alpha-\ell}  \kappa_{k,h_n,\alpha, \ell, \nu} (-2)^{\nu}=\frac{ k!}{(2\alpha!) (k-2\alpha)!},
\end{equation}
where $\kappa_{k,h_n,\alpha, \ell, \nu}$ is given in \eqref{kappa}.
\end{proposition}
\begin{proof}
	By Corollary \ref{appp}, with $x=u+Iv$ and formula \eqref{Real} we deduce that
\begin{equation}
\label{Stef1}
	\Delta^\alpha_{n+1}(x^k)\Big |_{\underline x=0}=k_\alpha \sum_{\ell=0}^{h_n-\alpha} \sum_{\nu=0}^{h_n-\alpha-\ell} \kappa_{k,h_n,\alpha, \ell, \nu} (-2)^{\nu} u^{k-2\alpha}.
\end{equation}
	By Theorem \ref{lim_laplacian}, with $f(x)=x^k$, we have
$$
\Delta^{\alpha}_{n+1} (x^k)\Big |_{\underline x=0}=\lim_{v\to 0} \Delta^{\alpha}_{n+1} (x^k)=\frac {4^\alpha (-1)^\alpha \alpha! (h_n-\alpha+1)_\alpha}{(2\alpha)! }\partial^{2\alpha}_u(u^{k}).
$$
This implies that
\begin{equation}
\label{Stef2}
\Delta^{\alpha}_{n+1} (x^k)\Big |_{\underline x=0}=\frac {4^\alpha (-1)^\alpha \alpha! (h_n-\alpha+1)_\alpha}{(2\alpha)! } \frac{k!}{(k-2\alpha)!} u^{k-2\alpha}.
\end{equation}
By putting equal \eqref{Stef1} and \eqref{Stef2} we obtain
	$$
	k_\alpha\sum_{\ell=0}^{h_n-\alpha} \sum_{\nu=0}^{h_n-\alpha-\ell} \kappa_{k,h_n,\alpha, \ell, \nu} (-2)^{\nu}=\frac {4^\alpha (-1)^\alpha \alpha! (h_n-\alpha+1)_\alpha}{(2\alpha)! } \frac{k!}{(k-2\alpha)!}.
	$$
Finally the result follows by the definition of $k_{\alpha}$, see \eqref{ka}.
\end{proof}
Now, we provide characterization the kernel of the operator $ \Delta_{n+1}^\alpha$.
\begin{lemma}
	\label{kernel02}
	Let $n$ be an odd number, set $h_n:=(n-1)/2$ and let $\alpha\in \mathbb{N}$ be such that $1 \leq \alpha \leq h_n-1$. If we assume that $U$ is a connected slice Cauchy domain, and $f \in \mathcal{SH}_L(U)$ (resp. $f \in \mathcal{SH}_R(U)$). Then $f$ belongs to the kernel of the operator $ \Delta_{n+1}^\alpha$ if and only if $f(x)= \sum_{\nu=0}^{2 \alpha-1} x^{\nu} \alpha_{\nu}$ (resp. $ f(x)= \sum_{\nu=0}^{2 \alpha-1} \alpha_{\nu} x^{\nu}$), with $\{\alpha_{\nu}\}_{1 \leq \nu \leq 2 \alpha-1} \subseteq \mathbb{R}_n$.
\end{lemma}
\begin{proof}
	We focus on proving only the case of left slice hyperholomorphic functions, since for right slice hyperholomorphic functions the result follows from similar arguments. We proceed by double inclusion. If $f(x)= \sum_{\nu=0}^{2 \alpha-1} x^{\nu} \alpha_{\nu}$ it is clear that $f$ belongs to the kernel of the operator $\Delta_{n+1}^\alpha$. Now, we suppose that $f$ belongs to the kernel of $\Delta_{n+1}^\alpha$. Since $f$ is slice hyperholomorphic we deduce that, after a change of variable if it is needed, there exists $r>0$ such that the function $f$ can be expanded in a convergent series at the origin
	$$ f(x)= \sum_{k=0}^{\infty} x^k \alpha_k, \quad \hbox{for} \, \{\alpha_k\}_{k \in \mathbb{N}_0} \subset \mathbb{R}_n,$$
	for any $x \in B_r(0)$. Now, we apply the differential operator $ \Delta_{n+1}^{\alpha}$ and we get
	\begin{equation}
		\label{aux_2}
		0= \Delta_{n+1}^{\alpha}f(x)= \sum_{k=2 \alpha}^{\infty} \Delta_{n+1}^{\alpha} (x^k) \alpha_k, \qquad \forall x \in B_r(0).
	\end{equation}
	Now, we restrict the polynomial $\Delta_{n+1}^{\alpha} (x^k)$ to a neighbourhood of the real axis $\Omega$, by Corollary \ref{appp}, formula \eqref{Real} and Proposition \ref{summ1} we get
	$$
	\left( \Delta^{\alpha}_{n+1}x^k\right)\Big|_{\underline{x}=0}= c_k x_0^{k-2\alpha}, \qquad c_k:= \frac {4^\alpha (-1)^\alpha \alpha! (h_n-\alpha+1)_\alpha}{(2\alpha)! } \frac{k!}{(k-2\alpha)!}, \qquad k\geq 2\alpha.
	$$
	Thus by restricting \eqref{aux_2} to a neighbourhood of the real axis we get
	$$ \alpha_k=0, \qquad \forall k \geq 2 \alpha.$$
	Therefore we have that $f(x)= \sum_{\nu=0}^{2 \alpha-1} x^{\nu} \alpha_{\nu}$ in $\Omega$. Finally since $U$ is a connected set we get that $f(x)= \sum_{\nu=0}^{2 \alpha-1} x^{\nu} \alpha_{\nu}$ for any $x \in U$.
\end{proof}

Next we show that holomorphic Cliffordian function has an integral representation.

\begin{theorem}
Let $n$ be an odd number, set $h_n:=(n-1)/2$ and let $\alpha\in \mathbb{N}$ be such that $1 \leq \alpha \leq h_n-1$.
 Let $U \subset \mathbb{R}^{n+1}$ be a bounded slice Cauchy domain, for $I \in \mathbb{S}$, we set $ds_I=ds(-I)$.
\begin{itemize}
\item If $f$ is a left slice hyperholomorphic function on a set that contains $\overline{U}$, then the left axially holomorphic Cliffordian function $\breve{f}_0(x)=\Delta_{n+1}^\alpha f(x)$, of order $(h_n-\alpha,1)$, has the integral representation
 \begin{equation}
\label{cliff1}
\breve{f}_0(x)=\frac{1}{2\pi} \int_{\partial (U \cap \mathbb{C}_I)} \mathcal{K}^L_{\alpha}(s,x) ds_I f(s),
\end{equation}
\item
If $f$ is a right slice hyperholomorphic function on a set that contains $\overline{U}$, then the right axially holomorphic Cliffordian function $\breve{f}_0(x)=\Delta_{n+1}^\alpha f(x)$, of order $(h_n-\alpha,1)$,
  has the integral representation
\begin{equation}
\label{cliff2}
\breve{f}_0(x)=\frac{1}{2\pi} \int_{\partial (U \cap \mathbb{C}_I)} f(s) ds_I \mathcal{K}^R_{\alpha}(s,x).
\end{equation}
\end{itemize}
Moreover, the integrals in \eqref{cliff1} and \eqref{cliff2} depend neither on $U$ and nor on the imaginary unit $I \in \mathbb{S}$ and not even on the kernel of $\Delta_{n+1}^\alpha $.
\end{theorem}
\begin{proof}
The result follows by combining \eqref{holCliff} and the Cauchy formula for slice hyperholomorphic functions, see Theorem \ref{Cauchy}. The independence from the set $U$ and the imaginary unit $I \in \mathbb{S}$ follows from the Cauchy theorem, similar arguments as in Theorem \ref{IP} show the independence on the kernel of $\Delta_{n+1}^\alpha $.
\end{proof}

\section{The holomorphic Cliffordian functional calculus based on the $S$-spectrum}\label{HOLCLIFFUNCAL}

We illustrate
with the diagram shows the holomorphic Cliffordian fine structure that we construct in this section, the function spaces and the related functional calculi. In this section we investigate the central part of the diagram, specifically the holomorphic Cliffordian functional calculus.
\begin{equation*}
    \begin{CD}
        \mathcal{SH}_L(U)  @> \Delta_{n+1}^\alpha >> \mathcal{AHC}_{h_n-\alpha,1}^L(U) @>\Delta_{n+1}^\alpha>> \mathcal{AM}_L(U) \\
        @V VV    @V VV  @V VV \\
        S\text{-Functional calculus} @. \text{Holomorphic Cliff. functional calculus } @. F\text{-Functional calculus} \\
    \end{CD}
\end{equation*}
A similar diagram holds for functions in $\mathcal{SH}_R(U)$.
We start with the definition of the resolvent operators for this fine structure.
\begin{definition}
\label{cliffresolvent}
Let $n$ be an odd number, set $h_n:=(n-1)/2$ and let $\alpha\in \mathbb{N}$ be such that $1 \leq \alpha \leq h_n-1$.
Let $T \in \mathcal{BC}^{0,1}(V_n)$ and $s\in \rho_S(T)$.  We define the left $\mathcal K$-resolvent operator as
$$ \mathcal{K}^L_{\alpha}(s,T):= k_{\alpha} (s\mathcal{I}-\overline{T}) \mathcal{Q}_{c,s}(T)^{-\alpha-1}$$
and the right  $\mathcal{K}$-resolvent operator as
$$ \mathcal{K}^R_{\alpha}(s,T):= k_{\alpha} \mathcal{Q}_{c,s}(T)^{-\alpha-1}(s\mathcal{I}-\overline{T}),$$
where $k_\alpha$ is defined in \eqref{ka}.
\end{definition}

Now, we write the $\mathcal{K}$-resolvent operator in terms of the $F$-resolvent operator.

\begin{proposition}
\label{FK}
Let $n$ be an odd number and set $h_n:=(n-1)/2$. For $s \in \mathbb{R}^{n+1}$, $T \in \mathcal{BC}^{0,1}(V_n)$ with $1 \leq \alpha \leq h_n-1$ we have
\begin{equation}
\label{kleft}
\!\! \! \! \! \! \! \! \mathcal{K}^L_{\alpha}(s,T)= \frac{k_{\alpha}}{\gamma_n} \sum_{\ell=0}^{h_n-\alpha} \sum_{\nu=0}^{h_n-\alpha-\ell} \binom{h_n-\alpha}{\ell} \binom{h_n-\alpha-\ell}{\nu} F_n^L(s,T)s^{h_n-\alpha+\ell-\nu} (-2T_0)^{h_n-\alpha-\ell-\nu} |T|^{2\nu},
\end{equation}
and
\begin{equation}
\label{kright}
\!\! \! \! \! \! \! \!
\mathcal{K}^R_{\alpha}(s,T)= \frac{k_{\alpha}}{\gamma_n} \sum_{\ell=0}^{h_n-\alpha} \sum_{\nu=0}^{h_n-\alpha-\ell} \binom{h_n-\alpha}{\ell} \binom{h_n-\alpha-\ell}{\nu}   (-2T_0)^{h_n-\alpha-\ell-\nu} |T|^{2\nu}s^{h_n-\alpha+\ell-\nu}F_n^R(s,T),
\end{equation}
where $k_{\alpha}$ are defined in \eqref{ka} and $\gamma_n$ are given in \eqref{gamman}.
\end{proposition}
\begin{proof}
We prove only \eqref{kleft} since \eqref{kright} can be obtained by using similar results. By the definitions of the left $\mathcal{K}$-resolvent operator, the $F$-resolvent operator and the binomial theorem we have
\begin{eqnarray*}
\mathcal{K}_{\alpha}(s,T)&=& \frac{k_{\alpha}}{\gamma_n} F_n^L(s,T) \mathcal{Q}_{c,s}^{h_n-\alpha}(T)\\
&=&  \frac{k_{\alpha}}{\gamma_n} F_n^L(s,T) \sum_{\ell=0}^{h_n-\alpha} \binom{h_n-\alpha}{\ell} \binom{h_n-\alpha}{\ell} s^{2 \ell} (|T|^2-2T_0s)^{h_n-\alpha-\ell}\\
&=& \frac{k_{\alpha}}{\gamma_n} \sum_{\ell=0}^{h_n-\alpha} \sum_{\nu=0}^{h_n-\alpha-\ell} \binom{h_n-\alpha}{\ell} \binom{h_n-\alpha-\ell}{\nu} F_n^L(s,T)s^{h_n-\alpha+\ell-\nu} (-2T_0)^{h_n-\alpha-\ell-\nu} |T|^{2\nu}.
\end{eqnarray*}

\end{proof}

\begin{proposition}
\label{series42}
Let $n$ be an odd number, set $h_n:=(n-1)/2$ and let $\alpha\in \mathbb{N}$ be such that $1 \leq \alpha \leq h_n-1$.
Then, for $s \in \mathbb{R}^{n+1}$ such that $\|T\| <|s|$,
the left $\mathcal K$-resolvent operator admits the series expansion
\begin{equation}
		\label{Cliffser2}
		\mathcal{K}^{L}_{\alpha}(s,T)=  k_{\alpha}\sum_{k=2 \alpha}^\infty  \sum_{\ell=0}^{h_n-\alpha} \sum_{\nu=0}^{h_n-\alpha-\ell}\kappa_{k,h_n, \alpha, \ell, \nu}P^n_{k-2 \alpha-2 \ell+\nu}(T)(-2T_0)^{\nu} |T|^{2(\ell-\nu)}s^{-1-k},
	\end{equation}
	and, the right $\mathcal K$-resolvent operator admits the series expansion
	\begin{equation}
		\label{Cliffser12}
		\mathcal{K}_{R}^{\alpha}(s,T)= k_{\alpha}\sum_{k=2 \alpha}^\infty  \sum_{\ell=0}^{h_n-\alpha} \sum_{\nu=0}^{h_n-\alpha-\ell}\kappa_{k,h_n, \alpha, \ell, \nu}s^{-1-k}P^n_{k-2 \alpha-2 \ell+\nu}(T)(-2T_0)^{\nu} |T|^{2(\ell-\nu)},
	\end{equation}
	where $k_\alpha$ are defined in \eqref{ka} and $\kappa_{k,h_n, \alpha, \ell, \nu}$ are given in \eqref{kappa}.
\end{proposition}
\begin{proof}
The proof follows by combining Proposition \ref{FK} and the expansion in series of the $F$-resolvent operator, see Proposition \ref{exseriesOPR}.
\end{proof}

Now, we show the regularity in the variable of the left and right $\mathcal{K}$-resolvent operators.

\begin{lemma}
\label{Kreg}
Let $n$ be an odd number, set $h_n:=(n-1)/2$ and $T \in \mathcal{BC}^{0,1}(V_n)$. The left (resp. right) $\mathcal{K}$-resolvent operator is a $ \mathcal{B}(V_n)$-valued left (resp. right) slice hyperholomorphic function in the variable $s$.
\end{lemma}
\begin{proof}
We know by Proposition \ref{reg21} that the left (resp. right) $F$-resolvent operator is a $ \mathcal{B}(V_n)$-valued left (resp. right) slice hyperholomorphic function in the variable $s$. Thus by Proposition \ref{FK} and the fact that the pointwise product between a left (resp. right) slice hyperholomorphic function and a slice intrinsic hyperholomorphic function is left (resp. right) slice hyperholomorphic we get the result.
\end{proof}

\begin{definition}[Holomorphic Cliffordian functional calculus for bounded operators]
\label{Clifffun0}
Let $n$ be an odd number, set $h_n:=(n-1)/2$ and let $\alpha\in \mathbb{N}$ be such that $1 \leq \alpha \leq h_n-1$.
 Let us consider $T \in \mathcal{BC}^1(V_n)$ be such that its components $T_i$, $i=1,...,n$ have real spectra. Let $U$ be a bounded slice Cauchy domain as in Definition \ref{SHONTHEFS}.
We set $ds_I=(-I)ds$ where $I \in \mathbb{S}$.
\begin{itemize}
\item
Then for any $f \in \mathcal{SH}_{\sigma_S(T)}^L(U)$ and $\breve{f}_0(x)=\Delta_{n+1}^\alpha f(x)$ we define the holomorphic Cliffordian functional calculus for bounded operators as
\begin{equation}
\label{Clifffun}
\breve{f}_0(T)= \frac{1}{2 \pi} \int_{\partial (U\cap \mathbb{C}_I)} \mathcal{K}^L_{\alpha}(s,T)ds_I f(s).
\end{equation}
\item
For any $f \in \mathcal{SH}_{\sigma_S(T)}^R(U)$ and $\breve{f}_0(x)=\Delta_{n+1}^\alpha f(x)$ we define the holomorphic Cliffordian functional calculus for bounded operators as
\begin{equation}
\label{Clifffun2}
\breve{f}_0(T)= \frac{1}{2 \pi} \int_{\partial (U\cap \mathbb{C}_I)} f(s)ds_I \mathcal{K}^R_{\alpha}(s,T).
\end{equation}
\end{itemize}
\end{definition}

\begin{theorem}
Let $n$ be an odd number, set $h_n:=(n-1)/2$ and let $\alpha$ be such that $1 \leq \alpha \leq h_n-1$.
 Let us consider $T \in \mathcal{BC}^1(V_n)$ be such that its components $T_i$, $i=1,...,n$ have real spectra. Let $U$ be a bounded slice Cauchy domain as in Definition \ref{SHONTHEFS}.
We set $ds_I=(-I)ds$ where $I \in \mathbb{S}$.
 Then for any $f \in \mathcal{SH}_{\sigma_S(T)}^L(U)$ (resp. $f \in \mathcal{SH}_{\sigma_S(T)}^R(U)$) the integral \eqref{Clifffun} (resp \eqref{Clifffun2}) is independent of the slice Cauchy domain $U$ and the imaginary unit $I\in\mathbb S$.
\end{theorem}
\begin{proof}
The result follows by using similar arguments of Theorem \ref{well1}. Indeed also in this case the resolvent operator is left (resp. right) slice hyperholomorphic in $s$, see Lemma \ref{Kreg}, and $ \lim_{s\to\infty} \mathcal{K}_\alpha^L(s,T)= \lim_{s\to\infty}\mathcal{K}_\alpha^R(s,T)=0.$

\end{proof}

Owing to non-commutativity, the right and left forms of the holomorphic Cliffordian-functional calculus typically result in different operators. The sole exception arises when the F-functional calculus is applied to intrinsic slice hyperholomorphic functions.

\begin{proposition}
	\label{KLR}
	Let $n$ be an odd number, set $h_n:=(n-1)/2$ and let $\alpha$ be such that $1 \leq \alpha \leq h_n-1$. Suppose $T \in \mathcal{BC}^{0,1}(V_n) $ and $f \in \mathcal{N}_{\sigma_S(T)}(U) $, where $ U $ is a bounded slice Cauchy domain as specified in Definition \ref{SHONTHEFS}.
	Then we have
	$$\frac{1}{2\pi}\int_{\pp(U\cap \mathbb{C}_I)} \mathcal{K}^L_{\alpha}(s,T) \, ds_I\, f(s)=\frac{1}{2\pi}\int_{\pp(U\cap \mathbb{C}_I)} f(s) ds_I  \mathcal{K}^R_{\alpha}(s,T) .$$
\end{proposition}
\begin{proof}
	The result follows by using similar arguments used to prove Theorem \ref{FLR}.
\end{proof}

Before to show that the holomorphic Cliffordian functional calculus is independent from the kernel of the operator $\Delta_{n+1}^\alpha$ we need the following technical result.

\begin{theorem}
	\label{zerothm2}
	Let $n$ be an odd number, set $h_n:=(n-1)/2$ and let $\alpha\in \mathbb{N}$ be such that $1 \leq \alpha \leq h_n-1$.
	Let $T \in \mathcal{BC}^1(V_n)$ be such that its components $T_i$, $i=1,...,n$ have real spectra. Let $G$ be a bounded slice Cauchy domain such that $\partial G\cap \sigma_S(T)= \emptyset$. Thus, for every $I \in \mathbb{S}$, we have
	$$ \int_{\partial (G \cap \mathbb{C}_I)} s^{\beta}ds_I \mathcal{K}^R_{\alpha}(s,T)  =\int_{\partial (G \cap \mathbb{C}_I)} \mathcal{K}^L_{\alpha}(s,T) ds_I s^{\beta}=0, \qquad \hbox{if} \quad 0 \leq \beta \leq 2\alpha-1.$$
\end{theorem}
\begin{proof}
	The first equality follows from the fact that $s^\beta$ is slice intrinsic hyperholomorphic, see Theorem \ref{KLR}. Now we prove the first equality. By the Fueter-Sce theorem, see Theorem \ref{intform},  we have
	\begin{equation}
		\label{intzero2}
		\int_{\partial(G \cap \mathbb{C}_I)} F_n^L(s,x) ds_J s^{m}= \Delta^{h_n}_{n+1}(x^m)=0, \qquad \hbox{if} \quad m \leq 2h_n-1,
	\end{equation}
	for all $x \notin \partial G$ and $I \in \mathbb{S}$. By Proposition \ref{NewR} and Proposition \ref{FK} we have
	\begingroup\allowdisplaybreaks

	\begin{eqnarray*}
		 \int_{\partial(G \cap \mathbb{C}_I)} \mathcal{K}^L_{\alpha}(s,T) ds_J s^{\beta}&=& \frac{k_{\alpha}}{\gamma_n} \sum_{\ell=0}^{h_n-\alpha}\sum_{\nu=0}^{h_n-\alpha-\ell}  \binom{h_n-\alpha}{\ell}   \binom{h_n -\alpha-\ell}{\nu}  (-2T_0)^{h_n-\alpha-\ell-\nu} |T|^{2\nu}\times\\
		 	&& \times  \int_{\partial(G \cap \mathbb{C}_I)} F_n^L(s,T) \, ds_I \, s^{h_n-\alpha+\ell-\nu+\beta} \\
		 	&=& \frac{k_{\alpha}}{\gamma_n} \sum_{\ell=0}^{h_n-\alpha}\sum_{\nu=0}^{h_n-\alpha-\ell}  \binom{h_n-\alpha}{\ell}   \binom{h_n -\alpha-\ell}{\nu}  (-2T_0)^{h_n-\alpha-\ell-\nu} |T|^{2\nu} \times\\
		 	&& \times  \int_{\partial(G \cap \mathbb{C}_I)} \left(\int_{\partial W}G(\omega, T) D \omega F_n^L(s, \omega)\right)  \, ds_I \, s^{h_n-\alpha+\ell-\nu+\beta}\\
		 		&=& \frac{k_{\alpha}}{\gamma_n} \sum_{\ell=0}^{h_n-\alpha}\sum_{\nu=0}^{h_n-\alpha-\ell}  \binom{h_n-\alpha}{\ell}   \binom{h_n -\alpha-\ell}{\nu}  (-2T_0)^{h_n-\alpha-\ell-\nu} |T|^{2\nu} \times\\
		 	&& \times \int_{\partial W}G(\omega, T) D \omega \left(\int_{\partial(G \cap \mathbb{C}_I)}  F_n^L(s, \omega) \, ds_I \, s^{h_n-\alpha+\ell-\nu+\beta}\right)\\
		 	&=&0,
	\end{eqnarray*}
	\endgroup
	where $W$ as in Definition \ref{locmongspetrum}. The last equality holds to be true by \eqref{intzero2}, since $h_n-\alpha+\ell-\nu+\beta\leq 2h_n-1$.
\end{proof}

\begin{proposition}\label{kernel_indip}
	Let $n$ be an odd number, set $h_n:=(n-1)/2$ and let $\alpha\in \mathbb{N}$ be such that $1 \leq \alpha \leq h_n-1$.  We assume $T \in \mathcal{BC}^1(V_n)$ be such that its components $T_i$, $i=1,...,n$ have real spectra. Let $U$ be a slice Cauchy domain with $\sigma_S(T) \subset U$.
	Suppose that  $f$, $g \in \mathcal{SH}_{\sigma_S(T)}^L(U)$ (resp.$f$, $g \in \mathcal{SH}_{\sigma_S(T)}^R(U)$) and $ \Delta_{n+1}^{\alpha} f(x)= \Delta_{n+1}^{\alpha} g(x)$ (resp. $ \Delta_{n+1}^{\alpha} f(x)= \Delta_{n+1}^{\alpha} g(x)$ ). Then, we have $\breve{f}_0 (T)=\breve{g}_0(T)$.
\end{proposition}
\begin{proof}
We show the result only for left slice hyperholomorphic functions, the case for right slice hyperholomorphic functions follows by using similar arguments. We start assuming that $U$ is connected. By the definition of the  holomorphic Cliffordian functional calculus, see Definition \ref{Clifffun0}, we have
$$ \breve{f}_0(T)-\breve{g}_0(T)= \frac{1}{2\pi} \int_{\partial(U \cap \mathbb{C}_I)} \mathcal{K}^{L}_\alpha (s,T) \, ds_I(f(s)-g(s)).$$
By Lemma \ref{Kreg} we know that the operator $\mathcal{K}_{L}^\alpha(s,T)$ is a right slice hyperholomorphic function in $s$, thus we can change the domain of integration to
$\partial (B_r(0) \cap \mathbb{C}_J)$ for $r>0$ such that $\| T \| <r$.  We know by hypothesis that $f-g$ belongs to the kernel of the operator $\Delta_{n+1}^{\alpha}$, thus by Lemma \ref{kernel02} we get
\begin{eqnarray*}
\breve{f}_0(T)-\breve{g}_0(T)&=& \frac{1}{2\pi} \int_{\partial(B_r(0) \cap \mathbb{C}_I)} \mathcal{K}^{L}_\alpha(s,T) ds_I(f(s)-g(s))\\
&=&\frac{1}{2\pi}  \sum_{\nu=0}^{2 \alpha-1} \int_{\partial(B_r(0) \cap \mathbb{C}_I)} \mathcal{K}^{L}_\alpha(s,T)\, ds_I s^{\nu} \alpha_{\nu}
\end{eqnarray*}
where $\{\alpha_{\nu}\}_{1 \leq \nu \leq 2 \alpha-1} \subseteq \mathbb{R}_n$. By Theorem \ref{zerothm2} (since $ \nu \leq 2\alpha-1$) the previous integral is zero.

Now, we suppose taht $U$ is not connected so we can write $U=\bigcup_{\ell=1}^m U_{\ell}$ where $U_{\ell}$ are the connected components of the set $U$. Thus by Lemma \ref{kernel02} we have $f(s)-g(s)=\sum_{\ell=1}^{m} \sum_{\nu=0}^{2 \alpha-1} \chi_{U_{\ell}} s^{\nu}\alpha_{\nu,\ell}$. Hence by the definition of the Cliffordian functional calculus, see Definition \ref{Clifffun0}, we have
$$ \breve{f}_0(T)-\breve{g}_0(T)= \frac{1}{2\pi} \sum_{\ell=1}^{m} \sum_{\nu=0}^{2\alpha-1} \int_{\partial(U_{\ell} \cap \mathbb{C}_J)} \mathcal{K}_{\alpha}^L(s,T)ds_J s^{\nu}\alpha_{\nu,\ell}.$$
Finally, by Theorem \ref{zerothm2}  we get that $\breve{f}_0(T)=\breve{g}_0(T)$.
\end{proof}

Now, we show some algebraic properties of the holomorphic Cliffordian functional calculus.

\begin{lemma}
Let $n$ be an odd number, set $h_n:=(n-1)/2$ and let $\alpha\in \mathbb{N}$ be such that $1 \leq \alpha \leq h_n-1$.
We assume $T \in \mathcal{BC}^1(V_n)$ be such its components $T_{i}$, $i=1,...,n$ have real spectra.
Suppose that $\breve{f}_0(x) = \Delta_{n+1}^{\alpha}f(x)$ and $\breve{g}_0(x)=  \Delta_{n+1}^{\alpha}g(x)$
 with $f$, $g \in \mathcal{SH}_{\sigma_S(T)}^L(U)$ (resp. $\mathcal{SH}_{\sigma_S(T)}^R(U)$) and let $a \in \mathbb{R}_n$.
Then, we have
$$ (\breve{f}_0a+\breve{g}_0)(T)=\breve{f}_0(T)a+\breve{g}_0(T), \qquad (resp. (a\breve{f}_0+\breve{g}_0)(T)=a\breve{f}_0(T)+\breve{g}_0(T)).$$
\end{lemma}
\begin{proof}
The result follows from the linearity of the integrals in Definition \ref{Clifffun0}.
\end{proof}

\begin{proposition}
Let $n$ be an odd number, set $h_n:=(n-1)/2$ and let $\alpha \in \mathbb{N}$ be such that $1 \leq \alpha \leq h_n-1$.
We assume $T \in \mathcal{BC}^1(V_n)$ be such that its components $T_{i}$, $i=1,...,n$ have real spectra.
\begin{itemize}
\item Let  $\breve{f}_0(x)= \Delta_{n+1}^{\alpha}f(x)$
where  $f(x)= \sum_{m=0}^{\infty} x^m a_m$ with $a_m \in \mathbb{R}_n$,  converges on a ball $B_r(0)$, for $r>0$, with $\sigma_S(T) \subset B_r(0)$.
Then, we have
$$
\breve{f}_0(T)= \frac{k_{\alpha}}{\gamma_n}\sum_{k=2 \alpha}^\infty  \sum_{\ell=0}^{h_n-\alpha} \sum_{\nu=0}^{h_n-\alpha-\ell}\kappa_{k,h_n, \alpha, \ell, \nu}P^n_{k-2 \alpha-2 \ell+\nu}(T)(-2T_0)^{\nu} |T|^{2(\ell-\nu)}a_k.
$$
\item Let  $\breve{f}_0(x)= \Delta_{n+1}^{\alpha} f(x) $ where  $f(x)= \sum_{m=0}^{\infty} a_m x^m $ with $a_m \in \mathbb{R}_n$,  converges on a ball $B_r(0)$, for $r>0$, with $\sigma_S(T) \subset B_r(0)$. Then, we have
$$
\breve{f}_0(T)=\frac{k_{\alpha}}{\gamma_n}\sum_{k=2 \alpha}^\infty  \sum_{\ell=0}^{h_n-\alpha} \sum_{\nu=0}^{h_n-\alpha-\ell}a_k\kappa_{k,h_n, \alpha, \ell, \nu}P^n_{k-2 \alpha-2 \ell+\nu}(T)(-2T_0)^{\nu} |T|^{2(\ell-\nu)}.
$$
\end{itemize}
\end{proposition}
\begin{proof}
We prove only the first statement since the second one follows from similar arguments. We pick an imaginary $I \in \mathbb{S}$ and a radius $0<R< r$ such that $\sigma_S(T) \subset B_R(0)$. Thus the series expansion of $f$ converges uniformly on $\partial (B_R(0) \cap \mathbb{C}_I)$, and so we have
\begin{align*}
\breve{f}_0(T)&=\frac{1}{2\pi} \int_{\partial(B_R(0) \cap \mathbb{C}_I)}  \mathcal{K}^{L}_\alpha (s,T)ds_I \sum_{m=0}^{\infty} s^m a_{m}\\
&= \frac{1}{2\pi}\sum_{m=0}^{\infty} \int_{\partial(B_R(0) \cap \mathbb{C}_I)} \mathcal{K}^{L}_\alpha (s,T) ds_I  s^m a_{m}.
\end{align*}
By using the series expansion of the resolvent operator $\mathcal{K}^{L}_\alpha (s,T)$ obtained in Proposition \ref{series42} we have
\begin{align*}
\breve{f}_0(T)&=\frac 1{2\pi} \frac{k_{\alpha}}{\gamma_n}\sum_{m=0}^{\infty}\sum_{k=2 \alpha}^\infty  \sum_{\ell=0}^{h_n-\alpha} \sum_{\nu=0}^{h_n-\alpha-\ell}\kappa_{k,h_n, \alpha, \ell, \nu}P^n_{k-2 \alpha-2 \ell+\nu}(T)(-2T_0)^{\nu} |T|^{2(\ell-\nu)}\\
&\times \int_{\partial(B_R(0) \cap \mathbb{C}_I)} s^{-1-k+m} a_{m} ds_I.
\end{align*}
Now, by using the fact that
$$ \int_{\partial(B_R(0) \cap \mathbb{C}_I)} s^{-1-k+m}ds_I=\begin{cases}
2 \pi, \qquad \hbox{if} \quad m=k \\
0, \qquad \hbox{if} \quad m\neq k,
\end{cases}$$
 we get
$$ \breve{f}_0(T)=\frac{k_{\alpha}}{\gamma_n}\sum_{k=2 \alpha}^\infty  \sum_{\ell=0}^{h_n-\alpha} \sum_{\nu=0}^{h_n-\alpha-\ell}\kappa_{k,h_n, \alpha, \ell, \nu}P^n_{k+2\ell+\nu-2h_n}(T)(-2T_0)^{\nu} |T|^{2h_n-2\alpha-2\ell-2\nu}a_k.$$
This proves the result.
\end{proof}

\subsection{Product rule for the $F$-functional calculus}
\label{PROD}

In order to get a product formula for the $F$-functional calculus we need to use the polyharmonic functional calculus and the holomorphic Cliffordian functional calculus developed in the previous sections. We need the following preliminary result, see \cite[Lemma 3.18]{ACGS}.

\begin{lemma}
\label{commope}
Let $n\in \mathbb{N}$ and $B \in \mathcal{B}(V_n)$. Let $G$ be an axially symmetric domain and  assume that $f \in \mathcal{N}(G)$. Then, for $p \in G$, we have
$$ \frac{1}{ 2 \pi} \int_{\partial (G \cap \mathbb{C}_I)}f(s)ds_I (\bar{s}B-Bp)(p^2-2s_0p+|s|^2)^{-1}=Bf(p).$$
\end{lemma}

\begin{theorem}[Product formula for the $F$-functional calculus]
Let $n$ be an odd number and set $h_n:=(n-1)/2$.
We assume $T \in \mathcal{BC}^1(V_n)$ be such that its components $T_{i}$, $i=1,...,n$ have real spectra.
Suppose that $f \in \mathcal{N}^L_{\sigma_S(T)}(U)$ and $g \in \mathcal{SH}^L_{\sigma_S(T)}(U)$. Then, we have
\begin{eqnarray}
\nonumber
\Delta^{h_n}_{n+1} (fg)(T)&=&(\Delta^{h_n}_{n+1} f)(T) g(T)+f(T)( \Delta^{h_n}_{n+1} g)(T)+ \sum_{i=0}^{h_n-2} (\Delta^{h_n-i-1}_{n+1}f)(T) (\Delta_{n+1}^{i+1}g)(T)\\
\label{pf}
&&- \sum_{i=0}^{h_n-1} (D \Delta^{h_n-i-1}_{n+1}f)(T)(D\Delta^{i}_{n+1} g)(T).
\end{eqnarray}
\end{theorem}
\begin{proof}
We consider $G_1$ and $G_2$ two bounded slice Cauchy domains such that contain $\sigma_S(T)$ and $\overline{G}_1 \subset G_2$, with $\overline{G}_2 \subset \hbox{dom}(f) \cap \hbox{dom}(g)$. For every $I \in \mathbb{S}$ we pick $p \in \partial (G_1 \cap \mathbb{C}_I)$ and $s \in \partial(G_2 \cap \mathbb{C}_I)$. Then by the definition of the $S$-functional calculus (see Definition \ref{Sfun}), $F$-functional calculus (see Definition \ref{Ffun}), the polyharmonic functional calculus (see Definition \ref{polyH}) and the holomorphic Cliffordian functional calculus (see Definition \ref{Clifffun0}) we have
\begin{eqnarray}
\label{prodrule1}
&&(\Delta^{h_n}_{n+1} f)(T) g(T)+f(T)( \Delta^{h_n}_{n+1} g)(T)+ \sum_{i=0}^{h_n-2} (\Delta^{h_n-i-1}_{n+1}f)(T) (\Delta^{i+1}_{n+1}g)(T)\\
\nonumber
&&- \sum_{i=0}^{h_n-1} (D \Delta^{h_n-i-1}_{n+1}f)(T)(D\Delta^{i}_{n+1} g)(T)\\
\nonumber
&=&\frac{1}{(2 \pi)^2} \int_{\partial(G_2 \cap \mathbb{C}_I)} f(s) ds_I F_n^R(s,T)\int_{\partial(G_1 \cap \mathbb{C}_I)} S^{-1}_L(p,T)dp_I g(p)\\
\nonumber
&&+\frac{1}{(2 \pi)^2} \int_{\partial(G_2 \cap \mathbb{C}_I)} f(s) ds_I S^{-1}_R(s,T)\int_{\partial(G_1 \cap \mathbb{C}_I)} F_n^L(p,T)dp_I g(p)\\
\nonumber
&& +\frac{1}{(2 \pi)^2} \sum_{i=0}^{h_n-2}\int_{\partial(G_2 \cap \mathbb{C}_I)} f(s)ds_I \mathcal{K}^R_{h_n-i-1}(s,T)\int_{\partial(G_1 \cap \mathbb{C}_I)}  \mathcal{K}^L_{i+1}(p,T)dp_Ig(p)\\
\nonumber
&&- \frac{1}{4\pi^2}\sum_{i=0}^{h_n-1} \int_{\partial(G_2 \cap \mathbb{C}_I)} \sigma_{n,h_n-i}\sigma_{n,i+1}f(s)ds_I \mathcal{Q}_{c,s}(T)^{-h_n+i}\int_{\partial(G_1 \cap \mathbb{C}_I)} \mathcal{Q}_{c,p}(T)^{-i-1}dp_Ig(p),
\end{eqnarray}
where we have also used Proposition \ref{SLR}, Proposition \ref{FLR}, Proposition \ref{KLR} and Proposition \ref{QLR}.
From the definition of the left and right $\mathcal{K}$-resolvent operators, see Definition \ref{cliffresolvent}, we have
$$ \mathcal{K}_R^{h_n-i-1}(s,T)=k_{h_n-i-1} \mathcal{Q}_{c,s}^{-h_n+i}(T)(s\mathcal{I}-\overline{T})=k_{h_n-i-1} \mathcal{Q}_{c,s}^{-h_n+i+1}(T) S^{-1}_R(s,T),$$
and
$$ \mathcal{K}_L^{i+1}(p,T)=k_{i+1} (p\mathcal{I}-\overline{T})\mathcal{Q}_{c,p}^{-i-2}(T)=k_{i+1} S^{-1}_L(p,T)\mathcal{Q}_{c,p}^{-i-1}(T).$$
Now, by  \eqref{gammaL} and \eqref{gamman}, we have
 $k_{h_n-i-1}k_{i+1}=\gamma_n$ and $\sigma_{n,h_n-i}\sigma_{n,i+1}=-\gamma_n$, where the constants  $k_\alpha$,  $\gamma_n$ and $\sigma_{n,\ell}$ are defined
  in \eqref{ka}, (\ref{gamman}) and(\ref{gammaL}), respectively.
			  Hence we can write \eqref{prodrule1} as
\begin{eqnarray}
\label{prodrule2}
	&&(\Delta^{h_n}_{n+1} f)(T) g(T)+f(T)( \Delta^{h_n}_{n+1} g)(T)+ \sum_{i=0}^{h_n-2} (\Delta^{h_n-i-1}_{n+1}f)(T) (\Delta^{i+1}_{n+1}g)(T)\\
	\nonumber
	&&- \sum_{i=0}^{h_n-1} (D \Delta^{h_n-i-1}_{n+1}f)(T)(D\Delta^{i}_{n+1} g)(T)\\
	\nonumber
	&=& \frac{1}{2 \pi^2} \int_{\partial(G_2 \cap \mathbb{C}_I)} \int_{\partial(G_1 \cap \mathbb{C}_I)} f(s) ds_I \Big\{ F_n^R(s,T)S^{-1}_L(p,T)+S^{-1}_R(s,T)F_n^L(p,T)\\
	\nonumber
	&&  +\gamma_n \Big[ \sum_{i=0}^{h_n-2}\mathcal{Q}_{c,s}^{-h_n+i+1}(T)S^{-1}_R(s,T)S^{-1}_L(p,T) \mathcal{Q}_{c,p}^{-i-1}(T)+\sum_{i=0}^{h_n-1} \mathcal{Q}_{c,s}^{-h_n+i}(T) \mathcal{Q}_{c,p}^{-i-1}(T)\Big] \Big\}dp_I g(p).
\end{eqnarray}
Now, by using Theorem \ref{resolvent} we can write \eqref{prodrule2} as
\begin{eqnarray}
	\label{prodrule3}
&&(\Delta^{h_n}_{n+1} f)(T) g(T)+f(T)( \Delta^{h_n}_{n+1} g)(T)+ \sum_{i=0}^{h_n-2} (\Delta^{h_n-i-1}_{n+1}f)(T) (\Delta^{i+1}_{n+1}g)(T)\\
\nonumber
&&- \sum_{i=0}^{h_n-1} (D \Delta^{h_n-i-1}_{n+1}f)(T)(D\Delta^{i}_{n+1} g)(T)\\
	\nonumber
	&=& \frac{1}{2 \pi^2} \int_{\partial(G_2 \cap \mathbb{C}_I)} \int_{\partial(G_1 \cap \mathbb{C}_I)} f(s) ds_I \left[ \left( F_n^R(s,T)-F_n^L(p,T)\right)p \right.\\
	\nonumber
	&& \left. -\bar{s}\left( F_n^R(s,T)-F_n^L(p,T)\right) \right] (p^2-2s_0p+|s|^2)^{-1}.
\end{eqnarray}
By using the linearity of the integrals we get
\begin{eqnarray*}
&&(\Delta^{h_n}_{n+1} f)(T) g(T)+f(T)( \Delta^{h_n}_{n+1} g)(T)+ \sum_{i=0}^{h_n-2} (\Delta^{h_n-i-1}_{n+1}f)(T) (\Delta^{i+1}_{n+1}g)(T)\\
\nonumber
&&- \sum_{i=0}^{h_n-1} (D \Delta^{h_n-i-1}_{n+1}f)(T)(D\Delta^{i}_{n+1} g)(T)\\
	\nonumber
	&=& \frac{1}{(2 \pi)^2} \int_{\partial(G_2 \cap \mathbb{C}_I)} f(s) ds_I \int_{\partial(G_1 \cap \mathbb{C}_I)} F_n^R(s,T)p (p^2-2s_0p+|s|^2)^{-1}dp_I g(p)\\
	&& -\frac{1}{(2 \pi)^2} \int_{\partial(G_2 \cap \mathbb{C}_I)} f(s) ds_I \int_{\partial(G_1 \cap \mathbb{C}_I)} F_n^L(s,T)p (p^2-2s_0p+|s|^2)^{-1}dp_I g(p)\\
	&& -\frac{1}{(2 \pi)^2} \int_{\partial(G_2 \cap \mathbb{C}_I)} f(s) ds_I \int_{\partial(G_1 \cap \mathbb{C}_I)} \bar{s}F_n^R(s,T) (p^2-2s_0p+|s|^2)^{-1}dp_I g(p)\\
	&&+\frac{1}{(2 \pi)^2} \int_{\partial(G_2 \cap \mathbb{C}_I)} f(s) ds_I \int_{\partial(G_1 \cap \mathbb{C}_I)} \bar{s}F_n^L(s,T) (p^2-2s_0p+|s|^2)^{-1}dp_I g(p).
\end{eqnarray*}
Now, we observe that the functions $p \mapsto p(p^2-2s_0p+|s|^2)^{-1}$, $p \mapsto (p^2-2s_0p+|s|^2)^{-1}$ are intrinsic slice hyperholomorphic on $G_1$, by Theorem \ref{Cif}, we get
 $$ \int_{\partial(G_2 \cap \mathbb{C}_I)} f(s) ds_I \int_{\partial(G_1 \cap \mathbb{C}_I)} F_n^R(s,T)p (p^2-2s_0p+|s|^2)^{-1}dp_I g(p)=0,$$

$$ \int_{\partial(G_2 \cap \mathbb{C}_I)} f(s) ds_I \int_{\partial(G_1 \cap \mathbb{C}_I)} \bar{s}F_n^L(s,T) (p^2-2s_0p+|s|^2)^{-1}dp_I g(p)=0.$$
Thus we obtain
\begin{eqnarray*}
&&(\Delta^{h_n}_{n+1} f)(T) g(T)+f(T)( \Delta^{h_n}_{n+1} g)(T)+ \sum_{i=0}^{h_n-2} (\Delta^{h_n-i-1}_{n+1}f)(T) (\Delta^{i+1}_{n+1}g)(T)\\
\nonumber
&&- \sum_{i=0}^{h_n-1} (D \Delta^{h_n-i-1}_{n+1}f)(T)(D\Delta^{i}_{n+1} g)(T)\\
	&=& \frac{1}{(2 \pi)^2} \int_{\partial(G_1 \cap \mathbb{C}_I)}\int_{\partial(G_2 \cap \mathbb{C}_I)} f(s) ds_I  \left[\bar{s}F_n^L(p,T)-F_n^L(p,T)p \right](p^2-2s_0p+|s|^2)^{-1}dp_I g(p).
\end{eqnarray*}
Finally by applying Lemma \ref{commope} with $B:=F_n^L(p,T)$ and the definition of the $F$-functional calculus, see Definition \ref{Ffun}, we have
\begin{eqnarray*}
	&&(\Delta^{h_n}_{n+1} f)(T) g(T)+f(T)( \Delta^{h_n}_{n+1} g)(T)+ \sum_{i=0}^{h_n-2} (\Delta^{h_n-i-1}f)(T) (\Delta^{i+1}_{n+1}g)(T)\\
	&&- \sum_{i=0}^{h_n-1} (D \Delta^{h_n-i-1}_{n+1}f)(T)(D\Delta^{i}_{n+1} g)(T)\\
&=& \frac{1}{2 \pi} \int_{\partial(G_1 \cap \mathbb{C}_I)}F_n^L(p,T) dp_I f(p)g(p)\\
&=& \Delta^{h_n}_{n+1}(fg).
\end{eqnarray*}
This proves the result.
\end{proof}

\begin{remark}
If we consider $n=3$ in formula \eqref{pf} we get back to the formula obtained in \cite[Thm. 9.3 ]{CDPS}, since we understood the sum $ \sum_{i=0}^{-1}$ to be zero. Furthermore, if we consider $n=5$ we get back exactly the same formula of \cite[Them. 10.3]{Fivedim}.
\end{remark}

 \begin{remark}
By replacing the operator $T$ with the paravector $x$ in \eqref{pf}, and setting $f(x)=x$
and $g(x)=x^m$, with $m \in \mathbb{N}$, we re-obtain formula \eqref{pf1}.
 \end{remark}

\begin{remark}
To get the product formula for the $F$-functional calculus it is crucial the holomorphic Cliffordian functional calculus developed in this section.
\end{remark}

\section{Integral representation of polyanalytic functions}\label{polyanalytic functions}
The function spaces that we have considered in the previous sections, i.e. slice hyperholomorphic, axially polyharmonic and axially monogenic functions appear in many fields of pure and applied mathematics, the polyanalytic functions are less known but still have several applications.

\medskip
The complex version of this class of functions was first considered by G.V. Kolossov in connection with his research on elasticity and also have applications in signal analysis, particularly in the context of Gabor frames with Hermite functions, as shown by the results of Gröchenig and Lyubarskii. Polyanalytic functions provide explicit representation formulas for functions in the eigenspaces of the Euclidean Laplacian with a magnetic field, which are referred to as Landau levels. This likely has significant implications in quantum mechanics and related fields. For an overview of the applications of this class of functions see the paper \cite{FISH} and also \cite{A1, Balk,Russ,Bergmanvasi} for further material on this theory.
Finally, we remark that the study and development of polyanalytic boundary value problems have continued to flourish in recent times. For more details and further developments, the reader can refer to the papers \cite{Begehr2,Begehr3,Begehr1,Nechaev,Pessoa, RD} and the references therein.

\medskip
 A polyanalytic function belongs to the kernel of the powers of the Dirac operator and where first considered by F. Brackx in \cite{B1976}.

 \begin{definition}
Let $m \geq 1$. We assume that $U \subset \mathbb{R}^{n+1}$ is an open set and let $f:U \to \mathbb{R}_n$ be a function of class $\mathcal{C}^m(U)$ and is of axial type (see \eqref{axx}). We say that $f$ is left (resp. right) axially polyanalytic (of order $m$) on $U$ if
$$ D^mf(x)=0, \qquad \forall x \in U \qquad \left(\hbox{resp. } \quad  f(x)D^m=0, \qquad \forall x \in U \right).$$
\end{definition}

\medskip
One of the peculiarity of an axially polyanalytic function is that they can be decomposed in terms of axially monogenic functions.

\begin{theorem}
	\label{polydeco}
Let $U \subset \mathbb{R}^{n+1}$ be an open set. A function $f:U \to \mathbb{R}_n$ is left (resp. right) axially polyanalytic of order $m$ if and only if it can be decomposed in terms of unique left (resp. right) axially monogenic functions $f_0$,..., $f_{m}$ such that
$$ f(x)=\sum_{k=0}^{m-1}x_0^k f_k(x).$$
\end{theorem}
Now, we show that the class of left and right axially polyanalytic function arise naturally in the factorization of the second map $T_{FS2}$ of the Fueter-Sce extension theorem.

In this section, we will establish each result exclusively for left slice hyperholomorphic functions, since similar arguments can be used to obtain the corresponding results for right slice hyperholomorphic functions.

\begin{theorem}
\label{fact}
Let $n$ be an odd number and set $h_n:=(n-1)/2$ assume that $\ell\in \mathbb{N}$ is such that $0 \leq \ell \leq h_n-1$. Then, the action of the operator
$$ T_{FS}^{(III)}= \Delta_{n+1}^{\ell} \overline{D}^{h_n-\ell}$$
on a left (resp. right) slice hyperholomorphic function on an axially symmetric open set $U \subseteq \mathbb{R}^{n+1}$ gives a left (resp. right) axially polyanalytic function of order $h_n-\ell+1$.
\end{theorem}
\begin{proof}
Let $f$ be a left slice hyperholomorphic function. First we show that $T_{FS}^{(III)}f(x)$ is of axial form. By Lemma \ref{laplacian_sf} we get that $\Delta_{n+1}^{\ell}f(x)$ is of axial type. Finally by Proposition \ref{axxx} we obtain that also $T_{FS}^{(III)} f(x)$ is of axial type. Now, we have to show that
$$ D^{h_n-\ell+1}T_{FS}^{(III)}f(x)=0, \ \  \ \forall x \in U.$$
By the Fueter-Sce mapping theorem, see Theorem \ref{Laplacian_comp} and the fact that $\Delta_{n+1}=D \overline{D}$ we have
\begin{eqnarray*}
	D^{h_n-\ell+1} \left(T_{FS}^{(III)}f(x) \right)&=& D^{h_n-\ell+1} \left(\Delta^{\ell}_{n+1} \overline{D}^{h_n-\ell} f(x)\right)\\
	&=& D^{h_n-\ell+1} \left(D^{\ell} \overline{D}^\ell \overline{D}^{h_n-\ell} f(x)\right)\\
	&=& D \Delta^{h_n}_{n+1}f(x)\\
	&=&0, \ \  \ \forall x \in U.
\end{eqnarray*}
\end{proof}
\begin{definition}[Axially polyanalytic functions associated with $\mathcal{SH}_L(U)$ and $\mathcal{SH}_R(U)$]
Let $n$ be an odd number and set $h_n:=(n-1)/2$. We define left axially polyanalytic functions
of order $h_n-\ell+1$, where $0 \leq \ell \leq h_n-1$, as
$$
\mathcal{AP}_{h_n-\ell+1}^L(U)=\{\Delta_{n+1}^{\ell} \overline{D}^{h_n- \ell}f\ :\ f\in \mathcal{SH}_L(U)\}.
$$
We define right axially polyanalytic functions
of order $h_n-\ell+1$, where $0 \leq \ell \leq h_n-1$, as
$$
\mathcal{AP}_{h_n-\ell+1}^R(U)=\{ f\Delta_{n+1}^{\ell}\overline{D}^{h_n- \ell} \ :\ f\in \mathcal{SH}_R(U)\}.
$$

\end{definition}
\begin{remark}\label{LYTBYTKU}
Theorem \ref{fact} give the following diagram:
\begin{equation}
	\begin{CD}
		&& \textcolor{black}{\mathcal{SH}_L(U)}  @>\ \    \overline{D}^{h_n-\ell} \Delta^{\ell}_{n+1}>>\textcolor{black}{\mathcal{AP}_{h_n-\ell+1}^L(U)}@>\ \   D^{h_n-\ell}>>\textcolor{black}{\mathcal{AM}_L(U)},
	\end{CD}
\end{equation}
and similarly for the case $\mathcal{SH}_R(U)$.
\end{remark}
\begin{remark}
The diagram in Remark \ref{LYTBYTKU} can be further decomposed  as follows:
	\begin{equation}
		\begin{CD}
			&& \textcolor{black}{\mathcal{SH}_L(U)}  @>\ \   \overline{D}^{h_n-\ell} \Delta^{\ell}_{n+1}>>\textcolor{black}{\mathcal{AP}_{h_n-\ell+1}^L(U)}@>\ \   D>>\textcolor{black}{\mathcal{AP}_{h_n-\ell}^L(U)}@>\ \   D >>\textcolor{black}{\mathcal{AP}_{h_n-\ell-1}^L(U)}\\
			&&@>\ \ D>>\textcolor{black}{...}@>\ \  D>>\textcolor{black}{\mathcal{AP}_2^L(U)}@>\ \  D^2>>\textcolor{black}{\mathcal{AM}_L(U)}.
		\end{CD}
	\end{equation}
\end{remark}

Similarly to what we did for the axially polyharmonic and holomorphic Cliffordian functions, our goal is to provide an integral representation of axially polyanalytic functions. To address this problem, we need to apply the operator $T_{FS}^{(III)}$ to the left slice hyperholomorphic Cauchy kernel. However, this is a challenging problem, whose solution needs some preliminary results. To avoid overburdening the main text, we provide a detailed proof of the following result in the appendix, see Section 10.
In Proposition \ref{p1} we concentrate on the case of left slice hyperholomorphic function,
but this result can be re-adapted for the right slice hyperholomorphic setting as well.

\begin{proposition}\label{p1}
Let $m \in \mathbb{N}$.	Let $n$ be a fixed odd number and set $h_n:=(n-1)/2$. We assume that $s$, $x \in \mathbb{R}^{n+1}$ are such that $s \notin[x]$. If $\beta=2m+1$, we have
	\begin{eqnarray}\label{dbars}
		&& \overline{D}^{\beta} S^{-1}_L(s,x)= \frac{2^{2m+1}}{(h_n-2m-2)!} \left(\sum_{j=0}^{m} 2^{2j} (m+j)! (h_n-m-j-1)! \binom{m+j}{2j}(s-x_0)^{2j} \mathcal{Q}_{c,s}^{-m-j-1}(x) \right. \nonumber\\
		&& \left. +(s-\overline{x})\sum_{j=0}^{m-1} 2^{2j+1} (m+j+1)! (h_n-m-j-2)! \binom{m+j+1}{2j+1} (s-x_0)^{2j+1} \mathcal{Q}_{c,s}^{-m-j-2}(x)  \right)\nonumber\\
		&&+ 4^{2m+1} (2m+1)! (s-\overline{x}) (s-x_0)^{2m+1} \mathcal{Q}_{c,s}^{-2m-2}(x).
	\end{eqnarray}
If $\beta=2m$, we have
	\begin{eqnarray}\label{f2}
	\overline{D}^{\beta} S^{-1}_L(s,x)&=& \frac{2^{2m}}{(h_n-2m-1)!} \left(\sum_{j=0}^{m-1} 2^{2j+1} (m+j)! (h_n-m-j-1)! \binom{m+j}{2j+1}\right.
\nonumber\\
	&&\quad\quad\quad\quad\quad\quad\quad\quad\quad\quad\quad\quad\quad\quad\quad\quad\quad\times(s-x_0)^{2j+1} \mathcal{Q}_{c,s}^{-m-j-1}(x)\nonumber
\\
	&& \left. +(s-\overline{x})\sum_{j=0}^{m-2} 2^{2j} (m+j)!(h_n-m-j-1)! \binom{m+j}{2j} (s-x_0)^{2j} \mathcal{Q}_{c,s}^{-m-j-1}(x)  \right)
 \nonumber\\
	&&+ 4^{2m} (2m)! (s-\overline{x}) (s-x_0)^{2m} \mathcal{Q}_{c,s}^{-2m-1}(x).
\end{eqnarray}
\end{proposition}
The above result is crucial for proving the following statement.
\begin{proposition}
	\label{aux1}
	Let $n$ be a fixed odd number and set $h_n:=(n-1)/2$. For $s$, $x \in \mathbb{R}^{n+1}$ such that $s \notin[x]$ we have
	$$ \overline{D}^{h_n} S^{-1}_L(s,x)= \frac{(-1)^{h_n}}{h_n!} F_n^L(s,x)(s-x_0)^{h_n}.$$
\end{proposition}
\begin{proof}
Recall that the left $F$-resolvent kernel,  defined in (\ref{FFL}), written in terms of the Sce exponent $h_n$ is
		\begin{equation}
			F_n^L(s,x):=\gamma_n(s-\bar x)(s^2-2{\rm Re}(x)s +|x|^2)^{-h_n-1},
		\end{equation}
where $\gamma_n$ are defined as (\ref{gamman}), i.e.,
$\gamma_n:=\left[\Gamma\left(\frac{n+1}{2}\right)\right]^2 2^{n-1}(-1)^{\frac{n-1}{2}}.$
Recall that by assumption $n$ be an odd number as a consequence $h_n:=(n-1)/2$ are natural numbers.
So, to apply Proposition \ref{p1}, we have to distinguish the cases  where $h_n$ is even or odd.
	With this observation the proof is a consequence
of formulas \eqref{dbars} and \eqref{f2}, once it is observed that the factors:
$$
\frac{(h_n-m-j-2)!}{(h_n-2m-2)!}\ \ \text{and} \ \ \frac{(h_n-m-j-1)!}{(h_n-2m-2)!}
$$
 in \eqref{dbars} and
 $$
 \frac{(h_n-m-j-1)!}{(h_n-2m-1)!}
$$
in  \eqref{f2} are zero if $\beta=h_n$.
\end{proof}

\begin{remark}
One might inquire why Proposition \ref{aux1} was not proven using induction on $h_n$. The issue with employing such a method lies in the fact that, in the inductive step, the imaginary units of $\overline{D}$ increase due to their dependence on $h_n$. Consequently, it is as issue to apply the induction principle directly.
\end{remark}

\begin{theorem}
\label{polykernel1}
Let $n$ be an odd number and set $h_n:=(n-1)/2$ assume that $\ell\in \mathbb{N}_0$ is such that $0 \leq \ell \leq h_n-1$. For $s$, $x \in \mathbb{R}^{n+1}$ with $s \notin [x]$, we have
\begin{equation}
\label{kernel1}
\Delta^{\ell}_{n+1} \overline{D}^{h_n-\ell} S^{-1}_L(s,x)= \frac{(-1)^{\ell-h_n}}{(h_n-\ell)!} F_n^L(s,x) (s-x_0)^{h_n-\ell},
\end{equation}
and
$$
	 S^{-1}_R(s,x)\Delta^{\ell}_{n+1} \overline{D}^{h_n-\ell}= \frac{(-1)^{\ell-h_n}}{(h_n-\ell)!} (s-x_0)^{h_n-\ell}F_n^R(s,x).
$$
\end{theorem}
\begin{proof}
We prove the result by induction on $ \ell$. For $ \ell=0$ formula \eqref{kernel1} follows by Proposition \ref{aux1}. We suppose that statement is true for $\ell$ and we prove it for $\ell+1$. By the  induction principle we have
\begin{eqnarray*}
\Delta^{\ell+1}_{n+1} \overline{D}^{h_n- \ell-1}S^{-1}_L(s,x)&=& D^{\ell+1} \overline{D}^{h_n} S^{-1}_L(s,x)\\
&=&\frac{(-1)^{\ell-h_n}}{(h_n-\ell)!} D \left( F_n^L(s,x) (s-x_0)^{h_n-\ell} \right).
\end{eqnarray*}
By using the Leibniz formula for the operator $D$, since $(s-x_0)^{h_n-\ell}$ is real-valued in $x$, and the fact that the kernel $F_n^L(s,x)$ is axially monogenic in $x$, see Proposition \ref{reg}, we get
\begin{eqnarray*}
\Delta^{\ell+1}_{n+1} \overline{D}^{h_n- \ell-1}S^{-1}_L(s,x)&=&\frac{(-1)^{\ell-h_n}}{(h_n-\ell)!} F_n^L(s,x) D (s-x_0)^{h_n-\ell}\\
&=& \frac{(-1)^{\ell-h_n+1}}{(h_n-\ell-1)!} F_n^L(s,x) (s-x_0)^{h_n-\ell-1}.
\end{eqnarray*}
This proves the result.
\end{proof}

\begin{definition}[Axially polyanalytic kernels]
Let $n$ be an odd number and set $h_n:=(n-1)/2$ assume that $\ell$ is such that $0 \leq \ell \leq h_n-1$. For $s$, $x \in \mathbb{R}^{n+1}$ such that $s \notin [x]$.
We define the left and right axially polyanalytic kernels of order $h_n-\ell+1$  as
$$
\mathcal{P}^L_\ell(s,x):= \frac{(-1)^{\ell-h_n}}{(h_n-\ell)!}  F_n^L(s,x) (s-x_0)^{h_n-\ell}, \quad \hbox{and} \quad  \mathcal{P}^R_\ell(s,x)=\frac{(-1)^{\ell-h_n}}{(h_n-\ell)!} (s-x_0)^{h_n-\ell}F_n^R(s,x).
$$
\end{definition}
The above definition is motivated by the following proposition.
\begin{proposition}
Let $n$ be an odd number and set $h_n:=(n-1)/2$ assume that $\ell$ is such that $0 \leq \ell \leq h_n-1$.
Assume that $s$, $x \in \mathbb{R}^{n+1}$, with $s \notin [x]$.
Then,  we have that $\mathcal{P}^L_\ell(s,x)$ (resp. $\mathcal{P}^R_\ell(s,x)$) is left (resp. right) intrinsic hyperholomorphic in the variable $s$ and left (resp. right) axially polyanalytic of order $h_n-\ell+1$ in $x$.
\end{proposition}
\begin{proof}
Since the kernel $F_n^L(s,x)$ (resp. $F_n^R(s,x)$) is left (resp, right) in the variable $s$, see Proposition \ref{reg2}, and the function $(s-x_0)^{h_n-\ell}$ is slice intrinsic  we get that the kernel $\mathcal{P}^L_\ell(s,x)$ (resp. $\mathcal{P}^R(s,x)$) is left (resp. right) slice hyperholomorphic. Now, we prove the regularity of the kernel $\mathcal{P}^L_\ell(s,x)$ (resp. $\mathcal{P}^R_\ell(s,x)$) in the second variable. By Proposition \ref{reg} we know that the $F$-kernel $F_n^L(s,x)$ (resp. $F_n^R(s,x)$) is of axial type. Thus it is clear that also the function $ \mathcal{P}^L_\ell(s,x)$ (resp.  $\mathcal{P}^R_\ell(s,x)$) is of axial type. The fact that $ \mathcal{P}^L_\ell(s,x)$ (resp.  $\mathcal{P}^R_\ell(s,x)$) is polyanalytic in the variable $x$ follows by Theorem \ref{fact}.
\end{proof}
\begin{remark}
One can prove that the kernel $ \mathcal{P}^L_\ell(s,x)$ (resp. $ \mathcal{P}^R_\ell(s,x)$) is left (resp. right) axially polyanalytic of order $h_n- \ell+1$ by using the so-called polyanalytic decomposition.
Since the kernel $F_n^L(s,x)$ is left axially monogenic in the variable $x$, see Proposition \ref{reg2}, and by Theorem \ref{polydeco} we get that $\mathcal{P}^L_\ell(s,x)$ (resp. $ \mathcal{P}^R_\ell(s,x)$) is left (resp. right) axially polyanalytic of order $h_n-\ell +1$ in $x$.
\end{remark}

Now, by using the series expansion of the $F$-kernel it is possible to write an expansion in series of the axially polyanalytic kernels.

\begin{theorem}
Let $n$ be an odd number and set $h_n:=(n-1)/2$ assume that $\ell\in \mathbb{N}_0$ is such that $0 \leq \ell \leq h_n-1$.
For $s$, $x \in \mathbb{R}^{n+1}$ such that $|x|<|s|$ we have
	\begin{equation}
		\label{PR}
		\mathcal{P}^L_\ell(s,x)= \frac{\gamma_n(-1)^{h_n-\ell}}{(h_n-\ell)!}  \sum_{k=h_n+\ell}^{\infty}\sum_{i=0}^{h_n-\ell} \tilde{k}_{h_n, \ell, k,i} x_0^{i}  P_{k-i-\ell-h_n}^n(x) s^{-1-k},
	\end{equation}
	and
	\begin{equation}
		\label{PL}
		\mathcal{P}^R_\ell(s,x)= \frac{\gamma_n(-1)^{h_n-\ell}}{(h_n-\ell)!}  \sum_{k=h_n+\ell}^{\infty}\sum_{i=0}^{h_n-\ell}\tilde{k}_{h_n, \ell, k,i}s^{-1-k} x_0^{i}  P_{k-i-\ell-h_n}^n(x),
	\end{equation}
where
\begin{equation}
\label{cons}
\tilde{k}_{h_n, \ell, k,i}:=\binom{h_n-\ell}{i} \binom{k-i-\ell+h_n}{k-i-\ell-h_n}(-1)^{i+h_n-\ell},
\end{equation}
and the homogeneous polynomials $P_k^n(x)$  are defined in (\ref{cliffApp}).

\end{theorem}
\begin{proof}
By Theorem \ref{polykernel1}, the  expansion in series of the $F$-kernel, see Proposition \ref{exseries}, and the Newton binomial we get
\begin{eqnarray*}
	\mathcal{P}^L_\ell(s,x)&=&\Delta^{\ell}_{n+1}  \overline{D}^{h_n-\ell} S^{-1}_L(s,x)\\
	&=& \frac{(-1)^{\ell-h_n}}{(h_n-\ell)!} F_n^L(s,x) (s-x_0)^{h_n-\ell}\\
	&=& \frac{(-1)^{\ell-h_n}}{(h_n-\ell)!} \sum_{k=2h_n}^{\infty} (\Delta^{h_n}_{n+1}x^k) s^{-1-k}(s-x_0)^{h_n-\ell}\\
	&=& \frac{(-1)^{h_n-\ell}}{(h_n-\ell)!}  \sum_{i=0}^{h_n-\ell} \sum_{k=2h_n}^{\infty}\binom{h_n-\ell}{i} (\Delta^{h_n}_{n+1}x^k) x_0^{i} (-1)^i s^{-1-k-\ell-i +h_n}.
\end{eqnarray*}
By making the following change of index $k=u-\ell-i+h_n$ we get
$$ \mathcal{P}^L_\ell(s,x)= \frac{(-1)^{h_n-\ell}}{(h_n-\ell)!} \sum_{i=0}^{h_n-\ell} \sum_{u=h_n+\ell+i}^{\infty} \binom{h_n-\ell}{i} x_0^{i}(-1)^i \left(\Delta^{h_n}_{n+1} x^{u-i-\ell+h_n} \right) s^{-1-u}. $$
Now, we observe that if $h_n+\ell \leq u < h_n+\ell+i$ then $ \Delta^{h_n}_{n+1}(x^{u-i-\ell+h_n})=0$ due to the fact that $u-i-\ell+h_n <2h_n$. Thus we have
$$
	\mathcal{P}^L_\ell(s,x)= \frac{(-1)^{h_n-\ell}}{(h_n-\ell)!} \sum_{i=0}^{h_n-\ell} \sum_{u=h_n+\ell}^{\infty} \binom{h_n-\ell}{i} x_0^{i}(-1)^i \left(\Delta^{h_n}_{n+1} x^{u-i-\ell+h_n} \right) s^{-1-u},
$$
where we recall the relation $\Delta_{n+1}^{h_n} x^m= \gamma_n \binom{m}{m-2h_n}P^n_{m-2h_n}(x)$, for $m \geq 2h_n$ is given in (\ref{app11}).
\end{proof}

\begin{corollary}
\label{app1}
Let $n$ be an odd number and set $h_n:=(n-1)/2$ assume that $\ell\in \mathbb{N}_0$ is such that $0 \leq \ell \leq h_n-1$.
 Then, for $x \in \mathbb{R}^{n+1}$ and for $k\geq 1$,  we have
\begin{equation}
\label{polyapp}
\Delta^{\ell}_{n+1} \overline{D}^{h_n- \ell} x^k=\frac{\gamma_n(-1)^{h_n-\ell}}{(h_n-\ell)!}  \sum_{i=0}^{h_n-\ell} \tilde{k}_{h_n, \ell, k,i} x_0^{i}  P_{k-i-\ell-h_n}^n(x)
\end{equation}
where $\tilde{k}_{h_n, \ell, k,i}$ are defined in \eqref{cons}, $\gamma_n$ are defined in \eqref{gamman}, and the homogeneous polynomials $P_k^n(x)$  are defined in (\ref{cliffApp}).
\end{corollary}
\begin{proof}
We consider $s \in \mathbb{R}^{n+1}$ such that $s \notin [x]$. By Proposition \ref{cauchyseries}, for $|x|<|s|$ we have
\begin{equation}
\label{c0}
\Delta^{\ell}_{n+1} \overline{D}^{h_n-\ell} S^{-1}_L(s,x)= \sum_{k =h_n+\ell}^{\infty} \left(\Delta^{\ell}_{n+1}  \overline{D}^{h_n-\ell} x^k\right) s^{-1-k}.
\end{equation}
Now,  equalizing \eqref{c0} and \eqref{PR} and we obtain
$$ \Delta^{\ell}_{n+1}  \overline{D}^{h_n-\ell} (x^k)=\frac{(-1)^{\ell-h_n}}{(h_n-\ell)!} \sum_{i=0}^{h_n-\ell}  \binom{h_n-\ell}{i} x_0^{h_n-\ell-i}(-1)^i \left(\Delta^{h_n}_{n+1} x^{k+i} \right) .$$
\end{proof}
\begin{remark}
If we consider $n=3$ and $\ell=0$ in (\ref{polyapp}) we obtain the same result as in  \cite{Polyf1}. If we consider $n=5$ with $\ell=0$, and $n=5$ with $\ell=1$, we reobtain the same result proved in \cite{Fivedim}.
\end{remark}
Now, we provide a characterization of the kernel of the operator $\Delta_{n+1}^{\ell} \overline{D}^{h_n-\ell}$ but, before, we need the following technical lemma, whose proof requires the use of Jacobi polynomials.

\begin{lemma}\label{jacobi}
	Let $n$ be an odd number and set $h_n:=(n-1)/2$ assume that $\ell\in \mathbb{N}_0$ is such that $0 \leq \ell \leq h_n-1$. Then we have
	$$
	\sum_{i=0}^{h_n-\ell} \tilde{k}_{h_n, \ell, k,i}=\begin{cases}
(-1)^{h_n-\ell} \frac{(h_n-\ell)! k!}{(k-h_n-\ell)! (2h_n)!} \binom{2h_n}{h_n-\ell}, \qquad k\geq 2 h_n\\
\newline
(-1)^{h_n-\ell} \binom{k}{k-h_n-\ell}, \qquad h_n+\ell\leq k<2h_n
	\end{cases}
	$$
	where the constant $\tilde{k}_{h_n, \ell, k,i}$ is given in \eqref{cons} and we used the convention that $\binom{\gamma}{\delta}=0$ if $\gamma< \delta$.
\end{lemma}
\begin{proof}
	We start by changing indexes in the sum $u=-i+h_n-\ell$, and we get
	\begin{equation*}
	\begin{split}
	 \sum_{i=0}^{h_n-\ell}\tilde{k}_{h_n, \ell, k,i} & =\sum_{u=h_n-\ell}^{0}\binom{h_n-\ell}{h_n-\ell-u} \binom{k+u}{k+u-2h_n}(-1)^{-u+2(h_n-\ell)}\\
	& =\sum_{u=0}^{h_n-\ell} \binom{h_n-\ell}{u}\binom{k+u}{k+u-2h_n}(-1)^u.
	\end{split}
	\end{equation*}
	For the sake of the simplicity we call the index in the last sum with letter $i$. Thus we want to show that for $k\geq 2 h_n$ we have
	$$
	\sum_{i=0}^{h_n-\ell} \binom{h_n-\ell}{i}\binom{k+i}{k+i-2h_n}(-1)^i=(-1)^{h_n-\ell} \frac{(h_n-\ell)! k!}{(k-h_n-\ell)! (2h_n)!} \binom{2h_n}{h_n-\ell},
        $$
        and that, for $h_n+\ell\leq k<2h_n$, we have
	\begin{equation}\label{line2p}
	\sum_{i=0}^{h_n-\ell} \binom{h_n-\ell}{i}\binom{k+i}{k+i-2h_n}(-1)^i = (-1)^{h_n-\ell} \binom{k}{k-h_n-\ell},
	\end{equation}
	 First we consider the case $k\geq 2h_n$. For the proof we need the properties of the Jacobi polynomials (see \cite{S39}, Chapter 4). We recall that the Jacobi polynomials for $\alpha,\, \beta, n\in \mathbb N$ and $ z\in\mathbb C$, are defined as
	\begin{equation}\label{jac}
	P^{(\alpha,\beta)}_n(z):=\frac {\Gamma(n+\alpha+1)}{n! \Gamma(n+\alpha+\beta+1)} \sum_{\nu=0}^n \binom{n}{\nu} \frac{\Gamma(n+\alpha+\beta+\nu+1)}{\Gamma(\alpha+\nu+1)} \left (\frac{z-1}{2}\right)^\nu.
	\end{equation}
	For the Jacobi polynomials the following formula
	\begin{equation}\label{jacobi_prop}
		P_n^{(\alpha,\beta)}(-1)=(-1)^n\binom{n+\beta}{n},
	\end{equation}
	holds (see formula (4.1.4) in \cite{S39}). We observe that by considering $n:=h_n-\ell$, $\alpha:=k-2h_n$ and $\beta:=h_n+\ell$ we get
	$$
	\sum_{i=0}^{h_n-\ell}\binom{h_n-\ell}{i} \binom{k+i}{k+i-2h}(-1)^{i}=\frac{(h_n-\ell)! k!}{(k-h_n-\ell)(2h)!}P^{(k-2h_n, h_n+\ell)}_{h_n-\ell}(-1)
	$$
	which, together with \eqref{jacobi_prop}, implies that
	$$
	\sum_{i=0}^{h_n-\ell}\binom{h_n-\ell}{i} \binom{k+i}{k+i-2h_n}(-1)^{i}=(-1)^{h_n-\ell} \frac{(h_n-\ell)! k!}{(k-h_n-\ell)! (2h_n)!} \binom{2h_n}{h_n-\ell}.
	$$
	For the second case we have to consider $h_n+\ell \leq k <2 h_n$. Since we have zero terms in the series \eqref{line2p} if $i <2h_n-k$ we have
	$$ \sum_{i=0}^{h_n-\ell} \binom{h_n-\ell}{i} \binom{k+i}{2 h_n} (-1)^i= \sum_{i=2h_n-k}^{h_n-\ell} \binom{h_n-\ell}{i} \binom{k+i}{2h_n}(-1)^i.$$
By performing the change of index $i = \gamma + 2h_n - k$, the series on the right-hand side can be rewritten as:

	 $$\sum_{i=2h_n-k}^{h_n-\ell} \binom{h_n-\ell}{i} \binom{k+i}{2h_n}(-1)^i= \sum_{\gamma=0}^{k-h_n-\ell} \binom{h_n-\ell}{\gamma+2h_n-k} \binom{\gamma+2h_n}{2h_n} (-1)^{\gamma}.$$
	 By using the definition of the Jacobi polynomials, see \eqref{jac}, with $ \alpha := 2h_n - k $, $ \beta := h_n + \ell $, and $n := k - h_n - \ell$, we can rewrite the above expression on the right-hand side as:
	
	 $$ \sum_{\gamma=0}^{k-h_n-\ell} \binom{h_n-\ell}{\gamma+2h_n-k} \binom{\gamma+2h_n}{2h_n} (-1)^{\gamma}= (-1)^k  P^{(2h_n-k, h_n+\ell)}_{k-h_n-\ell}(-1).$$
	 By using the property \eqref{jacobi_prop} of the Jacobi polynomials we have
	 $$ \sum_{i=0}^{h_n-\ell} \binom{h_n-\ell}{i}\binom{k+i}{2h_n}(-1)^i= (-1)^{h_n-\ell} \binom{k}{k-h_n-\ell}.$$
\end{proof}

\begin{lemma}\label{ker_poly}
	Let $n$ be an odd number and set $h_n:=(n-1)/2$. Let $f\in \mathcal{SH}_L(U)$ (resp. $f\in \mathcal{SH}_R(U)$) and $U$ be a connected slice Cauchy domain. Then $f$ belongs to $\Delta^\ell_{n+1} \overline{D}^{h_n-\ell}$ if and only if $f(x)=\sum_{\nu=0}^{h_n+\ell-1} x^{\nu} \alpha_{\nu}$ (resp. $f(x)=\sum_{\nu=0}^{h_n+\ell-1} \alpha_{\nu}x^{\nu} $) where $\{\alpha_\nu\}_{0 \leq \nu \leq h_n+\ell-1}\mathbb \subseteq \mathbb{R}_n$.
\end{lemma}
\begin{proof}
	We prove the result in the case $f\in \mathcal{SH}_L(\Omega)$ since the case $f\in \mathcal{SH}_R(\Omega)$ follows by similar arguments. We proceed by double inclusion. It is obvious that $f(x)=\sum_{\nu=0}^{h_n+\ell-1} x^{\nu} \alpha_{\nu}$ belongs to the kernel $\Delta^\ell_{n+1} \overline{D}^{h_n-\ell}$ . The other inclusion can be proved observing that the function $f$ belongs to  $\Delta^\ell_{n+1} \overline{D}^{h_n-\ell}$, after a change of variable if needed, there exists $r>0$ such that the function $f$ can be expanded in a convergent series at the origin
	$$f(x)=\sum_{k=0}^{\infty}x^k \alpha_k\quad\textrm{for $\{\alpha_k\}_{k\in\mathbb N_0}\subset\mathbb R_n$ and for any $x\in B_r(0)$}$$
	where $B_r(0)$ is the ball centred at $0$ and of radius $r$. We apply the operator $\Delta^\ell_{n+1} \overline{D}^{h_n-\ell}$ to the function $f$ and we get
	$$
	0= \Delta^\ell_{n+1} \overline{D}^{h_n-\ell}f(x)\equiv\sum_{k=h_n+\ell}^{\infty} \Delta^\ell_{n+1} \overline{D}^{h_n-\ell} (x^k)\alpha_k,\quad \forall x\in B_r(0).
	$$
	We restrict $\Delta^\ell_{n+1} \overline{D}^{h_n-\ell} (x^k)$ to a neighbourhood of the real axisis $\Omega$. Thus by Corollary \ref{app1} we have that
	$$ \Delta^\ell_{n+1} \overline{D}^{h_n-\ell} (x^k) |_{\underline{x}=0}= \left(	\sum_{i=0}^{h_n-\ell} \tilde{k}_{h_n, \ell, k,i}\right) x_{0}^{k-h_n-\ell},$$
	where by Lemma \ref{jacobi} we know that $	\sum_{i=0}^{h_n-\ell} \tilde{k}_{h_n, \ell, k,i}$ is different from zero. Thus we have that
	$$\alpha_k=0, \qquad k\geq h_n+\ell.$$
	This implies that $f(x)=\sum_{\nu=0}^{h_n+\ell-1} x^{\nu} \alpha_{\nu}$ in $\Omega$, and since $U$ is a connected set we get that $f(x)=\sum_{\nu=0}^{h_n+\ell-1} x^{\nu} \alpha_{\nu}$ for any $x \in U$.
\end{proof}

Finally, we have all the necessary tools to provide an integral representation of a axially polyanalytic functions.
\begin{theorem}[Integral representation of axially polyanalytic functions]
\label{polya}

Let $n$ be an odd number and set $h_n:=(n-1)/2$ assume that $\ell$ is such that $0 \leq \ell \leq h_n-1$.
 Let $U \subset \mathbb{R}^{n+1}$ be a bounded slice Cauchy domain, for $I \in \mathbb{S}$, we set $ds_I=ds(-I)$.

\begin{itemize}
\item If $f$ is a left slice hyperholomorphic function on a set that contains $\overline{U}$, then $\breve{f}^{\circ}(x)=\Delta^{\ell}_{n+1} \overline{D}^{h_n -\ell}f(x)$ is axially polyanalytic of order $h_n-\ell+1$ and it admits the integral representation
\begin{equation}
\label{left0}
\breve{f}^{\circ}(x)= \frac{1}{2 \pi} \int_{\partial (U \cap \mathbb{C}_I)}  \mathcal{P}^L_\ell(s,x) ds_I f(s), \qquad \forall x \in U,
\end{equation}
\item If $f$ is a right slice hyperholomorphic function on a set that contains $\overline{U}$,
then $\breve{f}^{\circ}(x)=f(x)\Delta^{\ell}_{n+1} \overline{D}^{h_n -\ell}$  is axially polyanalytic of order $h_n-\ell+1$ and it admits the integral representation
\begin{equation}
\label{right0}
	\breve{f}^{\circ}(x)= \frac{1}{2 \pi} \int_{\partial (U \cap \mathbb{C}_I)} f(s)ds_I \mathcal{P}^R_\ell(s,x), \qquad \forall x \in U.
\end{equation}
\end{itemize}
Moreover, the integrals in \eqref{left0} and \eqref{right0} depend neither on $U$ and nor on the imaginary unit $I \in \mathbb{S}$ and are independent on the kernel of $\Delta^{\ell}_{n+1} \overline{D}^{h_n -\ell}$.
\end{theorem}
\begin{proof}
We show only formula \eqref{left0}. By Theorem \ref{Cauchy} and Theorem \ref{polykernel1} we have
$$ \breve{f}^{\circ}(x)=\Delta^{\ell}_{n+1} \overline{D}^{h_n -\ell}f(x)= \frac{1}{2 \pi} \int_{\partial(U \cap \mathbb{C}_I)} 	\Delta^{\ell}_{n+1} \overline{D}^{h_n -\ell}S^{-1}_L(s,x) ds_I f(s)=\frac{1}{2 \pi} \int_{\partial (U \cap \mathbb{C}_I)} \mathcal{P}^L_\ell(s,x) ds_I f(s).$$
Similar arguments as Theorem \ref{IP} show the independence on the kernel of $\Delta^{\ell}_{n+1} \overline{D}^{h_n -\ell}$.
\end{proof}

\section{Polyanalytic functional calculus based on the $S$-spectrum}\label{Polyanalytic functional}
In this section, we investigate a polyanalytic functional calculus based on the S-spectrum. The key element of this investigation is the integral representation of axially polyanalytic functions, see Theorem \ref{polya}.
The diagram below shows the polyanalytic fine structure.  The focus of our investigation here is the central part of the diagram, i.e., the polyanalytic functional calculus:
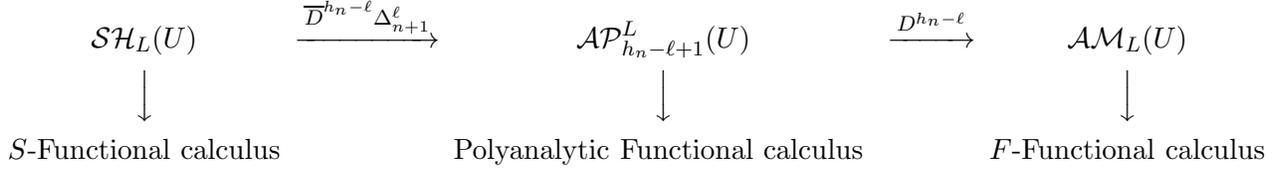
\begin{equation*}
    \begin{CD}
          \mathcal{SH}_L(U)@> \overline{D}^{h_n-\ell} \Delta^{\ell}_{n+1} >> \mathcal{AP}_{h_n-\ell+1}^L(U)@>D^{h_n-\ell}>> \mathcal{AM}_L(U) \\
        @V VV    @V VV  @V VV \\
        S\text{-Functional calculus} @. \text{Polyanalytic Functional calculus } @. F\text{-Functional calculus} \\
    \end{CD}
\end{equation*}
\newline
A similar diagram holds also for functions in $\mathcal{SH}_R(U)$.
\begin{definition}[Polyanalytic resolvent operators]
Let $n$ be an odd number and set $h_n:=(n-1)/2$ assume that $\ell\in \mathbb{N}_0$ is such that $0 \leq \ell \leq h_n-1$.
Let $T \in \mathcal{BC}^{0,1}(V_n)$ and $s \in \sigma_S(T)$. We define the left $\mathcal{P}$-resolvent operator as
$$
\mathcal{P}^L_\ell(s,T)=\frac{(-1)^{h_n-\ell}}{(h_n-\ell)!}F_n^L(s,T)(s\mathcal{I}-T_0)^{h_n-\ell},
$$
and the right $\mathcal{P}$-resolvent operator as
$$
\mathcal{P}^R_\ell(s,T)=\frac{(-1)^{h_n-\ell}}{(h_n-\ell)!}(s\mathcal{I}-T_0)^{h_n-\ell}F_n^R(s,T),
$$
where the $F$-resolvent operators $F_n^L(s,T)$  and $F_n^R(s,T)$  are defined in (\ref{FresBOUNDL}) and (\ref{FresBOUNDR}), respectively.
			
\end{definition}

Now, we provide an expansion in series for the left and right $\mathcal{P}$-resolvent operators.

\begin{proposition}\label{res_series}
Let $n$ be an odd number and set $h_n:=(n-1)/2$ assume that $\ell\in \mathbb{N}_0$ is such that $0 \leq \ell \leq h_n-1$.
 Let $T \in \mathcal{BC}(V_n)$ and $s \in \mathbb{R}^{n+1}$ be such that $\|T\|<|s|$.
 Then, we can write the left $\mathcal{P}$-resolvent operator as
\begin{equation}
\mathcal{P}^L_\ell (s,T)=\frac{\gamma_n}{(h_n-\ell)!}  \sum_{m=h_n+\ell}^{\infty} \sum_{k=0}^{h_n-\ell} \tilde{k}_{h_n, \ell, k,m}  T_0^{k} P^n_{m-h_n-\ell-k}(T)s^{-1-m},
\end{equation}
and the right $\mathcal{P}$-resolvent operator as
\begin{equation}
\mathcal{P}^R_\ell(s,T)=\frac{\gamma_n }{(h_n-\ell)!}  \sum_{m=h_n+\ell}^{\infty} \sum_{k=0}^{h_n-\ell}\tilde{k}_{h_n, \ell, k,m}s^{-1-m}T_0^{k} P^n_{m-h_n-\ell-k}(T),
\end{equation}
where the constant $\tilde{k}_{h_n, \ell, k,m}$ is given in \eqref{cons}, $\gamma_n$ is defined in \ref{gamman} and
the Clifford-Appell operators $P_k^n(T)$ are defined in (\ref{Appellope}).
\end{proposition}
\begin{proof}
By using the expansion in series of the $F$-resolvent operator, see Proposition \ref{exseriesOPR}, and  by the binomial formula for $(s-T_0)^{h_n-\ell}$, since $T_0$ commute with $s$, we have that
\[
\begin{split}
\mathcal{P}^L_\ell(s,T) &=\frac{\gamma_n}{(h_n-\ell)!} \sum_{m=2h_n}^{\infty} \binom{m}{m-2h_n} P^n_{m-2h_n}(T)s^{-1-m} (s\mathcal{I}-T_0)^{h_n-\ell}\\
&=\frac{\gamma_n }{(h_n-\ell)!} \sum_{m=2h_n}^{\infty} \sum_{k=0}^{h_n-\ell} \binom{h_n-\ell}{k} \binom{m}{m-2h_n} (-1)^k T_0^{k} P^n_{m-2h_n}(T)s^{-1-m+h_n-\ell-k}\\
&=\frac{\gamma_n }{(h_n-\ell)!} \sum_{u=h_n+\ell+k}^{\infty} \sum_{k=0}^{h_n-\ell} \binom{h_n-\ell}{k} \binom{u+h_n-\ell-k}{u-h_n-\ell-k} (-1)^k T_0^{k} P^n_{u-h_n-\ell-k}(T)s^{-1-u}.
\end{split}
\]
Now, if $h_n+\ell \leq u < h_n+\ell+k$ we have $P^n_{u-h_n-\ell-k}(T)=0$ since $u-h_n-\ell-k<0$, see Remark \ref{conv}. Thus we can write
$$\mathcal{P}^L_\ell (s,T)=\frac{\gamma_n }{(h_n-\ell)!} \sum_{u=h_n+\ell}^{\infty} \sum_{k=0}^{h_n-\ell} \binom{h_n-\ell}{k} \binom{u+h_n-\ell-k}{u-h_n-\ell-k} (-1)^k T_0^{k} P^n_{u-h_n-\ell-k}(T)s^{-1-u}.$$
The expansion for the right $ \mathcal{P}$-resolvent operator follows by using similar arguments.
\end{proof}

Now, we show a result regarding the regularity of the left (resp. right) $ \mathcal{P}$-resolvent operator.

\begin{lemma}\label{res_reg}
Let $n$ be an odd number and set $h_n:=(n-1)/2$ assume that $\ell\in \mathbb{N}_0$ is such that $0 \leq \ell \leq h_n-1$.
 We also assume that $T \in \mathcal{BC}^{0,1}(V_n)$. The polyanalytic left resolvent operator $\mathcal{P}^L_\ell (s,T)$ (resp. $\mathcal{P}^R_\ell(s,T)$) is a $ \mathcal{B}(V_n)$-valued right (resp. left) slice hyperholomorphic function for all $s$ in $ \rho_S(T)$.
\end{lemma}
\begin{proof}
To show the result it is sufficient to observe that $F_n^L(s,T)$ (resp. $F_n^R(s,T)$) is a right (resp. left) slice hyperholomorphic in the variable $s$, see Lemma \ref{regope}, and $(T_0-s)^{h_n-\ell}$ is an intrinsic function in the variable $s$ operator-valued.
\end{proof}

Now, inspired from the integral representation of polyanalytic functions, see Theorem \ref{polya} we give the definition of polyanalytic functional calculus.

\begin{remark}
Because of the particular role that $T_0$ plays in the expressions of the $\mathcal{P}$-resolvent operators, it is possible,
 in the definition of the polyanalytic functional calculus to take advantage of
Remark \ref{RESOLVprime} and consider $T \in \mathcal{BC}^{0,1}(V_n)$  such that its components $T_{i}$, for $i=0,1,...n-1$  have real spectra and $T_{n}=0$.
\end{remark}

\begin{definition}[Polyanalytic functional calculus on the $S$-spectrum]\label{def_poly}
	Let $n$ be an odd number and set $h_n:=(n-1)/2$ assume that $\ell\in \mathbb{N}_0$ is such that $0 \leq \ell \leq h_n-1$.
 Let us consider $T \in \mathcal{BC}^{0,1}(V_n)$ be such that its components $T_{i}$, for $i=0,1,...n-1$  have real spectra and $T_{n}=0$.
 Let $U$ be a bounded slice Cauchy domain and we set $ds_I=ds(-I)$ for $I \in \mathbb{S}$.
	\begin{itemize}
		\item For any $ f \in \mathcal{SH}_{\sigma_S(T)}^L(U)$ we set $ \breve{f}^{\circ}(x)= \Delta^{\ell}_{n+1} \overline{D}^{h_n -\ell} f$. Then we define the $ \mathcal{P}$-functional calculus for the operator $T$ as
		\begin{equation}
			\label{def_poly_1}
			\breve{f}^{\circ} (T):=\frac{1}{2\pi} \int_{\partial(U \cap \mathbb{C}_I)} \mathcal{P}^L_\ell(s,T) ds_I f(s),
		\end{equation}
		\item For any $ f \in \mathcal{SH}_{\sigma_S(T)}^R(U)$ we set $ \breve{f}^{\circ}(x)= \Delta^{\ell}_{n+1} \overline{D}^{h_n -\ell} f$. Then we define the $ \mathcal{P}$-functional calculus for the operator $T$ as
		\begin{equation}
			\label{def_poly_2}
			\breve{f}^{\circ} (T):=\frac{1}{2\pi} \int_{\partial(U \cap \mathbb{C}_I)} f(s) ds_I \mathcal{P}^R_\ell(s,T).
		\end{equation}
	\end{itemize}

\end{definition}

We now state and prove some results that show that the polyanalytic functional calculus is well defined.

\begin{proposition}
The polyanalytic functional calculus based on the $S$-spectrum is well defined, i.e. the integrals \eqref{def_poly_1} and \eqref{def_poly_2} depend neither on the imaginary units $I \in \mathbb{S}$ nor on the slice Cauchy domain $U$.
\end{proposition}
\begin{proof}
The result follows by employing arguments similar to those used in Theorem \ref{well1}. Specifically, in this case, the resolvent operator is left (respectively, right) slice hyperholomorphic in $ s $, as shown in Lemma \ref{Kreg}. Furthermore, $ \lim_{s \to \infty} \mathcal{P}^L_\ell(s,T) = \lim_{s \to \infty} \mathcal{P}^R_\ell(s,T) = 0 $.

\end{proof}

Due to the lack of commutativity, the right and left variants of the polyanalytic functional calculus often lead to different operators. The only exception occurs when the polyanalytic functional calculus is utilized for intrinsic slice hyperholomorphic functions.
\begin{proposition}
\label{PLR}
Let $n$ be an odd number and set $h_n:=(n-1)/2$ assume that $\ell\in \mathbb{N}_0$ is such that $0 \leq \ell \leq h_n-1$. We assume that  $T \in \mathcal{BC}^{0,1}(V_n)$ and $f \in \mathcal{N}_{\sigma_S(T)}(U)$, where $U$ is a bounded slice Cauchy domain as in
	Definition \ref{SHONTHEFS}. Then we have
	$$\frac{1}{2\pi}\int_{\pp(U\cap \mathbb{C}_I)} \mathcal{P}^L_\ell(s,T)  \, ds_I\, f(s)=\frac{1}{2\pi}\int_{\pp(U\cap \mathbb{C}_I)} f(s)  ds_I \mathcal{P}^R_\ell(s,T).$$
\end{proposition}
\begin{proof}
The result follows by using Theorem \ref{FLR}.
\end{proof}

The following technical tool is fundamental to show the independence of the polyanalytic functional calculus from the kernel of $\Delta^{\ell}_{n+1} \overline{D}^{h_n -\ell}$.

\begin{theorem}\label{ker_poly_funct}
Let $n$ be an odd number and set $h_n:=(n-1)/2$ assume that $\ell\in \mathbb{N}_0$ is such that $0 \leq \ell \leq h_n-1$.
 Let us consider $T \in \mathcal{BC}^{0,1}(V_n)$ be such that its components $T_{i}$, $i=0,1,...,n-1$  have real spectra and $T_{n}=0$.
Let $G$ be a bounded slice Cauchy domain such that $(\partial G) \cap \sigma_S(T)= \emptyset$.
Thus, for every $I \in \mathbb{S}$ and $1 \leq \ell \leq h_n-1$,
 we have
\begin{equation}
\label{inte}
\int_{\partial (G \cap \mathbb{C}_I)} s^{\alpha} ds_I \mathcal{P}^R_\ell(s,T)=\int_{\partial (G \cap \mathbb{C}_I)} \mathcal{P}^L_\ell(s,T) ds_I s^{\alpha}=0, \qquad \hbox{if} \quad  \alpha\in \mathbb{N}_0,  \quad \hbox{with} \quad 0 \leq \alpha \leq h_n+\ell-1.
\end{equation}
\end{theorem}
\begin{proof}
The first equality in \eqref{inte} follows form the fact that the function $s^{\alpha}$ is slice intrinsic holomorphic, see Proposition \ref{PLR}. Now, we prove the second equality in \eqref{inte}. By the Fueter-Sce mapping theorem, see Theorem \ref{intform},  we have
\begin{equation}
\label{intzero_poly}
 \int_{\partial(G \cap \mathbb{C}_I)} F_n^L(s,x) ds_I s^{m}= \Delta^{h_n}_{n+1}(x^m)=0, \qquad \hbox{if} \quad m \leq 2h_n-1,
\end{equation}
for all $x \notin \partial G$ and $I \in \mathbb{S}$. Now, by using the binomial theorem, the Fubini theorem and Proposition \ref{NewR} we have
\begingroup\allowdisplaybreaks
\begin{eqnarray}
\nonumber
\int_{\partial (G \cap \mathbb{C}_I)} \mathcal{P}^L_\ell(s,T) ds_I s^{\alpha}&=& \frac{1}{(h_n-\ell)!} \sum_{k=0}^{h_n-\ell} \int_{\partial(G \cap \mathbb{C}_I)}\binom{h_n-\ell}{k}(-1)^{k}T_0^{h_n-\ell-k}F_n^L(s,T)ds_I s^{k+\alpha}\\
\nonumber
&=&\int_{\partial W} G(\omega, T) D \omega\sum_{k=0}^{h_n-\ell}\binom{h_n-\ell}{k}(-1)^{k}T_0^{h_n-\ell-k} \times\\
&& \times \left(  \int_{\partial(G \cap \mathbb{C}_I)}F_n^L(s,\omega)ds_I s^{k+\alpha}\right),
 \nonumber
\end{eqnarray}
\endgroup
where $W \subset \mathbb{R}^{n+1}$ as in Definition \ref{locmongspetrum}. Since $ k+\alpha \leq 2 h_n-1$, by \eqref{intzero_poly} we have
$$ \int_{\partial (G \cap \mathbb{C}_I)} \mathcal{P}^L_\ell(s,T)ds_I s^{\alpha}=0. $$
\end{proof}

Now, we show that the polyanalytic functional calculus is independent from the kernel of the operator $\Delta^{\ell}_{n+1} \overline {D}^{h_n-\ell}$.

\begin{proposition}
	Let $n$ be an odd number and set $h_n:=(n-1)/2$ assume that $\ell\in \mathbb{N}_0$ is such that $0 \leq \ell \leq h_n-1$.
 Let us consider $T \in \mathcal{BC}^{0,1}(V_n)$ be such that its components $T_{i}$, $i=0,1,...n-1$  have real spectra and $T_{n}=0$.
Let $U$ be a slice Cauchy domain with $\sigma_S(T) \subset U$. Let $f$, $g \in \mathcal{SH}_L(U)$ (resp.$f$, $g \in \mathcal{SH}_R(U)$) and $ \Delta^{\ell}_{n+1} \overline {D}^{h_n-\ell} f(x)= \Delta^{\ell}_{n+1} \overline{D}^{h_n-\ell} g(x)$
(resp. $  f (x)\Delta^{\ell}_{n+1} \overline {D}^{h_n-\ell}= g(x) \Delta^{\ell}_{n+1} \overline {D}^{h_n-\ell}$ ). Then, we have $ \breve{f}^{\circ} (T)=\breve{g}^{\circ}(T)$.
\end{proposition}
\begin{proof}
We prove the result only for the left case since for the right case the result follows by similar arguments.
We start proving the result assuming that $U$ is connected. By the definition of the
polyanalytic functional calculus based on the $S$-spectrum, see Definition \ref{def_poly}, we have
$$\breve{f}^{\circ} (T)-\breve{g}^{\circ}(T)= \frac{1}{2 \pi} \int_{\partial(U \cap \mathbb{C}_I)} \mathcal{P}^L_\ell(s,T) ds_I(f(s)-g(s)).$$
By Lemma \ref{ker_poly} we know that the resolvent operator $\mathcal{P}^L_\ell(s,T)$ is $ \mathcal{B}(V_n)$-valued slice hyperholomorphic. This implies that we can change the integration to $B_r(0) \cap \mathbb{C}_I$with $r>0$ such that $\|T\| <r$. By hypothesis we have that $f(s)-g(s) \in \hbox{ker}(\Delta_{n+1}^\ell \overline{D}^{h_n-\ell})$, thus by Lemma \ref{ker_poly} we have
\begin{eqnarray*}
\breve{f}^{\circ} (T)-\breve{g}^{\circ}(T)=\frac{1}{2 \pi}  \sum_{\nu=0}^{h_n+\ell-1}\int_{\partial(U \cap \mathbb{C}_I)} \mathcal{P}^L_\ell(s,T) ds_I s^\nu \alpha_{\nu}
\end{eqnarray*}
and by Theorem \ref{ker_poly_funct} the previous integral is zero (since $ \nu \leq h_n+\ell-1$).  This implies that $\breve{f}^{\circ} (T)=\breve{g}^{\circ}(T).$
\\Now we consider the case when the set $U$ is not connected.
   In this case we can write $U=\bigcup_{\ell=1}^k U_{\ell}$ where $U_{\ell}$ are the connected components of the set $U$. So by Lemma \ref{ker_poly}, we have $f(s)-g(s)=\sum_{\ell=1}^{n} \sum_{\nu=0}^{h_n+ \ell-1} \chi_{U_{\ell}} s^{\nu}\alpha_{\nu,\ell}$. Hence by the definition of the polyanalytic functional calculus, see Definition \ref{def_poly}, we have
$$ \breve{f}^\circ(T)-\breve{g}^\circ(T)= \frac{1}{2\pi} \sum_{\ell=1}^{n} \sum_{\nu=0}^{h_n+\ell-1} \int_{\partial(U_{\ell} \cap \mathbb{C}_I)} \mathcal{P}^L_\ell(T) ds_I s^{\nu}\alpha_{\nu,\ell}.$$
Finally, by Theorem \ref{ker_poly_funct} we get that $\breve{f}^\circ(T)=\breve{g}^\circ(T)$.
\end{proof}

Now, we show some algebraic properties of the polyanalytic functional calculus.

\begin{lemma}
	Let $n$ be an odd number and set $h_n:=(n-1)/2$ assume that $\ell\in \mathbb{N}_0$ is such that $0 \leq \ell \leq h_n-1$.
 Let us consider $T \in \mathcal{BC}^{0,1}(V_n)$ be such that its components $T_{i}$, $i=0,1,...n-1$  have real spectra and $T_{n}=0$.
If $\breve{f}^{\circ}(x)= \overline{D}^{h_n-\ell} \Delta_{n+1}^{\ell}f(x)$ (resp. $\breve{f}^{\circ}(x)= f(x)\overline{D}^{h_n-\ell} \Delta_{n+1}^{\ell}$) and $\breve{g}^{\circ}(x)= \overline{D}^{h_n-\ell}\Delta_{n+1}^{\ell}g(x)$
(resp. $\breve{g}^{\circ}(x)= g(x)\overline{D}^{h_n-\ell} \Delta_{n+1}^{\ell}$) with $f$, $g \in\mathcal{SH}_{\sigma_S(T)}^L(U)$ (resp. $\mathcal{SH}_{\sigma_S(T)}^R(U)$) and $a \in \mathbb{R}_n$, then
	$$ (\breve{f}^{\circ}a+\breve{g}^{\circ})(T)=\breve{f}^{\circ}(T)a+\breve{g}^{\circ}(T), \qquad (resp. \, \, (a\breve{f}^{\circ}+\breve{g}^{\circ})(T)=a\breve{f}^{\circ}(T)+\breve{g}^{\circ}(T)).$$
\end{lemma}
\begin{proof}
	The result follows from the linearity of the integral \eqref{def_poly_1} (resp. \eqref{def_poly_2}).
\end{proof}

\begin{proposition}
	Let $n$ be an odd number and set $h_n:=(n-1)/2$ assume that $\ell\in \mathbb{N}_0$ is such that $0 \leq \ell \leq h_n-1$.
 Let us consider $T \in \mathcal{BC}^{0,1}(V_n)$ be such that its components $T_{i}$, $i=0,1,...n-1$  have real spectra and $T_{n}=0$.
	\begin{itemize}
		\item Let  $\breve{f}^{\circ}(x)= \overline{D}^{h_n-\ell} \Delta_{n+1}^{\ell}f(x)$ and assume that $f(x)= \sum_{m=0}^{\infty} x^m a_m$ with $a_m \in \mathbb{R}_n$ converges in a ball $B_r(0)$ with $\sigma_S(T) \subset B_r(0)$. Then, we have
		$$ \breve{f}^{\circ}(T)=\frac{\gamma_n}{(h_n-\ell)!}  \sum_{m=h_n+\ell}^{\infty} \sum_{k=0}^{h_n-\ell} \tilde{k}_{h_n, \ell, k,m} T_0^{k} P^n_{m-h_n-\ell-k}(T)  a_m$$
		\item Let  $\breve{f}^{\circ}(x)= f(x)\overline{D}^{h_n-\ell} \Delta_{n+1}^{\ell}$ and assume that $f(x)= \sum_{m=0}^{\infty} a_m x^m $ with $a_m \in \mathbb{R}_n$,  converges in a ball $B_r(0)$ with $\sigma_S(T) \subset B_r(0)$. Then, we have
		$$
\breve{f}^{\circ}(T)=\frac{\gamma_n}{(h_n-\ell)!}  \sum_{m=h_n+\ell}^{\infty} \sum_{k=0}^{h_n-\ell}a_m \tilde{k}_{h_n, \ell, k,m} T_0^{k} P^n_{m-h_n-\ell-k}(T)  .
$$
	\end{itemize}
\end{proposition}
\begin{proof}
	We prove only the first statement, as the second one follows from similar arguments. We choose an imaginary unit $ I \in \mathbb{S}$ and a radius $0 < R < r$ such that $\sigma_S(T) \subset B_R(0)$. Thus, the series expansion of $f$ converges uniformly on $\partial (B_R(0) \cap \mathbb{C}_I)$, and therefore we have
	\begin{eqnarray*}
		\breve{f}^{\circ}(T)&=&\frac{1}{2\pi} \int_{\partial(B_R(0) \cap \mathbb{C}_I)} \mathcal{P}^L_\ell(s,T)ds_I \sum_{\tau=0}^{\infty} s^\tau a_{\tau}\\
		&=& \frac{1}{2\pi}\sum_{\tau=0}^{\infty} \int_{\partial(B_R(0) \cap \mathbb{C}_I)}  \mathcal{P}^L_\ell(s,T)ds_I  s^\tau a_{\tau}.
	\end{eqnarray*}
By using the series expansion of the resolvent operator $\mathcal{P}^L_\ell(s,T)$, obtained in Proposition \ref{res_series}, we have
$$
		\breve{f}^{\circ}(T)= \frac{\gamma_n}{2 \pi (h_n-\ell)!}\sum_{\tau=0}^{\infty}  \sum_{m=h_n+\ell}^{\infty} \sum_{k=0}^{h_n-\ell} \tilde{k}_{h_n, \ell, k,m} T_0^{k} P^n_{m-h_n-\ell-k}(T) \int_{\partial(B_R(0) \cap \mathbb{C}_I)} s^{-1-m+ \tau} a_{\tau} ds_I.
$$
	Now, by using the fact that
	$$ \int_{\partial(B_R(0) \cap \mathbb{C}_I)} s^{-1-m+\tau}ds_I\begin{cases}
		2 \pi, \qquad \hbox{if} \quad \tau=m \\
		0, \qquad \hbox{if} \quad \tau\neq m,
	\end{cases}$$
we get
	$$ \breve{f}^{\circ}(T)=\frac{\gamma_n}{(h_n-\ell)!}  \sum_{m=h_n+\ell}^{\infty} \sum_{k=0}^{h_n-\ell} \tilde{k}_{h_n, \ell, k,m} T_0^{k} P^n_{m-h_n-\ell-k}(T)  a_m.$$
	This shows the result.
\end{proof}

\section{Concluding remarks}\label{Concluding remarks}

\begin{remark}[The quaternionic fine structures]\label{QUTFINSRUC}

In this remark, we aim to illustrate the quaternionic fine structures of Dirac type. This is an important case that has been widely studied in the literature, particularly regarding its various aspects on the functional calculi, and serves as a guideline for fine structures in the Clifford setting.
	We recall that fine structures in spectral theory on the $S$-spectrum consist of the set of function spaces and the associated functional calculi induced by all possible factorizations of the operator $T_{FS2}$ in the Fueter-Sce extension theorem.
We will denote the set of quaternionic slice hyperholomorphic functions by $\mathcal{SH}(U)$ and the class of axially Fueter regular functions on $U$ by $\mathcal{AM}(U)$.
	
		Let $ \mathcal{O}(\Pi)$ is the set of holomorphic functions defined on $\Pi \subseteq \mathbb{C}$ and $U$ is an axially symmetric set in quaternions $\mathbb{H}$ induced by $\Pi$.
	The Fueter extension theorem can be summarized in the following diagram:
	\begin{equation}
		\label{diagintro}
		\begin{CD}
			\textcolor{black}{\mathcal{O}(\Pi)}  @>T_{F1}>> \textcolor{black}{\mathcal{SH}(U)}  @>\ \   T_{F2}>>\textcolor{black}{\mathcal{AM}(U)},
		\end{CD}
	\end{equation}
	where the two maps $T_{F1}$ and $T_{F2}$ are explicit operators,
	in particular the operators $T_{F2}$ that maps $\mathcal{SH}(U)$ into $\mathcal{AM}(U)$ is the Laplace operator in dimension four.
	
	It is possible to factorize the Laplace operator in the following two different ways. Precisely,
	if $D$ and $\overline{{D}}$
	are the Cauchy-Fueter operator and its conjugate, respectively that are a particular case of (\ref{DIRACeBARDNN}) for $n=3$, then
	we have
	$$
	\Delta={D}\overline{{D}}=\overline{{D}}{D}.
	$$
	Even though the operators ${D}$ and $\overline{{D}}$ commute
	the factorizations ${D}\overline{{D}}$ and $\overline{{D}}{D}$
	give rise to different functions spaces according to the order
	${D}$ and $\overline{{D}}$ applied to a slice hyperholomorphic function.
	
	In modern terminology, the Fueter or Fueter-Sce extension theorem primarily refers to its second map denoted by $T_{F2}$ for the quaternions and by $T_{FS2}$ for the Clifford algebra setting.
	Specifically, this involves the application of the four dimensional Laplacian $\Delta$ in the quaternionic case, and the $(n+1)$-dimensional Laplacian $\Delta_{n+1}$ raised to the power of $h_n:=(n-1)/2$ in the Clifford algebra context, where $n$ is the number of imaginary units of the Clifford algebra $\mathbb{R}_n$ and the number $h_n$ is called Sce exponent.
	
	\medskip
	The notion of fine structures on the $S$-spectrum emerged from observing that various factorizations of the second map $T_{F2}$ in the Fueter-Sce extension theorem generated different function spaces with integral representations, allowing the definition of functional calculi for quaternionic operators or paravector operators in the Clifford algebra setting. This observation is formalized in the following definition of fine structures.

	\medskip
	Consequently, in our discussion, we will focus only on this second mapping
	and omit the original extension theorem that begins with holomorphic functions,
	as it is not directly relevant to our purposes.
	So, if we first apply operator ${D}$ to functions in $\mathcal{SH}(U)$ we have:
	$$
	\begin{CD}
		\textcolor{black}{\mathcal{SH}(U)} @>{D}>> \textcolor{black}{\mathcal{AH}(U)}  @>\ \    {\overline{D}} >>\textcolor{black}{\mathcal{AM}(U)},
	\end{CD}
	$$
	where $\mathcal{AH}(U)=\{Df \, : \, f\in \mathcal{SH}(U)\}$ and $\mathcal{AM}(U)=\{\Delta f \, : \, f\in \mathcal{SH}(U)\}$ are the class of axially harmonic function and of axially monogenic functions, respectively.
	In the case we apply $\overline{{D}}$ to functions in $\mathcal{SH}(U)$ we obtain
	$$
	\begin{CD}
		\textcolor{black}{\mathcal{SH}(U)} @>{\overline{D}}>> \textcolor{black}{\mathcal{AP}_2(U)}  @>\ \  {D} >>\textcolor{black}{\mathcal{AM}(U)},
	\end{CD}
	$$
	where $ \mathcal{AP}_2(U)=\{\overline{D}f \, : \, f\in \mathcal{SH}(U)\}$ is the class of axially polyanalytic functions of order 2.

	The integral representations of functions within the fine structures give rise to various functional calculi. The original idea stems from the complex Riesz-Dunford functional calculus \cite{RD}, which is based on the Cauchy integral formula. This calculus allows a complex variable $z$ in a suitable holomorphic function $f(z)$ to be replaced with a bounded linear operator $A$
	thereby defining $f(A)$. In the case of fine structures, the integral formulas of the function spaces  provide the fundamental tool to define functional calculus, following the same spirit as the Riesz-Dunford calculus.

\end{remark}

\begin{remark}[Developments of the spectral theory on the $S$-spectrum]
	
	Recently, it has been shown that both quaternionic and Clifford frameworks are specific cases within a more general context where spectral theory based on the $S$-spectrum can be developed, as detailed in \cite{ADVCGKS, PAMSCKPS} and related references.
	Using the concept of the $S$-spectrum, researchers have also successfully established a quaternionic version of the spectral theorem. For instance, the spectral theorem for unitary operators, leveraging Herglotz functions, is demonstrated in \cite{ACKS16}, while the quaternionic spectral theorem for normal operators is proved in \cite{ACK}. Furthermore, the spectral theorem grounded in the $S$-spectrum has recently been extended to Clifford operators, as discussed in \cite{ColKim}.

	The development of the spectral theory on the $S$-spectrum also has opened up several research directions in hypercomplex analysis and operator theory. Without claiming completeness, we mention the slice hyperholomorphic Schur analysis \cite{ACS2016}, the characteristic operator functions \cite{AlpayColSab2020}, the quaternionic perturbation theory and invariant subspaces \cite{CereColKaSab}, and new classes of fractional diffusion problems that are based on the $H^\infty$-version of the $S$-functional calculus \cite{ColomboDenizPinton2020,ColomboDenizPinton2021,CGdiffusion2018,FJBOOK,ColomboPelosoPinton2019}. Moreover, recently nuclear operators and Grothendieck-Lidskii formula for quaternionic operators has been studied in \cite{DEB_NU} and quaternionic triangular linear operators have been investigated in \cite{TRIANG}.  \medskip
	
\end{remark}

\begin{remark}[The functional calculi for unbounded operators]
	The generalization of the holomorphic functional calculus to sectorial operators leads
	to the $H^\infty$-functional calculus that, in the complex setting, was introduced in the paper \cite{McI1}, see also the books \cite{Haase, HYTONBOOK1, HYTONBOOK2}. Moreover, the boundedness  of the $H^\infty$ functional calculus depends on suitable quadratic estimates and this calculus has several applications to boundary value problems, see  \cite{MC10,MC97,MC06,MC98}.
	We remark that the $H^\infty$-functional calculus exists also for the monogenic functional calculus (see \cite{JM}) and it was introduced by A. McIntosh and his collaborators, see the books \cite{JBOOK,TAOBOOK} for more details.
	\medskip
	
	In the quaternionic setting the functional calculi for bounded operators and slice hyperholomorphic functions
	have already done in \cite{CS11,ColSab2006}, for bounded operators and axially monogenic functions in \cite{CDS,CDS1,CG,CS,CSF,CSS} and for polyanalytic functions and axially monogenic functions more recently in \cite{CDPS,Fivedim,Polyf1,Polyf2}.
	\medskip
	
	Unbounded operators for a restricted class of functions were considered in \cite{CDP23,CSF,ColSab2006,G17}. In the literature there exists the $H^\infty$-functional calculus for slice hyperholomorphic functions  in \cite{ACQS2016,CGK} and for harmonic functions in \cite{MPS23}.
	In the \cite{CPS1} we enlarges the class of admissible operators $T$ for the $S$- and the $Q$-functional calculus to operators of type $(\alpha,\beta,\omega)$. On the other hand we also treat the two $H^\infty$-functional calculus for polynomially growing functions which are polyanalytic of order $2$  and for axially monogenic functions.
\end{remark}

\begin{remark}[Summary]

In this paper we have proved that although the factorizations of the Fueter-Sce map $\Delta_{n+1}^{h}$, with $h_n:=(n-1)/2$, may seem intricate at first glance, they can be succinctly summarized in the following table, where we recall
$$
T_{FS}^{(I)}:= \Delta_{n+1}^{\ell-1} D, \quad \quad T_{FS}^{(II)}= D^s \overline{D}^{h_n-\ell-1}, \quad\quad T^{(III)}_{FS}=\Delta_{n+1}^{\ell} \overline{D}^{h_n- \ell}
$$
\footnotesize{
\begin{center}
	\begin{tabular}{| l | l |}
		\hline
		\rule[-4mm]{0mm}{1cm}
		{ \bf Operators} & {\bf Regularities} \\
		\hline
		\rule[-4mm]{0mm}{1cm}
		$T_{FS}^{(I)}, \quad 1 \leq \ell \leq h_n$& Polyharmonic of degree $h_n-\ell+1$\\
		\hline
		\rule[-4mm]{0mm}{1cm}
		$T_{FS}^{(II)}, \qquad 0 \leq \ell \leq h_n-2, \quad  0 \leq s \leq h_n-1 $ ,\quad $s+\ell<h_n$& Holomorphic Cliffordian of order $(\ell+1, h_n-s-\ell)$,
		 \\
		\hline
		\rule[-4mm]{0mm}{1cm}
		$T^{(III)}_{FS}, \quad 0 \leq \ell \leq h_n-1$& Polyanalytic of order $h_n-\ell+1$.
		\\
		\hline
	\end{tabular}
\end{center}
}
\end{remark}	
	The above table of operators is coherent with the constructions developed in the papers \cite{CDPS, Fivedim, Polyf1, Polyf2}.

\section{Appendix}\label{Appendix}
Here we give the proof of a crucial technical results.
We start with a simple observation collected in the following lemma.
\begin{lemma}
	\label{one}
	Let $s$, $x \in \mathbb{R}^{n+1}$ then we have
	$$ \mathcal{Q}_{c,s}(x)+(s- \overline{x})s- \overline{x}(s- \overline{x})=2(s- \overline{x})(s-x_0).$$
\end{lemma}
\begin{proof}
	It follows from the definition of $\mathcal{Q}_{c,s}(x)$.
\end{proof}
\begin{lemma}
	\label{leibniz1}
	Let $n$ be a fixed odd number and $h_n:=(n-1)/2$. We assume that $s$, $x \in \mathbb{R}^{n+1}$ such that $s \notin [x]$, and $\alpha \in \mathbb{N}$. Then we have
	$$ \overline{D} \left[(s-\overline{x}) \mathcal{Q}_{c,s}^{-\alpha}(x)\right]=2(h_n-\alpha) \mathcal{Q}_{c,s}^{-\alpha}(x)+4\alpha (s-\overline{x})(s-x_0)\mathcal{Q}_{c,s}^{-\alpha-1}(x),$$
\end{lemma}
\begin{proof}
	By using the fact that $\partial_{x_0}\qcsa{-\alpha}=\alpha(2s-2x_0)\qcsa{-\alpha-1}$ and $\partial_{x_i}(\qcsa{-\alpha})=-\alpha(2x_i)\qcsa{-\alpha-1}$, we have
	\[
	\begin{split}
		& \overline{D} \left[(s-\overline{x})  \mathcal{Q}_{c,s}^{-\alpha}(x)\right] = 2h_n \qcsa{-\alpha}+2\alpha(s-\bar x)(s-x_0)\qcsa{-\alpha-1} +2\alpha \underline x (s-\bar x)\qcsa{-\alpha-1}\\
		&= 2h_n \qcsa{-\alpha}+2\alpha [(s-\bar x)s-\bar x(s-\bar x)]\qcsa{-\alpha-1}\\
		&= 2(h_n-\alpha) \qcsa{-\alpha}+2\alpha [\qcs+(s-\bar x)s-\bar x(s-\bar x)]\qcsa{-\alpha-1}\\
	\end{split}
	\]
	By Lemma \ref{one}, we get the result in the following.
\end{proof}

We prove Proposition \ref{p1} stated in Section 8.

\begin{proposition}\label{p11}
	Let $m \in \mathbb{N}$.	Let $n$ be a fixed odd number and $h_n:=(n-1)/2$. We assume that $s$, $x \in \mathbb{R}^{n+1}$ such that $s \notin[x]$. If $\beta=2m+1$ we have
	\begin{eqnarray}\label{dbars1}
		&& \overline{D}^{\beta} S^{-1}_L(s,x)= \frac{2^{2m+1}}{(h_n-2m-2)!} \left(\sum_{j=0}^{m} 2^{2j} (m+j)! (h_n-m-j-1)! \binom{m+j}{2j}(s-x_0)^{2j} \mathcal{Q}_{c,s}^{-m-j-1}(x) \right. \nonumber\\
		&& \left. +(s-\overline{x})\sum_{j=0}^{m-1} 2^{2j+1} (m+j+1)! (h_n-m-j-2)! \binom{m+j+1}{2j+1} (s-x_0)^{2j+1} \mathcal{Q}_{c,s}^{-m-j-2}(x)  \right)\nonumber\\
		&&+ 4^{2m+1} (2m+1)! (s-\overline{x}) (s-x_0)^{2m+1} \mathcal{Q}_{c,s}^{-2m-2}(x).
	\end{eqnarray}
	If $\beta=2m$ we have
	\begin{eqnarray}\label{f21}
		\overline{D}^{\beta} S^{-1}_L(s,x)&=& \frac{2^{2m}}{(h_n-2m-1)!} \left(\sum_{j=0}^{m-1} 2^{2j+1} (m+j)! (h_n-m-j-1)! \binom{m+j}{2j+1}\right.\nonumber\\
		&&\quad\quad\quad\quad\quad\quad\quad\quad\quad\quad\quad\quad\quad\quad\quad\quad\quad\times(s-x_0)^{2j+1} \mathcal{Q}_{c,s}^{-m-j-1}(x)\nonumber\\
		&& \left. +(s-\overline{x})\sum_{j=0}^{m-2} 2^{2j} (m+j)!(h_n-m-j-1)! \binom{m+j}{2j} (s-x_0)^{2j} \mathcal{Q}_{c,s}^{-m-j-1}(x)  \right) \nonumber\\
		&&+ 4^{2m} (2m)! (s-\overline{x}) (s-x_0)^{2m} \mathcal{Q}_{c,s}^{-2m-1}(x),
	\end{eqnarray}
\end{proposition}
\begin{proof}
	We prove the result by induction over $m\in\mathbb N$. We prove first $m=0$. By Lemma \ref{leibniz1}, we have
	\[
	\overline{D} S^{-1}_L(s,x)=\overline{D} \left( (s-\bar x)\qcs^{-1}\right)= 2(h_n-1) \mathcal{Q}_{c,s}^{-1}(x)+4 (s-\overline{x})(s-x_0)\mathcal{Q}_{c,s}^{-2}(x),
	\]
	which is formula \eqref{dbars1} when $m=0$ if we understood the sum $\sum_{k=0}^{-1}$ to be zero. We have to prove the inductive step. Let us suppose formula \eqref{dbars1} holds to be true for $m\in\mathbb N$. We have to prove the formula \eqref{dbars1} for $m$ replaced by $m+1$. By the inductive hypothesis Lemma \ref{leibniz1}, and Proposition \ref{Dirac} we get:
	\[
	\begin{split}
		& \overline{D}^{\beta+1}\left (S^{-1}_L(s,x) \right)=  \frac{2^{2m+1}}{(h_n-2m-2)!} \Bigg(\sum_{j=0}^{m} 2^{2j} (m+j)! (h_n-m-j-1)! \binom{m+j}{2j}
\\
&
\quad\quad\quad\quad\quad\quad\quad\quad\quad\quad\quad\quad\quad\quad\quad \times\overline D\left( (s-x_0)^{2j} \mathcal{Q}_{c,s}^{-m-j-1}(x) \right)
\\
		& +\sum_{j=0}^{m-1} 2^{2j+1} (m+j+1)!(h_n-m-j-2)! \binom{m+j+1}{2j+1} \overline{D}\left((s-\overline{x})(s-x_0)^{2j+1} \mathcal{Q}_{c,s}^{-m-j-2}(x)  \right)\Bigg)\\
		& + 4^{2m+1} (2m+1)! \overline{D}\Bigg((s-\overline{x}) (s-x_0)^{2m+1} \mathcal{Q}_{c,s}^{-2m-2}(x) \Bigg)
	\end{split}
	\]
and by Proposition \ref{Dirac} and Lemma \ref{leibniz1} we get

\begingroup
\allowdisplaybreaks
	\begin{align}
			&\overline{D}^{\beta+1}\left (S^{-1}_L(s,x) \right)=\frac{2^{2m+1}}{(h_n-2m-2)!} \Bigg(\sum_{j=1}^{m} 2^{2j} (m+j)! (-2j) (h_n-m-j-1)! \binom{m+j}{2j} \label{sum1}\\
			&\quad\quad\quad\quad\quad\quad\quad\quad\quad\quad\quad\quad\quad\quad\quad\quad\quad\quad\quad\quad\quad\quad\quad\quad\quad\quad\quad\quad\times (s-x_0)^{2j-1} \mathcal{Q}_{c,s}^{-m-j-1}(x) \nonumber\\
			&+(s-\bar x)\sum_{j=0}^{m} 2^{2j+1} (m+j+1)! (h_n-m-j-1)! \binom{m+j}{2j} (s-x_0)^{2j} \mathcal{Q}_{c,s}^{-m-j-2}(x)\nonumber \\
			 &+(s-\overline{x})\sum_{j=0}^{m-1} 2^{2j+1} (-2j-1) (m+j+1)!(h_n-m-j-2)! \binom{m+j+1}{2j+1} (s-x_0)^{2j} \mathcal{Q}_{c,s}^{-m-j-2}(x)\nonumber\\
			 &+\sum_{j=0}^{m-1} 2^{2j+2} (m+j+1)!(h_n-m-j-2)! (h_n-m-j-2) \binom{m+j+1}{2j+1} (s-x_0)^{2j+1} \mathcal{Q}_{c,s}^{-m-j-2}(x)\nonumber\\
			&+(s-\bar x)\sum_{j=0}^{m-1} 2^{2j+3}  (m+j+2)!(h_n-m-j-2)! \binom{m+j+1}{2j+1} (s-x_0)^{2j+2} \mathcal{Q}_{c,s}^{-m-j-3}(x) \Bigg)\nonumber\\
			 &+ 4^{2m+1} (-2m-1)(2m+1)! (s-\overline{x}) (s-x_0)^{2m} \mathcal{Q}_{c,s}^{-2m-2}(x)\nonumber\\
			 &+ 4^{2m+1} 2(h_n-2m-2)(2m+1)! (s-x_0)^{2m+1} \mathcal{Q}_{c,s}^{-2m-2}(x)\nonumber\\
			 &+ 4^{2m+2} (2m+2)! (s-\bar x)(s-x_0)^{2m+2} \mathcal{Q}_{c,s}^{-2m-3}(x)\nonumber.
		\end{align}
	\endgroup

	Now we split the computations in two parts. In the first part we make the summation of the terms without the factor $(s-\bar x)$, and in the second part we make the summation of the terms with $(s-\bar x)$.
	\newline
	\newline
	\emph{First step A}: summation of all the terms without $(s-\bar x)$. The terms we want to sum are
	\[
	\begin{split}
		&\frac{2^{2m+1}}{(h-2m-2)!} \Bigg(\sum_{j=1}^{m} 2^{2j} (m+j)! (-2j) (h_n-m-j-1)! \binom{m+j}{2j} (s-x_0)^{2j-1} \mathcal{Q}_{c,s}^{-m-j-1}(x) \\
		& +\sum_{j=0}^{m-1} 2^{2j+2} (m+j+1)!(h_n-m-j-2)! (h_n-m-j-2) \binom{m+j+1}{2j+1} (s-x_0)^{2j+1} \mathcal{Q}_{c,s}^{-m-j-2}(x)\\
		& + 2^{2m+2} (h_n-2m-2)(h_n-2m-2)!(2m+1)! (s-x_0)^{2m+1} \mathcal{Q}_{c,s}^{-2m-2}(x)\Bigg).
	\end{split}
	\]
	In the second summation we make a change of index: $j'=j+1$, we have

\begingroup
\allowdisplaybreaks
	\begin{align}
		&\frac{2^{2m+1}}{(h-2m-2)!} \Bigg(\sum_{j=1}^{m} 2^{2j} (m+j)! (-2j) (h_n-m-j-1)! \binom{m+j}{2j} (s-x_0)^{2j-1} \mathcal{Q}_{c,s}^{-m-j-1}(x) \nonumber\\
		& +\sum_{j'=1}^{m} 2^{2j'} (m+j')!(h_n-m-j'-1)! (h_n-m-j'-1) \binom{m+j'}{2j'-1} (s-x_0)^{2j'-1} \mathcal{Q}_{c,s}^{-m-j'-1}(x)\nonumber\\
		& + 2^{2m+2} (h_n-2m-2)(h_n-2m-2)!(2m+1)! (s-x_0)^{2m+1} \mathcal{Q}_{c,s}^{-2m-2}(x)\Bigg)\nonumber\\
		&=\frac{2^{2m+1}}{(h_n-2m-2)!} \Bigg(\sum_{j=1}^{m} 2^{2j} (m+j)! (h_n-m-j-1)! \Bigg[ (-2j)\binom{m+j}{2j}\nonumber \\
		&\quad\quad\quad\quad\quad\quad\quad\quad\quad\quad+(h_n-m-j-1)\binom{m+j}{2j-1}\Bigg](s-x_0)^{2j-1} \mathcal{Q}_{c,s}^{-m-j-1}(x)\nonumber \\
		& + 2^{2m+2} (h_n-2m-2)(h_n-2m-2)!(2m+1)! (s-x_0)^{2m+1} \mathcal{Q}_{c,s}^{-2m-2}(x)\Bigg).\tag{A1}\nonumber
	\end{align}
	\endgroup
	Since $(-2j)\binom{m+j}{2j} +(h_n-m-j-1)\binom{m+j}{2j-1}=\binom{m+j}{2j-1}(h_n-2m-2)$, we have
	\begin{equation}\label{s6}
		\begin{split}
			& A_1=\frac{2^{2m+1}}{(h_n-2m-3)!} \Bigg(\sum_{j=1}^{m+1} 2^{2j} (m+j)! (h_n-m-j-1)! \binom{m+j}{2j-1}(s-x_0)^{2j-1} \mathcal{Q}_{c,s}^{-m-j-1}(x)\Bigg).
		\end{split}
	\end{equation}
	\emph{Second step A}. Now we make the summation of the terms in \eqref{sum1} with the factor $(s-\bar x)$. We note that in this summation we do not consider the last term in \eqref{sum1} since it appears in the correct form without any needs to be manipulated. The terms we want to sum are
	\begingroup\allowdisplaybreaks
	\[
	\begin{split}
		&\frac{2^{2m+1} (s-\bar x)}{(h-2m-2)!} \Bigg( \sum_{j=0}^{m} 2^{2j+1} (m+j+1)! (h_n-m-j-1)! \binom{m+j}{2j} (s-x_0)^{2j} \mathcal{Q}_{c,s}^{-m-j-2}(x) \\
		&+\sum_{j=0}^{m-1} 2^{2j+1} (-2j-1) (m+j+1)!(h_n-m-j-2)! \binom{m+j+1}{2j+1} (s-x_0)^{2j} \mathcal{Q}_{c,s}^{-m-j-2}(x)\\
		& +\sum_{j=0}^{m-1} 2^{2j+3}  (m+j+2)!(h_n-m-j-2)! \binom{m+j+1}{2j+1} (s-x_0)^{2j+2} \mathcal{Q}_{c,s}^{-m-j-3}(x)\\
		& + 2^{2m+1} (h_n-2m-2)! (-2m-1)(2m+1)!  (s-x_0)^{2m} \mathcal{Q}_{c,s}^{-2m-2}(x)\Bigg).
	\end{split}
	\]
	\endgroup
	Now we make a change of index in the third series: $j'=j+1$, and we get
	
	\begin{align}
		&\frac{2^{2m+1}(s-\bar x)}{(h_n-2m-2)!} \Bigg( \sum_{j=0}^{m} 2^{2j+1} (m+j+1)! (h_n-m-j-1)! \binom{m+j}{2j}(s-x_0)^{2j} \mathcal{Q}_{c,s}^{-m-j-2}(x)\nonumber \\
		&+\sum_{j=0}^{m-1} 2^{2j+1} (-2j-1) (m+j+1)!(h_n-m-j-2)! \binom{m+j+1}{2j+1} (s-x_0)^{2j} \mathcal{Q}_{c,s}^{-m-j-2}(x)\nonumber\\
		& +\sum_{j'=1}^{m} 2^{2j'+1}  (m+j'+1)!(h_n-m-j'-1)! \binom{m+j'}{2j'-1} (s-x_0)^{2j'} \mathcal{Q}_{c,s}^{-m-j'-2}(x)\nonumber\\
		& + 2^{2m+1} (h_n-2m-2)! (-2m-1)(2m+1)!  (s-x_0)^{2m} \mathcal{Q}_{c,s}^{-2m-2}(x)\Bigg)\tag{A2}.
	\end{align}
	
	Now we make the summation of the terms $j=0$ in formula $(A2)$:
		\begingroup\allowdisplaybreaks
	\begin{equation}\label{sum2}
		\begin{split}
			&\frac{2^{2m+1} (s-\bar x)}{(h_n-2m-2)!} \Bigg( 2 (m+1)! (h_n-m-1)!   \mathcal{Q}_{c,s}^{-m-2}(x) -2 (m+1)!(h_n-m-2)! (m+1)  \mathcal{Q}_{c,s}^{-m-2}(x)\Bigg)\\
			&=\frac{2^{2m+1}(s-\bar x)}{(h_n-2m-3)!} \Bigg(2(h_n-m-2)!(m+1)! \mathcal{Q}_{c,s}^{-m-2}(x) \Bigg).
		\end{split}
	\end{equation}
	\endgroup
	Now we make the summation of the terms $j=m$ plus the last term in formula $(A_2)$:

	\begingroup\allowdisplaybreaks
		\begin{align}
			& \frac{2^{2m+1}(s-\bar x)}{(h_n-2m-2)!} \Bigg( 2^{2m+1} (2m+1)! (h_n-2m-1)! (s-x_0)^{2m} \mathcal{Q}_{c,s}^{-2m-2}(x)\label{sum4} \\
			&+2^{2m+1}  (2m+1)!(h_n-2m-1)! 2m (s-x_0)^{2m} \mathcal{Q}_{c,s}^{-2m-2}(x)\nonumber\\
			& + 2^{2m+1} (h_n-2m-2)! (-2m-1)(2m+1)!  (s-x_0)^{2m} \mathcal{Q}_{c,s}^{-2m-2}(x)\Bigg)\nonumber \\
			&=\frac{2^{2m+1}(s-\bar x)}{(h_n-2m-2)!}\Bigg( 2^{2m+1} (2m+1)! (h_n-2m-2)! (s-\bar x)(s-x_0)^{2m} \mathcal{Q}_{c,s}^{-2m-2}(x)\nonumber \\
			&\times \left( h_n-2m-1+2m(h_n-2m-1)-2m-1 \right) \Bigg)\nonumber\\
			&=\frac{2^{2m+1} (s-\bar x)}{(h_n-2m-3)!}\Bigg( 2^{2m+1} (2m+1) (2m+1)! (h_n-2m-2)! (s-x_0)^{2m} \mathcal{Q}_{c,s}^{-2m-2}(x) \Bigg).\nonumber
		\end{align}
	\endgroup
	Now we make the summation of the terms in formula $(A2)$ with $j=1,\dots, m-1$:
	\begingroup\allowdisplaybreaks
	\begin{equation}\label{sum3}
		\begin{split}
			&\frac{2^{2m+1} (s-\bar x)}{(h-2m-2)!} \Bigg( \sum_{j=1}^{m-1} 2^{2j+1} (m+j+1)! (h_n-m-j-1)! \binom{m+j}{2j} (s-x_0)^{2j} \mathcal{Q}_{c,s}^{-m-j-2}(x) \\
			&+\sum_{j=1}^{m-1} 2^{2j+1} (-2j-1) (m+j+1)!(h_n-m-j-2)! \binom{m+j+1}{2j+1} (s-x_0)^{2j} \mathcal{Q}_{c,s}^{-m-j-2}(x)\\
			& +\sum_{j=1}^{m-1} 2^{2j+1}  (m+j+1)!(h_n-m-j-1)! \binom{m+j}{2j-1} (s-x_0)^{2j} \mathcal{Q}_{c,s}^{-m-j-2}(x)\Bigg)\\
			&=\frac{2^{2m+1}(s-\bar x)}{(h-2m-2)!} \Bigg( \sum_{j=1}^{m-1} 2^{2j+1} (m+j+1)! (s-x_0)^{2j}\mathcal{Q}_{c,s}^{-m-j-2}(x)(h_n-m-j-2)!\\
			&\quad\quad\times\left( \binom{m+j}{2j}(h_n-m-j-1)-(2j+1)\binom{m+j+1}{2j+1}+\binom{m+j}{2j-1}(h_n-m-j-1) \right)\Bigg)\\
			&=\frac{2^{2m+1}(s-\bar x)}{(h-2m-3)!} \Bigg( \sum_{j=1}^{m-1} 2^{2j+1} (m+j+1)! (h_n-m-j-2)!\binom{m+j+1}{2j}\\
			&\quad\quad\quad\quad\quad\quad\quad\quad\quad\quad\quad\quad\quad\quad\quad\quad\quad\quad\quad\quad\quad\quad\quad\times (s-x_0)^{2j}\mathcal{Q}_{c,s}^{-m-j-2}(x)\Bigg),
		\end{split}
	\end{equation}
	\endgroup
	where in the last equation we used the fact that
	\vspace{-5mm}
	\begingroup\allowdisplaybreaks
		\begin{equation*}
 \binom{m+j}{2j}(h_n-m-j-1)-(2j+1)\binom{m+j+1}{2j+1}+\binom{m+j}{2j-1}(h_n-m-j-1)= \binom{m+j+1}{2j} (h_n-2m-2).
		\end{equation*}
\endgroup
	Summing up \eqref{sum2}, \eqref{sum4}  and \eqref{sum3}, we obtain
		\vspace{-3mm}
	\begingroup\allowdisplaybreaks
	\begin{equation}
		\label{s5}
	A2	=\frac{2^{2m+2}}{(h_n-2m-3)!}\Bigg( \sum_{j=0}^{m} 2^{2j} (m+j+1)! (h_n-m-j-2)!\binom{m+j+1}{2j}(s-\bar x) (s-x_0)^{2j}\mathcal{Q}_{c,s}^{-m-j-2}(x)\Bigg).
	\end{equation}
	\endgroup
	Now, we make the sum of \eqref{s6}, \eqref{s5} and the last term in \eqref{sum1}, we obtain
	\begingroup\allowdisplaybreaks
\begin{align}
			& \overline{D}^{\beta+1}\left (S^{-1}_L(s,x) \right)=\frac{2^{2m+2}}{(h_n-2m-3)!} \Bigg(\sum_{j=0}^{m} 2^{2j+1} (m+j+1)! (h_n-m-j-2)! \binom{m+j+1}{2j+1}\label{s7}\\
			&\quad\quad\quad\quad\quad\quad\quad\quad\quad\quad\quad\quad\quad\quad\quad\quad\quad\quad\quad\quad\quad\quad\quad\quad\quad\quad\quad\quad\times(s-x_0)^{2j+1} \mathcal{Q}_{c,s}^{-m-j-2}(x)\nonumber\\
			&+\sum_{j=0}^{m} 2^{2j} (m+j+1)! (h_n-m-j-2)!\binom{m+j+1}{2j} (s-\bar x) (s-x_0)^{2j}\mathcal{Q}_{c,s}^{-m-j-2}(x)\Bigg)\nonumber\\
			& + 4^{2m+2} (2m+2)! (s-\bar x)(s-x_0)^{2m+2} \mathcal{Q}_{c,s}^{-2m-3}(x)\nonumber.
		\end{align}
		\endgroup
	Now we have to compute $\overline D^{\beta+2}(S^{-1}_L(s,x))$. By Lemma \ref{leibniz1} and Proposition \ref{Dirac}, we get
	\begingroup\allowdisplaybreaks
	\begin{equation}\label{s8}
		\begin{split}
			&\overline D^{\beta+2}(S^{-1}_L(s,x))=\frac{2^{2m+2}}{(h_n-2m-3)!} \Bigg(\sum_{j=0}^{m} 2^{2j+1} (m+j+1)! (h_n-m-j-2)! \binom{m+j+1}{2j+1}\\
			&\quad\quad\quad\quad\quad\quad\quad\quad\quad\quad\quad\quad\quad\quad\quad\quad\quad\quad\quad\quad\quad\quad\quad\quad\quad\quad\times\overline{D} \left((s-x_0)^{2j+1} \mathcal{Q}_{c,s}^{-m-j-2}(x)\right) \\
			&+\sum_{j=0}^{m} 2^{2j} (m+j+1)! (h_n-m-j-2)!\binom{m+j+1}{2j} \overline D\left( (s-\bar x) (s-x_0)^{2j}\mathcal{Q}_{c,s}^{-m-j-2}(x) \right)\Bigg)\\
			& + 4^{2m+2} (2m+2)!\overline D\left( (s-\bar x)(s-x_0)^{2m+2} \mathcal{Q}_{c,s}^{-2m-3}(x)\right) \\
			&=\frac{2^{2m+2}}{(h_n-2m-3)!} \Bigg(\sum_{j=0}^{m} 2^{2j+1} (m+j+1)! (h_n-m-j-2)! \binom{m+j+1}{2j+1}\\
			&\quad\quad\quad\quad\quad\quad\quad\quad\quad\quad\quad\quad\quad\quad\quad\quad\quad\quad\quad\quad\quad\quad\quad\quad\quad\quad\times (-2j-1) (s-x_0)^{2j} \mathcal{Q}_{c,s}^{-m-j-2}(x) \\
			&+(s-\bar x)\sum_{j=0}^{m} 2^{2j+2} (m+j+1)! (h_n-m-j-2)! \binom{m+j+1}{2j+1} (m+j+2) (s-x_0)^{2j+1} \mathcal{Q}_{c,s}^{-m-j-3}(x) \\
			&+(s-\bar x)\sum_{j=1}^{m} 2^{2j} (m+j+1)! (h_n-m-j-2)!(-2j)\binom{m+j+1}{2j} (s-x_0)^{2j-1}\mathcal{Q}_{c,s}^{-m-j-2}(x)\\
			&+\sum_{j=0}^{m} 2^{2j+1} (m+j+1)! (h_n-m-j-2)!(h_n-m-j-2)\binom{m+j+1}{2j} (s-x_0)^{2j}\mathcal{Q}_{c,s}^{-m-j-2}(x)\\
			&+(s-\bar x)\sum_{j=0}^{m} 2^{2j+2} (m+j+2)! (h_n-m-j-2)!\binom{m+j+1}{2j} (s-x_0)^{2j+1}\mathcal{Q}_{c,s}^{-m-j-3}(x)\Bigg)\\
			& + 4^{2m+2} (-2m-2) (2m+2)!  (s-\bar x)(s-x_0)^{2m+1} \mathcal{Q}_{c,s}^{-2m-3}(x)  \\
			& + 4^{2m+2} 2(h_n-2m-3) (2m+2)!  (s-x_0)^{2m+2} \mathcal{Q}_{c,s}^{-2m-3}(x)  \\
			& + 4^{2m+3} (2m+3)!  (s-\bar x)(s-x_0)^{2m+3} \mathcal{Q}_{c,s}^{-2m-4}(x).
		\end{split}
	\end{equation}
	\endgroup
	Now we split the computations in two parts. In the first part we make the summation of the terms without the factor $(s-\bar x)$, and in the second part we make the summation of the terms with $(s-\bar x)$.
	\newline
	\newline
	\emph{First step B}: summation of all the terms without $(s-\bar x)$. The terms we want to sum are
	\begingroup\allowdisplaybreaks
	\begin{eqnarray}
		\nonumber
		&&\frac{2^{2m+2}}{(h-2m-3)!} \Bigg(\sum_{j=0}^{m} 2^{2j+1} (m+j+1)! (h_n-m-j-2)! \\
\nonumber
		&&
\times \binom{m+j+1}{2j+1}(-2j-1) (s-x_0)^{2j} \mathcal{Q}_{c,s}^{-m-j-2}(x) \\
		\nonumber
		&&+\sum_{j=0}^{m} 2^{2j+1} (m+j+1)! (h_n-m-j-2)!(h_n-m-j-2)\binom{m+j+1}{2j} (s-x_0)^{2j}\mathcal{Q}_{c,s}^{-m-j-2}(x)\\
		\nonumber
		&& + 2^{2m+2} 2(h_n-2m-3)(h_n-2m-3)! (2m+2)!  (s-x_0)^{2m+2} \mathcal{Q}_{c,s}^{-2m-3}(x) \Bigg)\\
		\nonumber
		&& = \frac{2^{2m+2}}{(h_n-2m-3)!}\Bigg(\sum_{j=0}^m 2^{2j+1}(m+j+1)! (h_n-m-j-2)! (s-x_0)^{2j}\mathcal{Q}_{c,s}^{-m-j-2}(x)\\
		\nonumber
		&&\times \left( -(2j+1)\binom{m+j+1}{2j+1}+(h_n-m-j-2) \binom{m+j+1}{2j} \right) \Bigg) \\
		\nonumber
		&&+ 2^{2m+2} 2(h_n-2m-3)(h_n-2m-3)! (2m+2)!  (s-x_0)^{2m+2} \mathcal{Q}_{c,s}^{-2m-3}(x) \Bigg)\\
		\nonumber
		&& = \frac{2^{2m+3}}{(h_n-2m-4)!}\Bigg(\sum_{j=0}^{m+1} 2^{2j}(m+j+1)! (h_n-m-j-2)!
\\
\label{s14}
&&
\times (s-x_0)^{2j}\mathcal{Q}_{c,s}^{-m-j-2}(x) \binom{m+j+1}{2j} \Bigg).
	\end{eqnarray}
	\endgroup
	where we used the fact that:
		\begingroup\allowdisplaybreaks
\begin{equation*}
 \left( -(2j+1)\binom{m+j+1}{2j+1}+(h_n-m-j-2) \binom{m+j+1}{2j} \right)=(h_n-2m-3)\binom{m+j+1}{2j}.
\end{equation*}
\endgroup
	\emph{Second step B}. Now we make the summation of the terms in \eqref{s8} with the factor $(s-\bar x)$. We note that in this summation we do not consider the last term since it appears in the right form without any needs to be manipulated.  The terms we want to sum are
	\begingroup\allowdisplaybreaks
	\[
	\begin{split}
		&\frac{2^{2m+2} (s-\bar x)}{(h_n-2m-3)!}\Bigg(\sum_{j=0}^{m} 2^{2j+2} (m+j+1)! (h_n-m-j-2)! \binom{m+j+1}{2j+1} (m+j+2)\\
		& \quad\quad\quad\quad\quad\quad\quad\quad\quad\quad\quad\quad\quad\quad\quad\quad\quad\quad\quad\quad\quad\quad\quad\times (s-x_0)^{2j+1} \mathcal{Q}_{c,s}^{-m-j-3}(x) \\
		&+\sum_{j=1}^{m} 2^{2j} (m+j+1)! (h_n-m-j-2)!(-2j)\binom{m+j+1}{2j}  (s-x_0)^{2j-1}\mathcal{Q}_{c,s}^{-m-j-2}(x)\\
		&+\sum_{j=0}^{m} 2^{2j+2} (m+j+2)! (h_n-m-j-2)!\binom{m+j+1}{2j} (s-x_0)^{2j+1}\mathcal{Q}_{c,s}^{-m-j-3}(x)\\
		& + 2^{2m+2} (-2m-2) (2m+2)! (h_n-2m-3)!  (s-x_0)^{2m+1} \mathcal{Q}_{c,s}^{-2m-3}(x) \Bigg)
	\end{split}
	\]
	\endgroup
	In the second serie we make the change of index $j'=j-1$, we have
	\begingroup\allowdisplaybreaks	
	\begin{align}
		&\frac{2^{2m+2} (s-\bar x)}{(h_n-2m-3)!}\Bigg(\sum_{j=0}^{m} 2^{2j+2} (m+j+2)! (h_n-m-j-2)! \binom{m+j+1}{2j+1}\nonumber\\
		&\quad\quad\quad\quad\quad\quad\quad\quad\quad\quad\quad\quad\quad\quad\quad\quad\quad\quad\quad\quad\quad\quad\quad\quad\quad\quad\times (s-x_0)^{2j+1} \mathcal{Q}_{c,s}^{-m-j-3}(x) \nonumber\\
		&+\sum_{j'=0}^{m-1} 2^{2j'+2} (m+j'+2)! (h_n-m-j'-3)!(-2j'-2)\binom{m+j'+2}{2j'+2} \nonumber\\
		& \quad\quad\quad \quad\quad\quad\quad\quad\quad\quad\quad\quad\quad\quad\quad\quad\quad\quad\quad\quad\quad\quad\quad\quad\quad\quad\quad\times  (s-x_0)^{2j'+1}\mathcal{Q}_{c,s}^{-m-j-3}(x)\nonumber\\
		&+\sum_{j=0}^{m} 2^{2j+2} (m+j+2)! (h_n-m-j-2)! \binom{m+j+1}{2j} (s-x_0)^{2j+1}\mathcal{Q}_{c,s}^{-m-j-3}(x)\nonumber\\
		& + 2^{2m+2} (-2m-2) (2m+2)! (h_n-2m-3)!  (s-x_0)^{2m+1} \mathcal{Q}_{c,s}^{-2m-3}(x) \Bigg). \tag{B2}\nonumber
	\end{align}
	\endgroup	
	In formula $(B2)$, we make the summation of the term with $j=m$ and the last term. So we have
	\begingroup\allowdisplaybreaks
	\begin{eqnarray}
		\nonumber
		&&\frac{2^{2m+2} (s-\bar x)}{(h_n-2m-3)!}\Bigg(2^{2m+2} (2m+2)! (h_n-2m-2)! (s-x_0)^{2m+1} \mathcal{Q}_{c,s}^{-2m-3}(x) \\
		\nonumber
		&&+2^{2m+2} (2m+2)! (h_n-2m-2)! (2m+1) (s-x_0)^{2m+1}\mathcal{Q}_{c,s}^{-2m-3}(x)\\
		\nonumber
		&&+ 2^{2m+2} (-2m-2) (2m+2)! (h_n-2m-3)!  (s-x_0)^{2m+1} \mathcal{Q}_{c,s}^{-2m-3}(x) \Bigg) \\
		\nonumber
		&&=\frac{2^{2m+2} (s-\bar x)}{(h_n-2m-3)!}\Bigg(2^{2m+2} (2m+2)! (h_n-2m-3)!\\
		\nonumber
		&&\times \left[ (h_n-2m-2)+(2m+1)(h_n-2m-2)-2m-2 \right] (s-x_0)^{2m+1} \mathcal{Q}_{c,s}^{-2m-3}(x)\Bigg) \\
		\label{s12}
		&&=\frac{2^{2m+2} (s-\bar x)}{(h_n-2m-4)!}\Bigg(2^{2m+2} (2m+2)! (2m+2) (h_n-2m-3)!
 \\
		\nonumber
		&&\times (s-x_0)^{2m+1} \mathcal{Q}_{c,s}^{-2m-3}(x)\Bigg).
	\end{eqnarray}
	\endgroup
	Now in $(B2)$ we make the summation of the terms with $j=0,\dots, m-1$, we have

\begingroup\allowdisplaybreaks
	\begin{eqnarray}
		\nonumber
			&&\frac{2^{2m+2} (s-\bar x)}{(h_n-2m-3)!}\Bigg(\sum_{j=0}^{m-1} 2^{2j+2} (m+j+2)! (h_n-m-j-2)! \binom{m+j+1}{2j+1}\\
\label{s11}
			&&\quad\quad\quad\quad\quad\quad\quad\quad\quad\quad\quad\quad\quad\quad\quad\quad\quad\quad\quad\quad\quad\quad\quad\quad\quad\quad\times (s-x_0)^{2j+1} \mathcal{Q}_{c,s}^{-m-j-3}(x) \\
\nonumber
			&&+\sum_{j'=0}^{m-1} 2^{2j'+2} (m+j'+2)! (h_n-m-j'-3)!(-2j'-2)\binom{m+j'+2}{2j'+2} (s-x_0)^{2j'+1}\mathcal{Q}_{c,s}^{-m-j-3}(x)\\
\nonumber
			&&+\sum_{j=0}^{m-1} 2^{2j+2} (m+j+2)! (h_n-m-j-2)! \binom{m+j+1}{2j} (s-x_0)^{2j+1}\mathcal{Q}_{c,s}^{-m-j-3}(x)\Bigg)\\
\nonumber
			&&=\frac{2^{2m+2} (s-\bar x)}{(h_n-2m-3)!}\Bigg(\sum_{j=0}^{m-1} 2^{2j+2} (m+j+2)! (h-m-j-3)! (s-x_0)^{2j+1}\mathcal{Q}_{c,s}^{-m-j-3}(x)\\
\nonumber
			&&\times \Bigg[ (h-m-j-2) \binom{m+j+1}{2j+1}-(2j+2)\binom{m+j+2}{2j+2} +(h_n-m-j-2) \binom{m+j+1}{2j}\Bigg]\Bigg)\\
\nonumber
			&&=\frac{2^{2m+2} (s-\bar x)}{(h_n-2m-4)!}\Bigg(\sum_{j=0}^{m-1} 2^{2j+2} (m+j+2)! (h_n-m-j-3)!\binom{m+j+2}{2j+1}
\\
\nonumber
			&&
\times(s-x_0)^{2j+1}\mathcal{Q}_{c,s}^{-m-j-3}(x)\Bigg),\\
\nonumber
		\end{eqnarray}
	\endgroup
	where in the last equation we used the fact that
	$$
(h_n-m-j-2) \binom{m+j+1}{2j+1}-(2j+2)\binom{m+j+2}{2j+2}
$$
$$
+(h_n-m-j-2) \binom{m+j+1}{2j}=\binom{m+j+2}{2j+1}(h_n-2m-3).$$
	Now we put together  \eqref{s12} and  \eqref{s11}, we get
	\begin{equation}\label{s13}
		\begin{split}
			&B2 =\frac{2^{2m+2} (s-\bar x)}{(h_n-2m-4)!}\Bigg(\sum_{j=0}^{m} 2^{2j+2} (m+j+2)! (h_n-m-j-3)!\binom{m+j+2}{2j+1} \\
&
\times(s-x_0)^{2j+1}\mathcal{Q}_{c,s}^{-m-j-3}(x)\Bigg).
		\end{split}
	\end{equation}
	Finally we sum \eqref{s14}, \eqref{s13} and the last term in \eqref{s8}, we have
	\begin{equation}\label{lastsum}
		\begin{split}
			&\overline D^{\beta+2}\left( S^{-1}_L(s,x)\right)\\
			&= \frac{2^{2m+3}}{(h_n-2m-4)!}\Bigg(\sum_{j=0}^{m+1} 2^{2j}(m+j+1)! (h_n-m-j-2)! (s-x_0)^{2j}\mathcal{Q}_{c,s}^{-m-j-2}(x) \binom{m+j+1}{2j} \\
			&+\sum_{j=0}^{m} 2^{2j+1} (m+j+2)! (h_n-m-j-3)!\binom{m+j+2}{2j+1} (s-\bar x)(s-x_0)^{2j+1}\mathcal{Q}_{c,s}^{-m-j-3}(x)\Bigg)\\
			&+ 4^{2m+3} (2m+3)!  (s-\bar x)(s-x_0)^{2m+3} \mathcal{Q}_{c,s}^{-2m-4}(x).
		\end{split}
	\end{equation}
	This equation proves the inductive step and thus formula \eqref{dbars1} is proved.
	Finally formula \eqref{f21} follows by applying the operator $\overline D$ to formula \eqref{sum1}. These computations have been already done to prove formula \eqref{s7}.
\end{proof}

\section*{Declarations and statements}

{\bf Data availability}. The research in this paper does not imply use of data.

{\bf Conflict of interest}. The authors declare that there is no conflict of interest.

{\bf Acknowledgments}.
F. Colombo,  A. De Martino and S. Pinton are supported by MUR grant Dipartimento di Eccellenza 2023-2027.

\end{document}